\documentclass[letterpaper,11pt,oneside,reqno]{amsart}

\usepackage{enumitem}
\usepackage{bm}
\usepackage{mathrsfs}
\usepackage{amsfonts,amsmath, amssymb,amsthm,amscd,stmaryrd}

\usepackage{setspace}
\usepackage[mpexclude,DIV13]{typearea}
\usepackage{verbatim}
\usepackage{graphicx}
\usepackage[latin1]{inputenc}
\usepackage{latexsym}
\usepackage{lscape}
\usepackage{epsfig}
\include{bibtex}
\usepackage[colorlinks=true,linkcolor=blue]{hyperref}
\usepackage{parskip}
\begin{document}
\bibliographystyle{alpha}
\newcommand{\cn}[1]{\overline{#1}}
\newcommand{\e}[0]{\varepsilon}
\newcommand{\bbf}[0]{\mathbf}
\newcommand{\bX}{{\bf X}}
\newcommand{\Pfree}[5]{\ensuremath{\mathbb{P}^{#1,#2,#3,#4,#5}}}
\newcommand{\PfreeShort}{\ensuremath{\mathbb{P}^{BB}}}

%AAAAAAA
\newcommand{\WH}[8]{\ensuremath{\mathbb{W}^{#1,#2,#3,#4,#5,#6,#7}_{#8}}}
\newcommand{\Wfree}[5]{\ensuremath{\mathbb{W}^{#1,#2,#3,#4,#5}}}
\newcommand{\WHShort}[3]{\ensuremath{\mathbb{W}^{#1,#2}_{#3}}}
\newcommand{\WHShortCouple}[2]{\ensuremath{\mathbb{W}^{#1}_{#2}}}

\newcommand{\walk}[3]{\ensuremath{X^{#1,#2}_{#3}}}
\newcommand{\walkupdated}[3]{\ensuremath{\tilde{X}^{#1,#2}_{#3}}}
\newcommand{\walkfull}[2]{\ensuremath{X^{#1,#2}}}
\newcommand{\walkfullupdated}[2]{\ensuremath{\tilde{X}^{#1,#2}}}
%AAAAAAA

\newcommand{\PH}[8]{\ensuremath{\mathbb{Q}^{#1,#2,#3,#4,#5,#6,#7}_{#8}}}
\newcommand{\PHShort}[1]{\ensuremath{\mathbb{Q}_{#1}}}
\newcommand{\PHExp}[8]{\ensuremath{\mathbb{F}^{#1,#2,#3,#4,#5,#6,#7}_{#8}}}
%%%%

\newcommand{\D}[8]{\ensuremath{D^{#1,#2,#3,#4,#5,#6,#7}_{#8}}}
\newcommand{\DShort}[1]{\ensuremath{D_{#1}}}
\newcommand{\partfunc}[8]{\ensuremath{Z^{#1,#2,#3,#4,#5,#6,#7}_{#8}}}
\newcommand{\partfuncShort}[1]{\ensuremath{Z_{#1}}}
\newcommand{\bolt}[8]{\ensuremath{W^{#1,#2,#3,#4,#5,#6,#7}_{#8}}}
\newcommand{\boltShort}[1]{\ensuremath{W_{#1}}}
\newcommand{\boltNew}{\ensuremath{W}}
\newcommand{\QTLH}{\ensuremath{\mathfrak{H}}}
\newcommand{\QTLHgen}{\ensuremath{\mathfrak{L}}}

\newcommand{\whitenoise}{\ensuremath{\mathscr{\dot{W}}}}
\newcommand{\mf}{\mathfrak}
\newcommand{\di}{{\rm{disc}}}

\newcommand{\EE}{\ensuremath{\mathbb{E}}}
\newcommand{\PP}{\ensuremath{\mathbb{P}}}
\newcommand{\var}{\textrm{var}}
\newcommand{\N}{\ensuremath{\mathbb{N}}}
\newcommand{\R}{\ensuremath{\mathbb{R}}}
\newcommand{\C}{\ensuremath{\mathbb{C}}}
\newcommand{\Z}{\ensuremath{\mathbb{Z}}}
\newcommand{\Q}{\ensuremath{\mathbb{Q}}}
\newcommand{\T}{\ensuremath{\mathbb{T}}}
\newcommand{\E}[0]{\mathbb{E}}
\newcommand{\OO}[0]{\Omega}
\newcommand{\F}[0]{\mathfrak{F}}
\def \Ai {{\rm Ai}}
\newcommand{\G}[0]{\mathfrak{G}}
\newcommand{\ta}[0]{\theta}
\newcommand{\w}[0]{\omega}
\newcommand{\grad}{\nabla}
\renewcommand{\P}{\mathbb{P}}
\newcommand{\ra}[0]{\rightarrow}
\newcommand{\vectoro}{\overline}
\newcommand{\crairy}{\mathcal{CA}}
\newcommand{\nc}{\mathsf{NoTouch}}
\newcommand{\ncf}{\mathsf{NoTouch}^f}
\newcommand{\wxy}{\mathcal{W}_{k;\bar{x},\bar{y}}}
\newcommand{\AP}{\mathfrak{a}}
\newcommand{\cm}{\mathfrak{c}}
\newcommand{\bz}{\mathbf{z}}
\newtheorem{theorem}{Theorem}[section]
\newtheorem{partialtheorem}{Partial Theorem}[section]
\newtheorem{conj}[theorem]{Conjecture}
\newtheorem{lemma}[theorem]{Lemma}
\newtheorem{proposition}[theorem]{Proposition}
\newtheorem{corollary}[theorem]{Corollary}
\newtheorem{claim}[theorem]{Claim}
\newtheorem{experiment}[theorem]{Experimental Result}

\def\todo#1{\marginpar{\raggedright\footnotesize #1}}
\def\change#1{{\color{green}\todo{change}#1}}
\def\note#1{\textup{\textsf{\color{blue}(#1)}}}

\theoremstyle{definition}
\newtheorem{rem}[theorem]{Remark}

\theoremstyle{definition}
\newtheorem{com}[theorem]{Comment}

\theoremstyle{definition}
\newtheorem{definition}[theorem]{Definition}

\theoremstyle{definition}
\newtheorem{definitions}[theorem]{Definitions}

\theoremstyle{definition}
\newtheorem{conjecture}[theorem]{Conjecture}

\newcommand{\airysh}{\mathcal{A}}
\newcommand{\hfixed}{\mathcal{H}}
\newcommand{\afixed}{\mathcal{A}}
\newcommand{\canopynoarg}{\mathsf{C}}
\newcommand{\canopy}[3]{\ensuremath{\mathsf{C}_{#1,#2}^{#3}}}
\newcommand{\argmax}{x_{{\rm max}}}
\newcommand{\zmax}{z_{{\rm max}}}

\newcommand{\Rkle}{\ensuremath{\mathbb{R}^k_{>}}}
\newcommand{\Ronele}{\ensuremath{\mathbb{R}^k_{>}}}
\newcommand{\ewxy}{\mathcal{E}_{k;\bar{x},\bar{y}}}

\newcommand{\bxyf}{\mathcal{B}_{\bar{x},\bar{y},f}}
\newcommand{\bxyflr}{\mathcal{B}_{\bar{x},\bar{y},f}^{\ell,r}}

\newcommand{\bxyfone}{\mathcal{B}_{x_1,y_1,f}}

\newcommand{\ptac}{p}
\newcommand{\ptact}{v}

\newcommand{\fext}{\mathfrak{F}_{{\rm ext}}}
\newcommand{\gext}{\mathfrak{G}_{{\rm ext}}}
\newcommand{\xext}{{\rm xExt}(\mathfrak{c}_+)}

\newcommand{\dd}{\, {\rm d}}
\newcommand{\signc}{\Sigma}
\newcommand{\wxylr}{\mathcal{W}_{k;\bar{x},\bar{y}}^{\ell,r}}
\newcommand{\wxylrprime}{\mathcal{W}_{k;\bar{x}',\bar{y}'}^{\ell,r}}
\newcommand{\Rklezero}{\ensuremath{\mathbb{R}^k_{>0}}}
\newcommand{\XYfM}{\textrm{XY}^{f}_M}

\newcommand{\upright}{D}
\newcommand{\staircase}{SC}
\newcommand{\energy}{E}
\newcommand{\xmax}{{\rm max}_1}
\newcommand{\ymax}{{\rm max}_2}
\newcommand{\lppls}{\mathcal{L}}
\newcommand{\lpplsre}{\mathcal{L}^{{\rm re}}}
\newcommand{\lpplsarg}[1]{\mathcal{L}_{n}^{\fa \to #1}}
\newcommand{\larg}[3]{\mathcal{L}_{n}^{#1,#2;#3}}
\newcommand{\BP}{M}
\newcommand{\weight}{\mathsf{Wgt}}
\newcommand{\pairweight}{\mathsf{PairWgt}}
\newcommand{\sumweight}{\mathsf{SumWgt}}
\newcommand{\mpgood}{\mathcal{G}}
\newcommand{\mpg}{\mathsf{Fav}}
\newcommand{\mcgone}{\mathsf{Fav}_1}
\newcommand{\radnik}[2]{\mathsf{RN}_{#1,#2}}
\newcommand{\size}[2]{\mathsf{S}_{#1,#2}}
\newcommand{\pdr}{\mathsf{PolyDevReg}}
\newcommand{\pwr}{\mathsf{PolyWgtReg}}
\newcommand{\lwr}{\mathsf{LocWgtReg}}
\newcommand{\hwp}{\mathsf{HighWgtPoly}}
\newcommand{\fbr}{\mathsf{ForBouqReg}}
\newcommand{\bbr}{\mathsf{BackBouqReg}}
\newcommand{\fsc}{\mathsf{FavSurCon}}
\newcommand{\maxpoly}{\mathrm{MaxDisjtPoly}}
\newcommand{\maxswf}{\mathsf{MaxScSumWgtFl}}
\newcommand{\emaxswf}{\e \! - \! \maxswf}
\newcommand{\minswf}{\mathsf{MinScSumWgtFl}}
\newcommand{\eminswf}{\e \! - \! \minswf}
\newcommand{\surreg}{\mathcal{R}}
\newcommand{\scf}{\mathsf{FavSurgCond}}
\newcommand{\disjtpoly}{\mathsf{DisjtPoly}}
\newcommand{\intint}[1]{\llbracket 1,#1 \rrbracket}
\newcommand{\intinta}[2]{\llbracket #1,#2 \rrbracket}
\newcommand{\maxsym}{*}
\newcommand{\polynum}{\#\mathsf{Poly}}
\newcommand{\dlp}{\mathsf{DisjtLinePoly}}
\newcommand{\lowb}{\underline{B}}
\newcommand{\highb}{\overline{B}}
\newcommand{\tottt}{s_{1,2}^{2/3}}
\newcommand{\tot}{s_{1,2}}
\newcommand{\btone}{{\bf{t}}_1}
\newcommand{\bttwo}{{\bf{t}}_2}
\newcommand{\formerE}{C}
\newcommand{\rcon}{r_0}
\newcommand{\para}{Q}
\newcommand{\Cstrong}{E}

\newcommand{\mc}{\mathcal}
\newcommand{\vect}{\mathbf}
\newcommand{\bt}{\mathbf{t}}
\newcommand{\scB}{\mathscr{B}}
\newcommand{\scBres}{\mathscr{B}^{\mathrm{re}}}
\newcommand{\rightshadow}{\mathrm{RS}Z}
\newcommand{\dbm}{D}
\newcommand{\edgedbm}{D^{\rm edge}}
\newcommand{\gue}{\mathrm{GUE}}
\newcommand{\edgegue}{\mathrm{GUE}^{\mathrm{edge}}}
\newcommand{\eqdist}{\stackrel{(d)}{=}}
\newcommand{\geqdist}{\stackrel{(d)}{\succeq}}
\newcommand{\leqdist}{\stackrel{(d)}{\preceq}}
\newcommand{\scal}{{\rm sc}}
\newcommand{\fa}{x_0}
\newcommand{\hit}{H}
\newcommand{\scaledle}{\mathsf{Nr}\mc{L}}
\newcommand{\cenleup}{\mathscr{L}^{\uparrow}}
\newcommand{\cenledown}{\mathscr{L}^{\downarrow}}
\newcommand{\eln}{T}
\newcommand{\xmin}{{\rm Corner}^{\mfl,\mc{F}}}
\newcommand{\ymin}{{\rm Corner}^{\mfr,\mc{F}}}
\newcommand{\barxmin}{\overline{\rm Corner}^{\mfl,\mc{F}}}
\newcommand{\barymin}{\overline{\rm Corner}^{\mfr,\mc{F}}}
\newcommand{\qmin}{Q^{\mc{F}^1}}
%^{{\rm min}}}
\newcommand{\barqmin}{\bar{Q}^{\mc{F}^1}}
%^{{\rm min}}}
\setcounter{tocdepth}{2}
\newcommand{\test}{T}
\newcommand{\mfl}{\mf{l}}
\newcommand{\mfr}{\mf{r}}
\newcommand{\gfl}{\ell}
\newcommand{\gfr}{r}
\newcommand{\jre}{J}
\newcommand{\highfl}{{\rm HFL}}
\newcommand{\flyleap}{\mathsf{FlyLeap}}
\newcommand{\touch}{\mathsf{Touch}}
\newcommand{\notouch}{\mathsf{NoTouch}}
\newcommand{\close}{\mathsf{Close}}
\newcommand{\abovepar}{\mathsf{High}}
\newcommand{\vecint}{\bar{\iota}}
\newcommand{\cornthree}{{\rm Corner}^\mc{G}_{k,\mfl}}
\newcommand{\cornfour}{{\rm Corner}^\mc{H}_{k,\fa}}
\newcommand{\mpgg}{\mathsf{Fav}_{\mc{G}}}

\newcommand{\lefta}{M_{1,k+1}^{[-2\eln,\gfl]}}
\newcommand{\mida}{M_{1,k+1}^{[\gfl,\gfr]}}
\newcommand{\righta}{M_{1,k+1}^{[\gfr,2\eln]}}
\newcommand{\Var}{{\textrm{Var}}}
\newcommand{\ipdval}{d}
\newcommand{\ctemp}{d_0}

\newcommand{\wien}{W}
\newcommand{\pole}{P}
\newcommand{\pp}{p}

\newcommand{\const}{D_k}
\newcommand{\numcone}{14}
\newcommand{\numctwo}{13}
\newcommand{\numcthree}{6}
\newcommand{\rsC}{C}
\newcommand{\rsc}{c}
\newcommand{\cone}{c_1}
\newcommand{\Cone}{C_1}
\newcommand{\Ctwo}{C_2}
\newcommand{\smallc}{c_0}
\newcommand{\smallcprime}{c_1}
\newcommand{\smallcanother}{c_2}
\newcommand{\smallcnew}{c_3}
\newcommand{\Cda}{D}
\newcommand{\Kzero}{K_0}
\newcommand{\Rmac}{R}
\newcommand{\rmac}{r}
\newcommand{\conseqmac}{D}
\newcommand{\constn}{C'}
\newcommand{\coninit}{\Psi}
\newcommand{\condee}{\hat{D}}
\newcommand{\conbrac}{\hat{C}}
\newcommand{\Cnew}{\tilde{C}}
\newcommand{\Cbig}{C^*}
\newcommand{\Ctbd}{C_+}
\newcommand{\Ctbs}{C_-}
\newcommand{\cE}{\mathcal{E}}
\newcommand{\Cov}{Cov}

\newcommand{\imax}{i_{{\rm max}}}

\newcommand{\wlp}{{\rm WLP}}

\newcommand{\canopynumber}{\mathsf{Canopy}{\#}}

\newcommand{\cannum}{{\#}\mathsf{SC}}

\newcommand{\boundgood}{\mathsf{G}}
\newcommand{\lshift}{\mc{L}^{\rm shift}}
\newcommand{\deltapi}{\theta}
\newcommand{\rootneigh}{\mathrm{RNI}}
\newcommand{\rootneighuse}{\mathrm{RNI}}
\newcommand{\manycan}{\mathsf{ManyCanopy}}
\newcommand{\specialpt}{\mathrm{spec}}

\newcommand{\dist}{\vert\vert}
\newcommand{\fik}{\mc{F}_i^{[K,K+1]^c}}
\newcommand{\mcfa}{\mc{H}[\fa]}
\newcommand{\tent}{{\rm Tent}}
\newcommand{\goodk}{\mc{G}_{K,K+1}}
\newcommand{\pairsep}{{\rm PS}}
\newcommand{\mbf}{\mathsf{MBF}}
\newcommand{\nbd}{\mathsf{NoBigDrop}}
\newcommand{\bd}{\mathsf{BigDrop}}
\newcommand{\jleft}{j_{{\rm left}}}
\newcommand{\jright}{j_{{\rm right}}}
\newcommand{\smalljfluc}{\mathsf{SmallJFluc}}
\newcommand{\mfone}{M_{\mc{F}^1}}
\newcommand{\mfthree}{M_{\mc{G}}}
\newcommand{\rhomac}{P}
\newcommand{\phimac}{\varphi}
\newcommand{\chimac}{\chi}
\newcommand{\xnmac}{z_{\mathcal{L}}}
\newcommand{\Cwb}{E_0}
\newcommand{\initcond}{\mathcal{I}}
\newcommand{\neargeod}{\mathsf{NearGeod}}
\newcommand{\polyunique}{\mathrm{PolyUnique}}
\newcommand{\latecoal}{\mathsf{LateCoal}}
\newcommand{\nolatecoal}{\mathsf{NoLateCoal}}
\newcommand{\normalcoal}{\mathsf{NormalCoal}}
\newcommand{\regfluc}{\mathsf{RegFluc}}
\newcommand{\mdeltaweight}{\mathsf{Max}\Delta\mathsf{Wgt}}
\newcommand{\ovbar}[1]{\mkern 1.5mu\overline{\mkern-1.5mu#1\mkern-1.5mu}\mkern 1.5mu}
\newcommand{\fluc}{\mathsf{ScaledFluc}}
\newcommand{\fluct}{\mathsf{Fluc}}

\newcommand{\bigcon}{G_0}

\newcommand{\maxmin}{\pwr}
\newcommand{\nmac}{N}

\newcommand{\high}{{\rm High}}
\newcommand{\notlow}{{\rm NotLow}}

\newcommand{\rmreg}{{\rm Reg}}

\newcommand{\down}{\mathsf{Fall}}
\newcommand{\up}{\mathsf{Rise}}

\newcommand{\safeguard}{\mathsf{SafeGuard}}
\newcommand{\lowoverlap}{\mathsf{LowOverlap}}
\newcommand{\longex}{\mathsf{LongExcursions}}
\newcommand{\totexdur}{\mathsf{HighTotExcDur}_n^{0,t}}
\newcommand{\shortexc}{\mathsf{Short}_n^{0,t}}
\newcommand{\shortnonthinexc}{\mathsf{ShortNonThin}_n^{0,t}}
\newcommand{\longexc}{\mathsf{Long}_n^{0,t}}
\newcommand{\bulklongexc}{\mathsf{BulkLong}_n^{0,t}}
\newcommand{\edgelongexc}{\mathsf{EdgeLong}_n^{0,t}}
\newcommand{\sigover}{\mathsf{SigOver}}
\newcommand{\steady}{\mathsf{Steady}}

\newcommand{\paradelta}{\Delta^{\cup}\,}

\newcommand{\aplus}{a_+}
\newcommand{\xminus}{x^-}

\newcommand{\maxdist}{\mathsf{MaxDist}}
\newcommand{\proxymimicry}{\mathsf{LocalFluc}_n^{0,t}(L,H)}
\newcommand{\nowideexc}{\mathsf{NoWideExc}_n^{0,t}}
\newcommand{\manyunlucky}{\mathsf{ManyUnlucky}_n^{0,t}}

\newcommand{\consistsep}{\mathsf{ConsistSep}_n^{0,t}}
\newcommand{\dmac}{h}
\newcommand{\dmacd}{d}
\newcommand{\poscon}{\kappa}
\newcommand{\hata}{\sigma}
\newcommand{\hatapr}{\sigma}
\newcommand{\remainder}{r_0}
\newcommand{\nolow}{\mathsf{NoLow}_n}
\newcommand{\nohigh}{\mathsf{NoHigh}_n}
\newcommand{\duration}{{\rm dur}}
\newcommand{\scriptdur}{\mathscr{D}}

\newcommand{\redconst}{d_0}
\newcommand{\hmac}{h}
\newcommand{\heightmac}{s}
\newcommand{\taumac}{\hat{\tau}}
\newcommand{\tza}{\sigma}

\newcommand{\macroseventeen}{19}
\newcommand{\macrobig}{1370}

\title[Dynamical last passage percolation]{Stability and chaos in dynamical last passage percolation}

\author[S. Ganguly]{Shirshendu Ganguly}
\address{S. Ganguly\\
  Department of Statistics\\
 U.C. Berkeley \\
  401 Evans Hall \\
  Berkeley, CA, 94720-3840 \\
  U.S.A.}
  \email{sganguly@berkeley.edu}
  
\author[A. Hammond]{Alan Hammond}
\address{A. Hammond\\
  Departments of Mathematics and Statistics\\
 U.C. Berkeley \\
 899 Evans Hall \\
  Berkeley, CA, 94720-3840 \\
  U.S.A.}
  \email{alanmh@berkeley.edu}

\begin{abstract} Many complex disordered systems in statistical mechanics are characterized by intricate energy landscapes.
The ground state, the configuration with lowest energy, lies at the base of the deepest valley. 
 In important examples, such as Gaussian polymers and spin glass models, the landscape has many valleys and the abundance of near-ground states (at the base of valleys)
 indicates the phenomenon of {\em chaos}, under which the ground state alters profoundly when the disorder of the model is slightly perturbed.
  In this article, we compute the critical exponent that governs the onset of chaos in a dynamic manifestation of a canonical model in the Kardar-Parisi-Zhang [KPZ] universality class, Brownian last passage percolation [LPP]. 
  In this model in its static form,  semi-discrete polymers advance through Brownian noise, their energy given by the integral of the white noise encountered along their journey. A ground state is a {\em geodesic}, of extremal energy given its endpoints.
  We perturb Brownian LPP by evolving the disorder under an Ornstein-Uhlenbeck flow.
   We prove that, for polymers of length $n$,  a sharp phase transition marking the onset of chaos is witnessed at the critical time $n^{-1/3}$. Indeed, the  overlap between the geodesics at times zero and $t > 0$ that travel a given distance of order $n$ will be shown to be of order $n$ when $t\ll n^{-1/3}$; and to be of smaller order when $t\gg n^{-1/3}.$ We expect this exponent to be universal across a wide range of interface models. The present work thus sheds light on the dynamical aspect of the KPZ class; it builds on several recent advances.
These include Chatterjee's harmonic analytic theory~\cite{Chatterjee14} of equivalence of {\em superconcentration} and {\em chaos} in Gaussian spaces; a refined understanding of the static landscape geometry of Brownian LPP developed in the companion paper \cite{NearGroundStates}; and, underlying the latter, strong comparison estimates of the geodesic energy profile to Brownian motion in \cite{DeBridge}.
\end{abstract}

\maketitle
\begin{figure}[h]
\includegraphics[height=0.27\linewidth]{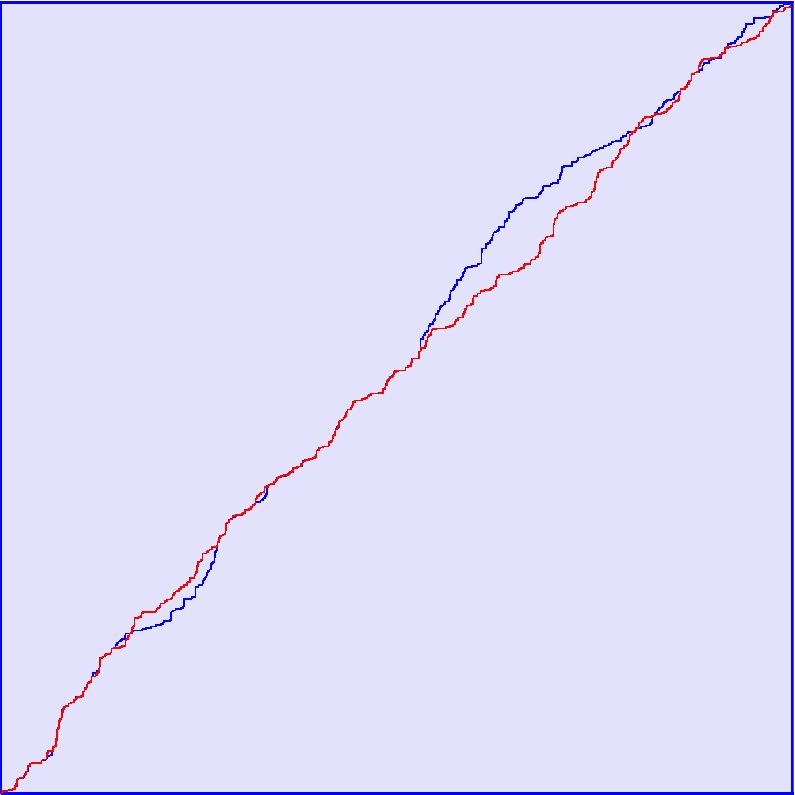}
\hspace{1cm} 
\includegraphics[height=0.27\linewidth]{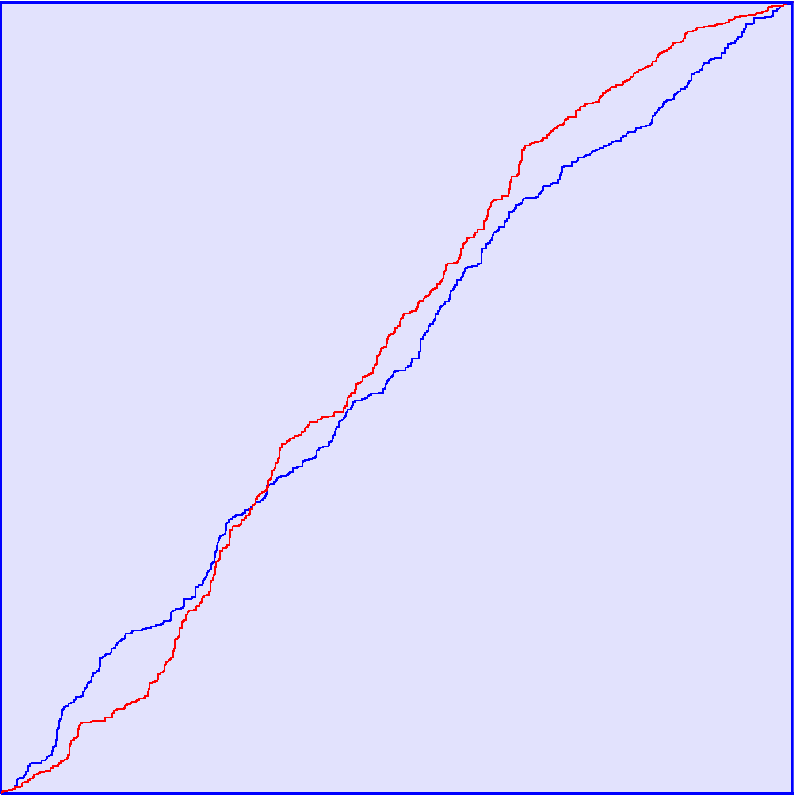}
\caption{Last passage percolation with uniform $U[0,1]$ weights is dynamically updated according to independent unit Poisson processes at each vertex. 
Depicted are snapshots at times $0.07$ and $0.3$ of a given dynamical simulation. 
The geodesic, blue at time zero, evolves to its present red state in each case. Since $1000^{-1/3} = 0.1$, the left sketch depicts a subcritical scenario and the right sketch a supercritical one, with the transition from high to low overlap evident in the images.}
\label{f.simulations}
\end{figure}

%\vspace{-8.5mm}

%\begingroup
%\hypersetup{linktocpage=false}
%\setcounter{tocdepth}{2}
%\renewcommand{\baselinestretch}{-1.5}\scriptsize
%\begin{spacing}{-2.0}
%\tableofcontents
%\end{spacing}
%\addtocontents{toc}{~\vspace{2\baselineskip}}
%\renewcommand{\baselinestretch}{-1.5}\tiny
%\endgroup
%%{\tableofcontents
%%\addtocontents{toc}{~\vspace{-2\baselineskip}}}
%\thispagestyle{plain}

%\tableofcontents 

%\begin{spacing}{-2}
%\scriptsize{\tableofcontents} 
%\end{spacing}

\begin{spacing}{-2}
\small{\tableofcontents} 
\end{spacing}

%\newpage

%\begingroup
%\hypersetup{linktocpage=false}
%\setcounter{tocdepth}{2}
%\renewcommand{\baselinestretch}{-1.5}\scriptsize
%\tableofcontents
%\addtocontents{toc}{~\vspace{2\baselineskip}}
%\renewcommand{\baselinestretch}{-1.5}\tiny
%\endgroup
%%{\tableofcontents
%%\addtocontents{toc}{~\vspace{-2\baselineskip}}}
%\thispagestyle{plain}

%\newpage
\section{Introduction}

A real-valued Hamiltonian $H$ on a finite set $X$ specifies a probability measure that charges $x \in X$ with weight proportional to $e^{-H(x)}$. 
In a viewpoint that is profitable for studying many complex systems in statistical mechanics, we may view the function $H$ as an energy specified over the landscape $X$.
In disordered systems, such as Sherrington-Kirkpatrick spin glasses, and polymer models, in which a route is randomly forged through a random medium, the energy landscape has a rich structure of local valleys and connecting pathways.

In such systems,  the ground state, namely the value $x \in X$ that minimizes $H$, may be joined in the landscape by a host of near ground states, resting at the bases of various local valleys, at more or less removed locations. The presence of multiple competing near-minimizers has significance in problems concerning the response of the system to perturbations of its law suffered by changes in parameters such as temperature or in the realization of the disorder. A guiding principle, on which we will elaborate, asserts that, when many valleys are present, \emph{chaos} reins: the system is fragile, with small perturbations causing profound changes in the form of the ground state.  

The random disorder in a statistical mechanical model may take a discrete or a continuous form. A discrete white noise field comprised of independent Bernoulli-$1/2$
random variables and a Gaussian field are canonical examples of these types. In critical percolation---bond percolation on the Euclidean lattice~$\Z^2$, or face percolation on the honeycomb lattice---edges or faces are independently declared open with probability one-half, and the large-scale structure of open connected components is investigated.  On the continuous side, the level sets of Gaussian processes on $\Z^2$ or $\R^2$ such as the Gaussian free field~\cite{Sheffield2007} or Gaussian analytic functions~\cite{BeffaraGayet} offer a counterpart model of random geometry.  Each type of randomness, discrete or Gaussian, may be perturbed in a canonical fashion.  In the Boolean case, the open bits may be independently updated at the ring-times of independent Poisson process of rate one.
A Gaussian field may be construed as a certain limit of a Boolean one; the discrete update dynamics passes in this limit to yield the Ornstein-Uhlenbeck~[OU] dynamics, 
which holds invariant the continuous field.  This continuous dynamics is a natural one, and may be viewed~\cite{JKO} profitably for the study of optimal transport problems~\cite{VillaniTopics} 
as the gradient flow of a functional of the relative entropy on the space of probability measures equipped with a Wasserstein metric. 

The last two decades have seen exciting progress in the rigorous study of chaos and the companion notion of noise sensitivity, wherein is investigated the degree of perturbation needed for an observable of interest to substantially lose correlation with  its initial value.
The influential work~\cite{BKS1999}  developed a general theory of noise sensitivity for Boolean functions by connecting their study to the theory of harmonic analysis on the discrete hypercube.
When such a function $f$ is noise sensitive, Fourier modes of high frequency are charged by a natural spectral measure attached to $f$.  Concentration of measure tools such as hypercontractivity were allied with this theory to generate beautiful applications, including the study of weighted majority functions, which are indicators of hyperplanes, proving them to be the canonical class of noise stable functions; and a demonstration that the event of crossing a large region in critical bond percolation on $\Z^2$
is sensitive to a small uniform perturbation in the open or closed status of edges.

A quantified assertion of the sensitivity of planar critical percolation to perturbation by noise was obtained by exploiting the same discrete Fourier analytic approach via a randomized algorithm in~\cite{SchrammSteif}. 
Schramm and Steif's technique has been applied to a dynamical OU perturbation of Bargmann-Fock percolation (a Gaussian analytic function) in~\cite{GarbanVanneuville}, to prove that this model is noise sensitive under perturbations that are polynomially small in the system diameter.
A further breakthrough for planar critical percolation was made in~\cite{GPS2010}, when a refined understanding of this problem was obtained, again by Fourier analytic means, with the value of exponents governing the onset of chaos  being rigorously obtained in terms of putatively universal exponents describing fractal properties of critical percolation on the hexagonal lattice.

The rigorous derivation of physically predicted critical percolation exponents~\cite{SmirnovWerner} via Schramm-Loewner evolution [SLE] methods~\cite{Schramm2000} 
provided a theoretical underpinning for the noise sensitivity study in~\cite{GPS2010} and indeed for the construction of the scaling limit for dynamical critical percolation in~\cite{GPS2018} via~\cite{GPS2013}.

Analytic tools have enabled important advances in the study of the transition to chaos in {\em random matrix theory}.
The real-valued quadratic form of any given  $n\times n$ Hermitian matrix is an energy landscape on the $n$-dimensional sphere. The eigenvector associated to the highest eigenvalue is the ground state in the landscape; it is a natural observable to monitor as the static model becomes dynamic.
 A canonical dynamics on  $n \times n$ matrices first randomly selects  a matrix according to the Gaussian unitary ensemble [GUE]
 and perturbs this initial condition by running OU dynamics independently in each Gaussian entry.
 Chatterjee~\cite{Chatterjee14} 
  found that the onset of chaos---the time  beyond which the eigenvectors become almost orthogonal to each other---has occurred by any time much exceeding $n^{-1/3}$.  Recently, Bordenave, Lugosi and Zhivotovskiy~\cite{BLZ2020}   
 studied a discrete dynamics, in which $k$ random entries in a random $n \times n$ Wigner matrix are updated. They 
 used resolvent analysis to prove that the counterpart to Chatterjee's bound is sharp: when $k \ll n^{5/3}$, the top eigenvector is highly correlated with its initial form; when $k \gg n^{5/3}$, the two vectors are almost orthogonal.
Returning to continuous dynamics, we mention that perturbing for short times the entries of a matrix along an OU flow has also been a central technique in the proofs of universality of spectral statistics and quantum ergodicity for Wigner matrices~\cite{ESY2011,ESY2012}. 

The OU dynamics on GUE induces a form of Dyson's Brownian motion~\cite{Dyson1962} on the eigenvalues. In fact this can be used to also prove the value of the highest eigenvalue starts de-correlating at the same time scale of $n^{-1/3}$.  This Dyson process of mutually avoiding Brownian motions is a leading player~\cite{O'ConnellYor} in
the study of a class of statistical mechanical models that has attracted massive mathematical interest since the 1990s. This is the $1+1$ dimensional Kardar-Parisi-Zhang [KPZ] universality class,
whose members include many models of local random growth; (in this way, a bridge runs from random matrix theory into KPZ).
 The object of study in KPZ is the scaled behaviour, on a large spatial scale and at advanced time, of  a wide range of interface models suspended over a one-dimensional domain, in which growth in a direction normal to the
surface competes with a smoothening surface tension in the presence of a local randomizing force
that roughens the surface.   In this article, we study the problem of stability and sensitivity under noise for a canonical model in this class.
Random matrix theory, equipped with Gaussian models and powerful analytical tools including linear algebraic formulas, has often formed a testbed for positing and proving conjectures also valid for geometrically complex models in the KPZ class. For example, the GUE Tracy-Widom law, which identifies the limiting point-to-point last passage percolation  geodesic energy fluctuation (a notion we will recall later), was derived~\cite{TracyWidom} by asymptotic determinantal analysis of large GUE matrices. In this vein, the attractive spectral theory available for random matrices, including resolvent methods, has led~\cite{BLZ2020} to 
the discovery and rigorous derivation of the exponent of $-1/3$ governing onset of chaos;  our purpose is to demonstrate
that this scale heralds chaos in the geometrically rich setting of dynamic KPZ.   

Last passage percolation [LPP], an important model for local random growth,  
is a zero temperature polymer model in which oriented paths in $\Z^2$ or $\R^2$  progress through independent disorder and accrue a random weight as the integral of the noise encountered on their journey.
The KPZ class is expected to contain many LPP models, a few of which have integrable features that have permitted rigorous demonstration of their KPZ characteristics when viewed in suitably scaled coordinates. Recently, integrable analysis has been allied with probabilistic and geometric inquiry in order to give detailed quantified understanding of how certain scaled LPP models behave.
The Brownian Gibbs property~\cite{AiryLE} is a probabilistic resampling technique concerning the Airy line ensemble, which encodes scaled information about the energy of sets of disjoint geodesics emerging from a given point in LPP. It has permitted detailed quantified comparison of curves in the Airy line ensemble with Brownian motion, with the recently released~\cite{D23A} providing Radon-Nikodym derivative estimates that refine estimates proved in~\cite{BrownianReg,DeBridge}. The construction of the directed landscape~\cite{DOV}, a rich expression of universal scaled LPP structure, relies on estimates on the bulk behaviour~\cite{DV} of the Airy line ensemble arising via Brownian Gibbs analysis. Such probabilistic ideas in the setting of general line ensembles feature in \cite{LEgeneral,barraquand2023spatial}. The study of fractal geometry and Hausdorff dimension of exceptional sets in the Airy line ensemble and the directed landscape has been enabled by robust probabilistic tools: \cite{BGH21,BGH22,GZ} studied exceptional sets involving the existence of pairs of disjoint geodesics;~\cite{D22} proved conjectures of~\cite{CHHM} concerning Hausdorff dimension values for exceptional times in the KPZ fixed point~\cite{MQR} at which several maximizers exist;~\cite{B22,B23,BB23} study the interlocking of the primal and dual geodesic trees in the directed landscape; and~\cite{D23B} provides a full classification for the topological possibilities for geodesic networks in the landscape.  Other aspects of fractal behavior such as the multi-fractal spectrum associated to the laws of iterated logarithm for the KPZ equation or the KPZ fixed point have also received significant attention with detailed investigations carried out in \cite{das2021law, das2022long}.

 The field has achieved a degree of development that permits a robust range of aspects to be fruitfully addressed. In particular, static understanding of a certain model, Brownian last passage percolation---whose disorder is standard {\em Gaussian white noise}---is now advanced enough (in part because it is this prelimiting model that perfectly satisfies the Brownian Gibbs property)
  that an inquiry into its dynamical perturbation may be undertaken. 

Several physicists have adopted numerical and heuristic approaches to the study of dynamical perturbations
of lattice Gaussian polymer models since the 1980s. Their work shares certain ideas and themes with ours: it often addresses, as ours does, disorder chaos; but also {\em temperature chaos}, in which a small disturbance to the positive temperature of a model of randomly weighted paths has the effect of introducing new randomness. 
Zhang \cite{Zhang1987} considers clustering of near maximizers and how perturbation causes the ground state to jump from one cluster to the other; while M\'ezard \cite{Mezard} studies positive temperature models and properties of samples with interactions penalizing overlap between them. Entropy of valleys with bearing on the underlying Gibbs measure is investigated in Fisher and Huse~\cite{FisherHuse}.  A more recent work by da Silveira and Bouchaud~\cite{daSilveiraBouchaud}  studies temperature and disorder chaos both for the ground state and positive temperature polymers in low dimensions.

Chatterjee's monograph \cite{Chatterjee14}  offers a pertinent rigorous advance that will form an important tool in our treatment. In a unified treatment that includes  Gaussian polymer models---and, fortuitously for our purpose, Brownian LPP---he developed an equivalence between  \emph{superconcentration} of observables and \emph{chaos}.  The later articles \cite{DEZ} and~\cite{CHL2018} prove
stronger lower bounds on the number of peaks in the energy landscape of  superconcentrated Gaussian fields; these works establish chaos for the ground state in the Sherrington-Kirkpatrick spin glass, with~\cite{CHL2018} also addressing   all mixed
$p$-spin models with $p$ even.
Recently, in \cite{Eldan}, Eldan employed a different approach using analysis of Gaussian spaces to extend the results of~\cite{CHL2018} to all mixed $p$-spin models with $p\geq 2$. 

In the rest of the introduction, we review  the KPZ scaling of LPP; 
 specify Brownian LPP, the integrable model of our choice, and its dynamical perturbation; and state our main result, regarding this dynamics, pinning down the transition from stability to chaos of the geodesic.
%, but to set the stage we start by reviewing the KPZ scaling characteristics of LPP.}  
%In the following three subsections, we will expand on the preceding paragraph `Last passage percolation ...' by reviewing the KPZ scaling characteristcs of LPP; 
% indicate a key aspect of Chatterjee's theory that is enough to hint at least at why our main result has the form that it does; and review 
%a vital probabilistic technique, the Brownian Gibbs resampling property, which undergirds the geometric random analysis that will deliver our main conclusion.
%The article~\cite{NearGroundStates} is a companion to the present one that provides several important inputs: the last mentioned subsection also briefly reviews these.}

\subsection{The KPZ scaling characteristics of LPP}\label{s:kpzchar}
In LPP, the geodesic from $(0,0)$ 
to $(n,n)$, is the path between them with the maximum energy.  For a large class of LPP models, the energy of the geodesic grows linearly in $n,$ with standard deviation expected to be of order $n^{1/3}$. Further, the geodesic is predicted to fluctuate from the linear interpolation of the endpoints by an order of $n^{2/3}$. These assertions capture a sense of the $(1/3,2/3)$ exponent pair characteristic of KPZ. However, the integrable structure available in just a few LPP models  has permitted rigorous sense to be made of the assertions. The seminal work of Baik, Deift
and Johansson~\cite{BDJ1999} rigorously established the one-third exponent, as well as obtained the
GUE Tracy-Widom distributional limit, for Poissonian LPP. Johansson~\cite{Johansson2000} derived the transversal fluctuation of two-thirds for this model.

%However, the focus of this article will be yet another last passage percolation model driven by Brownian motions known as  Brownian LPP. This is because it is the unique integrable LPP model for which the scaled energy profile has been shown to closely resemble Brownian motion (via Brownian Gibbs analysis) and to which Chatterjee's theory of superconcentration and chaos applies. We start with an informal description of the model and statements of our results followed by a more elaborate formal description. We will then expand on the aforementioned useful properties that Brownian LPP exhibits and how they will play crucial roles in our proofs. 
%The role of this section is to expand on this notion; to offer heuristics for the timing of the onset of chaos; and to introduce valuable notation for the upcoming analysis.

We now move on to formally defining Brownian LPP, and its dynamical perturbation.

%\subsection{Brownian LPP, its dynamical perturbation, and the overlap transition: an informal summary}\label{s.verbalsummary}
%In Brownian LPP, an independent white-noise field is assigned to the system of horizontal planar lines $\R \times \Z$. Each upright (or non-decreasing) path valued in this set is ascribed an energy by integrating the field along the path. A geodesic is a path of maximum energy given its endpoint pair. 
%
%
%
%A dynamical perturbation of Brownian LPP has a noise environment which is a stochastic process indexed by $[0,\infty)$ whose invariant measure is the just mentioned white noise field on $\R \times \Z$. This white noise may be viewed as a fine mesh, high~$n$, limit of independent  Bernoulli random variables valued in $n^{-1/2} \cdot \{-1, 1 \}$ and indexed by $n^{-1}\Z \times \Z$. This discrete field admits a natural dynamics where the Bernoulli random variables are independently renewed at Poisson times; in the fine mesh limit, this dynamics becomes an OU process---which limiting process will be our dynamical Brownian LPP model. 

\subsection{Brownian last passage percolation [LPP]}\label{s.brlpp}
%We have informally presented Brownian LPP, its  
%dynamic enhancement, and our main theorem, and 
%offered a short guide to pertinent concepts and history. We end the introduction 

On a probability space carrying a law labelled~$\PP$,
let $B:\R \times \Z \to \R$ denote an ensemble of 
independent  two-sided standard Brownian motions $B(\cdot,k):\R\to \R$, $k \in \Z$. The indexing of the 
domain in the form $\R \times \Z$ is unusual, with the other choice $\Z \times \R$ being more conventional. 
The choice of $\R \times \Z$ is made because it permits us to visualize this index set for the ensemble $B$'s 
curves as a subset of $\R^2$ with the usual Cartesian coordinate order being respected by the notation. 

Let $i,j \in \Z$ with $i \leq j$.
We denote the integer interval $\{i,\cdots,j\}$ by $
\llbracket i,j \rrbracket$.
Further let $x,y \in \R$ with $x \leq y$.
Consider the collection of  non-decreasing lists 
 $\big\{ z_k: k \in \llbracket i+1,j \rrbracket \big\}$ of values $z_k \in [x,y]$. 
With the convention that $z_i = x$ and $z_{j+1} = y$,
we associate an energy $\sum_{k=i}^j \big( B ( z_{k+1},k ) - B( z_k,k ) \big)$ to any such list.
We then define  the maximum energy
\begin{equation}\label{e.energydef}
M \big[ (x,i) \to (y,j) \big] \, = \, \sup \, \bigg\{ \, \sum_{k=i}^j \Big( B ( z_{k+1},k ) - B( z_k,k ) \Big) \, \bigg\} \, , 
\end{equation}
where the supremum is taken over all such lists. The random process $M \big[ (0,1) \to (\cdot,n) \big] : [0,\infty) \to \R$ was introduced by~\cite{GlynnWhitt} and further studied in~\cite{O'ConnellYor}.

\subsubsection{Staircases}\label{s.staircase}

Set $\N = \{ 0,1,\cdots \}$.
For $i,j \in \N$ with $i \leq j$, and 
$x,y \in \R$ with $x \leq y$,
an energy has been ascribed to  any non-decreasing list 
$\big\{ z_k: k \in \llbracket i+1,j \rrbracket \big\}$ of 
values $z_k \in [x,y]$.
In order to emphasise the geometric aspects of this definition, 
we associate to each list a  subset of $[x,y] \times [i,j] \subset \R^2$, which will be the range of a piecewise affine path,  
that we call a staircase.
 
 To define the staircase associated to $\big\{ z_k: k \in \llbracket i+1,j \rrbracket \big\}$, 
we again adopt the convention that $z_i = x$ and  $z_{j+1} = y$. 
The staircase is specified as the union of  
the horizontal segments of the form $[ z_k,z_{k+1} ] \times \{ k \}$ for $k \in \llbracket i , j \rrbracket$ as well as the vertical planar 
line segment of unit length interpolating the right and left endpoints of each consecutive pair of horizontal segments.

The resulting staircase may be depicted as the range of 
an alternately rightward and upward moving path from starting point $(x,i)$ to ending point $(y,j)$. 
The set of such staircases will be denoted by $\staircase \big[ (x,i) \to (y,j) \big]$.
Such staircases are in bijection with the collection of non-decreasing lists already considered. Thus, any 
staircase $\phi \in \staircase \big[ (x,i) \to (y,j) \big]$
is assigned an energy $E(\phi) = \sum_{k=i}^j \big( B ( z_{k+1},k ) - B( z_k , k  ) \big)$ via the associated $z$-list.

\subsubsection{Energy maximizing staircases are called geodesics.}\label{s.geodesics}
A staircase  $\phi \in \staircase \big[ (x,i) \to (y,j) \big]$ 
whose energy  attains the maximum value $M \big[ (x,i) \to (y,j) \big]$ is called a geodesic from $(x,i)$ to~$(y,j)$.
It is a simple consequence of the continuity of the constituent Brownian paths $B(\cdot,k)$
that this geodesic exists for all choices of $x,y \in \R$ with $x \leq y$.
 For any given such choice of the pair $(x,y)$, there is by~\cite[Lemma~$A.1$]{Patch},
   an almost surely unique geodesic from  $(x,i)$ to~$(y,j)$. We denote it by $\Gamma \big[ (x,i) \to (y,j) \big]$.

We next specify the details of the dynamics.
%Our principal result, which we now informally indicate, identifies the scale at which stability is interrupted by chaos in dynamical Brownian LPP. For $t \geq 0$, let $\Gamma_n(t)$ denote the geodesic on the route from $(0,0)$ to $(n,n)$
%in the copy of static Brownian LPP offered by the dynamics at time $t$. We measure similarity between these various paths by {\em overlap}. The overlap $\mc{O}(\gamma,\phi)$ between two upright paths $\gamma$ and  $\phi$ is the one-dimensional Lebesgue measure of the subset of $\R \times \Z$ given by the intersection of the ranges of these two paths. Set $\mc{O}_n(t) = \mc{O}\big(\Gamma_n(0),\Gamma_n(t)\big)$ to be the dynamic overlap process, indexed by $t \geq 0$. Note that $\mc{O}_n(0) = n$. The overlap process eventually drops far below order $n$. 

\subsection{Dynamical Brownian LPP}\label{s.dynamicalbrownian} 
An Ornstein-Uhlenbeck [OU] process is a simple stochastic process whose invariant measure is Gaussian. The stationary standard one-dimensional OU process $X: [0,\infty) \to \R$
is a stationary process that solves the  stochastic differential equation 
\begin{equation}\label{e.ou}
{\rm d}X(t)=-X(t) {\rm d}t+ 2^{1/2}{\rm d}W(t) \, 
\end{equation}
where $W: [0,\infty) \to \R$ is standard Brownian motion. 

It is sometimes useful to note that 
$X$ may be coupled to a further standard Brownian motion $W':[0,\infty) \to \R$, which is independent of the value $X(0)$, in such a way that, for $t \geq 0$,
$$
X(t) = e^{-t}X(0)+e^{-t}W'\big({e^{2t}-1}\big) \, ;
$$ 
incidentally, the processes $W$ and $W'$ are not equal.
For given $t \geq 0$, we may thus represent 
\begin{equation}\label{e.twotimedist}
X(t)\overset{d}{=}e^{-t}X(0)+ \big( 1-e^{-2t} \big)^{1/2} X' \, ,
\end{equation}
where $X(0)$ and $X'$ are independent standard Gaussian random variables. It readily follows that the correlation ${\rm Corr}(X(0),X(t))$ equals $e^{-t}$; so that $X(t)$ offers a perturbation of $X(0)$ that is slight when $t \geq 0$ is small.

 Now suppose given a Gaussian process $X: \mc{I} \to \R$ whose domain $\mc{I}$ is an arbitrary set. The Ornstein-Uhlenbeck dynamics 
whose invariant measure is the law of $X$ is the Gaussian process
 $\mathbf{X}:\mc{I} \times [0,\infty) \to \R$ such that, for $t \geq 0$, 
\begin{equation}\label{e.correlation}
\big(\mathbf{X}(\cdot,0), \mathbf{X}(\cdot,t)\big) \, \overset{d}{=} \, \Big( \, \mathbf{X}(\cdot,0) \, , \, e^{-t}\mathbf{X}(\cdot,0)+ \big( 1-e^{-2t} \big)^{1/2} \mathbf{X}'(\cdot) \, \Big) \,  ,
\end{equation} 
where $\mathbf{X}(\cdot,0)$ and 
$\mathbf{X}'(\cdot)$  are independent random variables that share the law of $X: \mc{I} \to \R$. Note that the above information is sufficient to determine the covariance structure of $\mathbf{X}$. 

As we noted in Section \ref{s.brlpp}, Brownian LPP is specified by an ensemble of independent two-sided Brownian motions denoted by $B:\R \times \Z \to \R$. This ensemble will be called the {\em static noise environment} as we turn to specify its dynamical enhancement.

Our dynamical model will be denoted by $B: \R \times \Z \times [0,\infty) \to \R$. The third argument $t \in [0,\infty)$ is the time parameter for our dynamics which will keep the underlying noise environment stationary.
The reuse of the symbol $B$ is an abuse, but explicit references to the dynamical model or to the static noise environment will be occasional, and the context will clarify what is meant. 

\begin{definition}\label{d.dynamics}
The dynamical model $B: \R \times \Z \times [0,\infty) \to \R$ is specified by the process ${\bf X}$ in~(\ref{e.correlation})
in the case that $\mathcal{I} = \R\times \Z$, where the Gaussian process $X: \mathcal{I} \to \R$ is the product over~$\Z$ of independent standard two-sided Brownian motions mapping $\R$ to $\R$. 
\end{definition}

To present an explicit construction 
of the dynamical model, a little notation is needed. A dyadic interval of scale $j \in \Z$ is a closed interval whose endpoints are consecutive elements in the set $2^{-j}\Z$. The midpoint of such an interval is a dyadic rational of scale $j+1$.
Thus, for example, odd integers are dyadic rationals of scale zero, and half-integers---elements of $2^{-1}\Z \setminus \Z$---are dyadic rationals of scale one. 

Let $\mc{D}_j$ denote the set of dyadic intervals of scale $j \in \Z$, and set $\mc{D} =  \cup_{j \in \Z} \mc{D}_j$. 
For $I \in \mc{D}_j$ of the form $[x,x+ 2^{-j}]$, set $I^- = [x,x+2^{-j-1}]$ and $I^+ = [x + 2^{-j-1},x+2^{-j}]$. 
Let $h_I:\R \to \R$ equal $2^{j/2}$ on $I^-$; $-2^{j/2}$ on $I^+$; and zero otherwise. Let $f_I:\R \to \R$ denote the definite integral of $h_I$.
This function takes the form of a tent map, taking the value zero outside $I$, equalling $2^{-j/2-1}$ at $x + 2^{-j-1}$, and having an affine form on $I^-$ and $I^+$.

\begin{lemma}\label{l.brownianou}
\begin{enumerate}
\item 
Let $\big\{ \xi_I: I \in \mc{D} \big\}$ denote an independent collection of standard Gaussian random variables. The random process $Z:\R  \to \R$ given by $Z(x) = \sum_{I \in \mc{D}} \xi_I f_I(x)$ is standard two-sided Brownian motion. 
\item  The Ornstein-Uhlenbeck dynamics whose invariant measure is standard two-sided Brownian motion is the  random process $W:\R \times [0,\infty) \to \R$ given by $W(x,t) = \sum_{I \in \mc{D}} \zeta_I(t) f_I(x)$, where 
$\zeta_I: [0,\infty) \to \R$ indexed by $I \in \mc{D}$ denote independent  standard Ornstein-Uhlenbeck processes.
\end{enumerate} 
\end{lemma}
{\bf Proof: (1).} This is similar to L\'evy's construction of Brownian motion as given in the proof of~\cite[Theorem~$1.3$]{MortersPeres}.  The present construction has doubly infinite indexing, however, so it is an exercise to adapt the proof. \\
{\bf (2).} We wish to verify that $W$ and ${\bf X}$ coincide, 
where ${\bf X}$ satisfies (\ref{e.correlation}) in the case that $\mc{I} = \R$ and $X:\mc{I} \to \R$ is standard Brownian motion.
Since $W$ and ${\bf X}$ are stationary Gaussian processes, it is enough to confirm that  ${\rm Cov} \big( W(x_1,0), W(x_2,t) \big)= {\rm Cov} \big( {\bf X}(x_1,0),  {\bf X}(x_2,t) \big)$ for $x_1,x_2 \in \R$ and $t \geq 0$. Note first that
 ${\rm Cov} \big( {\bf X}(x_1,0),  {\bf X}(x_2,t) \big)$
 equals  $e^{-t} \min \big\{ \vert x_1 \vert, \vert x_2 \vert \big\} \cdot {\bf 1}_{x_1 x_2 > 0}$ by~(\ref{e.correlation}) and the covariance formula for standard Brownian motion.  
 Next note that ${\rm Cov} \big( \zeta_I(0)  , \zeta_I(t)  \big) = e^{-t}$
and  ${\rm Cov} \big( \zeta_I(0)  , \zeta_I(0)  \big)  = 1$
for $I \in \mc{D}$. 
Thus, 
\begin{eqnarray*}
 {\rm Cov} \Big( W(x_1,0), W(x_2,t) \Big)
 & = & \sum_{I \in \mc{D}}  {\rm Cov} \big( \zeta_I(0)  , \zeta_I(t)  \big) f_I(x_1) f_I(x_2)\\ 
 & = & e^{-t}  \sum_{I \in \mc{D}}  {\rm Cov} \big( \zeta_I(0)  , \zeta_I(0)  \big) f_I(x_1) f_I(x_2) \\
 & = & e^{-t}  \, {\rm Cov} \Big( W(x_1,0), W(x_2,0) \Big) = e^{-t} \min \big\{ \vert x_1 \vert, \vert x_2 \vert \big\} \cdot {\bf 1}_{x_1 x_2 > 0} \, .
\end{eqnarray*}
The desired coincidence of covariances has thus been shown. \qed

Thus our dynamical model $B: \R \times \Z \times [0,\infty) \to \R$  satisfying Definition~\ref{d.dynamics} 
may be constructed by applying Lemma~\ref{l.brownianou}(2) to specify each marginal $B(\cdot,j,\cdot): \R \times [0,\infty) \to \R$ for $j \in \Z$ independently.

\subsubsection{Notation for the dynamical model}\label{s.dynamicalnotation}
Since the random model $B:\R \times \Z \times [0,\infty) \to \R$
couples copies of the static noise environment indexed by its third coordinate $t \geq 0$, we will use a superscript $t$ to indicate a random variable designated by the corresponding static noise. For example, for $i,j \in \Z$, $i \leq j$, and $x,y \in \R$, $x\leq y$, $M^t \big[ (x,i) \to (y,j) \big]$ denotes the maximum energy of a staircase from $(x,i)$ to $(y,j)$ in the noise environment $B(\cdot,\cdot,t)$; and similarly, $\Gamma^t \big[ (x,i) \to (y,j) \big]$ denotes,  the almost surely unique geodesic in the same environment. When no superscript appears in such notation, it is understood that a given static noise environment is involved. 

\subsection{Main result}\label{s.mainresult}
We now formally state our main result.  Let $S_1$ and $S_2$ denote two staircases. The set $S_1  \cap S_2  \cap ( \R \times \Z )$ is the intersection of the union of the horizontal segments of $S_1$ with the counterpart union for $S_2$.
The {\em overlap} $\mc{O}(S_1,S_2)$ between $S_1$ and $S_2$
is the one-dimensional Lebesgue measure of this set. For two staircases beginning at $(0,0)$ and ending at $(n,n)$, the overlap lies in~$[0,n]$. 

For $n \in \N$ and $t \geq 0$, let $\mc{O}_n(t) = \mc{O} \big( \Gamma^0 ( [0,0] \to [n,n] ) , \Gamma^t ( [0,0] \to [n,n] ) \big)$
denote the overlap  between the geodesics at times zero and $t$.
Note that when $t =0$, $\mc{O}_n(t) = n$. 
% As dynamic time $t$ rises, a transition may occur  
%from a regime of stability to one of chaos---
%from a subcritical regime at which $\mc{O}_n(t)$ typically has order $n$ to a supercritical regime with the behaviour $\mc{O}_n(t) = o(n)$ being typical. 
%Our principal result asserts that the overlap transitions from $O(n)$ to $o(n)$ at time $t = n^{-1/3 + o(1)}$.

\begin{theorem}\label{t.main}
$\empty$
\begin{enumerate}
\item 
There exist $d \in (0,1)$ and $n_0 \in \N$ such that, for 
$\lambda > 0$, we may find $h > 0$ for which
$t \in \big[0,n^{-1/3} \exp \big\{ - h (\log \log n)^{68} \big\} \big]$
and $n \geq n_0$ imply that
$$
 \PP \Big( n^{-1} \mc{O}_n(t)  \geq  d  \Big) \geq 1 - 2 (\log n)^{-\lambda}  \, .
$$
\item There exists a constant $D > 0$ such that, for $n^{-1/3}< t \leq 1$, $$
 \PP \Big(  n^{-1} \mc{O}_n ( t ) \leq D  \tau^{-1/2}  \Big) \geq 1 - \tau^{-1/2} \, ,
$$
where  $t = n^{-1/3} \tau$. 
\end{enumerate}
\end{theorem}
The bound for $t=1$, namely   $\PP \Big(  n^{-1} \mc{O}_n ( t ) \leq D  n^{-1/6}  \Big) \geq 1 - n^{-1/6}$, continues to hold when $t>1$. Indeed, our proof will show that 
\begin{align}\label{e.larget}
\E(n^{-1}\mc{O}_n ( t ))\le 
\left\{\begin{array}{ccc}
D\tau^{-1} & \text{for }& n^{-1/3}< t\le 1 \, , \\
Dn^{-1/3} & \text{for}&t> 1 \, ,
\end{array}
\right.
\end{align}
from which the just noted bound follows from Markov's inequality. In fact, it may be expected that $n^{-1} \mc{O}_n(t)$ decreases in mean  until it reaches an equilibrium value of order~$n^{-2/3}$.

%Our principal result, Theorem~\ref{t.main}---to be found at the end of the introduction---will assert that the transition from a stable phase where $\mc{O}_n$ has order $n$ to a chaotic one where its order is much smaller takes place at a time~$t$ of order $n^{-1/3}$. 
The order of the scale for the onset of chaos in planar LPP has been anticipated by da Silveira and Bouchaud~\cite{daSilveiraBouchaud}: we will discuss their prediction after developing a heuristic for the exponent value $-1/3$ in Section~\ref{s.onset}. 
%It is also notable that this exponent value appears in the onset of chaos in random matrix dynamics identified by Chatterjee and by Bordenave, Lugosi and Zhivotovskiy. This equality of exponents offers an inviting prospect for inquiry; we will briefly comment further in Section~\ref{s.onset}.

{\bf Acknowledgments.}  The authors thank two referees for their attention to the manuscript and for helpful comments.
 The first author thanks Sourav Chatterjee for several useful discussions on dynamical aspects of Brownian motion as well as on his work \cite{Chatterjee14}. He is partially supported by NSF grant DMS-$1855688$, NSF CAREER Award
DMS-$1945172$, and a Sloan Research Fellowship. The second author thanks Milind Hegde and G\'abor Pete for useful discussions. 
He is supported by NSF grant DMS-$1512908$ and a Miller Professorship from the Miller Institute for Basic Research in Science.

\section{Chaos onset heuristics, and scaled coordinates}

Brownian LPP is our model of study because it is the unique integrable LPP model for which the scaled energy profile has been shown to closely resemble Brownian motion and to which Chatterjee's theory of superconcentration and chaos applies. 
In four subsections, we present a heuristic argument for the scale of transition demonstrated in Theorem~\ref{t.main}; and an alternative heuristic for this conclusion, via discrete harmonic analysis;
offer a summary of the principal elements of Chatterjee's work on harmonic analysis of Gaussian spaces; and indicate roughly how we will need Brownian resemblance for energy profiles, and how such results arise from recent probabilistic and geometric research in KPZ (via the Brownian Gibbs resampling technique). The second of these subsections is specific to this arXiv release.

%The counterpart to the present section in the arXiv version~\cite{DynamicsArXiv} of this article contains the upcoming heuristic in an expanded form; a second heuristic for the transition scale; and a more detailed overview of the Brownian Gibbs technique. 
%offers further heuristic discuss
%The role of this section is to expand on this notion; to offer heuristics for the timing of the onset of chaos (first subsection), as well offer a glimpse of the key technical inputs from harmonic analysis of Gaussian spaces (second subsection) as well as geometry of geodesics that our arguments will crucially make use of (third subsection). 

%In the first subsection, we present a heuristic argument for why the onset of chaos occurs at time-scale $n^{-1/3}$.  
%In the second, we state a crucial input from  harmonic analysis of Gaussian spaces, and indicate how this article's main proof will depend on this input. 
%%{We will then present an alternative strategy using crucial inputs from harmonic analysis of Gaussian spaces, which the actual proof will implement rigorously.} 
%In the third, we will briefly indicate results on geometry of geodesics in Brownian LPP and  

%Our analysis of Brownian LPP will be performed in KPZ scaled coordinates. The third subsection introduces notation for the scaled description.  
% This permits us in the fourth and final subsection to explain roughly the assertion from the companion paper~\cite{NearGroundStates} (obtained via \cite{DeBridge})  of Brownian regularity for the LPP energy profile that is an underpinning for our analysis.    

\subsection{The time-scale of the onset of chaos explained via dynamical Bernoulli LPP}\label{s.onset}
%
%
%
%Theorem~\ref{t.main} indicates that, when Brownian last passage percolation is perturbed (and maintained at equilibrium) by  Ornstein-Uhlenbeck dynamics, the {\em transition from stability to chaos} for the geodesic from $(0,0)$ to $(n,n)$, in which scaled overlap $n^{-1} \mc{O}_n(t)$ changes from a unit to a vanishing order,  
% occurs as time $t$ increases through the order $n^{-1/3}$. Such times may thus naturally be written $t = n^{-1/3}\tau$, so that the parameter $\tau$ denotes time in units suitable for scaling under the dynamics.
%Theorem~\ref{t.main}(1) addresses short time, of the form $\tau \ll 1$, or more precisely $\tau \leq \exp \big\{ - h (\log \log n)^{68} \big\}$; and Theorem~\ref{t.main}(2), long time, of the form $\tau > 1$. 
%For the sake of exposition, the heuristic argument in this section (which our method of proof does not follow) will be provided for a discrete model of dynamical LPP,
%%---the Boolean form of the model's environment will also offer a useful bridge into tools of harmonic analysis, tools which are crucial to our study and to the proof of Theorem~\ref{t.main}.
%which we call dynamical Bernoulli LPP with the static model (i.e., at any fixed time) called 
In Bernoulli LPP, independent Bernoulli$(1/2)$ variables are assigned to the elements of $\Z^2$.
For each $n \in \N$,  nearest-neighbour {\em upright} paths that begin at $(0,0)$, travel to the right or upwards at each step, and end at $(n,n)$ accrue energy equal to the  number of one values encountered along the path.
In a model we call dynamical Bernoulli LPP, independent Poisson clocks renew the Bernoulli randomness at each vertex at unit rates.
As in Brownian LPP, the time-zero geodesic energy $M_n^0$ is equal to the {\em maximum} energy over all such paths, with the optimizing path called a geodesic and denoted by $\Gamma_n^0$. Since it is discrete,  the geodesic is typically not unique. But it is straightforward to see that there exists a geodesic that is uppermost---and also leftmost---in the natural sense, and, for definiteness, it is this geodesic that $\Gamma_n^0$ denotes. 

%\begin{figure}[t]
%\centering{\epsfig{file=bernoullitriple2.pdf, scale=0.8}}
%\caption{In the left sketch, the lattice $\Z^2$ has been rotated counterclockwise by forty-five degrees, and contracted by a factor of $2^{1/2}$. The geodesic thus passes from $(0,0)$ to $(0,n)$. The formerly anti-diagonal midlife line $y = n/2$ witnesses the geodesic's passage at location $z_{{\rm max}}$. In the middle sketch, the inverted profile $-Z^\swarrow_{n,n}$ has been translated vertically so as to touch, but not to cross, the profile $Z^\nearrow_{0,0}$. Horizontal coordinates of contact between the two graphs are locations of passage for geodesics through the midlife line $y = n/2$. The two profiles make jumps valued in $\{-1,0,1\}$ and resemble random walk, in a similiar fashion to  Bernoulli-$p$ measure being invariant for the totally asymmetric simple exclusion process. The routed weight profile $Z$ is depicted in the right sketch: its maximizers are the same locations of passage.}
%\label{f.bernoullitriple}
%\end{figure} 

% We will shortly view this model as the zero-time slice of a dynamical variant. 
%Note that the usages $M_n^0$
%and $\Gamma_n^0$ are counterpart to $M^0 \big[ (0,0) \to (n,n) \big]$ and $\Gamma^0 \big[ (0,0) \to (n,n) \big]$ specified in Subsection~\ref{s.dynamicalnotation} for dynamical Brownian LPP. 

%The shorthand notation will be used sometimes, also for the Brownian model, when we consider, as we do now, the standard route $(0,0)$ to $(n,n)$ that is the object of study of the main result, Theorem~\ref{t.main}.

In this section, we present  in terms of dynamical  Bernoulli LPP a heuristic argument  for the time-scale of transition from stability to chaos (which our method of proof does not follow).  
While no integrable properties of Bernoulli LPP are known, it is predicted to lie in the Kardar-Parisi-Zhang universality class, and thus it is expected that
\begin{equation}\label{e.timezeroenergy}
 M_n^0 = a n + W_n n^{1/3} \, , 
\end{equation}
where $a \in  (0,\infty)$
is a first-order growth coefficient, and $\big\{ W_n: n \in \N \big\}$ converges to a constant multiple of the Gaussian unitary ensemble [GUE] Tracy-Widom distribution. The path $\Gamma_n^0$ contains $2n$ vertices, so that  $a \leq 2$. On the other hand, although the weights are independent Bernoulli$(1/2)$,  it is not hard to see that $\Gamma_n^0$ is populated by one-values at a rate uniformly higher  than one-half, so that $a$ strictly exceeds one.
KPZ considerations also predict the fluctuations of $\Gamma^0_n$, implying that, at the midway height $y = n/2$, the distance of $\Gamma^0_n$ from $(n/2,n/2)$ is of order $n^{2/3}$.

%Stationary dynamics is introduced via independent unit-rate Poisson processes indexed by vertices in $\Z^2$ with the value attached to any vertex being refreshed with an independent Bernoulli$(1/2)$ whenever the corresponding clock rings.

Incidentally, it is not hard to see that, by passing suitably to the limit, the dynamics in this discrete model converges to Ornstein-Uhlenbeck dynamics with invariant measure given by standard Brownian motion. Regarding notation: 
as in the Brownian case, the time-$t$ energy (for $t \geq 0$) in dynamical Bernoulli LPP
%, here governed by the unit Poisson process dynamics,
 ascribed to any upright path~$\Phi$ will be denoted $E^t(\Phi)$; the time-$t$ maximum energy is $M_n^t$; and the time-$t$ geodesic is~$\Gamma_n^t$.

%To see this informally, if $S_n(t)$ denotes the sum of values assigned at time $t \geq 0$ to vertices in the interval $[0,n] \times \{ 0 \}$, then for given $t \geq 0$,  $[0,\infty)  \to \R: s \to  n^{-1/2} S_{ns} \big( t \big)$  may be coupled for high~$n$ to Brownian motion $[0,\infty) \to \R: s \to B(s,t)$ with a typically small error in the uniform on compact norm. The process $[0,\infty)^2 \to \R: (s,t) \to B(s,t)$ is .)

In a heuristic counterpart to our main result Theorem~\ref{t.main}, we ask: why does the stability-chaos transition occur at time of  of order $n^{-1/3}$?

In answer, first consider the dynamical changes suffered by the time-zero geodesic $\Gamma_n^0$ by a given time $t > 0$.
At this time, roughly $(1 - e^{-t})\cdot 2n$ (or $2tn$ for small $t$) vertices on the path $\Gamma_n^0$ have been resampled. 
% suffered changes in energy insofar as the values associated to aits vertices have been resampled and changed. The number of resampled vertices along $\Gamma_n^0$ is to first order equal to $(1 - e^{-t})\cdot 2n$; for small $t$, $2tn$. 
These vertices were more likely one than zero (with odds $a/2$ versus $1 - a/2$). But now they are equally likely zero or one.
Thus, the time-$t$ energy $E^t(\Gamma_n^0)$ of the initial geodesic $\Gamma_n^0$ experiences a precipitous decline after time zero, dropping at a rate of order $n$ from its initial value of $an$, until unit-order values of $t$ with the value approaching $n$ for high $t$.

We next want to understand how the maximum energy difference  $M^t_n - M^0_n = E^t\big(\Gamma^t_n\big) -  E^0\big(\Gamma^0_n\big)$  evolves between times zero and $t$ for  $t = \Theta(1) n^{-1/3}$. Note that 
\begin{equation}\label{e.energychange}
  M^t_n - M^0_n  =     - \Big(    E^0\big(\Gamma^0_n\big) -   E^t\big(\Gamma^0_n\big) \Big)    +  \Big(  E^t\big(\Gamma^t_n\big)  -  E^t\big(\Gamma^0_n\big) \Big) \, .
\end{equation}

On a shorter time-scale, when $t = \Theta(1) n^{-2/3}$,
the number of vertices on the initial geodesic that have been resampled is to first order given by $2tn = \Theta(1) n^{1/3}$; so that the time-$t$ energy $E^t\big(\Gamma^0_n\big)$  has dropped from its initial value $E^0\big(\Gamma^0_n\big)$  by an order of $n^{1/3}$---which, by ~(\ref{e.timezeroenergy}), is the same as the order of fluctuations of $E^0 (\Gamma_n^0).$
Thus, on any even slightly shorter time-scale, $t = o(1) n^{-2/3}$, $\Gamma_n^0$ remains a near-geodesic, whose time-$t$ energy differs from the maximum $M_n^t$ by at most $o(n^{1/3})$. 
But when $t \gg n^{-2/3}$, the initial geodesic~$\Gamma_n^0$ has become hopelessly uncompetitive.
Despite this, as Theorem~\ref{t.main} attests, the transition from stability to chaos occurs only at the much greater time-scale when $t = \Theta(1) n^{-1/3}$.
Towards this, note that our stability result claims only that the geodesics $\Gamma_n^t$ and $ \Gamma_n^0$  have significant overlap for $t\ll n^{-1/3}$, and asserts nothing about $E^t(\Gamma_n^0)$.
Indeed, even on the longer time-scale $t = \Theta(1)n^{-1/3}$, dynamic updates along $\Gamma_n^0$
are sparse---they number $n^{2/3}$ in total, with a typical distance from one to the next of order $n^{1/3}$. Such separation suggests that the energetic changes thus suffered resemble an independent and identically distributed sequence of small changes; a central limit theorem would thus dictate the combined energetic loss $E^t(\Gamma_n^0) - E^0(\Gamma_n^0)$ takes the form of a dominant non-random drift effect of order $nt = \Theta(1) n^{2/3}$ and a Gaussian fluctuation of order $(nt)^{1/2} = \Theta(1) n^{1/3}$. 

However, since dynamical Bernoulli LPP is in equilibrium, it is reasonable to expect that this effect of negative drift will be complemented by an opposing force  once one considers the maximum energy of all paths that run close to $\Gamma_n^0$---at distance smaller than $n^{2/3}$, say. 
Moreover, the form of~(\ref{e.energychange}) suggests that this equation's latter right-hand term represents this restorative process. That is, when updates occur to vertices that lie off but close to~$\Gamma_n^0$, an assignation of the value of one may create opportunities for a path of greater time-$t$ energy than~$\Gamma_n^0$ by a local rewiring that visits the newly updated vertex.
This ought to lead to an exact cancellation of the drift term, leading to the overall conclusion that  $  
E^t\big(\Gamma^t_n\big)  -  E^0\big(\Gamma^0_n\big) \approx (nt)^{1/2}.$
% newly furbished with a value of one. 
%In fact, further decomposing the just mentioned term as 
%$$  
%E^t\big(\Gamma^t_n\big)  -  E^t\big(\Gamma^0_n\big) =\Big(  E^t\big(\Gamma^t_n\big)  -E^0\big(\Gamma^t_n\big)\Big)  + \Big(E^0\big(\Gamma^t_n\big)-E^t\big(\Gamma^0_n\big) \Big)
%$$
%yields a right-hand side whose first term  has the same distribution as the first term in \eqref{e.energychange} up to a change of sign.  This implies an exact cancellation of the linear drift.  Further, the second term, by symmetry, has mean zero and contributes to the fluctuation term.  
%
% Moreover, the affected vertices are sparse, as were those hit along $\Gamma_n^0$, so that the fluctuation in the restorative process may be said to be Gaussian, on the order $\Theta(1)n^{1/3}$ for times  $t = \Theta(1) n^{-1/3}$, an order that is then shared with the negative energetic drift occurring along $\Gamma_n^0$. This indicates that, for these times $t = \Theta(1) n^{-1/3}$, the two parenthetical right-hand expressions in~(\ref{e.energychange}) together yield a Gaussian fluctuation of order  $\Theta(1)n^{1/3}$.
Thus, from time zero up until the short time-scale $t = \Theta(1) n^{-2/3}$, the original geodesic~$\Gamma_n^0$ remains competitive. Beyond this, it becomes less and less so. But until the much longer time-scale $t = o(1) n^{-1/3}$, the present energetic maximum $t \to M_n^t$ is in essence unchanged from its original value, differing from it by an insignificant error of the form $o(n^{1/3})$. 
%This square-root change $n^{-2/3} \to n^{-1/3}$ from short to long time-scale---from mimicry of energy on the part of the original geodesic to such mimicry on the part of all alternatives---is suggested by Gaussian energetic fluctuation after the cancellation of equal and opposing first-order drift effects, of loss along the original geodesic, and of gain from its surroundings. 

On scales greater than $t = \Theta(1) n^{-1/3}$, this mechanism of Gaussian energetic fluctuation cannot continue to hold since the system is in equilibrium and hence by (\ref{e.timezeroenergy}) the energetic change between any pair of times cannot exceed order $n^{1/3}$. Thus, the Gaussian effect presumably dominates on scales $t = o(1) n^{-1/3}$ while it competes on an equal basis with static energetic fluctuation on the scale $t = \Theta(1) n^{-1/3}$. This heuristically justifies why is it that the time-scale $t = \Theta(1) n^{-1/3}$ heralds the transition from stability to chaos in dynamical Bernoulli last passage percolation.

{\em Remark.} Following the statement of Theorem \ref{t.main}, we had mentioned that da Silveira and Bouchaud \\\cite{daSilveiraBouchaud} anticipated the $n^{-1/3}$ scale for the onset of chaos.
We discuss their argument in the language of Brownian LPP.  The perturbation considered adds noise, so that a standard Brownian noise environment $B$ becomes $B + B'$
 where, in each vertical coordinate, $B' = W(\e \cdot)$ for an independent Brownian motion $W$. (The parameter $\e > 0$ equals $n^{-1/3}$ on the critical scale.)
 da Silveira and Bouchaud argue that the time zero geodesic suffers a mean zero, Gaussian-order, change as time evolves until this change reaches the level of GUE fluctuation.
 Thus, they obtain the accurate $n^{-1/3}$ prediction. The reasoning appears imperfect, however, because, this form of perturbation is not stationary, and hence the variance of the Brownian noise environment rises with time causing the geodesic energy to rise linearly. As we have indicated, Gaussian fluctuation is indeed the correct local mechanism dictating dynamic change, but the maximizer is a local perturbation of the initial geodesic before the critical time, rather than being the initial geodesic itself. 
  
  The principle that the scale on which Gaussian and GUE fluctuation meet heralds the onset of chaos also serves to explain the $-1/3$ exponent in random matrix dynamics observed in~\cite{Chatterjee14} and~\cite{BLZ2020}. The uppermost curve~$D(1,\cdot)$ in Dyson's Brownian motion $D:\intint{n} \times [0,\infty) \to \R$ has dominantly Brownian increments on $[n,n+\Theta(n^{2/3})]$ since $D(1,\cdot) - D(2,\cdot) = \Theta(n^{1/3})$ on this scale. These scale $n^{1/3}$ increments are comparable to GUE fluctuation. The domain increment $n \to n+n^{2/3}$
  corresponds to the order $n^{-1/3}$ evolution of the OU flow on $n \times n$ GUE matrices.

\subsection{The harmonic analysis of last passage percolation}\label{s.harmonicanalysislpp}
The direct goal of this section is to present an alternative, and also non-rigorous, argument for the stability of geodesic energy on time-scales shorter than $t = \Theta(1)n^{-1/3}$. We will argue  that the energy maximum $[0,\infty) \to \R: t \to M_n^t$ satisfies the two-point bound that, for any $t \geq 0$,
\begin{equation}\label{e.twopoint}
 \E \big( M_n^0 - M_n^t \big)^2 \leq D t n \, ,
\end{equation}
for a given large constant $D > 0$. Although the present discussion is not rigorous, a vital element in our proofs will be a Brownian counterpart of \eqref{e.twopoint}---namely, Proposition \ref{p.onepoint} in Section \ref{s.deductions}.

We see by taking $t = \tau n^{-1/3}$ in the last display that this mean-squared energy difference is at most $D \tau n^{2/3}$, which indicates that the typical energy difference $M_n^0 - M_n^t$ is smaller than the typical static energetic fluctuation $\Theta(1) n^{1/3}$ when $\tau \ll 1$ is small. So, indeed,~(\ref{e.twopoint}) yields subcritical $t \ll n^{-1/3}$ energetic stability.

The claim that equality holds in~(\ref{e.twopoint}) up to a unit-order factor is consistent with the view, advanced in the preceding section, that the energetic fluctuation between time zero and any given subcritical time $t \ll n^{-1/3}$ is Gaussian (asymptotically in high~$n$). We will not comment further on the lower bound counterpart to~(\ref{e.twopoint}), however.

In order to argue in favour of~(\ref{e.twopoint}), we will discuss the fundamental role of the  tool of harmonic analysis in the study of dynamical Bernoulli LPP. A counterpart of this tool for dynamical Brownian LPP will play an essential role in proving our principal result Theorem~\ref{t.main}.  In fact, the broader goal of the present section is to introduce or reacquaint the reader with tools of discrete harmonic analysis,  in order that our later presentation of the pertinent rigorous techniques   for dynamical Brownian LPP may be adequately motivated.  As we have discussed, harmonic analytic tools also form the technical backbone in the  study of dynamical critical percolation in \cite{GPS2010}.

Setting $\Lambda_n = \llbracket 0, n \rrbracket^2$, the geodesic energy in Bernoulli LPP may be viewed as a function $M_n$ mapping $\{0,1\}^{\Lambda_n}$ to $\N$, assigning to $\omega \in \{0,1\}^{\Lambda_n}$
the maximum of the sum of $\omega$-values lying along any upright path between $(0,0)$ and~$(n,n)$.

To each subset $S \subseteq \Lambda_n$, we may associate the map 
$\chi_S: \{0,1\}^{\Lambda_n} \to \{-1,1\}$, given by 
$\chi_S(\omega) = \Pi_{x \in S}\big(2\omega(x) - 1\big)$; 
the convention that $\chi_\emptyset = 1$ is adopted. The collection of functions $\chi_S$ indexed by subsets $S \subseteq \Lambda_n$ is an orthonormal basis for the $L^2$-space of functions mapping  $\{0,1\}^{\Lambda_n}$ to the real line. As such, any such function $f$ has a resulting Fourier-Walsh decomposition, which we will study in the case that $f = M_n$. Indeed, we may write
\begin{equation}\label{e.fourierwalsh.en}
 M_n(\omega) = \sum_{S \subseteq \Lambda_n} \alpha(S) \chi_S(\omega) \, ,
\end{equation}
where note that $\alpha(\phi) = \E \, M_n$  and that, according to Parseval's formula, the system of coefficients $\big\{ \alpha(S): S \subseteq \Lambda_n \big\}$ satisfies
$$
 {\rm Var}(M_n) \, = \,  \sum_{S \not= \emptyset} \alpha(S)^2 \, .
$$
Here and henceforth, the condition that the summand $S$ is a subset of $\Lambda_n$ is understood. This identity permits us to introduce the {\em  spectral sample}, a random variable $\mathscr{S}$ distributed according to a probability measure---to be labelled $Q$---that is thus canonically associated to the function $M_n: \{0,1\}^{\Lambda_n} \to \R$.  The definition specifies that, for any given $S \subseteq \Lambda_n$,
\begin{equation}\label{e.spectralsample}
 Q \big( \mathscr{S} = S \big) \, = \, \frac{\alpha(S)^2}{{\rm Var}(M_n)} \, . 
\end{equation} 
The next proposition indicates the fundamental role of the spectral sample in the study of dynamics. Recall that, in dynamical Bernoulli LPP, we have specified a random process, which we now denote by $\lambda$, that maps time $t \in [0,\infty)$
to an element~$\lambda_t$ in the configuration space $\{0,1\}^{\Lambda_n}$. For any function $f: \{0,1\}^{\Lambda_n} \to \R$, we set $f^t(\lambda) = f(\lambda_t)$. We adopt the shorthand $\chi_S^t$ for $(\chi_S)^t$.
\begin{proposition}\label{p.bernoullispectralsample}
For any $t \geq 0$,
\begin{enumerate}
\item  the covariance of geodesic energies at times zero and $t$ satisfies
$$
  {\rm Cov} \big( M_n^0 , M_n^t \big) = \sum_{S \not= \emptyset} \alpha(S)^2 e^{-t \vert S \vert} \, ;
$$
\item and the associated correlation is given by
$$
{\rm Corr} \big( M_n^0 , M_n^t \big) =
 \E_Q \big[ e^{-t \vert \mathscr{S} \vert} \big] \, .
$$
\end{enumerate}
\end{proposition}
{\bf Proof: (1).} The quantity  ${\rm Cov} \big( M_n^0 , M_n^t \big)$ equals 
$\E \, \big( M_n^0 - \E M_n^0 \big) \big( M_n^t - \E M_n^t \big)$. Expanding $M_n^0$
and $M_n^t$ via~(\ref{e.fourierwalsh.en}), we see that this quantity equals 
\begin{equation}\label{e.meandoublesum}
 \E \, \sum_{S_1,S_2 \not= \emptyset} \alpha(S_1) \alpha(S_2) \chi^0_{S_1} \chi^t_{S_2} \, .
\end{equation}
Taking the mean $\E$ when $S_1 \not= S_2$ returns zero because any bit in the symmetric difference $S_1 \Delta S_2$ contributes a factor of $-1$ or $1$ to   $\chi^0_{S_1} \chi^t_{S_2}$ with equal probability. On the other hand, $\E \chi^0_{S} \chi^t_{S} = e^{-t\vert S \vert}$, because the mean reports zero if a bit in $S$ updates during $[0,t]$, and otherwise reports the value one. Thus the diagonal contribution to~(\ref{e.meandoublesum}) is seen to equal $\sum_{S \not= \emptyset} \alpha(S)^2 e^{-t \vert S \vert}$, and the first part of the proposition is proved.

{\bf (2).} The correlation of two random variables of equal variance is equal to the ratio of their covariance and their shared variance. Thus Proposition~\ref{p.bernoullispectralsample}(2)
follows from the first part of the proposition and the definition~(\ref{e.spectralsample}) of the spectral sample. \qed   

Our next claim is an expression for the mean cardinality of the spectral sample. By $\E_Q$, we denote the expectation with respect to the law $Q$. 

For $\omega \in \{0,1\}^{\Lambda_n}$ and $v \in \Lambda_n$, we write $\omega[v]$ for the element of $\{0,1\}^{\Lambda_n}$ that differs from $\omega$ at $v$ but that coincides with $\omega$ at arguments other than $v$. For $f: \{0,1\}^{\Lambda_n} \to \R$, we set $f[v]: \{0,1\}^{\Lambda_n} \to \R$ according to $f[v](\omega) = f(\omega[v])$.

\begin{lemma}\label{l.meanspec}
$$
 \E_Q \vert \mathscr{S} \vert  =  \frac{1}{4 \, {\rm Var}(M_n)}  \sum_{v \in \Lambda_n} \E \big(M_n - M_n[v]\big)^2 \, .
$$
\end{lemma}
{\bf Proof.} The value $\E_Q \vert \mathscr{S} \vert$ equals 
$\sum_{v \in \Lambda_n} Q\big(v \in \mathscr{S}\big)$.
 By~(\ref{e.spectralsample}), it further equals
$$
 \frac{1}{{\rm Var}(M_n)} \sum_{v \in \Lambda_n} \sum_{S: v \in S} \alpha(S)^2 \, .
$$
 Thus the task of proving Lemma~\ref{l.meanspec} will be accomplished if we prove that, for any $v \in \Lambda_n$,
\begin{equation}\label{e.enven}
 \E \big( M_n - M_n[v] \big)^2 = 4 \sum_{S: v \in S} \alpha(S)^2 \, .
\end{equation}
We have that $M_n = \sum_S \alpha(S) \chi_S$ and $M_n[v] = \sum_S \alpha(S) \chi_S[v]$. The left-hand side of~(\ref{e.enven}) is thus seen to equal $A -B$, where $A = 2 \sum_S \alpha(S)^2$ and $B = 2\E \sum \alpha(S)^2 \chi_S \chi_S[v]$. Note that $B$ equals the sum of $-2 \sum_{S: v \in S} \alpha(S)^2$ and $2 \sum_{S: v \not\in S} \alpha(S)^2$, so that $A - B$ is seen to equal the right-hand side of~(\ref{e.enven}). Thus is~(\ref{e.enven}) shown and the proof of the lemma completed. \qed

Our upper bound~(\ref{e.twopoint}) on mean-squared energetic difference  may be derived from the rigorous Proposition~\ref{p.bernoullispectralsample} and Lemma~\ref{l.meanspec} via the next two-part claim, for which we argue in a merely heuristic way. We state the claim, close out the derivation of~(\ref{e.twopoint}), and then present our case for the claim.

{\bf Claim: 1.} The variance ${\rm Var}(M_n)$ is $\Theta(1) n^{2/3}$.  
{\bf 2.} The sum $\sum_{v \in \Lambda_n} \E \big(M_n - M_n[v]\big)^2$ is $\Theta(1)n$.

Notably, the two parts of the claim alongside Lemma~\ref{l.meanspec} yield the order of the cardinality of the spectral sample:
\begin{equation}\label{e.meanspecsample}
 \E_Q \vert \mathscr{S} \vert = \Theta(1) n^{1/3} \, .
\end{equation}

Note now that $\E \big( M_n^0 - M_n^t \big)^2$ equals $2 {\rm Var}(M_n) - 2 {\rm Cov}(M_n^0,M_n^t)$, so that Proposition~\ref{p.bernoullispectralsample}(2) implies that
$$
\E \big( M_n^0 - M_n^t \big)^2 = 2 {\rm Var}(M_n) \big( 1 - \E_Q e^{-t\vert S  \vert} \big) \, .
$$
By Jensen's inequality, $\E_Q  e^{-t\vert S  \vert} \geq e^{- t \E_Q \vert \mathscr{S} \vert}$, whose right-hand side equals $e^{-\Theta(1)t n^{1/3}}$ in view of~(\ref{e.meanspecsample}).
We see then that
$$
\E \big( M_n^0 - M_n^t \big)^2 \leq 2 {\rm Var}(M_n) \big( 1 - e^{-\Theta(1)t n^{1/3}} \big) \, .
$$
In the subcritical regime, where $t = o(1) n^{-1/3}$, this right-hand side is  $2 {\rm Var}(M_n)\Theta(1) tn^{1/3}$. By Claim~$1$, this is  $\Theta(1) t n$. 

With the claim admitted, we have derived~(\ref{e.twopoint}); it remains to offer some justification of the claim. The first part is the assertion of energy fluctuation consistent with the belief that the static Bernoulli LPP model is a member of the KPZ universality class. Regarding the second, it is useful to consider the geodesic $\Gamma_n$ from $(0,0)$ to $(n,n)$, whose energy attains the value $M_n$. Any vertex  $v$ on $\Gamma_n$ that is assigned the value zero by the static environment $\omega$ necessarily satisfies $M_n - M_n[v] = -1$; since such~$v$ number $\Theta(1)n$, we see the lower bound in Claim~$2$. To argue for the upper bound, we should establish that vertices $v$ that do not lie on $\Gamma_n$ make no more than a comparable contribution to the sum in Claim~$2$ than do those $v$ that reside on $\Gamma_n$. In fact, if $v \in \Lambda_n$ has horizontal distance to $\Gamma_n$ equal to $k \in \intint{n}$, we claim non-rigorously that  $\E \big(M_n - M_n[v]\big)^2$ has order $k^{-3/2}$. Suppose further for simplicity and in essence without loss of generality that $v$ has a given coordinate along the diagonal $y=x$ such as the midway value $2^{-1/2}n$. We may consider the {\em routed geodesic energy profile}~$Z$, namely the function $Z(z)$ 
of the second, anti-diagonal, coordinate $z$ that reports the maximum energy among upright paths that pass---are {\em routed}---through coordinate $z$ on reaching the diagonal coordinate~$2^{-1/2}n$. This function takes the form $Z = Z^\nearrow_{(0,0)}(z) + Z^\swarrow_{(n,n)}(z)$, where
$$
 Z^\nearrow_{(0,0)}(z) = M \big[ (0,0) \to (n/2 + z,n/2 - z) \big] \, \, \, \textrm{and}  \,\,\, Z^\swarrow_{(n,n)}(z) = M \big[ (n/2 + z,n/2 - z) \to (n,n) \big]  \, ;
$$
the summands are energy profiles with one endpoint fixed,  the variable endpoint being either advanced $\nearrow$ or retarded $\swarrow$ along the diagonal relative to the fixed endpoint. 
The profile~$Z$ achieves its maximum at the location $z_{\rm max}$ of passage of the geodesic $\Gamma_n$ through the anti-diagonal indexed by $2^{-1/2}n$. In a heuristical view that is advanced in Figure~\ref{f.bernoullitriple}, the profile~$Z$ about $z_{\rm max}$
resembles a random walk around its maximum value.
\begin{figure}[t]
\centering{\epsfig{file=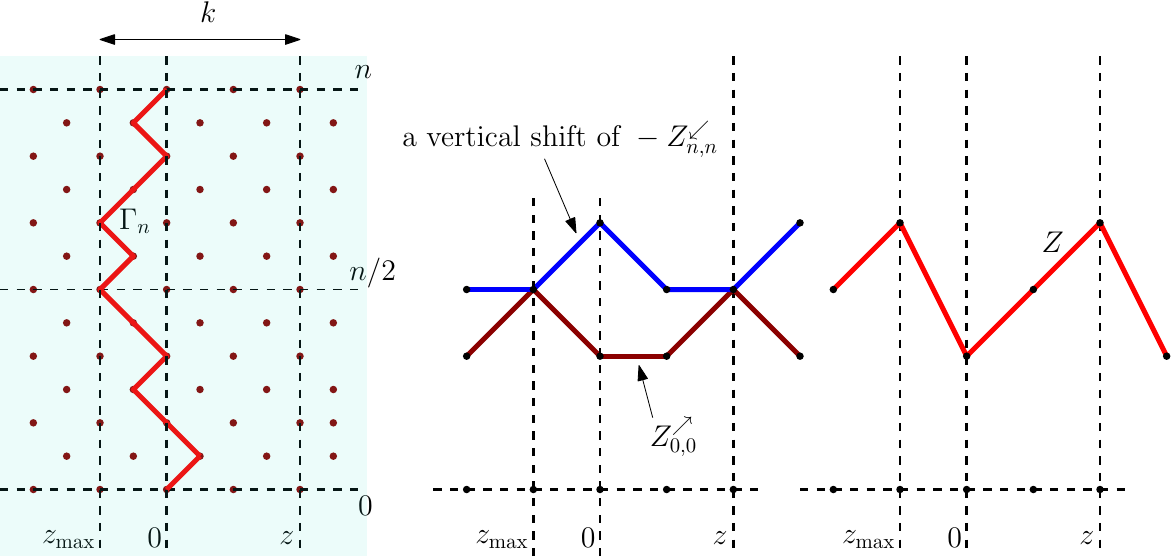, scale=0.8}}
\caption{In the left sketch, the lattice $\Z^2$ has been rotated counterclockwise by forty-five degrees, and contracted by a factor of $2^{1/2}$. The geodesic thus passes from $(0,0)$ to $(0,n)$. The formerly anti-diagonal midlife line $y = n/2$ witnesses the geodesic's passage at location $z_{{\rm max}}$. In the middle sketch, the inverted profile $-Z^\swarrow_{n,n}$ has been translated vertically so as to touch, but not to cross, the profile $Z^\nearrow_{0,0}$. Horizontal coordinates of contact between the two graphs are locations of passage for geodesics through the midlife line $y = n/2$. The two profiles make jumps valued in $\{-1,0,1\}$ and resemble random walk, in a similiar fashion to  Bernoulli-$p$ measure being invariant for the totally asymmetric simple exclusion process. The routed weight profile $Z$ is depicted in the right sketch: its maximizers are the same locations of passage.}
\label{f.bernoullitriple}
\end{figure} 
The quantity $M_n - M_n[v]$ is either zero or minus one, since the present choice of $v$ does not belong to $\Gamma_n$. If this quantity is not zero, the random walk profile~$Z$ returns to its maximum $Z(z_{\rm max})$ at exactly $k$ steps from $z_{\rm max}$---an occurrence depicted for $k=3$ in Figure~\ref{f.bernoullitriple}.  
It is an exercise that a simple random walk on a domain of length much exceeding~$k$ makes a return to its maximum value at distance exactly $k$ from its maximizer with probability of order $k^{-3/2}$. We have thus argued that, indeed, the order of $\E \big(M_n - M_n[v]\big)^2$ is $k^{-3/2}$.
Since $\sum_{k \geq 1} k^{-3/2}$ is finite, we see that off-geodesic~$v$ 
collectively contribute order $n$ to the sum with which Claim~$2$ is concerned. This is the same order as the contribution offered by their on-geodesic cousins; the case of Claim~$2$ rests here.

\subsection{Dynamical formula for variance in Gaussian LPP}\label{s.lowoverlapgaussian}

%We end our heuristic overview of tools of harmonic analysis in LPP by highlighting an inference that is known to be valid only for Gaussian models of LPP---namely, the typically low  overlap between a pair of geodesics at times that differ by a supercritical value $t \gg n^{-1/3}$. This inference corresponds to Theorem~\ref{t.main}(2) for our model of rigorous study. The model that lies most easily at hand (given the preceding discussion) for the expression of the relevant formula is one that we may call dynamical Gaussian LPP.
The harmonic analytic tool that we will use was developed by Chatterjee~\cite{Chatterjee14} in the context of another discrete LPP model whose noise variables are independent  standard Gaussians. While we will eventually extend them to Brownian LPP, our model of choice, by suitably passing to the limit, momentarily we state the versions of basic inputs we will be relying on in the setting of the discrete Gaussian model. 
 %Thus, while we will eventually adapt the same to our framework of Brownian LPP, presently we choose to state it in the setting of Gaussian LPP following Chatterjee. 

In this case, the independent standard Gaussians which form the time-zero noise environment 
are  updated independently according to Ornstein-Uhlenbeck dynamics~(\ref{e.ou}). Denoting the time-$t$ geodesic $\Gamma_n^t$ and its energy $M^t_n$, 
 just as we did for other models, the overlap $\mc{O}(\Gamma_n^0,\Gamma_n^t)$ is defined to be the cardinality of the vertices in $\Z^2$ that lie in both $\Gamma_n^0$ and $\Gamma_n^t$.

We now come to a critical formula, namely the  following dynamical formula for variance:
\begin{equation}\label{e.gaussiandynamicalvariance}
 {\rm Var}(M_n) = 2\int_0^\infty e^{-t} \E \mc{O}(\Gamma_n^0,\Gamma_n^t) \, {\rm d} t \, .
\end{equation}
Further, it is known that the mean overlap function $[0,\infty) \to \R: t \to \E \mc{O}(\Gamma_n^0,\Gamma_n^t)$ is non-increasing. 

%These assertions follow from the theory .
% We make no attempt at present to justify them and simply mention that they follow from fairly short arguments that use harmonic analysis in Gaussian space. 

Two important consequences follow from this formula. They will be stated and proved in Section~\ref{s.deductions}; here, we briefly indicate them. 

%Using the above and a variant of it, one can deduce two  important consequences---namely low supercritical overlap, and subcritical weight stability. These will be stated and proved (by approximating Brownian LPP by a discrete Gaussian model) in Section \ref{} and hence presently we will be rather brief. 

%
%Again, while the following discussion will greatly guide our proofs, the assertions made in this section aren't completely rigorous since the formal verifications of the KPZ predictions are still missing for Gaussian LPP. Nonetheless, we would be

%i.e., that they demonstrate that the overlap between geodesics at times zero and $t$ is typically small at supercritical times $t \gg n^{-1/3}$ and further the geodesic weight does not alter much. In Section~\ref{s.deductions}, we will give rigorous arguments for the counterpart assertions in the closely related model, dynamical Brownian LPP, that is our object of study. 
\textbf{Low supercritical overlap}: On the basis of the KPZ-inspired scaling ${\rm Var}(M_n) = \Theta(1)n^{2/3}$, alongside~(\ref{e.gaussiandynamicalvariance}) and the decrease in mean overlap, it follows that there is a constant $D>0$ such that for any $t\ge 0,$ 
\begin{equation}\label{e.lowoverlap}
 t e^{-t} \E \mc{O}(\Gamma_n^0,\Gamma_n^t)\le D n^{2/3}
 \end{equation}
  which along with Markov's inequality suffices to show low supercritical overlap. 
%
%Recalling that $\tau = t n^{1/3}$, we learn by a simple application of Markov's inequality that, for  $t \in [n^{-1/3},1]$, 
%\begin{equation}\label{e.lowoverlap}
% \P \Big(  \mc{O}(\Gamma_n^0,\Gamma_n^t) \geq D n \tau^{-1/2} \Big) \leq \tau^{-1/2}. \,  
%\end{equation}
%Note that $te^{-t}$ is maximized at $t=1$, which is the reason why Theorem \ref{t.main} is stated up to $t\le 1.$ Nonetheless, as remarked in \eqref{e.larget}, the above stated monotonicity of the mean overlap allows the same bound to continue to hold for all large times as well.

%The principal technical task of this article is to prove the much more demanding assertion that $\mc{O}_n(t)$ typically has order $n$ when $t$ is smaller than order $n^{-1/3}$. Next is a key input for this subcritical stability, whose role will be indicated in Section \ref{s.conceptsoverlap}.

\textbf{Subcritical energy stability:} A generalization of \eqref{e.gaussiandynamicalvariance} further implies 
\begin{equation}\label{subcrit12}
\E(M^0_n-M_n^t)^2 = 2\int_0^t e^{-s} \E \mc{O}(\Gamma_n^0,\Gamma_n^s) \, {\rm d} s \,=O(nt) \, ,
\end{equation} which in the  subcritical case implies that, except with small probability, $|M^0_n-M_n^t|=o(n^{1/3})=o \big( \vert M^0_n - \E M_n^0) \vert \big)$, or  ${\rm Corr}(M_n^0,M_n^t)=1-o(1).$

There are thus two natural order parameters by which we may seek to verify the transition in  dynamical Brownian LPP at time-scale $t = \Theta(1) n^{-1/3}$. These involve energy and overlap.  Given the above, one may term subcritical energy stability in \eqref{subcrit12} or supercritical low overlap in \eqref{e.lowoverlap} as the ``easy directions'', with establishing any form of dynamical transition requiring the proof of one of the ``hard directions'', namely, subcritical-high-overlap or supercritical-energy-decorrelation.

We pursue the former and prove $\mc{O}(\Gamma_n^0,\Gamma_n^t)\ge dn$ when $t\ll n^{-1/3}$ (for some constant $d>0$), and thereby establish the overlap transition.
%or, more precisely and in scaled language, polymer weight correlation ${\rm Corr}\big(\weight_n^0,\weight_n^t \big)$ and scaled polymer overlap $O_n\big( \rho_n^0, \rho_n^t \big)$, where the shorthand usage $\weight_n^t = \weight_n^t \big[ (0,0) \to (0,1) \big]$ is adopted. These two quantities may be expected to transition from a $\Theta(1)$ level to values close to zero as time $t$ rises from the subcritical $t \ll n^{-1/3}$ to the supercritical $t \gg n^{-1/3}$ phase. 
The other hard direction, i.e., proving ${\rm Corr}(M_n^0,M_n^t)=o(1)$ when $t\gg n^{-1/3}$ remains an attractive open problem that may require significant new ideas.

%It has come to our attention that there is need to be clear about the contribution that the preprint~\cite{ADS} makes relative to that of this article; that the title of~\cite{ADS} almost coincides with ours may engender a reflexive belief in similarity. The text in this section seeks to be clear in this sense: to emphasise, the transition in a single order parameter requires proving two assertions, one of which is easy (given the existing literature), and the other of which is hard. The principal contribution of the submission is to carry out such a hard assertion, via the overlap order parameter. There will then be two ways to complete the task of finding an order parameter whose behaviour heralds the transition to chaos.
%We may seek either to verify supercritical weight decorrelation or subcritical high overlap.

A recent preprint \cite{ADS} studies dynamical versions of more general first and last passage percolation models, in which each variable is refreshed at rate one as in dynamical Bernoulli LPP  
from Section~\ref{s.onset}. Chatterjee~\cite{Chatterjee14}  also considered this type of update, calling it  independent flip dynamics.  Deriving a version of the dynamical formula for variance~\eqref{e.gaussiandynamicalvariance}, \cite{ADS} proves counterparts of  \eqref{e.lowoverlap} and \eqref{subcrit12} for the models it examines. The hard-directions continue to remain open beyond the results of this paper, except, as already mentioned, in the case of random matrices in \cite{BLZ2020} where the availability of a much richer class of tools facilitates the proof of universality of this dynamical transition.

\subsection{Geometric and random LPP tools}\label{geomlpp}
%Indeed, these considerations will readily yield the conclusion $\E \mc{O}_n(t) \ll n$ of low overlap when $t \gg n^{-1/3}$ is supercritical. By far the more demanding task for us is to establish a subcritical stability assertion that $\mc{O}_n(t)$ typically has order $n$ when $t$ is smaller than order $n^{-1/3}$. 
%Chatterjee's theory and KPZ scaling here furnishes some pertinent information. Indeed the above dynamical formula for variance can be used to assert that geodesic energies, such as $M_n$, when viewed as functions of dynamic time $t$, are stable (in a suitable sense) under time increments much smaller than $n^{-1/3}$. 
We finish this section with a brief overview of the remaining tools we will be relying on to implement the central idea in the paper.
As will be elaborated shortly in Section \ref{s.conceptsoverlap}, the key innovation in this paper seeks to capture the energetic and geometric shadow that the geodesic $\Gamma_n(t)$ casts at time zero by a proxy construction.
More precisely, when $t \ll n^{-1/3}$, relying on the stability input in \eqref{subcrit12}, we will build a path---a proxy for 
the time-$t$ geodesic~$\Gamma_n(t)$---that mimics rather closely the route of $\Gamma_n(t)$ along with the property that its time-{\em zero} energy is close to the time-$t$ energy of $\Gamma_n(t)$. This technique will substantially transport time~$t$ geodesic geometry and energy data into the time zero model and will permit results about dynamics to be characterized in terms of the time-zero copy of static Brownian LPP.  
The need then arises to show that certain static events for Brownian LPP are suitably rare. A particularly important event will be the existence of well separated peaks that are close rivals in height in the LPP energy landscape.
%While this will be extensively expanded upon in Section \ref{}, for now, we emphasize that 
Strong quantitative assertions concerning the geodesic energy profiles will be needed for Brownian LPP to obtain probability bounds for the above event. These assertions have become available in recent years due to an important geometric and probabilistic tool for the study of KPZ universality: the Brownian Gibbs resampling technique.
 %A brief review of the history of this technique providing a vital conceptual underpinning for our inquiry can be found in \cite{}.
 For $n \in \N$ and $x \geq 0$, let $M_n(x)$ denote the geodesic energy for the route $(0,0) \to (x,n)$. The energy profile $[0,\infty) \to \R: x \to M_n(x)$
may, by the Robinson-Schensted-Knuth  correspondence, be embedded as the uppermost curve in an $(n+1)$-curve system of mutually avoiding Brownian motions. This system thus satisfies an attractive Brownian Gibbs resampling property that expresses the law of a curve fragment given the ensemble remainder in terms of Brownian bridge conditioned on suitable avoidance. On passing to the limit of high~$n$ one obtains the parabolic Airy line ensemble~\cite{PrahoferSpohn} exhibiting a similar resampling property shown in \cite{AiryLE}. However for our purposes, we will be making important use of the recent work ~\cite{DeBridge}  where Brownian Gibbs analysis has yielded strongly quantified comparisons between scaled energy profiles in Brownian LPP and Brownian motion. 
%As such, the Brownian Gibbs technique is a vital conceptual underpinning for our study.
%: a review of the technique's history and its role in recent advances may be found in the arXiv version~\cite{} of this article at the corresponding place.
%It has played a foundational role in several recent advances which the reader can find a brief account of in \cite{}.

%Beyond the Brownian similarity assertion of~\cite{DeBridge}, 
This study also relies on several aspects of the geometry of near-ground states in Brownian LPP. These have been developed in a companion paper,~\cite{NearGroundStates}. Viewed under KPZ scaling, geodesics have modulus of continuity with H\"older exponent $2/3-$. The results that we rely on from~\cite{NearGroundStates} include a robust, all-scale, assertion
to this effect, and a counterpart concerning energy. Other novel geometric inputs needed from~\cite{NearGroundStates}  include a bound on the typical energetic shortfall incurred by a path that mimics a geodesic route while making lengthy but slender excursions away from it; and a quantitative claim that the geodesic advances in a regular fashion, even microscopically.

 \section{The scaled version of the main result and a sketch of the proof ideas}\label{s.scaledversion}
 
The previously mentioned Brownian Gibbs analysis as well other inputs, say those developed in \cite{NearGroundStates}, are naturally suited to a scaled coordinate system. We will thus be primarily working in these coordinates as well even though the main result Theorem \ref{t.main} is most naturally expressed in unscaled language. 
We develop the scaled notation in a first subsection. The proxy construction mooted in Section \ref{geomlpp}, the technical backbone of the proof of our main result, Theorem~\ref{t.main}, is delicate and involves several steps and we expand on that  as well as other ideas involved in the proof of Theorem~\ref{t.main} next. We will then indicate several main tools that we will need. The section ends with a guide to where these tools will be proved and to the structure of the rest of the paper.

\subsection{Brownian LPP in scaled coordinates}

Recalling the KPZ scaling from Section \ref{s:kpzchar} note that the one-third exponent governs the energetic 
%\\
fluctuation of the geodesic between $(0,0)$ and $(n,n)$: 
%\\
indeed, recalling the definition \eqref{e.energydef}, if we write
\begin{equation}\label{e.weightfirst}
M \big[ (0,0) \to (n,n) \big] = 2n +  n^{1/3} \weight  \big[ (0,0) \to (n,n) \big] \, ,
\end{equation}
then the term $\weight  \big[ (0,0) \to (n,n) \big]$ is a 
%\\
random but unit-order quantity, tight in $n$\footnote{That the leading order term grows as $2n$ follows from that $M \big[ (0,0) \to (n,n) \big]$ is distributed as the leading eigenvalue of an $n\times n$ GUE ensemble (which follows from \cite{Baryshnikov, GTW} and \cite{Grabiner}) and the well known asymptotic behavior of the latter (see e.g. \cite{TracyWidom} or \cite{Ledoux,Aubrun} for tail estimates).}. This is the scaled geodesic energy, which we will call {\em weight}. 
%\\
When geodesic energy $[0,\infty) \to \R: x \to 
M \big[ (0,0) \to (x,n) \big]$ is varied from $x =n$, it is changes of order~$n^{2/3}$ in~$x$ that result in non-trivial correlation. Regarding notation, note that we are presently discussing aspects of the {\em static} 
%\\
Brownian LPP model; our notation carries no superscripts---we refer to $M$, not $M^t$---in 
%\\
accordance with the usage set out in Subsection~\ref{s.dynamicalnotation}.

We will specify scaled coordinates under which the 
%\\
journey between $(0,0)$ and $(n,n)$ corresponds to the 
%\\
unit vertical journey between $(0,0)$ and $(0,1)$, and for which horizontal perturbation of the endpoint $(n,n)$ 
%\\
by magnitude $n^{2/3}$ corresponds to unit-order scaled horizontal perturbation. Moreover, we will 
%\\
associate a scaled energy, or weight, to the image of any path in scaled coordinates. In the next paragraphs, 
%\\
we specify the scaling map $R_n:\R^2 \to \R^2$ whose range specifies scaled coordinates; introduce notation
%\\
for scaled paths; specify the form of scaled energy; and discuss notation for the dynamical version of the scaled 
%\\
model.
\subsubsection{The scaling map.}\label{s.scalingmap}
For $n \in \N$, the $n$-indexed {\em scaling} map $R_n:\R^2 \to \R^2$ is given by
\begin{equation}\label{e.scalingmap}
 R_n \big(v_1,v_2 \big) = \Big( 2^{-1} n^{-2/3}( v_1 - v_2) \, , \,   v_2/n \Big) \, .
\end{equation}
%\\ 
The scaling map acts on subsets $C$ of $\R^2$ with 
%\\
$R_n(C) = \big\{ R_n(x): x \in C \big\}$.
%\\
\subsubsection{Scaling transforms staircases to zigzags.}\label{s.staircasezigzag}
%\\
The image of any staircase under $R_n$
will be called an $n$-zigzag. The starting and ending 
%\\
points of an $n$-zigzag $Z$ are defined to be the image under $R_n$
%\\
of the corresponding points for the staircase $S$ such that $Z = R_n(S)$.
 
Observe that the set of horizontal lines is invariant under $R_n$, while vertical lines are mapped to lines of 
%\\
gradient  $- 2 n^{-1/3}$.
%\\
Thus, an $n$-zigzag is the range of a piecewise affine path from the starting point to the ending point which 
%\\
alternately moves rightwards  along horizontal line 
%\\
segments  and northwesterly along sloping line 
%\\
segments of gradient  $- 2 n^{-1/3}$.

Note for example that, for given real choices of $x$ and 
%\\
$y$, a journey which in the original coordinates occurs 
%\\
between  $(2n^{2/3}x,0)$ and $(n  + 2n^{2/3}y,n)$ takes place 
%\\
in scaled coordinates between $(x,0)$ and $(y,1)$.
%\\
We may view the first coordinate as space and the 
%\\
second as time, though the latter interpretation should not be confused with dynamic time $t$; with this view in 
%\\
mind, the journey at hand is between $x$ and $y$ over the unit lifetime $[0,1]$.

 \begin{figure}[ht]
\begin{center}
\includegraphics[height=7cm]{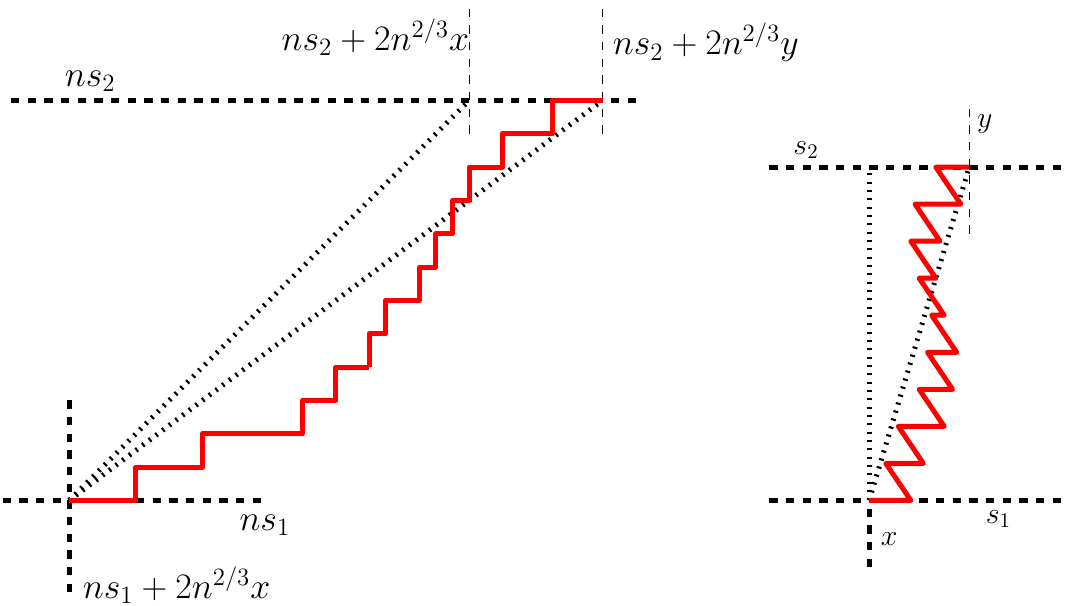}
\caption{Let $(n,s_1,s_2)$ be a compatible triple and let $x, y \in \R$. The endpoints of the geodesic in the left sketch are such that, when the scaling map~$R_n$ is applied to produce the right sketch,  the result is the $n$-polymer $\rho_n \big[ (x,s_1) \to (y,s_2) \big]$ from $(x,s_1)$ to $(y,s_2)$.
  }
\label{f.scaling}
\end{center}
\end{figure}  

 \subsubsection{Compatible triples}\label{s.compatible}
 Let $(n,s_1,s_2) \in \N \times \R_\leq^2$, 
%\\
 where we write $\R_\leq^2 = \big\{ (s_1,s_2) \in \R^2: s_1 \leq s_2 \big\}$.
%\\
 Taking $x,y \in \R$, 
does there exist an $n$-zigzag from $(x,s_1)$ and $
%\\
(y,s_2)$? Two conditions must be satisfied if there is to be.
%\\

First: as far as the data $(n,s_1,s_2)$ is concerned, such an $n$-zigzag may exist only if 
%\\
 \begin{equation}\label{e.ctprop}
     \textrm{$s_1$ and $s_2$ are integer multiplies of $n^{-1}$} \, .
\end{equation}
%\\
We say that data $(n,s_1,s_2)  \in \N \times \R_\leq^2$ 
%\\
is a {\em compatible triple} if it verifies the last condition. 
We will consistently impose this condition, whenever we 
%\\
seek to study $n$-zigzags whose lifetime is $[s_1,s_2]$.
The use of compatible triples should be considered to 
%\\
be a fairly minor detail. As the index $n$ increases, the $n^{-1}$-mesh becomes finer, so that the space of $n$-
%\\
zigzags better approximates a field of functions, defined on arbitrary finite intervals of the vertical coordinate, and taking values in the horizontal coordinate.  
%\\

%\\
Associated to a compatible triple is the notation $\tot$, which will denote the difference $s_2 - s_1$. The law of 
%\\
the underlying Brownian ensemble $B: \R \times \Z  \to \R$ is invariant under integer shifts in the latter, curve 
%\\
indexing, coordinate. This translates to an invariance in law of scaled objects under vertical shifts by multiples of $n^{-1}$, something that makes the parameter $\tot$
%\\
of far greater relevance than the individual values $s_1$ or $s_2$.

%\\
Returning to the above posed question, the second 
%\\
needed condition is that the horizontal coordinate of the unscaled counterpart of the latter endpoint must be at 
%\\
least the former, which translates into requiring,
\begin{equation}\label{e.xycond}
y -x \geq - 2^{-1}n^{1/3} \tot \, . 
\end{equation}

\subsubsection{Zigzag subpaths}\label{s.zigzagsubpaths}
%\\
Let $\phi$ denote an $n$-zigzag between elements $(x,s_1)$ and $(y,s_2)$ in $\R \times n^{-1}\Z$.
%\\
Let $(u,s_3)$ and $(v,s_4)$ be elements in $\phi \cap \big( \R \times [s_1,s_2] \cap n^{-1}\Z  \big)$. Suppose that $s_3 \leq s_4$ (and that $u \leq v$ if equality here holds), so 
%\\
that $(u,s_3)$ is encountered before $(v,s_4)$ in the 
%\\
journey along $\phi$. The removal of $(u,s_3)$ and $(v,s_4)$  from $\phi$ results in three connected 
%\\
components. The closure of one of these contains  contains these two points and this closure will be 
%\\
denoted by $\phi_{(u,s_3) \to (v,s_4)}$. This is the zigzag subpath, or sub-zigzag, of $\phi$ between  $(u,s_3)$ and $(v,s_4)$.
%\\

\subsubsection{Staircase energy scales to zigzag weight.}
Let $n \in \N$ and $i,j \in \N$ satisfy $i < j$.
%\\
Any $n$-zigzag $Z$ from $(x,i/n)$ to $(y,j/n)$  is ascribed a scaled energy, which we will refer to as its 
%\\
weight, 
$\weight(Z) = \weight_n(Z)$, given by 
\begin{equation}\label{e.weightzigzag}
 \weight(Z) =  2^{-1/2} n^{-1/3} \Big( E(S) - 2(j - i)  - 2n^{2/3}(y-x) \Big) 
\end{equation}
%\\
where $Z$ is the image under $R_n$ of the staircase $S$.

\subsubsection{Maximum weight.}\label{s.maxweight} Let $n \in \N$. The quantity
%\\ 
 $\weight_n \big[ (0,0) \to (0,1) \big]$ specified in~(\ref{e.weightfirst}) is simply
%\\
 the maximum weight ascribed to any $n$-zigzag from $(0,0)$ to $(0,1)$.
 
 Let $(n,s_1,s_2) \in \N \times \R_\leq^2$ be a compatible triple. 
%\\
Suppose that $x,y \in \R$ satisfy 
 $y \geq x - 2^{-1} n^{1/3} \tot$. We now define  $\weight_n \big[ (x,s_1) \to (y,s_2) \big]$ such that this 
 %\\
 quantity equals the maximum weight of any $n$-zigzag from $(x,s_1)$ to $(y,s_2)$.
 We must set 
 %\\
\begin{eqnarray}
 & &  \weight_n \big[ (x,s_1) \to (y,s_2) \big]  \label{e.weightgeneral} \\
 &    =  &  2^{-1/2} n^{-1/3} \Big(  M \big[ (n s_1 + 2n^{2/3}x,n s_1) \to (n s_2 + 2n^{2/3}y,n s_2) \big] - 2n \tot -  2n^{2/3}(y-x) \Big) \, . \nonumber
\end{eqnarray}

The quantity  $\weight_n \big[ (x,s_1) \to (y,s_2) \big]$ 
%\\
may be expected to be, for given real choices of $x$ and $y$ that differ by order $\tot^{2/3}$, a unit-order 
%\\
random quantity; this collection of random variables is tight in the scaling parameter $n \in \N$ and in such 
%\\
choices of $s_1, s_2 \in n^{-1}\Z$ and $x,y \in \R$.

 \subsubsection{Highest weight zigzags are called polymers.}\label{s.polymer}
 An $n$-zigzag that attains the maximum weight given its endpoints will be called an $n$-polymer, or, usually, 
%\\ 
simply a polymer.
 Thus, geodesics map to polymers under the scaling map.
 %\\ 
As we recalled in Subsection~\ref{s.geodesics}, the geodesic with any given pair of endpoints is almost surely unique. For $x,y \in \R$ and $(n,s_1,s_2) \in \N \times \R_\leq^2$ a compatible triple, 
%\\
we will write~$\rho_n \big[ (x,s_1) \to (y,s_2) \big]$ for 
%\\
the almost surely unique $n$-polymer from $(x,s_1)$ to $(y,s_2)$; see Figure~\ref{f.scaling}. 

Though not standard, since the term `polymer' is often used to refer to typical realizations of the path measure 
%\\
in LPP models at positive temperature, the above 
usage of the term `polymer' for `scaled geodesic' is 
%\\  
 apt for our study, owing to the central role played by these objects.

\subsubsection{Zigzags as near functions of the vertical coordinate}\label{s.zigzagfunction}
%\\
Suppose again that $\phi$ is an $n$-zigzag between points $(x,s_1)$ and $(y,s_2)$ in $\R \times n^{-1}\Z$.  
%\\
For $s \in [s_1,s_2] \cap n^{-1}\Z$, we will write $\phi(s)$
for the supremum of values $x \in \R$ for which $(x,s) \in \phi$. This abuse of notation permits $\phi(s)$ to denote 
%\\
the horizontal coordinate of the point of departure from vertical coordinate $s$
%\\
in the journey along $\phi$ from $(x,s_1)$ to $(y,s_2)$. This convention is adopted partly because it captures 
%\\
the notion that the typical zigzags~$\phi$ we will consider---polymers or concatenations thereof---are 
%\\
closely approximable by a real-valued function of the vertical coordinate $s \in [s_1,s_2]$, at least when $n$ 
%\\
is high---indeed, the maximum length of the horizontal line segments in an $n$-polymer is readily seen to decay to zero in~$n$ with high probability. 
%\\
(Corollary~\ref{c.notallcliffs}, a result proved in~\cite{NearGroundStates}, quantifies this assertion.)

\subsubsection{Dynamical Brownian LPP in scaled notation} 
We have introduced notation for several aspects of static Brownian LPP in scaled coordinates. Notation for the dynamical model is inherited via the convention laid out in Subsection~\ref{s.dynamicalnotation}.
Thus, $\rho_n^t \big[ (x,s_1) \to (u,s_2) \big]$ is the $n$-polymer from $(x,s_1)$ to $(y,s_2)$
in the copy of Brownian LPP offered by the marginal noise environment $B(\cdot,\cdot,t)$. This polymer has weight $\weight_n^t \big[ (x,s_1) \to (u,s_2) \big]$ according to that noise environment. We make two comments about this notation.

First, notation that includes the superscript is not fully scaled, in the sense that the dynamic time parameter $t \geq 0$ is unscaled. A fully scaled notation would write $\rho_n^\tau$ in place of $\rho_n^t$, where $\tau = t n^{1/3}$ is scaled time. Indeed, an inviting prospect is the study of the dynamical KPZ scaling limit formally specified by the random fields $\weight_\infty^\tau$ and $\rho_\infty^\tau$ indexed by pairs of planar points (with unequal vertical coordinates) as $\tau \geq 0$ varies. Theorem~\ref{t.main} may represent a significant first step in this study.  

Second, we mention a source of confusion regarding notation specific to the study of dynamical LPP models to which we alluded briefly in Subsection~\ref{s.staircasezigzag}. In static Brownian LPP, it is sometimes natural to regard the vertical planar coordinate as time, and, for example, to speak of the lifetime $[s_1,s_2]$ of the polymer $\rho_n \big[ (x,s_1) \to (y,s_2) \big]$.
This usage may however conflict with reference to dynamic time $t$. In order to alleviate confusion, we reserve the letter $s$ for reference to the vertical coordinate, and $t$ for dynamic time. We make occasional reference to the vertical coordinate in the temporal sense. These usages are limited to the term `lifetime' in the above sense; to the `duration' $\tot = s_2 - s_1$ of such a lifetime; and to the coordinate $s$ of a planar point $(x,s)$, which we sometimes call the `moment'~$s$.    

For the ease of readability, we restate our principal conclusion Theorem~\ref{t.main} in scaled language next.
 \subsection{The main result scaled}\label{s.resultscaled}
 
Indeed, let $\rho$ and $\phi$ denote two $n$-zigzags. The set $\rho \, \cap \, \phi \, \cap \big( \R \times n^{-1} \Z \big)$ is the intersection of the union of the horizontal segments of $\rho$ with the counterpart union for $\phi$. Define the scaled overlap $O_n(\rho,\phi) \in [0,\infty)$ to be the product of $2 n^{-1/3}$ and the one-dimensional Lebesgue measure of the set $\rho \, \cap \, \phi \, \cap \big( \R \times n^{-1} \Z \big)$. Note from the form~(\ref{e.scalingmap}) of the scaling map $R_n$
that this scaled overlap is equal to $n^{-1} \mc{O}\big(R_n^{-1}(\rho),R_n^{-1}(\phi)\big)$; namely, it is the $(1/n)\textsuperscript{th}$ multiple of the overlap of the preimage staircases $R_n^{-1}(\rho)$ and $R_n^{-1}(\phi)$ specified in Section~\ref{s.mainresult}. Thus, for example, $O_n(\rho,\rho) = 1$ for any $n$-zigzag~$\rho$ from $(0,0)$ to $(0,1)$.

 For $t \geq 0$, the time-$t$ polymer $\rho_n^t \big[ (0,0) \to (0,1) \big]$ will be denoted by the shorthand notation $\rho_n^t$. (The unscaled journey is from $(0,0)$ to $(n,n)$, so that this shorthand is consistent with the usage made in Section~\ref{s.onset}.) The dynamic scaled overlap function is specified to be $[0,\infty) \to [0,\infty): t \to O_n\big( \rho_n^0, \rho_n^t \big)$.  

\begin{theorem}\label{t.main.scaled}
$\empty$
\begin{enumerate}
\item 
There exist $d \in (0,1)$ and $n_0 \in \N$ such that, for 
$\lambda > 0$, we may find $h > 0$ for which $t \in \big[0,n^{-1/3} \exp \big\{ - h (\log \log n)^{68} \big\} \big]$
and $n \geq n_0$ imply that
$$
 \PP \Big( O_n\big( \rho_n^0, \rho_n^t \big) \geq d  \Big) \geq 1 - (\log n)^{-\lambda}  \, .
$$
\item 
There exists a constant $D > 0$ such that, for $n^{-1/3}< t \leq 1$, 
$$
 \PP \Big(   O_n\big( \rho_n^0, \rho_n^t \big)  \leq D \tau^{-1/2}  \Big) \geq 1 - \tau^{-1/2} \, ,
$$
where $\tau \in (1,n^{1/3}]$ is specified by $t = n^{-1/3} \tau$. 
\end{enumerate}
\end{theorem}

\subsection{Concepts for proving high subcritical overlap via proxy construction}\label{s.conceptsoverlap}
% Recalling the discussion in Section \ref{s.lowoverlapgaussian} in scaled language, the two natural order parameters to gauge the dynamical transition at time-scale $t = \Theta(1) n^{-1/3}$ are weight correlation ${\rm Corr}\big(\weight_n^0,\weight_n^t \big)$ and polymer overlap $O_n\big( \rho_n^0, \rho_n^t \big)$, where the shorthand $\weight_n^t = \weight_n^t \big[ (0,0) \to (0,1) \big]$ is used. These quantities should drop from unit order to close to zero between the subcritical $t \ll n^{-1/3}$ and the supercritical $t \gg n^{-1/3}$ phases. 

%It is not too difficult to adapt the subcritical weight stability assertion \eqref{subcrit12} to this setting leading to ${\rm Corr}\big(\weight_n^0,\weight_n^t \big) = 1-o(1)$ for $t \ll n^{-1/3}$, while, the consequence \eqref{e.lowoverlap} of the dynamical formula for variance~(\ref{e.gaussiandynamicalvariance}) makes it plausible to obtain supercritical low overlap in the sense that that $O_n\big( \rho_n^0, \rho_n^t \big) \ll 1$ is typical when $t \gg n^{-1/3}$.
As indicated in Section \ref{s.lowoverlapgaussian}, by far the most substantial technical contribution of this article  is the proof of Theorem~\ref{t.main.scaled}(1) which is an assertion of high subcritical overlap.
In this section, we provide a detailed overview of the proof. 

%Given that Theorem~\ref{t.main.scaled}(2) will be a straightforward consequence of the Brownian LPP version of the dynamical formula for variance as indicated in \eqref{e.lowoverlap},
%the derivation of Theorem~\ref{t.main.scaled}(1) is the principal technical task of this article. Below we provide a somewhat detailed overview of the several ideas involved.
%Theorem~\ref{t.main.scaled}(1). 

During the overview, we naturally fix time $t = \tau n^{-1/3}$ at a subcritical value, so that the scaled time parameter~$\tau$ satisfies $\tau \ll 1$.  Recall the two  inputs discussed in more or less precise terms in Sections \ref{s.lowoverlapgaussian} and \ref{geomlpp}:
% We now present the actual form in which we will use them 
%It will be helpful to recall in the ensuing discussion two assertions for which we have presented an informal case:
\begin{enumerate}
\item Subcritical weight stability: a weight such as $\weight_n^{t'} \big[ (x,0) \to (y,1) \big]$ for given unit-order $x,y \in \R$ typically varies little relative to its initial value between times $t' =0$ and $t' =t$.   Indeed,~(\ref{subcrit12}) indicates that the mean squared difference between these $t' = 0$ and $t' = t$ weights is at most $\Theta(1) t n^{1/3}$. That is,
\begin{equation}\label{e.typicaltau}
\textrm{$\weight_n^t \big[ (x,0) \to (y,1) \big]$ and $\weight_n^0 \big[ (x,0) \to (y,1) \big]$ typically differ by order $\tau^{1/2}$.}
\end{equation}
\item Brownian resemblance for routed polymer weight profiles: for any given $a \in n^{-1}\Z \cap (0,1)$, the random function 
\begin{equation}\label{e.routedprofile}
x \to Z_n(x,a) = \weight_n \big[ (0,0) \to (x,a) \big] + \weight_n \big[ (x,a) \to (0,1) \big]
\end{equation}
 bears a strong resemblance to  Brownian motion. Thus, so is the profile specified by this formula with the replacement  $\weight_n \to \weight_n^t$ for any given time $t \geq 0$.
\end{enumerate} 
As indicated rather quickly in Section \ref{geomlpp}, the broad approach in our arguments is to track the shadow of a time-$t$ event at time zero.
Namely, we will demonstrate that the event of low subcritical overlap---that $O_n\big( \rho_n^0, \rho_n^t \big)$ is at most one-hundredth, say---typically forces the occurrence of a static event, at time zero, which we show to be rare. 

Our proof will classify the event of low subcritical overlap according to several cases for the geometric relationship between the polymers $\rho_n^0$ and $\rho_n^t$; in different cases, different rare time-zero events will be shown to be forced.
These time-zero events share a certain feature, however---that there exists a zigzag whose geometry is substantially different from that of $\rho_n^0$ but which is a near polymer at time zero, in the sense that its time-zero weight differs from $\rho_n^0$'s by an insignificant margin of error. We now expand more on this overarching theme. First, we will indicate more precisely the form of the time-zero event that we will utilise, and explain why it is a rare event. Then we will indicate two important cases for the relative geometry of $\rho_n^0$ and $\rho_n^t$ in the event that these polymers experience low overlap and how that forces a rare time-zero event. 
\subsubsection{The static rarity of a near polymer that significantly escapes the polymer's route}\label{s.staticrarity}
The location of time-zero polymer $\rho_n^0$ between $(0,0)$ and $(0,1)$ at the mid-life time $1/2$ is a unit-order quantity which for convenience we suppose here to exceed the value one; that is, we suppose that $(x,1/2) \in \rho_n^0$ for some $x \geq 1$. A time-zero near polymer is a zigzag $\phi$ between $(0,0)$ and $(0,1)$ whose weight is close to that of $\rho_n^0$, characterized by the condition that $\weight^0(\phi) \geq \weight_n^0 - {\rm error}$, where ${\rm error} = o(1)$ is a given small quantity. What is the probability that such a near polymer exists whose geometry differs substantially from $\rho_n^0$'s?
While delicate geometric possibilities need to be analyzed in the proof, as a warm up, let us consider a vanilla version: what is the probability of the existence of a near polymer $\phi$ that visits $(-\infty,-1]$ at the mid-life time; i.e., that such $\phi$ exists for which $(y,1/2) \in \phi$ for some $y \leq -1$? 

Brownian resemblance for the routed polymer weight profile provides an answer to this question. In the event under discussion, the routed weight profile $Z = Z_n(\cdot,1/2): \R \to \R$
 in~(\ref{e.routedprofile}) is maximized at the value $ x \geq 1$. But it is also nearly maximized at the value $y \leq -1$, since $Z(y) = Z(x) - {\rm error} = Z(x) - o(1)$. Our mooted tool of Brownian resemblance indicates that the probability of this {\em twin peaks} circumstance is accurately modelled by the probability that Brownian motion $B: [-2,2] \to \R$
satisfies the condition that the suprema of $B$ on the intervals $[-2,-1]$ and $[1,2]$ differ by at most a small quantity ${\rm error} = o(1)$.  (The process $B$ has been restricted to the interval $[-2,2]$ because the profile $Z:\R \to \R$ globally tracks a parabola of the form $-2^{3/2} x^2$; thus, the possibility that it is nearly maximized at a value outside a compact set may be neglected.) This Brownian probability of near coincidence of suprema is readily seen to have the form $\Theta(1) \cdot {\rm error}$. 

Thus we see that the tool of Brownian resemblance indicates that the probability of a static near polymer, of weight deficit ${\rm error} = o(1)$, with substantial mid-life deviation from the route of the polymer, is at most $\Theta(1)\cdot  {\rm error} = o(1)$.

\subsubsection{Low subcritical overlap: the case of one long excursion}\label{s.oneexcursion}
Take the event of low overlap to be  $O_n\big( \rho_n^0 , \rho_n^t \big) \leq 1/2$, say. 
We here consider a simple and instructive special case (or subevent) of low overlap: 
that the time-zero polymer $\rho_n^0$ adopts a globally rightward trajectory, intersecting the planar interval $[1,\infty) \times \{1/2 \}$; and that its time-$t$ counterpart $\rho_n^t$
instead moves leftward, intersecting the interval $(-\infty,-1] \times \{1/2\}$. In essence, this event captures the case that low overlap is accounted for by a single interval of excursion of $\rho_n^t$ away (and to the left of) $\rho_n^0$ whose duration takes the form of a macroscopic subinterval of $[0,1]$ (that contains the time one-half): see Figure~\ref{f.proxy}(left). 

We want to argue that this scenario typically forces a rare time zero event of the twin peaks kind whose probability has just been bounded above. The most direct means of seeking to verify this is simply to consider the time-$t$ polymer as a candidate for a near polymer at time zero. The discussion in Section~\ref{s.onset} is pertinent for evaluating how viable this approach is. We argued there that, up to  $t \ll n^{-2/3}$, a typical initial polymer will suffer no significant weight change. For such choices of $t$, $\rho_n^t$ is thus plausibly a time-zero near polymer, so that the event of a substantially separated near-maximizer in the routed weight profile $Z = Z_n \big( \cdot, 1/2 \big)$ is indeed forced. 

However for a given time $t$ on the longer subcritical scale $t \ll n^{-1/3}$, the above reasoning fails, since  
as we indicated in Section~\ref{s.onset}, beyond the shorter scale $t = O(1) n^{-2/3}$, the time-zero weight of $\rho_n^t$ will be hopelessly uncompetitive. Nonetheless, despite the naive strategy failing at this point, one of the key innovations in this article is the development of an approach to resolve this difficulty. 

The main input that we employ is the subcritical weight stability estimate (\ref{e.typicaltau}). 
In particular, writing $(y,1/2)$ with $y \leq -1$ for an element of $\rho_n^t$, it seems to follow from three applications of (\ref{e.typicaltau}) that
\begin{eqnarray*}
 &  & \weight_n^{t'} \big[ (0,0) \to (0,1) \big] \, , \, \weight_n^{t'} \big[ (0,0) \to (y,1/2) \big] \, \, \textrm{and} \, \, \weight_n^{t'} \big[ (y,1/2) \to (0,1) \big] \\
  & & \qquad \ \textrm{typically vary between $t' =0$ and $t' = t$
 by order $\tau^{1/2}$.}   
\end{eqnarray*}
There is a flaw in this reasoning: $(y,1/2)$ is a special planar point, selected to lie on $\rho_n^t$, so one of the endpoints in~(\ref{subcrit12}) is not given in two of the applications. Handling this precisely will lead to a deteriorated bound of the form $\tau^{\alpha}$ for some $\alpha<1/2,$ but we will not address this difficulty here.

We will exploit the displayed stability effect to construct a zigzag whose geometry mimics $\rho_n^t$'s but whose time-zero weight is close to the maximizer $\rho_n^0$'s.  We label this construct $\rho_n^{t \to 0}$ and  call it the time-zero {\em proxy} of $\rho_n^t$; see Figure~\ref{f.proxy}(left). 
The proxy is the zigzag between $(0,0)$ and $(0,1)$ that has maximum weight at time zero subject to passing through $(y,1/2)$. As such, $\rho_n^{t \to 0}$ equals the union of 
$\rho_n^0 \big[ (0,0) \to (y,1/2) \big]$ and $\rho_n^0 \big[ (y,1/2) \to (0,1) \big]$. The proxy mimics the pertinent geometric discrepancy of $\rho_n^t$ from $\rho_n^0$ by passing through the leftward location $y$ at the midlife time one-half. But it also accurately mimics in its time-zero weight the time-zero maximum weight $\weight_n^0$ for the journey from $(0,0)$ to $(0,1)$. Indeed, we claim that, on the last displayed typical event,
\begin{equation}\label{e.mimic}
 \weight_n^0 \big( \rho_n^{t \to 0} \big) - \weight_n^0 \big[ (0,0) \to (0,1) \big] \, \,  \textrm{has order $\tau^{1/2}$.}
\end{equation}
This claim is confirmed by noting that, on the same typical event,  
\begin{eqnarray*}
 \weight_n^0 \big( \rho_n^{t \to 0} \big)  & = &  \weight_n^0 \big[ (0,0) \to (y,1/2) \big] + \weight_n^0 \big[ (y,1/2) \to (0,1) \big] \\
  & = &   \weight_n^t \big[ (0,0) \to (y,1/2) \big] + \weight_n^t \big[ (y,1/2) \to (0,1) \big] + \Theta(\tau^{1/2}) \\
  & = & \weight^t_n(\rho_n^t) + \Theta(\tau^{1/2}) =  \weight_n^t \big[ (0,0) \to (0,1) \big] + \Theta(\tau^{1/2}) \\
  &  = & \weight_n^0 \big[ (0,0) \to (0,1) \big] + \Theta(\tau^{1/2})  \, ,
\end{eqnarray*}
where weight stability along each of the three journeys addressed by the typical event has been invoked.

\begin{figure}[t]
\centering{\epsfig{file=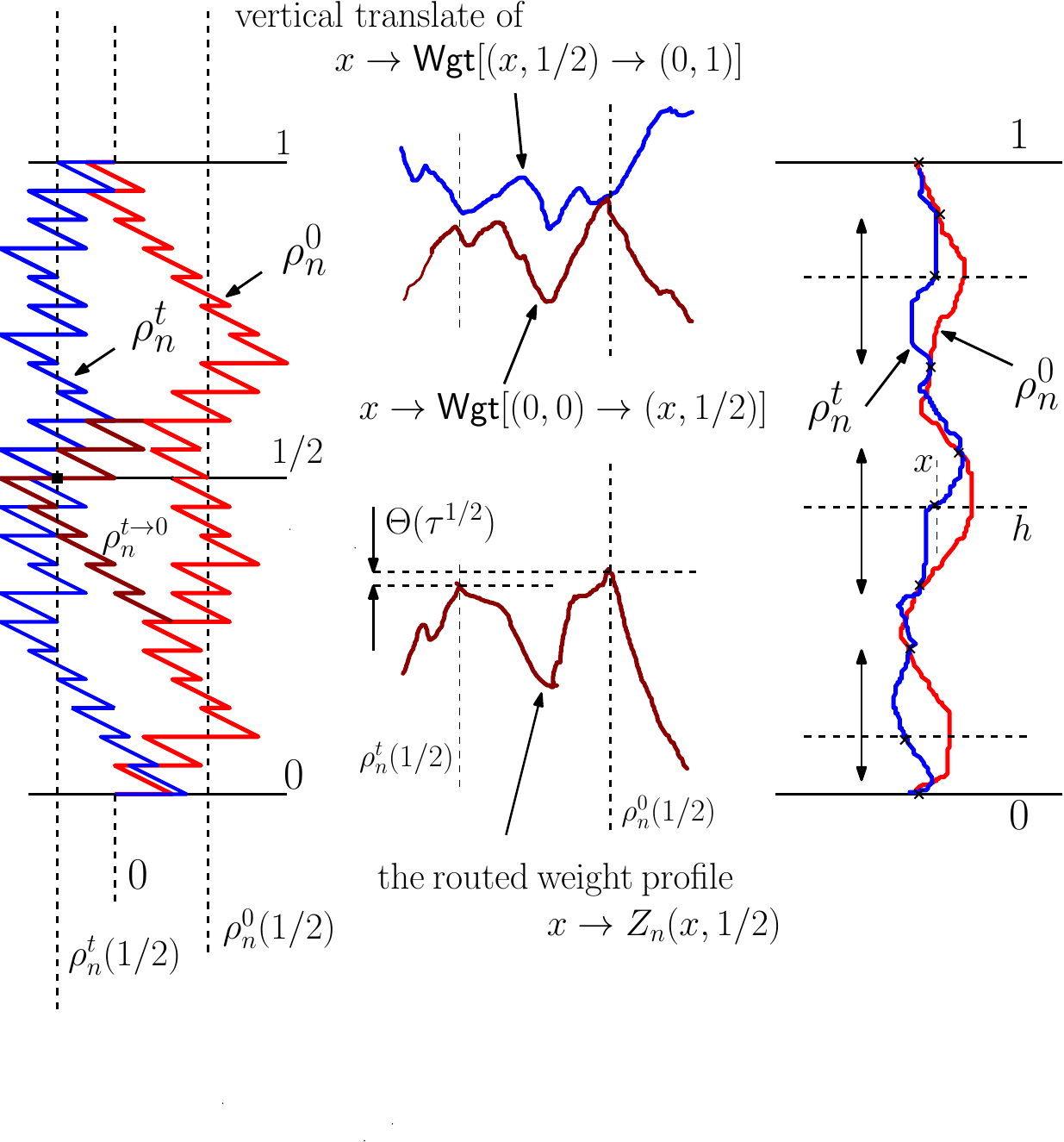, scale=0.53}}
\caption{In the left sketch, there is one long excursion between the red polymer at time zero and the blue polymer at time $t$, so that the brown proxy $\rho_n^{t \to 0}$ (which merges at both ends with the red curve) is composed of two time-zero polymers that abut an element (marked with a small square) of $\rho_n^t$ at height one-half. The middle sketch illustrates behaviour for certain weight profiles that is typically consistent with the left sketch:  above, {\em near touch}, namely a close encounter between the indicated transformations of narrow weight profiles routed at $(0,0)$ or $(0,1)$;  below, and equivalently, {\em twin peaks}, namely a value at $\rho_n^t(1/2)$ for the time-zero routed weight profile $x \to Z^0_n(x,1/2)$
rivalling the peak at $\rho_n^0(1/2)$ to a height of the order $\tau^{1/2}$ seen in~(\ref{e.mimic}).
 The right sketch indicates  the polymers at time zero and time~$t$
in the case of several excursions of roughly equal duration that is the subject of Subsection~\ref{s.severalexcursions}. The three vertical double-arrowed intervals illustrate excursions of duration of order~$2^{-\ell}$. The proxy $\rho_n^{t \to 0}$
is not drawn, but it interpolates by means of time-zero polymers the various small crosses.
Let $h$ denote the $y$-coordinate of one of the displayed dotted horizontal intervals~$I$; and let $x$ denote the horizontal coordinate of the cross that lies on $I$. Associated to $I$ is a shortfall in weight 
of the proxy relative to the time-zero polymer $\rho_n^0$, since the value
$Z^0_n(x,h)$ lies below the supremum
of the routed weight profile~$Z^0_n(\cdot,h)$.}\label{f.proxy}
\end{figure} 

We see then how the proxy $\rho_n^{t \to 0}$ lives as a near polymer---with weight deficit of order $\tau^{1/2}$--- alongside~$\rho_n^0$ at time zero. We are considering the case that $\rho_n^0$ intersects  $[1,\infty) \times \{1/2 \}$, and that $\rho_n^t$ intersects $(-\infty,-1] \times \{1/2\}$; but the latter intersection is maintained by $\rho_n^{t \to 0}$ in light of the proxy's construction. Thus (\ref{e.mimic}) implies the twin peaks event that the routed profile $Z = Z_n \big( \cdot, 1/2 \big)$ from~(\ref{e.routedprofile}) has suprema on the intervals $(-\infty,-1]$ and $[1,\infty)$
that differ by order $\tau^{1/2}$. It is not hard to show that with probability $\Theta(1) \tau^{1/2}$ Brownian motion of rate two on say $[-2,2]$ has suprema on $[-2,-1]$ and $[1,2]$ that differ in this manner. But given the strong Brownian resemblance of the routed weight profile indicated around~(\ref{e.routedprofile}), this inference in essence transmits to the profile~$Z$.
Thus a special case of low overlap between $\rho_n^0$ and $\rho_n^t$ has heuristically been shown to be rare. We may summarise the obtained inference by the informally expressed bound
$$
 \PP \Big( O_n\big(\rho_n^0,\rho_n^t \big) \leq 1/2 \, , \, \textrm{$\rho_n^t$ avoids $\rho_n^0$ for a unit-order duration around time $1/2$} \Big) \leq \Theta(\tau^{1/2}) \, .
$$

\subsubsection{Low subcritical overlap: the case of several excursions of roughly equal duration}\label{s.severalexcursions}

Low overlap may occur in the form of not merely one long excursion of $\rho_n^t$ relative to $\rho_n^0$ but of several such excursions.
We next offer an argument treating the case of several excursions of roughly equal duration. This duration will have scale $2^{-\ell}$ for a dyadic scale parameter $\ell \in \N$
that is supposed to satisfy $2^{-\ell} \gg n^{-1}$, in order to ensure that excursion duration is greater than the microscopic vertical scale. Indeed in the actual proof, we will choose a scale $\ell$ of excursions that dominate the low-overlap and hence one can think of the preceding case of one single excursions as simply corresponding to $\ell=O(1).$
%The argument to be offered in this {\em long excursions' case} will inspire our rigorous treatment more directly than the preceding analysis of a single excursion and in fact will  subsume the latter case. 
(The name `long excursion' may thus seem to be a misnomer, because the argument will treat excursions much shorter than the single excursion already discussed. But the names `long' and `short' are convenient monikers for the two cases that drive our analysis with the latter reserved for the case of microscopic excursions of size comparable to $1/n.$)
The latter case where a macroscopic fraction of the vertical interval $[0,1]$ is populated by excursions between $\rho_n^0$ and $\rho_n^t$ whose durations barely exceed the microscopic scale $n^{-1}$, will be treated separately.

It is useful to now provide a more precise specification of excursions---the actual prescription, which entails some further detail, appears in Definition \ref{d.excursion}. 
The symmetric difference $\rho_n^0 \Delta \rho_n^t$ consists of sub-zigzags of $\rho_n^0$ and $\rho_n^t$ that may be paired up when the starting and ending points are shared. The closure of the union of paths in a pair forms a connected set---and this is what we call an excursion. (See the left sketch in Figure \ref{f.proxy}, and Figure \ref{f.excursion}, for illustrations.) The height, or duration, of an excursion is the length of the interval of vertical coordinates that it occupies. Consider an event of low overlap, that $O_n \big(\rho_n^0,\rho_n^t \big)$ is at most one-half, say. For a dyadic scale parameter $\ell \in \N$, we will consider 
an event under which such low overlap is largely due to excursions whose height lies in $[2^{-\ell-1},2^{-\ell}]$; to be specific, the event that the summed height of such excursions is at least one-quarter, say. In practice, non-overlap may occur due to excursions populating several different dyadic scales, and by pigeonhole principle, some scale contributes a $1/\log n$ fraction (which will suffice for our argument) but presently we work in the more extreme case.

Consider given $\ell \in \N$ for which $2^\ell \ll n$; in our later rigorous argument, this hypothesis will essentially take the form that 
\begin{equation}\label{e.inverselog}
\textrm{$2^{-\ell}n$ grows faster than any power of $\log n$} \, .
\end{equation}
The precise condition appears in \eqref{e.proxynlowerbound}.
We {\em construct a proxy} $\rho^{t \to 0}_n$ to the time-$t$ polymer $\rho_n^t$ that mimics geometry 
at vertical scale $2^{-\ell}$ (note that $\ell$ in the application will be random).  To do so, we mark points along $\rho_n^t$ at vertical separation given by a small constant multiple of $2^{-\ell}$. (In fact, in the rigorous implementation, the separation will be taken to be $o(2^{-\ell})$ to ensure that the geometry is mimicked with high probability; 
a detail we ignore here for the sake of exposition.)
The proxy $\rho_n^{t \to 0}$ is defined to be the union of the time-zero polymers between consecutive pairs of marked points. 

The proxy will demonstrate the rarity of the event of low overlap via the dominant mechanism of scale-$\ell$ excursions if we argue in favour of two bounds: first, an assertion of weight mimicry of the time-$t$ polymer on the part of the proxy; namely, for a positive exponent $a$, 
\begin{equation}\label{e.proxymimicry}
 \big\vert \weight_n^0 \big( \rho_n^{t \to 0} \big) - \weight^t_n \big( \rho_n^t \big) \big\vert \, \, \, \textrm{typically has order at most $2^{2\ell/3} \tau^a$} \, .
\end{equation}
Second, that the proxy's time-zero weight is far below the maximum:
\begin{equation}\label{e.proxytimezero}
\weight_n^0 \big( \rho_n^{t \to 0} \big) - \weight^0_n \big( \rho_n^0 \big) \, \, \, \textrm{is typically negative and of order at least $2^{2\ell/3}$} \, .
\end{equation}
For all but rather small $\ell \in \N$,  the weight mimicry~(\ref{e.proxymimicry}) of the proxy is not strong enough to render the proxy a near polymer at time-zero as its counterpart was in the case of one long excursion. However, since $\tau \ll 1$, this mimicry is adequate when
allied with~(\ref{e.proxytimezero}) 
 to force the proxy's weight to be much closer to $\weight_n^t \big( \rho_n^t \big)$ than to $\weight^0_n \big( \rho_n^t \big)$. Consequently, when the two displayed bounds hold, 
  \begin{equation}\label{e.wdlb}
  \weight_n^t(\rho_n^t) - \weight^0_n(\rho_n^0) \geq \Theta(1) 2^{2\ell/3} \, . 
  \end{equation}
  Why is this
  improbable? The concerned weights are unit-order random quantities with tails governed by the GUE Tracy-Widom distribution, so we have an answer unless $\ell \in \N$ is low. But we want to include the case of low $\ell$. We thus instead invoke 
  subcritical weight stability~(\ref{e.typicaltau}) with $x=y=0$ to find that $\weight_n^t(\rho_n^t) - \weight^0_n(\rho_n^0)$  is typically at most of order $\tau^{1/2}$.   Since $\tau \ll 1$, this is inconsistent with~(\ref{e.wdlb}), so we have a conclusion of the desired form.

We have reduced then to explaining why (\ref{e.proxymimicry}) and~(\ref{e.proxytimezero}) hold. Regarding~(\ref{e.proxymimicry}), 
we will present an  argument with some aspects fallacious which nonetheless capture the key ideas.  Note that   
$\big\vert \weight_n^0 \big( \rho_n^{t \to 0} \big) - \weight^t_n \big( \rho_n^t \big) \big\vert$ is bounded above by a sum of order $2^\ell$ terms, each of which has the form
$\omega_i = \big\vert \weight_n^0 \big[ (x_i,s_i) \to (x_{i+1},s_{i+1}) \big] -  \weight_n^t \big[ (x_i,s_i) \to (x_{i+1},s_{i+1}) \big]  \big\vert$ where the planar points $(x_i,s_i)$ and $(x_{i+1},s_{i+1})$  are a consecutive pair used in the interpolative construction of the proxy. Thus, $\vert s_{i+1} - s_i \vert$ has order $2^{-\ell}$. 
It would seem that the typical value of $\omega_i$ can be inferred from the assertion of subcritical weight stability, which was used crucially in the preceding subsection. In~(\ref{e.typicaltau}), we indicated the order of the typical weight difference of time-zero and time-$t$ polymers between shared endpoints at unit vertical distance. The typical value of $\omega_t$
is such a weight difference where now the vertical separation of the shared endpoints has order $2^{-\ell}$. Replacing $n$ by $n 2^{-\ell}$ in (\ref{subcrit12}), we obtain $\E \big( M_{2^{-\ell}n}^0 -  M_{2^{-\ell}n}^t \big)^2 \leq \Theta(1) tn 2^{-\ell}$. Recalling the definition of polymer weight
from Subsection~\ref{s.maxweight},
 and that $\tau$ denotes $t n^{1/3}$, we find that, for $x,y \in \R$ with $\vert x - y \vert = 2^{-2\ell/3}$,
$$
  \Big\vert \weight_n^0 \big[ (x,0) \to (y, 2^{-\ell}) \big] -   \weight_n^t \big[ (x,0) \to (y, 2^{-\ell}) \big]   \Big\vert^2 = \Theta(1) \tau 2^{-\ell} \, .
$$
Thus, it would seem that $\big\vert \weight_n^0 \big[ (x,0) \to (y, 2^{-\ell}) \big] -   \weight_n^t \big[ (x,0) \to (y, 2^{-\ell}) \big] \big\vert$ typically has order at most $\tau^{1/2} 2^{-\ell/2}$. We find then that the summand $\omega_i$ is typically of order $\tau^{1/2} 2^{-\ell/2}$. Since the number of summands $\omega_i$ has order $2^\ell$, we obtain a strengthened form of~(\ref{e.proxymimicry}) with right-hand side $2^{\ell/2} \tau^{1/2}$. Our argument is flawed because the endpoint pairs to which we have applied~(\ref{e.typicaltau}) are potentially exceptional. We will develop and utilise robust counterparts to~(\ref{e.typicaltau}) that are capable of handling such pairs; the outcome will be a bound of the form~(\ref{e.proxymimicry}) for a small but explicit value of $a > 0$.

We now explain why~(\ref{e.proxytimezero}) holds. As Figure~\ref{f.proxy}(right) suggests, the mimicry at vertical scale $2^{-\ell}$ of the polymer $\rho_n^t$ by the proxy $\rho_n^{t \to 0}$ has the consequence that to each excursion between $\rho_n^0$ and $\rho_n^t$ of height of order $2^{-\ell}$ will correspond an excursion between $\rho_n^0$ and $\rho_n^{t \to 0}$ whose geometry, including height, is not significantly different. As such, the symmetric difference  $\rho_n^{t \to 0} \Delta \rho_n^0$ is a disjoint union of components, which includes an order of $2^\ell$ excursions between the proxy and the time-zero polymer of height of order $2^{-\ell}$. The reader may glance ahead to Figure \ref{f.proxyconstruction} which illustrates this aspect of the construction.
Any such excursion $E$ is a union $E^{t \to 0} \cup E^0$, where $E^{t \to 0}$ and $E^0$ are $n$-zigzags with a common pair of endpoints that are respective subpaths of $\rho_n^{t \to 0}$ and~$\rho_n^0$. Define the time-zero excursion weight difference $\Delta(E)$ to be $\weight^0(E^{t \to 0}) - \weight^0(E^0)$.  
 It is not difficult to see that the weight difference in~(\ref{e.proxytimezero}) is what we may call {\em excursion additive}. Namely,
\begin{equation}\label{e.excursionadditive}
\weight_n^0 \big( \rho_n^{t \to 0} \big) - \weight^0_n \big( \rho_n^0 \big) \, = \, \sum_E \Delta(E) \, , 
\end{equation}
where the right-hand sum is over the just identified excursions $E$ between $\rho_n^0$ and $\rho_n^{t \to 0}$. 

We will now seek to bound below  the preceding right-hand side. 
At this point, we rely on another assertion that an excursion $E$ of height of order $2^{-\ell}$ have the geometric feature that the two paths $E^0$ and $E^{t \to 0}$ are horizontally typically at order $2^{-2\ell/3}$ from each other at the mid-point height $h = h(E)$ of $E$. This is a consequence of the KPZ characteristic exponent of two-thirds that governs transversal fluctuations;
but the consequence isn't trivial (see the upcoming point~$(1$)).

We hence assume that this typical behaviour holds for  a positive fraction of the order $2^\ell$ excursions. Fixing such an excursion $E$, let $(y,h)$ and $(x,h)$ be respective elements of $E^{t \to 0}$ and $E^0$.  As we will see (in Section \ref{s.additive124}), we may express $\Delta(E)$ in terms of the routed profile $Z = Z_n \big( \cdot, h(E)\big)$ from~(\ref{e.routedprofile}) as $\Delta(E)\le Z(y) - Z(x)$. Since $x$ is a maximizer of $Z$, $\Delta(E)$ is seen to be at most zero; and we may determine its typical order by invoking Brownian similarity for $Z$ alongside $\vert y - x \vert = \Theta(2^{-2\ell/3})$,  indicating that the negative quantity $\Delta(E)$ has order $2^{-\ell/3}$ (see e.g. Figure~\ref{f.proxy} (middle,lower) for an illustration of the Brownianity in a different case).

The expression~(\ref{e.excursionadditive}) is thus seen to be negative and of order at least $2^\ell \cdot 2^{-\ell/3} = 2^{2\ell/3}$. This completes our argument in favour of~(\ref{e.proxytimezero}); and our discussion of low overlap via scale-$\ell$ excursions. In the actual proof, we will apply a union bound over all possible $O(\log n)$ many scales.

\subsection{Key remaining ingredients}\label{s.summaryoverview}
%We have argued how two significant subevents of the event of low subcritical overlap typically force provably rare static events. 
%The first argument was expository and it is the second that will find rigorous expression in the proof of Theorem~\ref{t.main.scaled}(1). 
 We conclude this overview by highlighting some points of incompleteness in our discussion.
\begin{enumerate}
\item In arguing for the bound~(\ref{e.proxytimezero}),
we invoked the typical geometric feature that the two paths that constitute an excursion of scale $2^{-\ell}$ would have horizontal separation of order $2^{-2\ell/3}$ at typical vertical coordinates. If many excursions are {\em slender}, with $\rho_n^t$ consistently running closer than the characteristic displacement for much of the lifetime of excursions of this polymer from $\rho_n^0$, the argument for~(\ref{e.proxytimezero}) breaks down. We will present a separate argument that slender excursions between the two polymers $\rho_n^0$ and $\rho_n^t$ are a rarity. 
\item We have mentioned that the several excursions' argument in Subsection~\ref{s.severalexcursions} will run into difficulty when $2^\ell$ is close to $n$---when low overlap is due to many excursions whose duration barely exceeds the microscopic scale~$n^{-1}$ (the short excursions case).
This case will require a  separate analysis that does not employ the proxy.
\end{enumerate}  

%According to a precise definition that we defer, the latter excursions will be called {\em short}. The others, whose treatment forms our principal concern and which has been the object of this heuristical overview, will be called {\em long}. The more unwieldly names {\em extremely short} and {\em not extremely short}
%would however better express the conceptual role of the two cases.

\subsection{Structure of the rest of the paper}\label{s.structure}

Several tools must be developed or recalled in order to carry out the plan for proving Theorem~\ref{t.main.scaled} that we have indicated. Four main techniques may be highlighted, to which the next four sections are devoted in turn. 
\begin{enumerate}
\item
The use of Fourier theory for deriving low supercritical overlap, namely Theorem~\ref{t.main.scaled}(2), as well as for proving an assertion, Proposition~\ref{p.onepoint}, of subcritical weight stability, which we have seen to be a cornerstone of our proposed proof approach to proving the high subcritical overlap. As already alluded to in Section \ref{s.lowoverlapgaussian}, this necessitates the extension of the tools from Chatterjee's theory of chaos and superconcentration  to the setting  of Brownian LPP.
%from the monograph~\cite{Chatterjee14}, and in particular the dynamical formula for variance \eqref{e.gaussiandynamicalvariance} therein, plays a notable role.
\item  Static Brownian LPP results, presented in  Section~\ref{s.staticpolymer}. Needed assertions about the static model take two forms.
First, statements about the geometry and weight of polymers uniformly as the endpoint pair of the static polymer is permitted to vary over a suitably scaled region, and have been developed in~\cite{NearGroundStates}.
Second, the resemblance to Brownian motion of the routed profile weight profile~(\ref{e.routedprofile}). This result has been obtained in \cite{NearGroundStates} relying on the Brownian resemblance of geodesic weight profile results from~\cite{DeBridge}.   
\item Proposition~\ref{p.onepoint} asserts that the weight of a polymer with given endpoints typically changes little under time increments of the form $t \ll n^{-1/3}$. A more robust assertion of subcritical weight stability makes a comparable claim uniformly as the endpoint pair of the polymer is varied suitably. Theorem \ref{t.stable} and Proposition~\ref{p.stablehorizontal} are such assertions; their proofs harness the tools recalled in Section~\ref{s.staticpolymer}. 
\item In Theorem~\ref{t.proxy},
the construction and main properties of the proxy $\rho_n^{t \to 0}$ that captures time-$t$ polymer geometry with competitive time-zero weight. The construction employs the three preceding sets of tools.
\end{enumerate} 

The verification that Proposition~\ref{p.stablehorizontal}'s hypotheses are sufficient to permit use of the elements needed for its proof occupies a few paragraphs. We present this verification in Appendix~\ref{s.calcder}.

With the indicated tools developed, the proof of Theorem~\ref{t.main.scaled}(1) 
is derived in the final three sections of the paper. 
%As we indicated in Section~\ref{s.summaryoverview}, separate arguments will be made in the short excursions' case, where $2^{-\ell}$ is very close to $n$; and in the opposing
%case of long excursions.
Section~~\ref{s.highsubcriticaloverlap} begins with an overview of how the cases of long and short excursions will be analysed, and reduces the proof to statements concerning the two.
Section~\ref{s.longexc} treats long excursions and the rarity of slender excursions. Section~\ref{s.shortexc} provides the analysis  for the short excursions' case.

\section{Deductions from harmonic analysis: 
%\\ 
\protect\linebreak[1] the transition from weight stability to low overlap}\label{s.deductions}

%The theory of harmonic analysis furnishes two vital inputs for our study: the dynamical variance formula already indicated in~(\ref{e.gaussiandynamicalvariance}) (and the related assertion of decrease in mean overlap); and stability for polymer weights over subcritical time-increments $t \ll n^{-1/3}$ which in fact is a consequence of a variant of the dynamical variance formula. 

%We next state the two crucial inputs \eqref{e.lowoverlap} and \eqref{subcrit12}. 
Here, we develop the approximation theory needed to prove  \eqref{e.lowoverlap} and \eqref{subcrit12} for Brownian LPP. The form of the expressions stay exactly the same which we record next in the interests of clarity. 
\begin{theorem}\label{t.dynvaroverlap}
\begin{enumerate}
Dynamical Brownian LPP satisfies two properties.
\item It enjoys a dynamical formula for variance:
$$
 {\rm Var}(M_n) = \int_0^\infty e^{-t} \E \mc{O}(\Gamma_n^0,\Gamma_n^t) \, {\rm d} t \, .
$$
\item The mean overlap function $[0,\infty) \to \R: t \to \E \mc{O}(\Gamma_n^0,\Gamma_n^t)$ is non-increasing. 
\end{enumerate}
\end{theorem}
Note that the result is stated in terms of  geodesic energy~$M_n$ and unscaled overlap~$\mc{O}$; this unscaled expression is more natural for this result, despite our overall preference for scaled coordinates.  

We next turn to subcritical energetic stability. 
%for geodesics with given endpoints is, in essence, a consequence of this dynamical variance formula. The next result first expresses this consequence in unscaled coordinates; in this regard, note that, taking $(x,i) = (0,0)$ and $(y,j) = (n,n)$ with $t \ll n^{-1/3}$ yields a right-hand side that is smaller than the static variance $\Theta(1) n^{2/3}$. A scaled rephrasing of this bound---of subcritical weight stability for polymers---is provided in the result's second part.  
\begin{proposition}\label{p.onepoint} 
\begin{enumerate}
$\empty$
\item {\em Unscaled stability.} For $i,j \in \N$ and $x,y \in \R$ with $i \leq j$ and $x \leq y$,  and for $t \geq 0$, 
 $$
 \E \, \Big\vert \, M^t \big[ (x,i) \to (y,j) \big] - M^0 \big[ (x,i) \to (y,j) \big] \, \Big\vert^2 \, \leq \, 2\vert x-y \vert t \, .
$$
\item {\em Scaled stability.}
 For $(n,s_1,s_2) \in \N \times \R_\leq^2$ a compatible triple;  $x,y \in \R$ that satisfy $\vert y - x \vert \leq 2^{-1} n^{1/3} \tot$;  and $t \geq 0$ written in the form $t = n^{-1/3} \tau$, 
 
$$
 \tot^{-2/3} \, \E \, \Big\vert \, \weight_n^0\big[ (x,s_1)\to(y,s_2) \big] - \weight_n^t\big[ (x,s_1)\to(y,s_2) \big] \,  \Big\vert^2 \, \leq \, 2\tot^{1/3} \tau \, .
$$
\end{enumerate}
\end{proposition}

In the next section we recall aspects of Chatterjee's monograph~\cite{Chatterjee14} needed to prove the above statements for discretely indexed Gaussian random variables, which will then yield the Brownian LPP versions on passing to the limit.
%This along with suitably discretizing Brownian LPP will suffice to prove the above results. We will then finish the short proof of low supercritical overlap Theorem~\ref{t.main}(2) in a manner already indicated in the discussion around \eqref{e.lowoverlap}.
%and prove Proposition~\ref{p.onepoint}.

\subsection{Some general tools for Markovian dynamics}

The monograph~\cite{Chatterjee14} examines Markovian dynamics, exploring relations between chaotic behaviour; the presence of many near-minima in the energy landscape specified by the Markov chain's equilibrium measure; and the phenomenon of superconcentration, under which observables have variance significantly below the scaling compatible with a central limit theorem.  
In this section, we specify certain general elements of the theory as Lemmas~\ref{l.dynamicalvariance} and~\ref{l.dec}, with a view to specializing them to prove Theorem~\ref{t.dynvaroverlap} in the following one.

First, we specify the general apparatus for the inputs needed from~\cite{Chatterjee14}. Let $X=(X_t)_{t\ge 0}:[0,\infty) \to K$ be a Markov process, valued in a set~$K$. For our applications, $K$ will be an Euclidean space, so we will not worry about the topological properties of $K$; for the moment, assume that it is a Polish space, equipped with the Borel $\sigma$-algebra. Assume also that $X$ has an equilibrium measure, which is a probability measure~$\mu$ on $K$. Specify an inner product $\langle f,g \rangle  = \int fg \, {\rm d}\mu$
for functions $f,g: K \to \R$ in the space $L^2(\mu)$. Suppose that each element in the semigroup of operators $\big\{ P_t: t \geq 0 \big\}$ is well-defined in its action on any $f \in L^2(\mu)$ via the formula 
$(P_t f)(x) = \E \big( f(X_t) \big\vert X_0 = x \big)$ for $x \in K$. The semigroup's generator $L$ acts on such $f$ via $Lf = \lim_{t \searrow 0} (P_t f - f)/t = \partial_t P_t f \big\vert_{t = 0}$, assuming that the right-hand side is well defined. In a simple consequence of the definition of $L$ and the semigroup property of $P$, the heat equation takes the form $\partial_t P_t = L P_t$.

The Dirichlet form of the Markov semigroup $P_t$ is specified via the inner product $\langle \cdot,\cdot \rangle$:
\begin{equation}\label{e.dirichletform}
 \mc{E}(f,g) : = - \langle f,Lg \rangle = - \int f Lg \, {\rm d} \mu \, .
\end{equation}

\begin{lemma}[Dynamical formula for covariance]\label{l.dynamicalvariance}
For $f,g\in L^2(\mu)$,
$$
{\rm Cov}(f,g):=\langle fg\rangle-\langle f \rangle \langle g \rangle=\int_{0}^{\infty} \cE \big(f,P_t(g) \big) \, {\rm d}t \, .
$$
In particular, we have the dynamical formula for variance:
$$
{\rm Var} (f):= \langle f^2\rangle-\langle f \rangle^2=\int_{0}^{\infty} \cE \big( f,P_t(f) \big) \, {\rm d}t \, .
$$
\end{lemma}
{\bf Proof.} Note that 
${\rm Cov}(f,g)=-\int_{0}^{\infty} \partial _t \langle f, P_t g\rangle \, {\rm d}t$.
By the heat equation  $\partial_t P_t(g)=L P_t(g)$, we have
\begin{equation}\label{keydynamical}
\partial _t \langle f, P_t g\rangle \, = \, - \, \mc{E}\big(f,P_t g \big) \, ,
\end{equation}
so that the former assertion is obtained. The latter follows trivially. \qed

Towards specializing to the particular case of the Ornstein-Uhlenbeck semigroup, we assume that the Markov process $X$ is reversible and the generator $L$ is self-adjoint, and negative semidefinite with a discrete spectrum (it is well known that the Ornstein-Uhlenbeck semigroup satisfies these properties. for instance, ). Note that zero is always an eigenvalue because $L 1 = 0$. 
Consequently the  eigenvalues of $-L$ may be ordered as a non-decreasing sequence $\big\{ \lambda_n: n \in \N \big\}$, with $\lambda_0 = 0$, and the corresponding sequence $\big\{ u_n: n \in \N \big\}$ of eigenfunctions, where $u_0 = 1$, is a complete orthogonal basis for $L^2(\mu)$. 
 
Here is the second general tool that we will need. 

\begin{lemma}\label{l.dec}
The function $[0,\infty) \to \R: t \to e^{\lambda_1 t}\cE\big(f,P_t(f)\big)$ is decreasing. Consequently,  so is  $[0,\infty) \to \R: t \to \cE\big(f,P_t(f)\big)$.
\end{lemma}
\begin{proof}
From the heat equation follows the semigroup expression  $P_t=e^{tL}$. Apply $-LP_t$ via this representation to the  
decomposition $f=\sum_{k} \langle f,u_k \rangle u_k$ to learn that 
$-L P_t f=\sum_{k} \lambda_ k e^{-\lambda_k t}\langle f,u_k \rangle u_k$. 
Applying~(\ref{e.dirichletform}), we find that 
$\cE\big(f,P_t(f)\big) = \sum_{k} \lambda_ k e^{-\lambda_k t}{\langle f,u_k \rangle}^2$.
The lemma's first assertion then follows from $\lambda_k\ge \lambda_1$ for $k \geq 1$; its second assertion follows trivially.  
\end{proof}

\subsection{The dynamical formula for variance: proving Theorem~\ref{t.dynvaroverlap}}
\subsubsection{Introducing approximating Gaussian models}\label{s.introgauss}
As already indicated, to prove Theorem~\ref{t.dynvaroverlap}, we must adapt results from \cite{Chatterjee14} that treat discrete Gaussians so that they apply to our {\em continuous} Gaussian dynamic setting given by  dynamical Brownian LPP.

To do this, we will need approximating discrete counterparts.
It is easy enough to specify them, indexed by $m \in \N^+ = \N \setminus \{ 0 \}$, with dynamical Brownian LPP formally obtained by taking $m =\infty$. Recall the noise environment $B:\R \times \Z \to \R$ of static Brownian LPP (under a law labelled~$\P$). 
We will consider these curves' increments on a grid of intervals of length $m^{-1}$.

For $m \geq 1$, $i \in \Z$ and $u \in m^{-1} \intinta{0}{nm-1}$, set $X[m](u,i) = B(u+ m^{-1},i) - B(u,i)$. Under $\P$, $\big\{ X[m](u,i): (u,i) \in m^{-1} \intinta{0}{nm-1} \times \Z \big\}$ is an independent collection of Gaussian random variables of mean zero and variance $m^{-1}$. The variable name $u$ is intended to suggest a real variable, in view of our interest in the high $m$ limit; but we specify a discrete Gaussian model via the noise environment $X[m]$. It is convenient to think of the Gaussian variable $X[m](u,i)$ as being attached to the horizontal edge connecting $(u,i)$ and $(u+ m^{-1},i).$ For given $m \in \N^+$,
an {\em east-north} path from $(0,0)$ to $(n,n)$ is a path between these endpoints in the $m$-indexed lattice $m^{-1} \intinta{0}{nm} \times \Z$ each of whose steps is either an easterly movement along a horizontal edge between adjacent lattice points, with displacement $(m^{-1},0)$, or a northerly movement, one level up with displacement $(0,1)$.  
Any such path $\gamma$ is ascribed an energy  $E[m](\gamma)$ as the sum of the $X[m]$-assigned values on the horizontal edges that it encounters.  
There is rather trivially, almost surely, only one such path that attains the maximum energy $M_n[m]$; this is the geodesic, to be called $\Gamma_n[m]$. 

The field $X[m]$ is a coarsening of the white noise field ${\rm d}B$ attached to the $\Z$-indexed system of lines. In a formal sense, the static Brownian LPP geodesic $\Gamma_n$ equals $\Gamma_n[\infty]$. 

We next extend our static coupling of Brownian LPP and its approximants to the dynamical version. Recall the notational abuse by which $B: \R \times \Z  \times [0,\infty) \to \R$ denotes the dynamical Brownian LPP noise environment; so that $B(x,i,t)$ denotes the time-$t$ value at $x \in \R$ of the $i$-indexed Brownian motion. We simply specify $X[m]: m^{-1} \intint{nm} \times \Z \times \R \to \R$  
 via 
 $$
 X[m](u,i,t) =  B(u+ m^{-1},i,t) - B(u,i,t) \, . 
$$ 
 Natural notational extensions such as $M^t_n[m]$ and $\Gamma^t_n[m]$ result. The overlap of two east-north paths in the $m$-indexed lattice is said to be the product of $m^{-1}$ and the cardinality of the set of horizontal edges visited by both paths.
 
  In order to apply the stated Markovian dynamics tools to the approximants $X[m]$, we begin by discussing the one-dimensional dynamics of the components $X[m](u,i,\cdot):[0,\infty) \to \R$. 
  The stochastic differential equation~(\ref{e.ou}) is one-dimensional Ornstein-Uhlenbeck dynamics 
  whose invariant measure is a standard Gaussian law. Each component  $X[m](u,i,\cdot):[0,\infty) \to \R$ evolves according to OU dynamics 
  with invariant measure equal to a Gaussian of mean zero and variance~$m^{-1}$ whose corresponding SDE is obtained from \eqref{e.ou} by a simple scaling:
$${\rm d}X(t)=-X(t) {\rm d}t+ \big( 2 m^{-1} \big)^{1/2}{\rm d}W(t),
$$
where $W: [0,\infty) \to \R$ is standard Brownian motion.
  
We now record the properties of these dynamics that we will need.
 
\subsubsection{The Gaussian Ornstein-Uhlenbeck dynamics}
Let $\nu_{0,\sigma^2}$ denote the Gaussian law on $\R$ of mean zero and variance $\sigma^2$.
Next we state two well-known facts about  the resulting Ornstein-Uhlenbeck semigroup on $L^2\big(\nu_{0,\sigma^2}\big)$ and {the natural generalization to the $d$-dimensional product space generated by independent and identically distributed copies of this Gaussian law, where  each coordinate evolves independently.} 
For $\sigma=1$, these facts can be found in, for example,~\cite[Chapter~$2$, Section~$2$]{Chatterjee14}. 
 Note however that, by scaling properties of Gaussian distributions, $f(\cdot)\to f(\frac{1}{\sigma}\cdot)$ is an isometry from $L^2\big(\nu_{0,1}\big)$ to $L^2\big(\nu_{0,\sigma^2}\big)$ which  preserves the indicated Dirichlet form and hence also the spectrum (the analogous statement holds in any dimensions). Thus, the statements are valid for any positive $\sigma$.
\begin{enumerate}
\item {\em Dirichlet form.} For any $d\ge 1,$ the Dirichlet form is given by 
\begin{equation}\label{dirichletformfinite}
\cE(f,g)=\sigma^2\E \langle \grad f, \grad g \rangle \, .
\end{equation} where $\grad$ denotes the gradient vector in $d$ dimensions. 
\item {\em Ornstein-Uhlenbeck coupling.} Let $t \geq 0$. Consistently with~(\ref{e.twotimedist}),
 there exists a $t$-determined collection $\big\{ Z^t[m](u,i):(u,i) \in  m^{-1} \intint{nm} \times \Z \big\}$ of random variables that is distributed as, while being independent of, the collection $X[m](\cdot,\cdot,0)$,
 such that 
$$
X[m](u,i,t)=e^{-t}X[m](u,i,0) + \big( 1-e^{-2t} \big)^{1/2} Z^t[m](u,i)  \, .
$$ 
\end{enumerate}

\subsubsection{Dynamical variance formula and mean overlap decrease for the approximating models}
Next is the manifestation of the general tools Lemmas~\ref{l.dynamicalvariance} and~\ref{l.dec} for the discrete Gaussian LPP models.
\begin{lemma}\label{l.approxmodel}
$\empty$ Let $\Gamma^0_n[m] \cap  \Gamma^{t}_n[m]$ be the set of horizontal edges $e_{(u,i)}:=\big((u,i),(u+m^{-1},i)\big),$ where $(u,i) \in  m^{-1} \intinta{0}{nm-1} \times \Z$, that are shared by the geodesics $\Gamma^0_n[m]$ and  $\Gamma^{t}_n[m].$
\begin{enumerate}
\item The variance of geodesic energy is given by
$$
 {\rm Var} \big( M_n[m] \big) =
 m^{-1}  \int_0^\infty e^{-t} \E \big\vert \Gamma^0_n[m] \cap  \Gamma^{t}_n[m] \big\vert \, {\rm d}t \, .
$$
\item
The mean overlap function $[0,\infty) \to [0,\infty): t \to  m^{-1}\E \big\vert \Gamma^0_n[m] \cap  \Gamma^{t}_n[m] \big\vert$ is non-increasing.
\end{enumerate}
\end{lemma}
{\bf Proof. (1).}
For convenience of expression in this argument, we record $X[m][\cdot,\cdot,t]$ in the form $X^t[m]$, a notation that runs in parallel with the usage $Z^t[m]$.

The Ornstein-Uhlenbeck coupling indicated a few moments ago implies that, for any $f$ belonging to $L^2\big(\nu_{0,\frac{1}{m}}^{\otimes\, m^{-1} \intinta{0}{nm-1} \times \Z}\big)$, namely to  the $L^2$-space of the product space generated by all the Gaussian variables $\{X[m](u,i,0): (u,i) \in  m^{-1} \intinta{0}{nm-1} \times \Z\}$,
 $$
 P_t(f)(X^0[m])= \E_{Z^t[m]} f\big( e^{-t}X^0[m] + (1-e^{-2t})^{1/2} Z^t[m] \big) \, ,
 $$ 
where the expectation is taken over $Z^t[m]$.
Since the geodesic energy $M_n[m]$ is a piecewise linear function of the variables $X[m](u,i)$, it is absolutely continuous with the component-wise gradient, almost surely,  given by  
$$
\grad_{u,i} M_n[m](X^0[m]) = \mathbf{1}_{e_{(u,i)} \in \Gamma^0_n[m]}  \, .
$$

By exchanging expectation and differentiation in the above expression for the Dirichlet form, we find that, with $f = M_n[m]$,
$\cE\big( M_n[m], P_t M_n[m] \big)$ equals 
\begin{eqnarray}
 &  &  m^{-1} \E_{X^0[m]}\E_{Z^t[m]} \sum_{u,i} \partial_{X^0[m](u,i)} f(X^0[m])  \partial_{X^0[m](u,i)}  f\Big(e^{-t}{X^0[m]}+ (1-e^{-2t})^{1/2}Z^t[m]\Big) \nonumber \\
&= & m^{-1} \E_{X^0[m]} \E_{Z^t[m]} \sum_{u,i} \mathbf{1}_{e_{(u,i)} \in \Gamma^0_n[m]}e^{-t}\mathbf{1}_{e_{(u,i)} \in \Gamma^t_n[m]} \, =  \, m^{-1} e^{-t}\E \big\vert \Gamma^0_n[m] \cap  \Gamma^{t}_n[m] \big\vert  \, . \label{e.discdirichlet}
\end{eqnarray}
Thus, Lemma~\ref{l.dynamicalvariance} yields Lemma~\ref{l.approxmodel}(1).

{\bf (2).} This is due to the above expression for $\cE\big( M_n[m], P_t M_n[m] \big)$,  Lemma~\ref{l.dec} and  $\lambda_1$ equalling one for the OU dynamics (see \cite[Chapter 2, Section 4]{Chatterjee14}). \qed

\subsubsection{Convergence of the discrete Gaussian approximants}\label{s.approx}
\begin{proposition}\label{p.approx}
Let $n,m \in \N$. The geodesic energy $M_n[m]$ in the $m\textsuperscript{th}$ approximating model is at most its limiting counterpart $M_n$. Regarding the opposing bound, we have that
$$
\PP \Big( M_n - M_n[m] \geq 8 n  \big( \tfrac{\log m}{\log n} + 1 \big)^{1/2} m^{-1/2} (\log n)^{1/2}  \Big) \leq 
2^{7/2} \pi^{-1/2} ( \log m )^{-1/2}m^{-1} \, .
$$ 
\end{proposition}
{\bf Proof.} To any east-north path $\gamma$ between $(0,0)$ and $(n,n)$,
we consider the naturally associated staircase $S(\gamma) \subset \R^2$ that is the range of a motion from $(0,0)$ to $(n,n)$ that alternates between rightward and upward movement by associating to each horizontal edge of $\gamma$,  the horizontal planar line segment with the same endpoints as the edge. Then $S(\gamma)$ is defined to be the union of the planar sets associated to the steps made by $\gamma$.  
Clearly, for any such path $\gamma$, the energy that is assigned to $S(\gamma)$ by the  
Brownian LPP noise environment  is nothing other than the energy $E[m](\gamma)$ ascribed to $\gamma$ by the field $X[m]$. Thus, $M_n \geq M_n[m]$.

To prove an opposing inequality, 
recall that the geodesic $\Gamma_n \subset \R^2$ is a staircase from $(0,0)$ to $(n,n)$. We will associate to it an $m$-indexed lattice approximant  $\Gamma_n[m]$. This will be an east-north
path between $(0,0)$ and $(n,n)$ in the $m$-lattice $m^{-1} \intinta{0}{nm} \times \Z$. 
Informally, $\Gamma_n[m]$ will be the east-north path lying to the right of $\Gamma_n$, and being the leftmost in this set. Note that the above set is non-empty since the lattice path which passes through the points $(0,0),(n,0),(n,n)$ is in the set. 

We need a little notation to give a formulaic specification of $\Gamma_n[m]$.
For $i \in \llbracket 0, n \rrbracket$, we write $z_i$ for the supremum of $x \in [0,n]$ for which $(x,i) \in \Gamma_n,$ and $z_{-1}=0$.
Likewise, if, for $m \in \N^+$, we specify $z_i[m] = \sup\big\{ x \in m^{-1}\llbracket 0, nm \rrbracket: (x,i) \in \Gamma_n[m] \big\}$, then the task of specifying $\Gamma_n[m]$ guided by the above informal description is a matter of defining  $z_i[m]$ for $i \in \llbracket 0, n-1 \rrbracket$---and we define this quantity to be the smallest element in the set $m^{-1}\llbracket 0, nm \rrbracket$ which is greater equal to $z_i$. And we set $z_{-1}[m]=0$. Note that, since the $z_i$-sequence is non-decreasing, so is the $z_i[m]$-sequence.

Clearly, for all $i$, 
\begin{equation}\label{difference}
0\le z_{i}{[m]}- z_i\le 1/m \, .
\end{equation}

Note that the intersections of $S\big(\Gamma_n[m]\big)$ and  $\Gamma_n$ with the level $i$ are the intervals $[z_{i-1},z_i]$ and $\intinta{z_{i-1}[m]}{z_i[m]}$ respectively.
Thus the symmetric difference $S\big(\Gamma_n[m]\big) \Delta \Gamma_n$ 
contains numerous horizontal segments;
these number at most $2(n-1)+2 =2n$, since there are at most two at each vertical level  in $\intint{n-1}$, and one from the levels indexed by zero and $n$.
By \eqref{difference}, each segment is a planar interval
contained in a horizontal planar interval of 
length~$1/m$, delimited by consecutive $m$-lattice points.

To $S\big(\Gamma_n[m]\big)$, the Brownian LPP environment assigns energy $E \big( S\big(\Gamma_n[m]\big) \big)$; but the latter quantity equals  $E[m]\big( \Gamma_n[m] \big)$, which is at most $M_n[m]$. Thus, we see that
 $M_n - M_n[m] \leq M_n - E \big( S\big(\Gamma_n[m]\big) \big)$;
 the right-hand side here is at most the sum of the absolute values of the increments of the Brownian motions in the Brownian LPP noise environment indexed by the numerous horizontal segments of which  $S\big(\Gamma_n[m]\big) \Delta \Gamma_n$ is comprised.
  We find then that $M_n - M_n[m]$ is at most the product of $2n$ and the quantity ${\rm Osc}_n[m]$, which measures the maximum oscillation witnessed by the Brownian motions $B(\cdot, i)$ on the intervals of length $m^{-1}$  delimited by consecutive $m$-lattice points.
Namely,
 $$
 {\rm Osc}_n[m] = \sup \Big\{ \big\vert B(u+\eta_1,i)-B(u + \eta_2,i) \big\vert : u \in m^{-1} \llbracket 0,nm-1 \rrbracket \, , \,  i \in \llbracket 0,n \rrbracket \, , \, \eta_1,\eta_2 \in [0,m^{-1}] \Big\} \, .
 $$
 Control on this oscillation is offered next.
\begin{lemma}\label{l.osc}
For $\kappa >0$,
$$
\PP \Big( {\rm Osc}_n[m] \geq \kappa  m^{-1/2} (\log n)^{1/2}   \Big) \leq 
 m  \cdot 2^{7/2} \pi^{-1/2} \kappa^{-1}  \big( \log n \big)^{-1/2} n^{2 - \kappa^2/8} \, .
$$
\end{lemma}
{\bf Proof.}
Writing $\mc{B}$ for the law of standard Brownian motion $B$, we find from Brownian symmetry, the reflection principle, Brownian scaling and a Gaussian tail bound stated in \cite[Section~$12.4$]{Williams91}  that, for $h \geq 0$,
\begin{eqnarray*}
& & \PP \bigg( \sup_{0\le x,y\le m^{-1}} \big\vert B(i,x)-B(i,y) \big\vert \geq h \bigg) 
 \leq  2 \mc{B} \Big( \sup_{0 \leq x \leq m^{-1}}B(x) \geq h/2 \Big) \\
 & = & 4 \nu_{0,m^{-1}}(h/2,\infty)
= 4 \nu_{0,1}\big(2^{-1}m^{1/2}h,\infty\big) \leq 2^{5/2} \pi^{-1/2}  m^{-1/2} h^{-1} \exp \big\{ - 8^{-1}m h^2 \big\} \, .
\end{eqnarray*}
The random variable ${\rm Osc}_n[m]$ is the maximum of $(n+1)nm$ such random quantities; using $n \geq 1$, Lemma~\ref{l.osc} thus follows from a union bound with $h = \kappa m^{-1/2} (\log n)^{1/2}$. \qed

Taking $\kappa = 4 \big( \tfrac{\log m}{\log n} + 1 \big)^{1/2}$
renders $n^{2 - \kappa^2/8}$ equal to $m^{-2}$.
Since 
$M_n - M_n[m] \leq 2n {\rm Osc}_n[m]$, we apply Lemma~\ref{l.osc} with this choice of $\kappa$ 
to obtain the latter assertion of Proposition~\ref{p.approx}. \qed

The discrete geodesic converges almost surely to its continuum counterpart.

\begin{lemma}\label{l.almostsure}
Let $i \in \llbracket 0, n \rrbracket$. Almost surely, $z_i[m] \to z_i$ as $m \to \infty$.
\end{lemma}
{\bf Proof.} Note that $M_n[m] = E[m] \big( \Gamma_n[m] \big) = E \big( S(\Gamma_n[m]) \big)$, where the notation $S(\Gamma_n[m])$ has been specified at the outset of the ongoing  proof of Proposition~\ref{p.approx}.
Let $\big\{ z'_i : i \in \llbracket 0, n \rrbracket \big\}$ denote a subsequential limit point in $m$ of $\big\{ z_i[m] : i \in \llbracket 0, n \rrbracket \big\}$, and let $S'$ denote the staircase associated to the non-decreasing list $z'$ in the sense of Subsection~\ref{s.staircase}. By Proposition~\ref{p.approx}, after possibly the extraction of a further subsequence $\big\{ m_j: j \in \N \big\}$, we have the almost sure convergence of $M_n[m_j]$
to $M_n$. But this sequence converges to $E(S')$, while $M_n$ equals $E(\Gamma_n)$.
Since the geodesic $\Gamma_n$ is almost surely unique by~\cite[Lemma~B$.1$]{Patch} with
this lemma's parameter $\ell$ set equal to one, we have that $S'$ equals $\Gamma_n$ almost surely. Thus, the list $z'$ equals $\big\{ z_i : i \in \llbracket 0, n \rrbracket \big\}$ almost surely.  \qed

\subsubsection{Proof of Theorem~\ref{t.dynvaroverlap}} {\bf (1).} \label{limjustify}
We aim to take the high $m$ limit in Lemma~\ref{l.approxmodel}(1). To do so, we need to establish first the left-hand convergence ${\rm Var}\, M_n[m] \to {\rm Var}\, M_n$ along a subsequence of $m \in \N$.
This follows from the almost sure subsequential convergence of $M_n[m]$ to $M_n$, which is due to Proposition~\ref{p.approx}, alongside $M_n[m] \leq M_n$ (which is also due to this   proposition) and $\E M_n^2 < \infty$. The convergence of the right-hand integrals to the limit in the theorem will follow from 
\begin{equation}\label{overlapconverge} m^{-1}\E \big\vert \Gamma^0_n[m] \cap  \Gamma^{t}_n[m] \big\vert \to \E \mc{O} \big( \Gamma^0_n  ,  \Gamma^{t}_n \big)
\end{equation}
 as $m \to \infty$ for each $t \geq 0$,
  since the integrands are uniformly bounded and 
the integrals over $[0,\infty)$ are approximable, uniformly in $m$, by counterpart integrals over compact intervals.    Since 
$$
m^{-1}\big\vert \Gamma^0_n[m] \cap  \Gamma^{t}_n[m] \big\vert \, = \, \sum_{i=0}^n \Big\vert  \big[ z_{i-1}^0[m], z_{i}^0[m] \big] \cap  \big[ z_{i-1}^t[m], z_{i}^t[m] \big] \Big\vert \, \, ,
$$
the indicated convergence of means follows from the almost sure convergence of the bounded sequences 
  $\big\{ z_i^{t'}[m]: m \in \N \big\}$ for $t' \in \{0,t\}$---a convergence that is implied by Lemma~\ref{l.almostsure}.

  {\bf (2).} The just noted convergence of means permits the high $m$ limit of  Lemma~\ref{l.approxmodel}(2) to be taken. The sought assertion is the outcome of doing so. \qed

\subsection{Low supercritical overlap: the proof of Theorem~\ref{t.main}(2)}
We must establish, for a suitably high choice of $D>0$, that $\P \big(  \mc{O}(\Gamma_n^0,\Gamma_n^t) \geq D n \tau^{-1/2} \big) \leq \tau^{-1/2}$ provided that $t \leq 1$. 
 But this bound is~(\ref{e.lowoverlap}) in Section~\ref{s.lowoverlapgaussian}, where it was proved on the assumptions that ${\rm Var}(M_n) = O(n^{2/3})$ 
 and that the conclusions of Theorem~\ref{t.dynvaroverlap} hold. That  ${\rm Var}(M_n) = O(n^{2/3})$ holds for Brownian LPP is due to  \cite[Corollary~$3$]{LedouxRider}; 
 and~\cite{Baryshnikov} or~\cite[Remark~$4$]{GTW}.
 Indeed,  these results respectively state  the bound, in scaled units, on variance for the uppermost GUE eigenvalue;  
 and assert that $M_n$ has the law of this eigenvalue. 
 Thus do we obtain Theorem~\ref{t.main}(2). \qed

\subsection{Subcritical weight stability: the proof of Proposition~\ref{p.onepoint}}

{(\bf 1).} We adopt the shorthand $M^t =  M^t \big[ (x,i) \to (y,j) \big]$ for $t \geq 0$ with a view to bounding $\E \big\vert M^t - M^0 \big\vert^2$. Via the notational device of Subsection~\ref{s.dynamicalnotation}, $M$, without superscript, denotes this random variable as determined by the static noise environment.

By an east-north path in the $m$-lattice {\em from $(x,i)$ to $(y,j)$}, we mean a path valued in $m^{-1} \intinta{0}{nm} \times \Z$  that begins at the first lattice point encountered at or to the right of $(x,i )$ at height $i$; that ends at the lattice point similarly at or to the right of $(y,j)$;
and each of whose steps is an easterly or northerly movement in the sense of Subsection~\ref{s.introgauss}. Any 
such path is ascribed an energy at any given time $t \geq 0$ as the sum of the quantities $X[m](u,i,t)$ attached to the horizontal edges~$e_{(u,i)}$ that it visits. 

We extend the shorthand notation by setting $M^t[m]$ equal to the maximum energy thus ascribed at time $t$ to paths over the route $(x,i) \to (y,j)$; and we will also use the static notation $M[m]$.

It is readily seen that 
\begin{equation}\label{diffsquare}
\E \Big( M^t[m]  -  M^0[m]  \Big)^2 \, = \, 2 \bigg(  {\rm Var} \big( M[m] \big)  - {\rm Cov}\Big( M[m] , P_t \big( M[m] \big) \Big) \bigg) \, .
\end{equation}
The right-hand side above is $-2\int_{0}^{t}\frac{{\rm d}}{{\rm d}s}{\rm Cov} \big( M[m] ,P_s(M[m]) \big) \, {\rm d}s$. Now, by \eqref{keydynamical}, 
\begin{align*}
\frac{{\rm d}}{{\rm d}s}{\rm Cov}\big( M[m] , P_s(M[m]) \big)=-\cE\big( M[m],P_s(M[m]) \big) \, .
\end{align*}
By  versions of \eqref{e.discdirichlet} and \eqref{overlapconverge} where $(0,0)$ and $(n,n)$ are replaced by $(x,i)$ and $(y,j)$ respectively,  we see that 
\begin{align*}
0 \leq \lim_{m\to \infty}\int_{0}^{t}\cE\big( M[m],P_s(M[m])  {\rm d}s &=\int_{0}^{t}e^{-s}\E\big\vert \Gamma^s \big[ (x,i) \to (y,j) \big] \cap \Gamma^0 \big[ (x,i) \to (y,j) \big]\big\vert {\rm d}s \\
&\leq |x-y|(1-e^{-t}) \leq  \vert x-y \vert t \, ,
\end{align*}
where in the last inequality we use the trivial bound 
$\big\vert \Gamma^s \big[ (x,i) \to (y,j) \big] \cap \Gamma^0 \big[ (x,i) \to (y,j) \big] \big\vert \le |x-y|$. 
To obtain Proposition~\ref{p.onepoint}(1), we reason as in Subsection \ref{limjustify}. Namely, by Proposition~\ref{p.approx}, and a simple application of the Borel-Cantelli lemma, $M^t[m]\to M^t$ almost surely along the subsequence of powers of two. Fatou's lemma now finishes the argument. 
%ThFurther $\E\big((M^t)^2\big).$
%By the same reasoning 
%\begin{equation}\label{e.mdifference}
%\E \big( M^t  -  M^0  \big)^2 \, = \, 2 \Big(  {\rm Var} ( M  )  - {\rm Cov}\big( M  , P_t ( M ) \big) \Big) \, .
%\end{equation}
%To justify this passage to the limit, we apply a similar reasoning as in Subsubsection \ref{limjustify}. Namely, by Proposition~\ref{p.approx}, for any $t$ and $m$, we have $M^t[m]\le M^t,$ as well as, by a simple application of the Borel-Cantelli lemma, the almost sure convergence $M^t[m]\to M^t$ as $m$ increases along the subsequence of powers of $2$. Further $\E\big((M^t)^2\big).$
%
% one can rely on Fatou's lemma, and  a finite bound on $\E M^{t}[m]^2$ that is independent of $m$ and $t$.
%(Such a bound arises because $M^{t}[m]$ is at most the sum of the absolute values of $(y-x)(j-i)m$ of Gaussian random variables each of variance $m^{-1}$ and hence has second moment at most $(y-x)(j-i)$.)

{\bf (2).} In~(\ref{e.weightgeneral}), the weight $\weight_n\big[ (x,s_1)\to(y,s_2) \big]$ is expressed in unscaled energetic units.
The relevant application of Proposition~\ref{p.onepoint}(1), namely that
$$
\E \Big( M^t  \big[ ( ns_1 + 2 n^{2/3} x , ns_1) \to ( ns_2 + 2 n^{2/3} y , ns_2) \big]- M^0  \big[ ( ns_1 + 2 n^{2/3} x , ns_1) \to ( ns_2 + 2 n^{2/3} y , ns_2) \big] \Big)^2
$$ 
is at most $2\big( n \tot + 2 n^{2/3} \vert y - x \vert \big) t$, thus yields
$$
\E \, \Big(  \weight_n^t\big[ (x,s_1)\to(y,s_2) \big] - \weight_n^0 \big[ (x,s_1)\to(y,s_2) \big]   \Big)^2 \leq  n^{-2/3} \big( n \tot + 2 n^{2/3} \vert y - x \vert \big) t \, .
$$
 Using $t = \tau n^{-1/3}$ and the hypothesis that $2 n^{2/3} \vert y - x \vert \leq n \tot$, we find that
 the latter right-hand side is at most $2\tot  \tau$. Dividing the resulting bound by $\tot^{2/3}$ yields Proposition~\ref{p.onepoint}(2). \qed

\section{Inputs concerning the geometry and weight of the static polymer}\label{s.staticpolymer}

Here, we record the main inputs that we will need. 
 Most of these have been derived in the companion article~\cite{NearGroundStates}.
There are three subsections. The first treats some rather basic aspects, and a notational convention. The second  records a series of results from~\cite{NearGroundStates} concerning geometry of polymers $\rho_n \big[ (x,s_1)\to (y,s_2) \big]$ and their weights $\weight_n \big[ (x,s_1)\to (y,s_2) \big]$. The third presents
a further important input, an assertion that the twin peaks' event is a rarity for the routed polymer weight profile. This result is also derived in~\cite{NearGroundStates} relying on Brownian regularity theorems in~\cite{BrownianReg} and~\cite{DeBridge}.  

\subsection{The scaling principle and basic polymer properties}
\subsubsection{The scaling principle}\label{s.scalingprinciple}
Write $\R^2_< = \big\{ (x,y) \in \R^2: x < y\big\}$. Let $(n,s_1,s_2) \in \N \times \R^2_<$ be a compatible triple. The quantity $n \tot$ is a positive integer, in view of the defining property~(\ref{e.ctprop}).
The scaling map $R_k: \R^2 \to \R^2$ has been defined whenever $k \in \N^+$, and thus we may speak of $R_n$  and $R_{n \tot}$.
The map $R_n$ is the composition of $R_{n \tot}$ and the transform $S_{\tot^{-1}}$ given by $\R^2 \to \R^2: (a,b) \to \big(a\tot^{-2/3},b\tot^{-1}\big)$. That is, the system of $n\tot$-zigzags is transformed into the system of $n$-zigzags
by an application of  $S_{\tot^{-1}}$.  Note 
that $\weight_{n;(x,s_1)}^{(y,s_2)} =  \tot^{1/3} \weight_{n \tot ;(x \tot^{-2/3},\kappa)}^{(y \tot^{-2/3},\kappa + 1)}$, where $\kappa = s_1 \tot^{-1}$; indeed this weight transformation law is valid for all zigzags, rather than just polymers, in view of~(\ref{e.weightzigzag}).
 
We may summarise these inferences by saying that the system of $n\tot$-zigzags, including their weight data, is transformed into the $n$-zigzag system, and its accompanying weight data, by the transformation $\big( a,b,c \big) \to \big( a \tot^{-1/3}  ,b \tot^{-2/3} , c \tot^{-1} \big)$, where the coordinates refer to the changes suffered in weight,  
and
 horizontal and  vertical coordinates.
 This fact leads us to what we call the {\em scaling principle}. 
 
 \noindent{\em The scaling principle.} 
 Let $(n,s_1,s_2) \in \N \times \R^2_<$ be a compatible triple.
 Any statement concerning the system of $n$-zigzags, including weight information, is equivalent to the corresponding statement concerning the system of $n\tot$-zigzags, provided that the following changes are made:
 \begin{itemize}
 \item the index $n$ is replaced by $n\tot$;
 \item any time is multiplied by $\tot^{-1}$;
 \item any weight is multiplied by $\tot^{1/3}$;
 \item and any horizontal distance is multiplied by $\tot^{-2/3}$.
 \end{itemize}

\subsubsection{Polymer uniqueness and ordering}\label{s.uniquenessordering}

A polymer with given endpoints is almost surely~unique.

\begin{lemma}[Lemma~$4.6(1)$, \cite{Patch}]\label{l.polyunique}
  Let $x,y \in \R$. 
 There exists an $n$-zigzag from $(x,0)$ to $(y,1)$
 if and only if $y \geq x - n^{1/3}/2$.
 When the last condition is satisfied, there is almost surely a unique $n$-polymer from $(x,0)$ to $(y,1)$.
 \end{lemma}

A rather simple sandwiching fact about polymers will also be needed. Let $(x_1,x_2),(y_1,y_2) \in \R^2$ and  consider
a zigzag  $Z_1$ from $(x_1,s_1)$ to $(y_1,s_2)$ and another
 $Z_2$ from $(x_2,s_1)$ to $(y_2,s_2)$. 
We declare that $Z_1 \preceq Z_2$ if `$Z_2$ lies on or to the right of $Z_1$': formally, if $Z_2$ is contained in the union of the closed horizontal planar line segments whose left endpoints lie in $Z_1$.
\begin{lemma}[Lemma $5.7$, \cite{NonIntPoly}]\label{l.sandwich}
Let $(n,s_1,s_2)$ be a compatible triple, and let  $(x_1,x_2)$ and $(y_1,y_2)$ belong to $\R_\leq^2$. Suppose that there is a unique $n$-polymer from $(x_i,s_1)$ to $(y_i,s_2)$, both when $i=1$ and $i=2$. (This circumstance occurs almost surely, and the resulting polymers have been labelled $\rho_n \big[ (x_1,s_1) \to (y_1,s_2) \big]$ and  $\rho_n \big[ (x_2,s_1) \to (y_2,s_2) \big]$.)
Now let $\rho$ denote any $n$-polymer that begins in $[x_1,x_2] \times \{ s_1\}$
and ends in  $[y_1,y_2] \times \{ s_2 \}$. Then 
$$
\rho_n  \big[ (x_1,s_1) \to (y_1,s_2) \big]  \, \preceq \, \rho \, \preceq \, \rho_n \big[ (x_2,s_1) \to (y_2,s_2) \big] \, .
$$
\end{lemma}

\subsubsection{Operations on polymers: splitting and concatenation}\label{s.split}
A polymer may be split into two pieces. 
Let  $(n,s_1,s_2) \in \N \times \R^2_\leq$ be a compatible triple, and let $(x,y) \in \R^2$
satisfy  $y \geq x - 2^{-1} n^{1/3} \tot$. Let $s \in (s_1,s_2) \cap n^{-1} \Z$.
Suppose that the almost sure event that
 $\rho_n \big[ (x,s_1) \to (y,s_2) \big]$ is well defined occurs.  Select any element $(z,s) \in \rho_n \big[ (x,s_1) \to (y,s_2) \big]$.
The removal of $(z,s)$ from $\rho_n \big[ (x,s_1) \to (y,s_2) \big]$  creates two connected components. 
Taking the closure of each of these amounts to adding the point $(z,s)$ to each of them. The resulting sets are $n$-zigzags from $(x,s_1)$ to $(z,s)$, and from $(z,s)$ to $(y,s_2)$; indeed,
it is straightforward to see that these are the unique $n$-polymers given their endpoints. We use a concatenation notation~$\circ$ to represent this splitting. In summary,   $\rho_n \big[ (x,s_1) \to (y,s_2) \big]  = \rho_n \big[ (x,s_1) \to (z,s) \big]  \, \circ \, \rho_n \big[ (z,s) \to (y,s_2) \big]$. Naturally, we also have  $\weight_n \big[ (x,s_1) \to (y,s_2) \big] = \weight_n \big[ (x,s_1) \to (z,s) \big]  + \weight_n \big[ (z,s) \to (y,s_2) \big]$.
Indeed, the concatenation operation may be applied to any two $n$-zigzags
for which the ending point of the first equals the starting point of the second. Since the zigzags are subsets of $\R^2$, it is nothing other than the operation of union. The weight is additive for the operation.

We end with a notational convention.
\subsubsection{Boldface notation for parameters in statement applications}\label{s.boldface}
The various results that we recall in the present section come equipped with parameters that must be set in any given application. When such applications are made later in the article, we employ a boldface notation to indicate the parameter labels of the results being applied. This device permits occasional reuse of symbols and disarms notational conflict.

\subsection{Geometric and weight fluctuations of the polymer}

\subsubsection{Inputs on tail bounds on one-point polymer weight}

We will have need for control on the upper and lower tail of the weight of a polymer with given endpoints.
\begin{lemma}\label{l.onepointbounds}
There exist constants $C \in [1,\infty)$ and $c \in (0,1]$, and $n_0 \in \N$, such that the following holds. Let $n \in \N$ and $x,y \in \R$
satisfy $n \geq n_0$ and $\vert x -y \vert \leq c n^{1/9}$.
\begin{enumerate}
\item For $t \geq 0$,
$$
\PP \Big( \weight_n \big[ (x,0) \to (y,1) \big] + 2^{-1/2}(y-x)^2 \geq t \Big) \leq C \exp \big\{ - c t^{3/2} \big\} \, .
$$
\item For $t \geq 0$,
$$
\PP \Big( \weight_n \big[ (x,0) \to (y,1) \big] + 2^{-1/2}(y-x)^2 \leq - t \Big) \leq C \exp \big\{ - c t^{3/2} \big\} \, .
$$
\end{enumerate}
\end{lemma}
{\bf Proof.} This result follows from \cite[Proposition~$3.6$]{DeBridge} and translation invariance of Brownian LPP. \qed

\subsubsection{The tail of polymer fluctuation at given height}

We now specify a measure of the fluctuation of the polymer $\rho_n \big[ (x,s_1)\to (y,s_2) \big]$ at the intermediate moment $h \in [s_1,s_2] \cap n^{-1} \Z$, measuring the horizontal distance between the polymer at this height $h$ relative to the height-$h$ location $\tfrac{s_2 - h}{\tot} x + \tfrac{h- s_1}{\tot} y$ of the line that interpolates $(x,s_1)$ and $(y,s_2)$. We set 
$$
\fluct_n \big[ (x,s_1) \to (y,s_2) ; h \big] 
  =   \sup \Big\{ \big\vert u - \tfrac{s_2 - h}{\tot} x - \tfrac{h- s_1}{\tot} y \big\vert:  u \in \R \, , \, (u,h) \in \rho_n \big[ (x,s_1) \to (y,s_2) \big] \Big\} \, .
 $$
The typical order of this quantity is $\lambda^{2/3}$, where $\lambda$ equals $(h - s_1) \wedge (s_2 - h)$, with $\wedge$ denoting minimum.

\begin{proposition}[Theorem 1.7, \cite{NearGroundStates}]\label{p.deviation}
There exists $r_0 >0$ such that, for any $H_0 > 0$, we may choose $D > 0$ so that the following holds. Let $K > 0$, $r \geq r_0$, $a \in (0,1/4]$, $n \in \N$ and $s_1,s_2 \in n^{-1}\Z \cap [0,1]$ satisfy $s_1 \leq s_2$; $n \tot a$ and $n \tot(1-a)$ are at least $D$; $K a^{1/3} \leq H_0$ and $\vert K \vert \leq (n \tot)^{2/3}$.
Then
\begin{equation}\label{e.deviationdisplay} 
  \PP \Big( \sup \fluct_n \big[ (x,h_1) \to (y,h_2) ; h \big] \geq  r (a \tot)^{2/3}\big(\log a^{-1} \big)^{1/3} \Big) \leq D K^2 a^{-10/3}  a^{d r^3 } \, ,
\end{equation}
where the supremum is taken over $x,y \in [-K,K]\cdot s_{1,2}^{2/3}$, $h_1 \in n^{-1}\Z \cap [s_1,s_1 + \tot/3]$, $h_2 \in n^{-1}\Z \cap [s_2 - \tot/3,s_2]$ and 
$h \in n^{-1}\Z$ such that $\tfrac{h-h_1}{h_{1,2}} \in [a,2a] \cup [1-2a,1-a]$.
\end{proposition}
Here, we have written $h_{1,2} = h_2 - h_1$, in an extension of the notation $\tot$ from Subsection~\ref{s.compatible}.

The next result delivers a control on the fluctuation properties of polymers uniformly as it traverses across different height levels. 
\subsubsection{Control on polymer fluctuation uniform across heights} 
\begin{proposition}[Theorem 1.4, \cite{NearGroundStates}]
\label{p.toolfluc} 
\leavevmode
 \begin{enumerate}
\item There exist positive $H$,  $h$ and $r_0$, and $n_0 \in \N$, such that, when $n \in \N$ satisfies $n \geq n_0$, $k \in \N$ satisfies $2^k \leq h n$ and $r \in \R$ satisfies  $r_0 \leq r \leq n^{1/10}$, it is with probability at least $1 - H \exp \big\{ - h r^3 k \big\}$ that the following event occurs.
Let $x,y \in \R$ be of absolute value at most~$r$.  Let $h_1,h_2 \in n^{-1}\Z \cap [0,1]$ satisfy $h_{1,2} \in (2^{-k-1},2^{-k}]$ and
let $u,v \in \R$ be such that $(u,h_1)$ and $(v,h_2)$ belong to $\rho_n \big[ (x,0) \to (y,1) \big]$.  
Then
$$
\big\vert v -u \big\vert \leq H h_{1,2}^{2/3} \big( \log (1 + h_{1,2}^{-1}) \big)^{1/3}r \, .
$$
\item  There exist positive $G$, $H$,  $h$ and $r_0$, and $n_0 \in \N$, such that, when $n \in \N$ satisfies $n \geq n_0$, and $r \in \R$ satisfies  $r_0 \leq r \leq n^{1/10}$, it is with probability at least $1 - H n^{- h r^3}$  that the following event occurs.
As above, let $x,y \in \R$ be of absolute value at most~$r$, and let $u,v \in \R$ be such that $(u,h_1)$ and $(v,h_2)$ belong to $\rho_n \big[ (x,0) \to (y,1) \big]$. 
 Consider any $h_1,h_2 \in n^{-1}\Z \cap [0,1]$ that satisfy  $h_{1,2} < H n^{-1}$.  
Then
$$
\big\vert v -u \big\vert \leq G n^{-2/3} ( \log n )^{1/3}r
 \, .
$$
\end{enumerate}
 \end{proposition}
This result will allow us to gain control on the maximum fluctuation of such polymers.

\subsubsection{The tail of maximum polymer fluctuation}

The probability of lateral movement of polymers to distance $r$ decays as $\exp \big\{ -\Theta(1)r^3 \big\}$.

 \begin{proposition}[Corollary~$1.5$, \cite{NearGroundStates}]\label{p.lateral}
There exist positive $H$,  $h$ and $r_0$, and $n_0 \in \N$, such that, when $n \in \N$ satisfies $n \geq n_0$, and $r \in \R$ satisfies  $r_0 \leq r \leq n^{1/10}$, it is with probability at least $1 - H \exp \big\{ - h r^3\big\}$ that the following holds.
Let $x,y \in \R$ be of absolute value at most $r$. If $(u,h') \in \R \times \big( n^{-1}\Z \cap [0,1] \big)$ lies in   $\rho_n \big[ (x,0) \to (y,1) \big]$, then  
$\vert u \vert \leq H r$.
 \end{proposition}
 
We express a corollary in terms of $\mathsf{MaxFluc}_{n} \big[ (x,s_1) \to (y,s_2)  \big]$, the {\em maximum} horizontal  {\em fluctuation} of a point on  $\rho_n \big[ (x,s_1) \to (y,s_2) \big]$ from the linear interpolation of this polymer's endpoints, namely  
$$
\sup \Big\{ \, \big\vert u - \tfrac{s_2 - h}{\tot} x - \tfrac{h- s_1}{\tot} y \big\vert:  u \in \R \, , h \in [s_1,s_2] \cap n^{-1}\Z \, , \, (u,h) \in \rho_n \big[ (x,s_1) \to (y,s_2) \big] \, \Big\} \, .
$$
\begin{proposition}\label{p.maxfluc}
There exist
 positive $h$,  $d$ and $R_0$, and $n_0 \in \N$, such that, when $n \in \N$, $s_1,s_2 \in n^{-1}\Z$, $s_1 \leq s_2$, satisfy $ns_{1,2} \geq n_0$, $R \in \R$ satisfies  $R_0 \leq R \leq (n s_{1,2})^{1/10}$,
and $z \in \R$, we have that
$$
 \PP \, \Big( \, \tot^{-2/3} \sup \mathsf{MaxFluc}_{n} \big[ (x,s_1) \to (y,s_2)  \big]  \geq R \, \Big) \, \leq \, \exp \big\{ -d R^3 \big\} \, ,
$$
where the supremum is taken over $x,y \in \R$ for which $\vert x - z \vert$ and $\vert y - z \vert$ are at most $h R s_{1,2}^{2/3}$.
\end{proposition}
{\bf Proof.}
The result follows by taking $z =0$ by translation invariance, and $s_{1,2}=1$ by the scaling principle from Subsection~\ref{s.scalingprinciple}, and applying Proposition~\ref{p.lateral}. Using the boldface notation of Section~\ref{s.boldface}, we set the proposition's parameter ${\bf r}$ equal to $hR$ for this application; the constant $h$ we seek is then equal to  $H^{-1}$. \qed

\subsubsection{Gaining control on the weight of polymers}
%Polymer weight tails uniform in the endpoint pair
 Uniformly over compact variation of both endpoints, polymer weight differs from parabolic curvature by more than $r$ with a probability that decays as $\exp \big\{ - \Theta(1) r^{3/2} \big\}$.

\begin{proposition}[Proposition~$3.15$,~\cite{NearGroundStates}]\label{p.tailweight}
There exist $H_1,h,R_0 \in (0,\infty)$ and $n_0 \in \N$ such that, for $n \geq n_0$, 
 $R_0 \leq R \leq n^{1/30}$ and $0 < K \leq n^{1/46}$, 
$$
 \P \left(   \sup_{\begin{subarray} c (x,h_1) \in [-K,K] \times [-3,-1] \\   (y,h_2) \in [-K,K] \times [1,3] \end{subarray}} \bigg\vert  \, \weight_n \big[ (x,h_1) \to (y,h_2) \big] + 2^{-1/2} \frac{(x-y)^2}{h_{1,2}} \, \bigg\vert \geq R   \right) \, \leq \, H_1 K^2 \exp \big\{ - h R^{3/2} \big\} \, .
$$
\end{proposition} 

The next result offers control on the weight of sub-paths of polymers.

 \begin{proposition}[Theorem~$1.6(1)$,\cite{NearGroundStates}]\label{p.weight}
There exist positive $H$,  $h$ and $r_0$, and $n_0 \in \N$, such that, when $n \in \N$ satisfies $n \geq n_0$; $k \in \N$ satisfies $2^k \leq hn$;  and $r \in \R$ satisfies $r \geq r_0$, it is with probability at least $1 - H \exp \big\{ - h r^3 k \big\}$ that the following occurs.
Let $h_1,h_2 \in n^{-1}\Z \cap [0,1]$ satisfy $h_{1,2} \in ( 2^{-k-1}, 2^{-k}]$ and 
  $r \leq (n h_{1,2})^{1/64}$;
let $x,y \in \R$ be of absolute value at most $r$; and let
$u,v \in \R$ be such that $(u,h_1)$ and $(v,h_2)$ belong to $\rho_n \big[ (x,0) \to (y,1) \big]$. Then
$$
\big\vert \weight_n \big[ (u,h_1) \to (v,h_2) \big] \big\vert \leq H^2 r^2 \cdot 
 h_{1,2}^{1/3} \big( \log  h_{1,2}^{-1} \big)^{2/3} 
\, .
$$
 \end{proposition}
 
\subsubsection{Modulus of continuity for the static polymer}
Viewed as a rough function of its vertical coordinate, the polymer has H{\"o}lder exponent~$2/3-$, with a logarithmic power correction of $1/3$. 

\begin{definition}\label{d.staticpolymer}
Let $\phi$ be an $n$-zigzag from $(0,0)$ to $(0,1)$. Let the parameters $\kappa \in (0,e^{-1})$ and $R > 0$ be given.
The zigzag $\phi$ is said to be $(\kappa,R)$-regular if, whenever a pair $(x,s_1)$ and $(y,s_2)$
of elements of $\phi \cap \big( \R \times n^{-1}\Z \big)$ satisfy $\tot = s_2 - s_1 \in [0, 6\kappa]$, we have that
\begin{equation}\label{e.regulardef}
 \big\vert y - x \big\vert \leq R \kappa^{2/3} \big( \log \kappa^{-1} \big)^{1/3} \, .
\end{equation}
\end{definition}
\begin{proposition}\label{p.regular}
There exist positive $d$, $K_0$ and $R_0$ such that, for $\kappa \in \big( K_0 n^{-1},e^{-1} \big)$ and $R \geq R_0$,
$$
 \PP \, \Big( \, \rho_n \big[ (0,0) \to (0,1) \big] \, \, \textrm{is not $(\kappa,R)$-regular} \, \Big) \, \leq \, \kappa^{dR^3} \, .
$$
\end{proposition}
{\bf Proof.} It is enough to prove the result with the notion of  $(\kappa,R)$-regular modified so that the stronger condition $\tot \in [3\kappa, 6\kappa]$ holds, because the case where $\tot < 3\kappa$
may be then treated by the triangle inequality.

The altered form of the result is obtained by applying Proposition \ref{p.toolfluc}(1) with  ${\bf r} =  d_0 R$,  %
where $d_0 > 0$ is a suitably small constant, and with the dyadic scale ${\bf k} \in \N$ ranging over two consecutive values in order that $[3\kappa,6\kappa]$ be contained in the union of the intervals $(2^{-k-1},2^{-k}]$.  \qed

\subsection{Twin peaks via Brownian regularity}\label{s.twinpeakresults} 

Our proof overview in Section~\ref{s.conceptsoverlap} indicated that Brownianity of the routed weight profile introduced in~(\ref{e.routedprofile}) would be a vital technical component of our analysis. In this section, we develop the needed apparatus. 

Let $n \in \N$ and  $a \in (0,1) \cap n^{-1}\Z$. 
The routed weight profile $x \to Z_n(x,a)$, as specified in~(\ref{e.routedprofile}), reports the maximum weight of an $n$-zigzag from $(0,0)$ to $(0,1)$ which passes through $(x,a)$. We would be troubled by an awkwardness in this definition were we to adopt it for rigorous analysis: the two weights on the right-hand side of~(\ref{e.routedprofile}) may be viewed as functions of $x$ for given $a$; but they then lack independence, because the randomness in the scaled noise environment indexed by level~$a$ contributes to both of these weight profiles. The  definition of $Z_n(\cdot,a)$ that we in fact adopt avoids this problem.
Polymers whose weight supremum equals $Z_n(x,a)$ not only visit $(x,a)$
but depart from level $a$ at $x$.
\begin{definition}\label{d.routedweightprofile}
Let $n \in \N$, $a \in (0,1) \cap n^{-1}\Z$ and $x \in \R$. Let $Z_n(x,a)$ denote the supremum of the weights of $n$-zigzags that begin at $(0,0)$; end at $(0,1)$; and contain the point $(x,a)$ but no point of the form $(x + u,a)$ for $u >0$. 
\end{definition}
The next result indicates how the routed weight profile $\R \to \R: x \to Z_n(x,a)$
as specified here is a slight perturbation of its informal cousin~(\ref{e.routedprofile}). 
\begin{lemma}\label{l.routedprofile}
Set $\aplus = a+ n^{-1}$ and $\xminus = x - 2^{-1}n^{-2/3}$.
\begin{enumerate}
\item The routed weight profile is given by 
\begin{equation}\label{e.routedweightprofile}
 Z_n(x,a) \, = \,  \weight_n \big[(0,0) \to (x,a) \big] \, + \, \weight_n \big[(\xminus,\aplus) \to (0,1) \big]  \, . 
\end{equation}
\item Almost surely, the maximizer of $Z_n(\cdot,a)$, namely the value of $x \in \R$ for which $Z_n(x,a)$ equals the supremum of $Z_n(z,a)$ over $z \in \R$, is unique and equals $\rho_n(a)$.
\end{enumerate}
\end{lemma}
{\bf Proof: (1).} 
Let $\psi$ denote an $n$-zigzag that begins at $(0,0)$, ends at $(0,1)$, and for which $x = \sup \big\{ z \in \R: (z,a) \in \psi \big\}$. Let $\psi^-$ denote the initial zigzag of $\psi$ that ends at $(x,a)$. Note that $\psi$ reaches $\R \times \{ a + n^{-1} \}$ at $(\xminus,\aplus)$. Let $\psi^+$ denote the final sub-zigzag of $\psi$ that begins at $(\xminus,\aplus)$. Thus, $\weight_n (\psi) = \weight_n (\psi^-)+ \weight_n (\psi^+)$. By definition, $Z_n(x,a)$ equals the supremum of  $\weight_n (\psi)$ over such $\psi$. We see that $Z_n(x,a)$ is at most the right-hand side of~(\ref{e.routedweightprofile}). But equality may be obtained by varying $(\psi^-,\psi^+)$ subject to the endpoint constraints that specify this pair.\\
{\bf (2).} The polymer $\rho_n$ is almost surely unique by~\cite[Lemma~$4.6(1)$]{Patch}.
Since 
$\rho_n(a)$ is by definition the location of departure of the polymer $\rho_n$ from $\R \times \{ a \}$, we see that it is the maximizer of $x \to Z_n(x,a)$. \qed

The notation $\aplus$ and $\xminus$ is adopted henceforth. It reflects the two denoted quantities being merely microscopically perturbed copies of $a$ and $x$. 

Next we ask: what is the probability of {\em twin peaks}, namely  that there exists $x \in \R$ such that $Z_n(x,a)$
rivals the maximum value of $Z_n(\cdot,a)$, with $Z_n(x,a)$ being less than this maximum by a small multiple~$\hata$ of the square-root distance $\big( x - \rho_n(a) \big)^{1/2}$? Our answer is 
obtained in~\cite{NearGroundStates}, underpinned by the Brownian weight profile regularity result~\cite[Theorem~$3.11$]{DeBridge},
which itself harnesses technique from~\cite{BrownianReg}\footnote{The recent preprint~\cite{D23A} strengthens~\cite[Theorem~$3.11$]{DeBridge}, and it would be natural that it offer the role of underlying input here. The implications of this modification for the present article are minor however, and we prefer to rely verbatim on the result in the published companion article~\cite{NearGroundStates}.}:
twin peaks are a rarity, with such probabilities being bounded above by the product of~$\hata$ and a lower-order correction $\exp \big\{ \Theta(1) \big( \log \hata^{-1} \big)^{5/6} \big\}$. This is the implication of the next result when the parameter $R$ is set to equal zero, which is the choice that we will make in applications; when $R$ is non-zero, the factor  $e^{-\Theta(1) R^2 \ell}$ represents a penalty for the maximizer being far from the origin.

\begin{theorem}[Theorem~$1.3$,\cite{NearGroundStates}]\label{t.nearmax}
For $K$ any compact interval of $(0,1)$,
there exist positive constants $H = H(K)$ and $h = h(K)$ and an integer $n_0 = n_0(K)$ such that
the following holds.
Let $n \in \N$, $R \in \R$, $\ell \geq 1$, $\ell' > 0$, $a \in n^{-1}\Z \cap K$, $\hata > 0$ and $\e > 0$.
 Suppose that $n \geq n_0$,
 $\vert R \vert \leq h n^{1/9}$,
$\ell \in (3\e, h n^{1/\macrobig})$ and $\ell' \in (3\e,\ell]$. 
Denoting $\hata \wedge 1$ by $\hata_*$, we have that  
 \begin{eqnarray*}
& & \PP \Big( M \in [R - \ell/3,R+\ell/3]  \, , \, \sup_{x \in \R: x - M  \in [\e,\ell'/3]} \big( Z_n(x,a) + \hata (x - M)^{1/2} \big) \geq Z_n(M,a) \Big) \\
 & \leq & \log \big( \ell' \e^{-1} \big) \max \Big\{ \hata_* \cdot 
 \exp \big\{  - hR^2 \ell + H \ell^{\macroseventeen} \big( 1 + R^2 + \log \hata_*^{-1}  \big)^{5/6} \big\} ,   \exp \big\{ - h n^{1/12}   \big\} \Big\} \, , 
\end{eqnarray*}
where $M$ denotes $\rho_n(a)$, the almost surely unique maximizer of $x \to Z_n(x,a)$.
\end{theorem}
The parameters $R$, $\e$, $\ell'$, $\ell$ and $\sigma$ must be set in any given application. For example: to gauge the probability of a scenario that is illustrated by Figure~\ref{f.proxy}(middle,lower),
in which the routed weight profile maximum on $[-2,2]$ is attained at some $M$ but rivalled to order $\tau^{1/2}$ at some $x \in [-2,2]$ with $\vert x - M \vert \in [1,2]$, we would set $R = 0$; $\ell' = \ell = 6$; $\e = 1$; and $\sigma = \Theta(\tau^{1/2})$. 
We learn that such twin peaks arise with probability at most $\tau^{1/2} \exp \big\{ \Theta(1) (\log \tau^{-1})^{5/6} \big\}$, uniformly in high $n$.
 \section{Subcritical weight stability, uniformly in the endpoint pair}

Our applications require assertions of weight stability under increments~$t$ in dynamic time of the form $t \ll n^{-1/3}$ that are more robust than Proposition~\ref{p.onepoint}, which concerns changes in polymer weight with a fixed endpoint pair. Theorem~\ref{t.stable} is a suitably robust tool, en route to whose derivation Proposition~\ref{p.stablehorizontal} is a useful halfway house.

 Let $n \in \N$, $x,y \in \R$ and $h_1,h_2 \in n^{-1}\Z$. Define the weight difference
\begin{equation}\label{e.delta}
 \Delta^{0,t} \big[ (x,h_1) \to  (y,h_2) \big] = \, \weight_n^t\big[ (x,h_1)\to(y,h_2) \big] - \weight_n^0\big[ (x,h_1)\to(y,h_2) \big] \,  ,
\end{equation}
and recall that $t = n^{-1/3} \tau$.
\begin{theorem}(Planar stability)\label{t.stable}
There is a universal constant $a>0$, such that for all  $0<\taumac < a$, with $t \in [0, \taumac n^{-1/3}]$,  $\tau = t n^{1/3}$, the following holds. 
Let $n \in \N$, $D \geq 1$  satisfy 
$n \geq H \cdot D^{18}  (\ell + 1)^{18} 2^\ell \taumac^{-1/5}$ for a suitably high constant $H > 0$;
let $\ell \in \N$ and $s_1,s_2 \in n^{-1}\Z$ satisfy $2^{-\ell-1} \leq \tot \leq 2^{-\ell}$; and let $I \subset \R$ be an interval of length $D \,2^{-2\ell/3} (\ell + 1)^{1/3}$. 
Then
\begin{equation}\label{e.stable}
  \PP \bigg(  \sup \, h_{1,2}^{-1/3}  \Big\vert \Delta^{0,t} \big[ (x,h_1) \to   (y,h_2) \big]  \Big\vert \geq \,  (\ell+1)D   \taumac^{2/1001}   \bigg) 
  \leq  D^2  2^{-\ell/7} \taumac^{1/12}  \, ,
\end{equation}
where $h_{1,2} = h_2 - h_1$, and where the supremum is taken over $x,y \in I$; $h_1$ in the first-third interval $[s_1, s_1 + \tot/3]$; and  $h_2$ in the final-third interval $[s_2 - \tot/3,s_2]$.
\end{theorem}

On account of the desired uniformity, the quality of the bound deteriorates from the estimate provided by Proposition \ref{p.onepoint}(2).  Further, while one expects such a bound to improve as $t$ tends to zero, our purpose is served by the presented $\hat\tau$-determined bound that is uniform in $0\le \tau \le \hat\tau.$

The halfway house assertion  is uniform in the endpoints merely under horizontal perturbation. The constants $C$ and $c$ in its statement arise from Lemma~\ref{l.onepointbounds}.

\begin{proposition}(Horizontal stability)\label{p.stablehorizontal}
There is a universal constant $a>0$, such that, for all  $0<\taumac < a$, the following holds.
Let $n \in \N$, $D \geq 1$ and $t \in [0, \taumac n^{-1/3}]$, with $\tau = t n^{1/3}$, satisfy $n \geq   10^{29} D^{18}  c^{-9} (\ell + 1)^{18} 2^\ell \big( \log \taumac^{-1}\big)^9$;  let $\ell \in \N$, and $s_1,s_2 \in n^{-1}\Z$ satisfy $2^{-\ell-1} \leq \tot \leq 2^{-\ell}$; and let $I$ and $J$  be intervals of length $D \,2^{-2\ell/3} (\ell + 1)^{1/3}$ whose left-hand endpoints differ in absolute value by at most 
$2^{-3}3^{-1} c  (n \tot)^{1/18}$. 
Then
$$
 \PP \, \bigg( \tot^{-1/3} \sup_{x \in I, y\in J} \Big\vert \Delta^{0,t} \big[ (x,s_1) \to (y,s_2) \big] \Big\vert \,   \geq  (\ell +1 ) \cdot 3000 c^{-1/2} \big( \log \taumac^{-1}\big)^{1/2}  \taumac^{1/500}  \bigg) 
 $$
is at most  $10^8 D^2 C  2^{-\ell/6}    \taumac^{49/100}$.
\end{proposition}

We derive Proposition~\ref{p.stablehorizontal} and then prove Theorem~\ref{t.stable} as a consequence.

\subsection{Horizontally uniform weight stability: proving Proposition~\ref{p.stablehorizontal}}

The result that we seek to demonstrate is a strengthening of Proposition~\ref{p.onepoint}(2), which asserts a companion claim in the case that the endpoint pair $\big\{ (x,s_1),(y,s_2) \big\}$
for the time zero and time-$t$ polymers is fixed. We will derive Proposition~\ref{p.stablehorizontal} from this precursor by invoking the latter as $x$ and $y$ vary over  suitably fine meshes in $I$ and $J$. An accompanying tool is then needed to treat the remaining $(x,y) \in I \times J$: an understanding that the weight function $I \times J \to \R: (x,y) \to \weight_n \big[ (x,s_1) \to (y,s_2) \big]$ has a degree of regularity in response to variation of the arguments $x$ and $y$. H\"older regularity with an exponent of $1/2-$ may be expected due to the locally Brownian form of polymer weight profiles. We begin by stating a suitable version, Proposition~\ref{p.fluc}, of this weight profile regularity tool. The basic input driving this result is  \cite[Theorem $1.1$]{ModCon}.

Let $(x_1,x_2)$ and $(y_1,y_2)$ belong to $\R^2_\leq$.
It is useful to define the {\em parabolically adjusted} weight difference
\begin{equation}\label{paraboladjust}
 \Delta^{\cup} \, \weight_n \big[ ( \{ x_1,x_2\} ,s_1) \to ( \{y_1,y_2\},s_2) \big] 
\end{equation}
to equal 
$$
\bigg( \weight_n \big[ (x_2,s_1) \to (y_2,s_2) \big] + 2^{-1/2} \frac{(y_2 - x_2)^2}{s_2-s_1} \bigg) \, - \, \bigg( \weight_n \big[ (x_1,s_1) \to (y_1,s_2)  \big] +  2^{-1/2} \frac{(y_1 - x_1)^2}{s_2-s_1} \bigg) \, ,
 $$
 since a slope arising from differences in the globally parabolic form of weight profiles is eliminated by working with $\Delta^{\cup} \, \weight_n$ in place of a difference of weights $\weight_n$; this permits much higher choices of $\vert x - y \vert$ in our assertion of square-root weight fluctuation under horizontal endpoint perturbation. 

\begin{proposition}\label{p.fluc}
Let $C$ and $c$ be the positive constants furnished by Lemma~\ref{l.onepointbounds}, and let $a \in (0,2^{-4}]$. 
Let $(n,s_1,s_2) \in \N \times \R_\leq^2$ be a compatible triple for which 
$n \tot \geq 10^{32} c^{-18}$ and let $x,y \in \R$ satisfy  $\big\vert x - y  \big\vert \tot^{-2/3} \leq 2^{-2} 3^{-1} \rsc  (n \tot)^{1/18}$. 
Let 
 $K \in \big[10^4 \, , \,   10^3 (n \tot)^{1/18} \big]$.
Then
$$
\PP \, \left( \, \sup_{\begin{subarray}{c} x_1,x_2 \in [x,x+a\tot^{2/3}] \, , \, x_1 < x_2 \\
    y_1,y_2 \in [y,y+a\tot^{2/3}] \, , \, y_1 < y_2 \end{subarray}}  \Big\vert  \Delta^{\cup} \, \weight_n^0 \big[ ( \{ x_1,x_2\} ,s_1) \to ( \{y_1,y_2\},s_2) \big]  \Big\vert \, \geq \, K a^{1/2} \tot^{1/3} \, \right) 
$$
is at most 
$10032 \, C \exp \big\{ - c 2^{-24} K^2 \big\}$.
\end{proposition}
{\bf Proof.}
The special case that $s_1 = 0$ and $s_2=1$ is implied by \cite[Theorem $1.1$]{ModCon}.
(The upper bound in the latter result is $10032 \, C  \exp \big\{ - c_1 2^{-21}   R^{3/2}   \big\}$. But $c_1 = 2^{-5/2} c \wedge 1/8$, where the constant $c > 0$ is at most one, so that
 we obtain the upper bound in Proposition~\ref{p.fluc}.) 
 The scaling principle from Section~\ref{s.scalingprinciple}
then yields the proposition from this special case. \qed

Equipped with this tool, we may indeed give the next proof.

{\bf Proof of Proposition~\ref{p.stablehorizontal}.}
Write the interval $I$ in the form $[u,v]$.
For $\eta > 0$, let $I_\eta$ denote the discrete mesh $I_\eta = [u,v] \cap \big( u + \eta \N \big)$, so that the leftmost element of $I_\eta$ is $u$. 
Similarly, we write $J = [u',v']$ and set $J_\eta =  [u',v'] \cap \big( u' + \eta \N \big)$.

Let $x \in I$ and $y \in J$ be given; and  
let $x_\eta$ and $y_\eta$ denote the respective elements of $I_\eta$ and $J_\eta$ that are encountered at or directly to the left of $x$ and $y$.

Recall~(\ref{e.delta}), and adopt the shorthand $\Delta^{0,t}[ x \to y]$ for  $\Delta^{0,t}\big[ (x,s_1) \to (y,s_2) \big]$. Note that  
\begin{eqnarray*}
 \Delta^{0,t} [ x \to  y ] -  \Delta^{0,t} [ x_\eta \to y_\eta ]  & = &  \Big( \, \weight_n^0\big[ (x_\eta,s_1)\to(y_\eta,s_2) \big] - \weight_n^0\big[ (x,s_1)\to(y,s_2) \big] \, \Big) \\
  & &  -  \,  \Big( \, \weight_n^t\big[ (x_\eta,s_1)\to(y_\eta,s_2) \big] \, - \, \weight_n^t\big[ (x,s_1)\to(y,s_2) \big] \, \Big) \, .
\end{eqnarray*}
Adding the term $2^{-1/2}\tfrac{(x_\eta-y_\eta)^2}{\tot} - 2^{-1/2}\tfrac{(x-y)^2}{\tot}$ inside the two sets of parentheses, we obtain
\begin{eqnarray}
 & & \Delta^{0,t}\big[  x \to y  \big] -  \Delta^{0,t}\big[ x_\eta  \to y_\eta \big]  \nonumber \\
 & =  & \Delta^{\cup} \, \weight_n^0 \big[ ( \{ x_\eta,x\} ,s_1) \to ( \{y_\eta,y\},s_2) \big] \, - \,  \Delta^{\cup} \, \weight_n^t \big[ ( \{ x_\eta,x\} ,s_1) \to ( \{y_\eta,y\},s_2) \big] \label{e.parenthesispair} \, .
\end{eqnarray}
We seek an upper bound on the tail of the random variable $\sup_{x,y \in I} \big\vert \Delta^{0,t} [  x \to y ] \big\vert$.
Fixed endpoints' weight stability Proposition~\ref{p.onepoint}(2) and a union bound over endpoint pairs lying in the mesh $I_\eta$ will lead to control on the term $\Delta^{0,t}[ x_\eta \to  y_\eta ]$. The pair of parabolically adjusted given-time weight differences in~(\ref{e.parenthesispair}) will then be controlled by means of Proposition~\ref{p.fluc}.

Two claims correspond with these steps. First, for any $r > 0$, we {\em claim} that

\begin{equation}\label{e.claimone}
 \PP \, \bigg( \sup_{x' \in I_\eta, y' \in J_\eta} \big\vert \,    \Delta^{0,t} [ x'  \to y'  ]     \big\vert \geq \tot^{1/2} \taumac^{1/2} r \bigg) \, \leq \, {\sqrt{2}}\Big( (v-u)\eta^{-1} + 1 \Big)^2 r^{-1} \, .
\end{equation}
Indeed, by $\tau \leq \taumac$, the Cauchy-Schwarz inequality and Proposition~\ref{p.onepoint}(2) with ${\bf x} = x'$ and ${\bf y} = y'$, we find that $\E \big\vert   \Delta^{0,t} [ x' \to y' ] \big\vert\leq {\sqrt{2}}\tot^{1/2}  \taumac^{1/2}$, so that Markov's inequality and a union bound that uses that $\vert I_\eta \vert$ and  $\vert J_\eta \vert$ are at most $(v-u)\eta^{-1} + 1$ yields~(\ref{e.claimone}). Here, we make further use of the notational convention indicated in Section~\ref{s.boldface}, whereby a boldface font is used to indicate the settings for the parameters of an input result. Note that the above application of Proposition~\ref{p.onepoint}(2) with ${\bf x} = x'$ and ${\bf y} = y'$ necessitates verifying that the hypotheses on the latter in fact hold in this case. This leads to certain mundane computations that are deferred to  Appendix~\ref{s.calcder}.

Next, gather the hypotheses on parameters $n \in \N$, $\eta > 0$, $s_1,s_2 \in n^{-1}\Z$,  $x,y \in \R$  and $K > 0$ that $\eta \tot^{-2/3} \in (0,2^{-4}]$; that $n \tot \geq 10^{32} c^{-18}$; that $D 2^{-2\ell/3} (\ell + 1)^{1/3} \tot^{-2/3} \leq 2^{-2} 3^{-1} \rsc  (n \tot)^{1/18}$; and that
 $K \in \big[10^4 \, , \,   10^3 (n \tot)^{1/18} \big]$. 
Our second {\em claim} is that, when these hypotheses hold, 
\begin{eqnarray}
 & & \PP \bigg( \sup_{x \in I,y \in J} \Big\vert \Delta^{\cup} \, \weight_n \big[ ( \{ x_\eta,x\} ,s_1) \to ( \{y_\eta,y\},s_2) \big] \Big\vert \geq K \eta^{1/2}  \bigg) \label{e.claimtwo} \\
 & \leq & \Big( (v-u)\eta^{-1} + 1 \Big)^2 \cdot 10032 \, C \exp \big\{ - c \, 2^{-24} K^2 \big\}
  \, . \nonumber
\end{eqnarray}
The bound follows from Proposition~\ref{p.fluc} with parameter settings ${\bf a} = \eta \tot^{-2/3}$, ${\bf x} = x_\eta$ and ${\bf y} = y_\eta$; and by a union bound over pairs $(x_\eta,y_\eta)$ valued in $I_\eta \times J_\eta$ (again,  the verification of hypotheses is carried out in Appendix~\ref{s.calcder}).

Revisiting~(\ref{e.parenthesispair}) equipped with the two claims---the latter applied at times zero and $t$---we learn that, for $r >0$ and for $K$ and $\eta$ satisfying the hypothesised constraints,
\begin{eqnarray*}
 & & \PP \bigg( \tot^{-1/3} \sup_{x \in I,y \in J}  \big\vert \Delta^{0,t}[ x \to y] \big\vert  \geq  \tot^{1/6} \taumac^{1/2} r + 2 K \eta^{1/2} \tot^{-1/3} \bigg) \\
 & \leq & \Big( (v-u)\eta^{-1} + 1 \Big)^2  \Big( 2^{1/2} r^{-1} \, + \, 2\cdot 10032 \, C \exp \big\{ - c \, 2^{-24} K^2 \big\} \Big) \, .
\end{eqnarray*}

We now select the parameters $r$, $K$ and $\eta$. Setting $\tot^{1/6} \taumac^{1/2} r = 2 K \eta^{1/2}  \tot^{-1/3}$, we may eliminate $r$ from the displayed bound; if we further insist that $\eta \leq v -u$, we obtain 
\begin{eqnarray*}
 & & \PP \bigg( \tot^{-1/3} \sup_{x \in I,y \in J}  \big\vert \Delta^{0,t}[x \to y] \big\vert  \geq  4 K \eta^{1/2}  \tot^{-1/3}  \bigg) \\
 & \leq &  4 (v-u)^2 \eta^{-2}  \Big( \tot^{1/2} \taumac^{1/2} 2^{-1/2} K^{-1} \eta^{-1/2} \, + \, 2\cdot 10032 \, C \exp \big\{ - c 2^{-24} K^2 \big\} \Big) \, .
\end{eqnarray*}

Recall from before~(\ref{e.claimtwo}) that we must impose $\eta \leq 2^{-4} \tot^{2/3}$, and from Proposition~\ref{p.stablehorizontal} that $2^{-\ell - 1 } \leq \tot \leq 2^{-\ell}$. We may thus set $\eta = 2^{-5} 2^{-2\ell/3}\phi$---where $\phi \leq 1$ remains to be selected. Recall further that $v -u = \vert I \vert = D \, 2^{-2\ell/3} (\ell + 1)^{1/3}$. Again using $2^{-\ell} \geq \tot \geq 2^{-\ell - 1}$,   
we see then that 
\begin{eqnarray*}
 & & \PP \bigg( \tot^{-1/3} \sup_{x \in I,y \in J}  \big\vert \Delta^{0,t}(x,y) \big\vert  \geq  2^{-1/6} K  \phi^{1/2}   \bigg) \\
 & \leq &  D^2  (\ell + 1)^{2/3}  \phi^{-2}   \Big( 2^{-\ell/6} \taumac^{1/2} K^{-1} 2^{14}  \phi^{-1/2}  \, + \, 2^{13} \cdot 10032  C \exp \big\{ - c 2^{-24} K^2 \big\} \Big) \, .
\end{eqnarray*}

{We now choose $\phi$ so that $\phi^{-5/2} = \taumac^{-1/100}$. That is, we set $\phi = \taumac^{1/250}$.} We obtain
\begin{eqnarray*}
 & & \PP \bigg( \tot^{-1/3} \sup_{x \in I,y \in J}  \big\vert \Delta^{0,t}(x,y) \big\vert  \geq  2^{-1/6} K  \taumac^{1/500}     \bigg) \\
 & \leq & D^2 (\ell + 1)^{2/3}    \Big(  2^{-\ell/6}  2^{14} \taumac^{1/2 -1/100} K^{-1}    \, + \,  2^{13} \cdot 10032 \,  \taumac^{-1/125}   C \exp \big\{ - c 2^{-24} K^2 \big\} \Big) \, .
\end{eqnarray*}
We now set $K = (\ell + 1) c^{-1/2} 2^{12} (1/2 + 1/125)^{1/2} \big( \log \taumac^{-1}\big)^{1/2}$. Note that, since $\ell \geq 0$, $e^{- c 2^{-24} K^2} \leq \taumac^{(\ell + 1)^2(1/125 + 1/2)} \leq \taumac^{1/125 + (\ell+1)^2/2}$. 
We find that
\begin{eqnarray*}
 & & \PP \bigg( \tot^{-1/3} \sup_{x \in I,y \in J}  \big\vert \Delta^{0,t}(x,y) \big\vert  \geq  (\ell + 1) c^{-1/2} 2^{12} (127/250)^{1/2} \big( \log \taumac^{-1}\big)^{1/2}  \taumac^{1/500}  \bigg) \\
 & \leq & D^2 (\ell + 1)^{2/3}    \Big(  2^{-\ell/6} (\ell +1)^{-1} 2^2 \taumac^{49/100} c^{1/2}  (250/127)^{1/2}   \big( \log \taumac^{-1}\big)^{-1/2}  \, + \, 2^{13} \cdot 10032    C \taumac^{(\ell+1)^2/2} \Big)  \, .
\end{eqnarray*}
We may suppose that $\taumac \leq e^{-1}$. Using this alongside $c \leq 1$, $C \geq 1$ and $2^{12} (127/250)^{1/2} \leq 3000$, 
 $2^{13} \cdot 10032 + 4(250/127)^{1/2}  \leq 10^8$
and  that $\taumac^{(\ell+1)^2/2}$  is at most $\taumac^{1/2} 2^{-\ell/6} (\ell + 1)^{-2/3}$ for $\ell \geq 0$ and $\taumac \leq 2^{-1}$, we arrive at 
$$
 \PP \bigg( \tot^{-1/3} \sup_{x \in I,y \in J}  \big\vert \Delta^{0,t}(x,y) \big\vert  \geq  (\ell +1) \cdot 3000 c^{-1/2} \big( \log \taumac^{-1}\big)^{1/2}  \taumac^{1/500}  \bigg) 
  \leq  10^8 D^2 C  2^{-\ell/6}    \taumac^{49/100}  \, ,
$$
and thus complete the proof of Proposition~\ref{p.stablehorizontal}. 

\qed

\subsection{Weight stability with planar uniformity: the proof of Theorem~\ref{t.stable}}

Given is a rectangle $I \times [s_1,s_2]$ of width $D \,2^{-2\ell/3} (\ell + 1)^{1/3}$ and height $\tot \in [2^{-\ell-1},2^{-\ell}]$; our task is to bound the upper tail of the supremum of the  absolute value of the  difference in weight between times zero and~$t$ of polymers moving from some point $(x,h_1)$ in the rectangle's lower third to a point $(y,h_2)$ in its upper third. 

\subsubsection{The method of proof in broad brushstrokes} We indicate the method, specifying some artefacts for now only roughly, and illustrating them in Figure~\ref{f.infancyprimedotage}.
Fix endpoints $(x,h_1)$ and $(y,h_2)$, and let $l = l \big[ (x,h_1) \to (y,h_2) \big]$ denote the planar line segment that interpolates them. Let $\eta \in n^{-1} \Z$ be a positive parameter that is rather small compared to $\tot$.  Choose $h_1^+[\eta] \in [h_1 + \eta,h_1 + 2\eta] \cap \Z \eta$ and   $h_2^-[\eta] \in [h_2 - 2\eta,h_2 - \eta] \cap \Z \eta$. These $\eta$-mesh points are {\em early} and {\em late} moments in the life of the polymer 
$\rho_n \big[ (x,h_1) \to (y,h_2) \big]$, separated from the polymer's starting and ending moments by order~$\eta$.
Indeed, the polymer's lifetime $[h_1,h_2]$ is composed of three epochs: a lengthy {\em prime} $\big[h_1^+[\eta],h_2^-[\eta]\big]$ bordered on either end by a brief {\em infancy} $\big[h_1,h_1^+[\eta]\big]$ and a brief {\em dotage} $\big[h_2^-[\eta],h_2\big]$. 
 Let $J \subset \R$ be an interval of length a little higher than~$\eta^{2/3}$ centred at the line $l$'s location at moment $h_1^+[\eta]$; and let $K$ be a similar interval centred at $l$'s location at moment $h_2^-[\eta]$. The plan of attack draws on three elements:
\begin{enumerate}
\item Via the two-thirds exponent for polymer fluctuation, argue that any polymer $\rho_n \big[ (x,h_1) \to (y,h_2) \big]$ typically passes through $J \times \{ h_1^+[\eta]  \}$ and $K \times \{ h_2^-[\eta] \}$.
\item Via the one-third exponent for weight, show that the infancy  
and dotage 
contribute a negligible order of $\eta^{1/3}$ to the weight $\weight_n  \big[ (x,h_1) \to (y,h_2) \big]$. 
\item Harnessing horizontal stability Proposition~\ref{p.stablehorizontal} for the pair of moments $(h_1^+[\eta],h_2^-[\eta])$,  we will see that  the weight difference $\big\vert \weight_n^t \big[ (u,h_1^+[\eta]) \to (v,h_2^-[\eta]) \big] -  \weight_n^0 \big[ (u,h_1^+[\eta]) \to (v,h_2^-[\eta]) \big] \big\vert$ is small relative to these weights' order, for any pair $(u,v) \in J \times K$. 
\end{enumerate}
\begin{figure}[t]
\centering{\epsfig{file=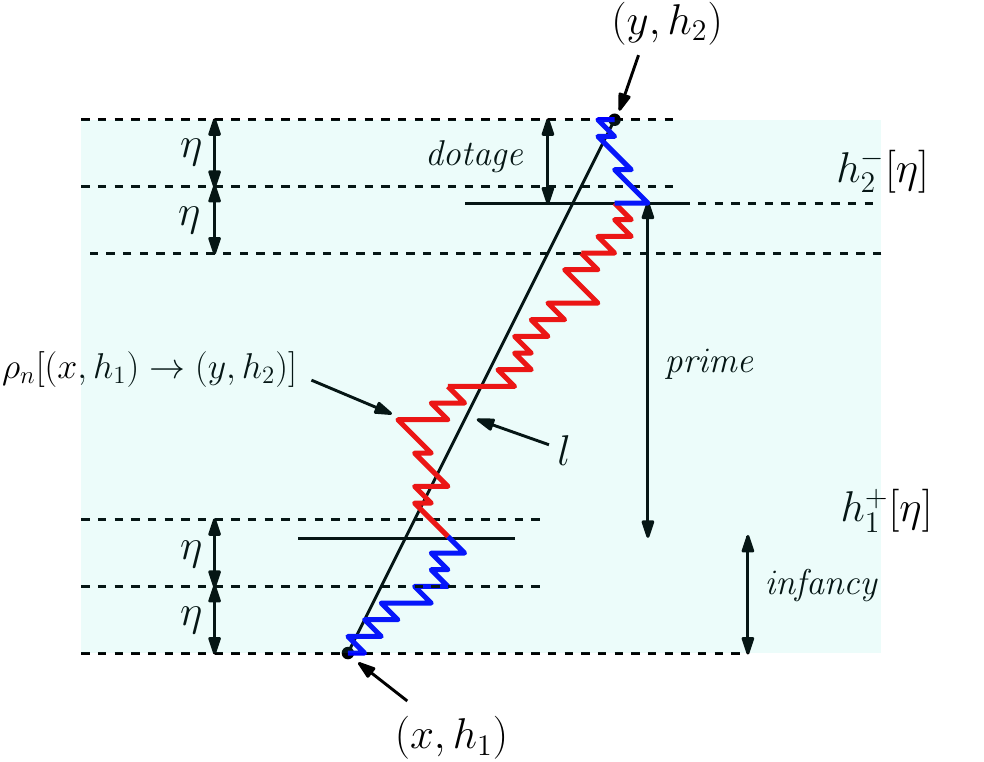, scale=0.8}}
\caption{In this instance, as typically, the polymer $\rho_n \big[ (x,h_1 ) \to (y,h_2) \big]$
visits the bold horizontal intervals $J \times \{ h_1^+[\eta] \}$ and $K \times \{ h_2^-[\eta]\}$ at the start and end of its prime.}\label{f.infancyprimedotage}
\end{figure} 
We will quote or derive rigorous renderings of each of these elements. For now, we explain heuristically how to combine them to prove Theorem~\ref{t.stable}.
Let $u$ and $v$ denote the locations of passage of the time-zero polymer $\rho_n^0 \big[ (x,h_1) \to (y,h_2) \big]$ at the early and late moments. By~(1), we may neglect the possibility that $u \not\in J$ or that $v \not\in K$. In its prime, during $\big[h_1^+[\eta],h_2^-[\eta]\big]$, the time-zero polymer passes from $u$ to $v$, and accrues weight 
$\weight_n^0 \big[ (u,h_1^+[\eta]) \to (v,h_2^-[\eta]) \big]$  along the way. Up to two additive terms indexed by the polymer's infancy and its dotage---terms that are negligible by (2)---this accrued weight equals the polymer's total weight $\weight_n^0 \big[ (x,h_1) \to (y,h_2) \big]$. But the accrued weight is little affected by the time change $0 \to t$ in view of~(3),
and the resulting slightly altered weight $\weight_n^t \big[ (u,h_1^+[\eta]) \to (v,h_2^-[\eta]) \big]$ approximates the maximum time-$t$
weight of a path from $(x,h_1)$ to $(y,h_2)$ that is constrained to pass via $(u,h_1^+[\eta])$ and $(v,h_2^-[\eta])$, because the two other contributions are negligible in light of~(2).
Thus,  it is typical that the bound $\weight_n^t \big[ (x,h_1) \to (y,h_2) \big] \geq \weight_n^0 \big[ (x,h_1) \to (y,h_2) \big]$ holds up to a negligible error.
But the opposing inequality must then also typically hold, because our dynamics is reversible. So the two weights differ negligibly---this, in outline, is how we will prove Theorem~\ref{t.stable}.

We continue with three subsections that rigorously enact the three displayed steps. A final subsection completes the derivation of Theorem~\ref{t.stable} by assembling these elements.

\subsubsection{Regular passage at the end of infancy and the start of dotage: implementing Step (1)}

The key tool for this step is 
Proposition~\ref{p.deviation}.
Recall that $(x,h_1) \in I \times ( n^{-1}\Z \times [s_1,s_1 + \tot/3] )$ and  $(y,h_2) \in I \times ( n^{-1}\Z \times [s_2 - \tot/3,s_2] )$. 
 Let $h_1^+[\eta] \in [h_1 + \eta,h_1 + 2\eta] \cap \Z \eta$ and   $h_2^-[\eta] \in [h_2 - 2\eta,h_2 - \eta] \cap \Z \eta$.  Set $\phi = \eta \tot^{-1} \in (0,1)$.  We have indicated that $\eta$ is rather smaller than $\tot$; we now quantify this smallness by setting
\begin{equation}\label{e.phi}
   \phi = \poscon D^{-18} (\ell + 1)^{-18} \taumac^{1/5} \, ,
  \end{equation}
 where $\poscon$ is a suitable positive constant whose value will be guided by smallness requirements of $\phi$ throughout the proof below. (The above choice of $\phi$ does not interact with the previous choice made in the proof of Proposition~\ref{p.stablehorizontal}.)
 
 Let $J_{x,y} \subset \R$ denote the length $2 \bigcon \big(\ell + 2 \big)^{1/3} \eta^{2/3} \big( \log \phi^{-1} \big)^{1/3}$ interval centred at $\tfrac{h_2 - h_1^+[\eta]}{h_{1,2}} x + \tfrac{h_1^+[\eta]- h_1}{h_{1,2}} y$, and let $K_{x,y} \subset \R$ be the interval of equal length centred at $\tfrac{h_2 - h_2^-[\eta]}{h_{1,2}} x + \tfrac{h_2^-[\eta]- h_1}{h_{1,2}} y$; here, $\bigcon$ is a positive constant that will be chosen to be suitably high a little later.  Define the {\em moderate fluctuation} event $\mathsf{ModFluc}^0_n \big[ I \times [s_1,s_2] ; \eta \big]$ to be the event that Step~(1) aims to prove is typical uniformly in the endpoint-pair index, namely
 \begin{eqnarray*} 
  & &
 \mathsf{ModFluc}^0_n \big[ I \times [s_1,s_2] ; \eta \big] \\
  & = & \bigcap \, \Big\{ \rho^0_n  \big[ (x,h_1) \to (y,h_2)  \big] \,\, \textrm{intersects both $J_{x,y} \times \{ h_1^+[\eta] \}$ and  $K_{x,y} \times \{ h_2^-[\eta] \}$} \Big\} \, ,
 \end{eqnarray*}
 where the intersection is taken over $(x,h_1) \in I \times ( n^{-1}\Z \times [s_1,s_1 + \tot/3] )$ and  $(y,h_2) \in I \times ( n^{-1}\Z \times [s_2 - \tot/3,s_2] )$, as well as over all concerned polymers in the case of exceptional endpoint pairs where the polymer is not unique. When the superscript~$0$ is absent in specifying this event, a static environment is the object of attention, according to the convention of Subsection~\ref{s.dynamicalnotation}.

Step (1) is enacted via 
Proposition~\ref{p.deviation}. Indeed, we will now argue that this proposition implies that 
\begin{equation}\label{e.negmodfluc}
 \PP \Big( \neg \, 
 \mathsf{ModFluc}_n \big[ I \times [s_1,s_2] ; \eta \big]  \Big) \leq 
  2^{-4/3} D^3 \bigcon^{-2}  \phi^{-10/3  + 2^{-3}d \bigcon^3  (\ell + 2)} \, .
\end{equation}
To verify this bound, note that, since  $h_1^+[\eta] - h_1 \in \eta \cdot [1,2]$, $h_{1,2} \in s_{1,2} \cdot [1/3,1]$ and $\phi = \eta \tot^{-1}$, we have that $\big( h^+_1[\eta] - h_1 \big) h_{1,2}^{-1} \in \phi \cdot [1,6]$.
Similarly, $\big( h_2  - h^-_1[\eta]  \big) h_{1,2}^{-1}$ also lies in $\phi \cdot [1,6]$. For this reason, we will apply Proposition~\ref{p.deviation} three times, taking ${\bf a} \in \phi \cdot \{ 1,2,4 \}$, because all possible values for  
$\big( h^+_1[\eta] - h_1 \big) h_{1,2}^{-1}$ and $\big( h_2  - h^-_1[\eta]  \big) h_{1,2}^{-1}$ are treated by so doing.
In order that every instance of fluctuation implicated by the non-occurrence of $\mathsf{ModFluc}_n \big[ I \times [s_1,s_2] ; \eta \big]$ be captured by at least one of the three applications, we impose that Proposition~\ref{p.deviation}'s parameter ${\bf r}$
satisfy 
$$
{\bf r} \big( {\bf a} \tot \big)^{2/3} \big( \log {\bf a}^{-1}\big)^{1/3} \geq \bigcon \big(\ell + 2 \big)^{1/3} \eta^{2/3} \big( \log \phi^{-1} \big)^{1/3}
$$  
where the right-hand side is one-half of the length of the intervals $J_{x,y}$ and $K_{x,y}$. Since the function $z \to z^{2/3} \big( \log z^{-1} \big)^{1/3}$
is increasing on $z \in (0,e^{-3/4})$, we see that by ensuring $\phi \leq 4^{-1}e^{-3/4}$ the preceding display is verified whenever ${\bf r}$ is at least 
$\bigcon \big(\ell + 2 \big)^{1/3}$ for ${\bf a}$ taking values in $\phi \cdot \{1,2,4 \}$.
Thus, we impose that, for each of the three applications,
$$
 {\bf r} = \bigcon \big( \ell + 2 \big)^{1/3} \, .
 $$ 
Proposition~\ref{p.deviation}'s parameter ${\bf K}$ must satisfy the condition that ${\bf K} \geq 2^{-1}D 2^{-2\ell/3} (\ell + 1)^{1/3} \tot^{-2/3}$, in order that $2 {\bf K} s_{1,2}^{1/3}$ be at least the length of $I$; here, the value of $D$ is provided by this proposition. Since $\tot \geq 2^{-\ell-1}$, we may choose
$$
 {\bf K} = 2^{-1/3} D (\ell + 1)^{1/3} \, ,
$$
The sum of the right-hand sides of the three applications of the proposition is at most 
$$ 
3 D \big( 2^{-1/3} D (\ell + 1)^{1/3} \big)^2 \phi^{-10/3} \big( 4\phi \big)^{d \bigcon^3 (\ell + 2)} \, . 
$$
It is now convenient to use $\phi^{2^{-1}d \bigcon^3 (\ell + 2)} \leq (\ell + 1)^{-2/3}$, which in view of $\phi \leq 1/2$ may be ensured by demanding that $\bigcon \geq 2^{2/3} 3^{-1/3} (\log 2)^{-1/3} d^{-1/3}$, since $\log (\ell +1) \leq \ell +2$. Respecifying the values of the positive quantities $D$ and $d$ yields the bound~(\ref{e.negmodfluc}). Finally, the hypotheses of Proposition~\ref{p.deviation} in all three applications are implied by the bound
that ${\bf r} \geq r_0$, which is assured by a suitably high choice of $\bigcon$, alongside the conditions
$$
 2^{-1/3}D (\ell + 1)^{1/3} \leq \big( n 2^{- \ell -1} \big)^{2/3} \, \, , \, \,
  2^{-1/3} D (\ell + 1)^{1/3} {\bf a}^{1/3} \leq H_0 \, \, \, \textrm{and} \, \, \,
  n 2^{-\ell -1} \big( {\bf a} \wedge (1-{\bf a}) \big) \geq D \, .
$$
The first of these conditions, $n \geq 2^{1/2} D^{3/2} 2^\ell (\ell + 1)^{1/2}$, is implied by $D \geq 1$, $c \leq 1$, $\taumac \leq e^{-1}$ and the lower bound on $n$ hypothesised in Theorem~\ref{t.stable}. The second amounts to the demand, satisfied in view of~(\ref{e.phi}) and $\taumac \leq 1 \leq D$, that~$\phi$ is at most a constant multiple of $(\ell + 1)^{-1}D^{-3}$. The third imposes on $\phi$ the condition that $n 2^{-\ell} \phi$ exceed the constant~$2D$.  
The lower bound on $n$ hypothesised in Theorem~\ref{t.stable}
takes the form $n 2^{-\ell} \phi \geq \poscon H$, where $H$ is the parameter in Theorem~\ref{t.stable}; 
so the condition in question holds if we take $\poscon \geq 2D H^{-1}$ in~(\ref{e.phi}). Later, we demand a constant upper bound on $\poscon$; the constant $H$ is thus forced to exceed a large constant, but we may make such a choice in setting the value of $H$ in Theorem~\ref{t.stable}.
%such a choice is permitted by Proposition~\ref{p.deviation}.

\subsubsection{Infancy and dotage little affect weight: implementing Step (2)}

Here, Proposition~\ref{p.tailweight} is the main input.
On the event  $\mathsf{ModFluc}_n \big[ I \times [s_1,s_2] ; \eta \big]$, we may select, for any given   $(x,h_1) \in I \times ( n^{-1}\Z \times [s_1,s_1 + \tot/3] )$ and  $(y,h_2) \in I \times ( n^{-1}\Z \times [s_2 - \tot/3,s_2] )$, $u \in J_{x,y}$ and $v \in K_{x,y}$ such that a polymer 
$\rho^0_n \big[ (x,h_1) \to (y,h_2) \big]$ contains $(u, h_1^+[\eta])$ and $(v, h_2^-[\eta])$. By the weight additivity noted in Section~\ref{s.split}, this implies that, on this event,
\begin{eqnarray}
 & &  \weight^0_n  \big[ (x,h_1) \to (y,h_2) \big] \label{e.tripleweight} \\
 & = &
 \weight^0_n  \big[ (x,h_1) \to (u, h_1^+[\eta]) \big] +
 \weight^0_n  \big[ (u, h_1^+[\eta]) \to (v, h_2^-[\eta]) \big] +
 \weight^0_n  \big[ (v, h_2^-[\eta]) \to (y,h_2) \big]  \, . \nonumber
\end{eqnarray}
Step (2) asserts that typically 
the infancy and dotage weights
 $\weight^s_n  \big[ (x,h_1) \to (u, h_1^+[\eta]) \big]$ and  $\weight^s_n  \big[ (v, h_2^-[\eta]) \to (y,h_2) \big]$ are of roughly order $\eta^{1/3}$ for times $s = 0$ and $s = t$.
 Indeed, we set 
 \begin{eqnarray*}
 & &  \mathsf{InfaDot}_n^{0,t}  \big( I \times [s_1,s_2] \big)  \\
 & = & \bigcap \Big\{ \textrm{these four weights are in absolute value at most $E_1 \big( \ell + 2 \big)^{2/3} \eta^{1/3} \big( \log \phi^{-1} \big)^{2/3}$} \Big\} \, ,
 \end{eqnarray*}
 where the intersection is again taken over $(x,h_1) \in I \times ( n^{-1}\Z \times [s_1,s_1 + \tot/3] )$ and  $(y,h_2) \in I \times ( n^{-1}\Z \times [s_2 - \tot/3,s_2] )$. 
  
 Several applications of Proposition~\ref{p.tailweight} will
 yield
 \begin{eqnarray}  
 & & \PP \, \Big( \, \mathsf{ModFluc}^0_n \big[ I \times [s_1,s_2] ; \eta \big] \, \cap \, \neg \,  \mathsf{InfaDot}_n^{0,t}  \big( I \times [s_1,s_2] \big)  \, \Big) \label{e.neginfadot} \\
  & \leq  & 256 \cdot \bigcon^2 H_1 (\ell +2)^{2/3} \phi^{-2 + h   3^{1/2} E_1^{3/2}  (\ell + 2)  } \nonumber 
   \, .
 \end{eqnarray}
 Indeed, by the use of the scaling principle from Section~\ref{s.scalingprinciple}, Proposition~\ref{p.tailweight} with  
  $$
 {\bf n} = 3\eta n \, \, , \, \, {\bf K} = 2^{8/3}  \bigcon \big(\ell + 2 \big)^{1/3}  \big( \log \phi^{-1} \big)^{1/3} \, \,  \textrm{and}  \, \, \, {\bf R} =  3^{1/3} E_1 \big(\ell + 2 \big)^{2/3} \big( \log \phi^{-1} \big)^{2/3}
 $$
 may be applied to rectangles that are translates of 
 $$
  \big[0,4 \bigcon \big(\ell + 2 \big)^{1/3} \eta^{2/3} \big( \log \phi^{-1} \big)^{1/3}\big] \times [0,3\eta] \, ;
 $$ 
 note that a suitably high choice of the constant $E_1$ that specifies the $\mathsf{InfaDot}$ event permits the parabolic expression in the display in Proposition~\ref{p.tailweight}
 to be absorbed by the weight upper bound in this event. 
 
In this paragraph, we confirm that the proposition's hypothesis are verified in the application just made. The hypothesis $3\eta n \geq n_0$ is in view of $\eta = \phi \tot$, $\tot \geq 2^{-\ell -1}$, $\phi = \poscon  D^{-18} (\ell + 1)^{-18} \taumac^{1/5}$ and the lower bound on $n$ hypothesised by Theorem~\ref{t.stable}
implied by  taking $\poscon = 6 n_0 H^{-1}$. The hypothesis $R \geq R_0$
is implied by choosing $\taumac > 0$ small enough. To verify the hypothesis $R \leq n^{1/30}$,
note that the lower bound on $n$ hypothesised in Theorem~\ref{t.stable} alongside $H$ and $D$ being at least one implies that $n \geq \taumac^{-1/5}$, so that our specification of $R$ implies that this quantity is at most a constant multiple of $(\ell + 2)^{2/3} \big( \log(\ell + 1) + \log n \big)^{2/3}$;  since $2^\ell \leq n$, the desired upper bound on $R$ thus arises if we assume that $n$ is high enough (via a suitable choice of $H$). 
The hypothesis $K \leq n^{1/46}$ holds when $n$ is high enough for the same reason as does the upper bound on $R$.  
 
 To bound the number of applications of  Proposition~\ref{p.tailweight}  made in deriving~(\ref{e.neginfadot}), it is convenient to specify the relationship between the parameters $D$ and $\bigcon$: we set $D = 2^{1/2} \bigcon$. The number of applications takes the form $2r_1 r_2$, where~$r_1$, the cardinality of horizontal coordinates, is at most 
\begin{equation}\label{e.horizontalbound}
 \frac{2^{-1} \bigcon^{-1}D 2^{-2\ell/3} (\ell + 1)^{1/3}}{(\ell + 2)^{1/3} \eta^{2/3} \big( \log \phi^{-1} \big)^{1/3}}  + 1 
 %\leq    \frac{2^{1/6}   \phi^{-2/3}}{\big( \log \phi^{-1} \big)^{1/3}}  + 1
 \, \leq \,   2\frac{\phi^{-2/3}}{\big( \log \phi^{-1} \big)^{1/3}}   \, ,
\end{equation}
where we used $\bigcon \geq 1$ and $\phi \leq 1/4$; 
and $r_2 = 3^{-1} \tot \eta^{-1} + 1 = 3^{-1}\phi^{-1} + 1 \leq \phi^{-1}$, due to $\phi \leq 2/3$.
By a union bound, the upper bound on the probability in~(\ref{e.neginfadot}) thus obtained is
\begin{eqnarray*}
& & 4   \cdot \frac{\phi^{-5/3}}{\big( \log \phi^{-1} \big)^{1/3}} \cdot   2^{16/3}    \bigcon^2 ( \ell + 2 )^{2/3} \big( \log \phi^{-1} \big)^{2/3} H_1  \phi^{h   3^{1/2} E_1^{3/2}   (\ell + 2)  } \\ 
&= & 
{2^{22/3}}    \bigcon^2 H_1  (\ell + 2)^{2/3}   \big( \log \phi^{-1} \big)^{1/3}   \phi^{-5/3 + h   3^{1/2} E_1^{3/2}  (\ell + 2) } \\
& \leq &  256 \cdot \bigcon^2 H_1 (\ell +2)^{2/3} \phi^{-2 + h   3^{1/2} E_1^{3/2}  (\ell + 2)  } \, .
\end{eqnarray*}
 Thus we obtain~(\ref{e.neginfadot}).

\subsubsection{Stability for the prime weight: implementing Step (3)} 
 
 This step is enacted via horizontal stability Proposition~\ref{p.stablehorizontal}. We set
 \begin{eqnarray*}
 & & \mathsf{Stable}^{0,t}_n \big( I,h_1^+[\eta] , h_2^-[\eta] \big) \\
 & = & 
 \bigcap \, \Big\{ \, \Big\vert \, \weight^t_n  \big[ (u, h_1^+[\eta]) \to (v, h_2^-[\eta]) \big] - \weight^0_n  \big[ (u, h_1^+[\eta]) \to (v, h_2^-[\eta]) \big] \Big\vert \leq (\ell+1)2^{-\ell/3} \taumac^{2/1001} \Big\} \, ,
\end{eqnarray*}
 where the intersection is taken over all $u \in J_{x,y}$ and $v \in K_{x,y}$ such that 
 $(x,h_1) \in I \times ( n^{-1}\Z \times [s_1,s_1 + \tot/3] )$ and  $(y,h_2) \in I \times ( n^{-1}\Z \times [s_2 - \tot/3,s_2] )$. 
The intersection is over a broader class of sets if instead we vary $u$ and $v$
over an interval~$I'$ that shares its midpoint with $I$ but has twice the length. This is because any interval of the form $J_{x,y}$ or $K_{x,y}$ for concerned pairs $(x,y)$ intersects $I$ and is of length at most that of~$I$---to wit, $2 \bigcon (\ell + 2)^{1/3} \eta^{2/3} \big(\log \phi^{-1}\big)^{1/3}$ is at most $\bigcon 2^{-2\ell/3} (\ell + 1)^{1/3}$, which bound is due to $\eta = \phi \tot$, $\tot \leq 2^{-\ell}$, and
$\phi^{2/3}\big(\log \phi^{-1}\big)^{1/3} (\ell + 2)^{1/3}$ being at most a small constant in light of~(\ref{e.phi}).  
 Proposition~\ref{p.stablehorizontal} with ${\bf D} = \bigcon$ and ${\bf I} = {\bf J} = I'$ implies that 
 \begin{equation}\label{e.negtau}
  \PP \Big( \neg \,  \mathsf{Stable}^{0,t}_n \big( I,h_1^+[\eta] , h_2^-[\eta] \big) \Big) \leq \phi^{-2} 
  \cdot 10^8 C \bigcon^2 \cdot 2^{-\ell/6} \taumac^{49/100} \, ,
 \end{equation}
 whose right-hand factor $\phi^{-2} = (\tot \eta^{-1})^2$ equals the number of concerned level-pairs for $\big( h_1^+[\eta],h_2^-[\eta] \big)$; 
 the other factor is provided by Proposition~\ref{p.stablehorizontal}.

  \subsubsection{Proving Theorem~\ref{t.stable} after taking these three steps}

 In view of~(\ref{e.tripleweight}),  the occurrence of 
 $$
 \mathsf{ModFluc}_n \big[ I \times [s_1,s_2] ; \eta \big] \cap  \mathsf{InfaDot}_n^{0,t}  \big( I \times [s_1,s_2] \big) \cap \mathsf{Stable}^{0,t}_n \big( I,h_1^+[\eta] , h_2^-[\eta] \big)
 $$ 
 entails that 
\begin{eqnarray*}
 & &  \weight^t_n  \big[ (x,h_1) \to (y,h_2) \big] \\
 & \geq &
 \weight^t_n  \big[ (x,h_1) \to (u, h_1^+[\eta]) \big] +
 \weight^t_n  \big[ (u, h_1^+[\eta]) \to (v, h_2^-[\eta]) \big] +
 \weight^t_n  \big[ (v, h_2^-[\eta]) \to (y,h_2) \big]  \\
 & \geq & 
 \weight^0_n  \big[ (x,h_1) \to (u, h_1^+[\eta]) \big] +
 \weight^0_n  \big[ (u, h_1^+[\eta]) \to (v, h_2^-[\eta]) \big] - (\ell+1)2^{-\ell/3} \taumac^{2/1001}  \\
 & & \qquad \qquad + \, \,
 \weight^0_n  \big[ (v, h_2^-[\eta]) \to (y,h_2) \big]  -  2 E_1 (\ell + 2)^{2/3} \eta^{1/3} \big( \log \phi^{-1} \big)^{2/3} \\
  & \geq & \weight_n^0 \big[  (x,h_1) \to (y,h_2) \big] - (\ell+1)2^{-\ell/3} \taumac^{2/1001} -   2 E_1 (\ell + 2)^{2/3} \eta^{1/3} \big( \log \phi^{-1} \big)^{2/3}  \\ 
  & \geq & \weight_n^0 \big[  (x,h_1) \to (y,h_2) \big] - 2^{-\ell/3} \Big( (\ell+1)\taumac^{2/1001} +  2 E_1 (\ell + 2)^{2/3} \phi^{1/3} \big( \log \phi^{-1} \big)^{2/3} \Big) \,  . 
\end{eqnarray*}
The final inequality invokes $\eta^{1/3} \leq 2^{-\ell/3} \phi^{1/3}$, which follows from $\phi = \eta \tot^{-1}$ and $\tot \leq 2^{-\ell}$.
Recall~(\ref{e.phi}).
 Since $\taumac \leq 1/2$, the term being subtracted in the last displayed line is less than $(\ell+1)2^{1-\ell/3} \taumac^{2/1001}$ when the positive constant $\poscon$ that specifies~$\phi$ is selected to be suitably small but without dependence on $D \geq 1$. 

Consider the probability that there exist 
 $(x,h_1) \in I \times ( n^{-1}\Z \times [s_1,s_1 + \tot/3] )$ and  $(y,h_2) \in I \times ( n^{-1}\Z \times [s_2 - \tot/3,s_2] )$ such that 
%\begin{equation}\label{e.weight.taumac}
$$
 \weight^t_n  \big[ (x,h_1) \to (y,h_2) \big] -  \weight^0_n  \big[ (x,h_1) \to (y,h_2) \big] \leq - D (\ell+1)\cdot 2^{1-\ell/3}  \taumac^{2/1001} \, .
$$
%\end{equation}
 From~(\ref{e.negmodfluc}),~(\ref{e.neginfadot}) and~(\ref{e.negtau}), we find that 
 this probability is at most 
\begin{eqnarray*}
 & & 
2^{-4/3} D^3 \bigcon^{-2}  \phi^{-10/3  + 2^{-3}d \bigcon^3  (\ell + 2)}  + 
256 \cdot \bigcon^2 H_1 (\ell +2)^{2/3} \phi^{-2 + h   3^{1/2} E_1^{3/2}  (\ell + 2)  } 
 \\
& & \qquad + \, \,  
  \phi^{-2} 
  \cdot10^8 C \bigcon^2 \cdot 2^{-\ell/6} \taumac^{49/100}  \, . 
\end{eqnarray*}
Recall that we have set $D = 2^{1/2} \bigcon$. 
Owing to this and to the choice made in \eqref{e.phi}, and with an increase, if need be, in the values of $\bigcon$ and $E_1$, the displayed expression  
 is at most a constant multiple of~$\bigcon^2 \cdot 2^{-\ell/7} \taumac^{1/12}$. 
 % and thus, for $n$ high enough, at most $2^{-3}3^{-2/3} \bigcon^2 \cdot 2^{-\ell/7} \taumac^{1/15}$. 

However, this conclusion may equally be asserted with an interchange of times zero and $t$, because the dynamics is reversible.
We obtain the variant of Theorem~\ref{t.stable} in which, in~(\ref{e.stable}), the quantity $D \taumac^{2/1001}$ is replaced by $3^{1/3} 2 D \taumac^{2/1001}$ and the right-hand term is $2\bigcon^2 \cdot 2^{-\ell/7} \taumac^{1/12}$. 
{A suitable choice of the upper bound $a \in (0,1]$ on the value of $\taumac$ permits us to omit the unwanted factor of $3^{1/3}2$.
Since $D^2 = 2 \bigcon^2$, we have completed the proof of this theorem. }\qed
 
%and choosing $\taumac > 0$ so that the hypothesis $\tau \in (0,\taumac)$ ensures the needed smallness conditions, we complete the proof of this theorem. \qed

\subsection{A weak and general form of weight stability}
A simple such result will be needed to treat the short excursions' case.
\begin{lemma}\label{l.crudebound} 
There exists $h > 0$ such that, when $n \in \N^+$ and  $\tau \le 1$, it is with probability at least $1 - e^{-hn}$ that
$$
\sup \, \big\vert \weight_n^t(\psi)-\weight_n^0(\psi)\big\vert \leq 4n^{1/2}
 $$
where the supremum is taken over all $n$-zigzags $\psi$ from $(0,0)$ to $(0,1)$.
\end{lemma}
{\bf Proof.}
The unscaled dynamical noise environment $B: \R \times \Z \times \R \to \R$
evolves in dynamic time according to Ornstein-Uhlenbeck dynamics whose invariant measure is static Brownian LPP, namely the law of $B(\cdot,\cdot,0):\R \times \Z \to \R$. 
 Let $t \geq 0$. By~(\ref{e.correlation}), an independent realization $W:\R \times \Z \to \R$ of  $B(\cdot,\cdot,0)$ may be coupled to $B(\cdot,\cdot,0)$ so that
 $$
  B \big( \cdot,\cdot,t \big)    = e^{-t}B \big( \cdot,\cdot,0 \big) + \big( 1-e^{-2t} \big)^{1/2} W \big( \cdot,\cdot \big) \, .
 $$

Let $\weight_n^*$ denote the weight determined by the noise environment $W$.  Let $\psi$ be an $n$-zigzag. Since $t \leq n^{-1/3}$, the preceding display implies that
$$
\big\vert \weight_n^t(\psi)-\weight_n^0(\psi) \big\vert \leq t \big\vert \weight_n^0(\psi) \big\vert+ 2^{1/2} t^{1/2}  \big\vert \weight_n^*(\psi) \big\vert \, .
$$
The weights $\weight_n^0(\psi)$ and $\weight^*_n(\psi)$ are equal in law. Lemma~\ref{l.onepointbounds}(1) with ${\bf x} = {\bf y} = 0$ 
implies that, for a suitably small positive $h$, it is  with probability at least $1-e^{-hn}$ that
$\sup \big\vert \weight^*_n(\psi) \big\vert$ is at most  $n^{2/3}$, where the supremum is taken over all $n$-zigzags~$\psi$ from $(0,0)$ to $(0,1)$.  
This completes the proof of Lemma~\ref{l.crudebound}, since $t\le n^{-1/3}$. \qed

\section{The construction and main properties of the proxy}\label{s.proxy}

We now specify precisely the time-zero proxy $\rho_n^{t \to 0}$ of the polymer $\rho_n^t$, the key object in the proof as described in Section~\ref{s.conceptsoverlap}. The value of $t \geq 0$ is fixed throughout this definition, with the scaled parameter $\tau \geq 0$ specified via $t = n^{-1/3} \tau$ as in Section~\ref{s.onset}. Since the proxy will be employed to demonstrate high {\em subcritical} overlap, $\tau$ may be viewed for now as a fixed positive quantity that is much smaller that one (though shortly we will demand that it tends to zero with high $n$ at a modest rate). The proxy has just been called  $\rho_n^{t \to 0}$, in accordance with the usage in the overview of Section~\ref{s.conceptsoverlap}. We now proceed to making the argument in  Subsection~\ref{s.severalexcursions} rigorous which demands the construction mimic the geometry of $\rho_n^t$ at a dyadic scale parameter $\ell \in \N$. This necessitates a scale-$\ell$ proxy, to be denoted by $\rho_{n,\ell}^{t \to 0}$ whose job will be to mimic the geometry of $\rho_n^t$ by replication of the excursion structure of $\rho_n^t$ relative to $\rho_n^0$ on vertical scales of the form~$2^{-\ell}$, or, in fact, on a slightly finer scale. 

%Recall that the geometric mimicry of $\rho_n^t$ that the proxy succeeds in performing entails the roughly accurate ,
%when the former polymer is replaced by the proxy. 
We start with some definitions that concern this excursion structure.

\begin{definition}\label{d.excursion}
Let $n \in \N$, and let $\gamma$ and $\phi$ be two $n$-zigzags. For $b,f \in n^{-1} \Z$ with $b \leq f$,
and $x,y \in \R$, suppose that $(x,b)$ and $(y,f)$ belong to $\gamma \cap \phi$. Recall zigzag subpath notation from Subsection~\ref{s.zigzagsubpaths}. The subpath union set $\gamma_{(x,b) \to (y,f)} \cup \phi_{(x,b) \to (y,f)}$
is called a {\em journey} of $\gamma$ and $\phi$, and it is denoted by $J \big( \gamma, \phi, (x,b) \to (y,f) \big)$. This journey $J$'s {\em duration} $\duration(J)$  is $f-b$; its {\em legs} are the two sets  $\gamma_{(x,b) \to (y,f)}$ and $\phi_{(x,b) \to (y,f)}$.  A journey is said to have {\em scale} $\ell \geq 0$ if $f - b \in (2^{-\ell-1},2^{-\ell}]$.

Consider three properties that a journey of $\gamma$ and $\phi$ with such endpoints may satisfy. The latter two are expressed in terms of three parameters: $\alpha > 0$, $\chi \in (0,1)$ and $\tau_0$. The latter parameter $\tau_0$ will act as an upper bound on $\tau$; we will shortly impose that it decays gently to zero with $n$. 
\begin{enumerate}
\item The legs  $\gamma_{(x,b) \to (y,f)}$ and  $\phi_{(x,b) \to (y,f)}$ are disjoint except at their common pair of endpoints.
\item For 
a proportion
of values $h \in n^{-1} \Z \cap [b,f]$ that is at least $1-\chi$, the departures of the legs from height $h$ differ horizontally by at least $(f-b)^{2/3}\tau_0^\alpha$; that is, for such $h$, $\vert \gamma(h) - \phi(h) \vert$ is at least $(f-b)^{2/3} \tau_0^{\alpha}$, where here the notation from Subsection~\ref{s.zigzagfunction} is used. See Figure~\ref{f.excursion}.

A variation of this property will be needed for technical purposes.
\item For 
a proportion 
of values $h \in n^{-1} \Z \cap [b,f]$ that is at least $1-\chi$, these departures differ horizontally by at least one-half of the preceding lower bound; that is, for such $h$, we have $\vert \gamma(h) - \phi(h) \vert \geq (f-b)^{2/3} \tau_0^{\alpha}/2$. 
\end{enumerate}
A journey that satisfies $(1)$ is called an {\em excursion}; an excursion is called {\em normal} if it satisfies $(2)$ and {\em slender}, or $(\alpha,1-\chi)$-slender, if it does not. A journey that satisfies $(3)$ is called a {\em weak excursion}. Note that, in a weak excursion, the legs 
are not constrained to be disjoint, but are merely supposed to gain a sufficient separation, often enough.
\end{definition}

\begin{figure}[t]
\centering{\epsfig{file=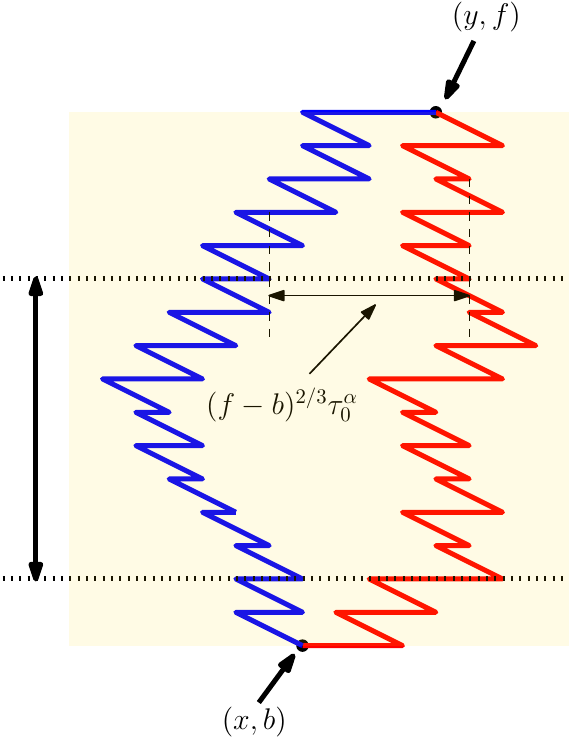, scale=0.6}}
\caption{An excursion between two zigzags $\gamma$ and $\phi$. The leg $\gamma_{(x,b) \to (y,f)}$ is blue and the leg $\phi_{(x,b) \to (y,f)}$ is red. The vertical double-arrowed line indicates the set of coordinates---an interval in this instance---at which the leg departures differ by at least the quantity appearing in the definition of a normal excursion.}\label{f.excursion}
\end{figure} 

We will employ these definitions in the case that the pair of zigzags is given by the polymers  $\rho_n^0$ and~$\rho_n^t$  at times zero and $t$. In accordance with the KPZ spatial exponent of two-thirds (whose role in geodesic energy correlation was indicated after~(\ref{e.weightfirst})), a journey between these polymers of lifetime $[b,f]$ may be expected to experience horizontal displacement between its legs of order $(f-b)^{2/3}$ at typical moments in say the middle-third of the lifetime $[b,f]$. 
The positive parameter $\alpha$ will shortly be fixed, and the subcritical time condition
$\tau \ll 1$ will also soon be expressed more precisely.  As such, most excursions between $\rho_n^0$ and $\rho_n^t$ may be expected to be normal, and only a few to be slender. Note further that, in keeping with analytic uses such as `weak solution', a weak excursion need not be an excursion.

%In our heuristical overview of proxy construction, we indicated two procedures, treating the cases of one long excursion between $\rho_n^0$ and $\rho_n^t$, and of many such excursions of given scale. 
%Our construction of the proxy will more closely follow the latter of these arguments, which was presented 
In Subsection~\ref{s.severalexcursions}, we outlined how to construct  the scale-$\ell$ proxy so that it mimics $\rho_n^t$
at scale~$\ell$---at vertical separations of order~$2^{-\ell}$: the proxy is formed by marking points along $\rho_n^t$ at consecutive vertical displacement slightly smaller than~$2^{-\ell}$ and interpolating these points via time-zero polymers. 
We now rigorously perform this construction. The order of the vertical discrepancy between marked points will take the form $2^{-m}$, where $m \in \N$ is a dyadic scale parameter that slightly exceeds $\ell$ in a manner that we now specify. Recall that scaled time $\tau$ is less than one in the relevant subcritical phase.
Let $q > 0$ be a parameter to be set later. We henceforth impose the condition that
\begin{equation}\label{e.taucondition}
 n^{1/3}t = \tau \leq \tau_0:=\big( \log n \big)^{-q} \, .
\end{equation} 
We further 
 stipulate the relation 
 \begin{equation}\label{e.mlrelation}
 2^{-m} = 2^{-\ell} \tau_0^\eta \, .  
 \end{equation}
{ Here, $\eta$ is a positive parameter that may be viewed as fixed and small, though its precise value must accommodate the need for $m$ to be an integer. This is in accordance with our practice of using $\eta$ to denote a small number as previously in \eqref{e.phi}.}

Three exponents govern our hypotheses: $\alpha$, which describes excursion geometry; $\eta$, which specifies how geometric mimicry will be realized on a vertical scale that is slightly shorter than $2^{-\ell}$; and $q$, which quantifies the smallness of subcritical time (while for some of our results, it will suffice to take $q$ as a large constant, eventually we will need to allow it to grow with $n$ in a manner made precise later in \eqref{e.qndependence}). On these exponents, we impose the condition that
\begin{equation}\label{e.exponentcondition}
 2 \eta/3 - \alpha - \tfrac{1}{3q} > 0 \, .
\end{equation}
We also stipulate a lower bound on $n$ of the form
 \begin{equation}\label{e.proxynlowerbound}
 n \, \geq \,  2^\ell H^{18}  \Theta(1) 
 \bigg( \Big(  (\ell + 1)^{18} 2^{18} +    \eta^{18} q^{18} (\log \log n)^{18} \Big) \vee \big(\log n \big)^4  \vee \big(\log n \big)^{64\big( q(2\eta/3 - \alpha) -1/3 \big)} \bigg)
   \tau_0^{-\eta} \big( \log \tau_0^{-1}\big)^9 \,.
\end{equation}  
 This hard-to-parse expression may be interpreted in practice by  pretending for now that $q$ is a constant---the right-hand side becomes, in essence, a power of $\log n$, so that, in this guise,~(\ref{e.proxynlowerbound}) is
 a concrete manifestation of the indication offered in ~(\ref{e.inverselog}). The role of the parameter $H > 0$ in~(\ref{e.proxynlowerbound}) will be addressed shortly.

The proxy  $\rho_{n,\ell}^{t \to 0}$ performs two feats of mimicry. We state these now as our principal assertion regarding the proxy.

\begin{theorem}\label{t.proxy}
Suppose that the bounds and relations~(\ref{e.taucondition}),~(\ref{e.mlrelation}),~(\ref{e.exponentcondition}) and~(\ref{e.proxynlowerbound}) are in force. The proxy  $\rho_{n,\ell}^{t \to 0}$ is an $n$-zigzag from $(0,0)$ to $(0,1)$ that typically mimics $\rho_n^t$ in two ways.
There exist $n_0 \in \N$ determined by the value  $2 \eta/3 - \alpha - \tfrac{1}{3q}$,
 and positive parameters $d_0,$ $d$, $h_0$ and $G$, such that, when   $n \geq  n_0$ and 
 $h_0 \tau_0^{- \big( 2 \eta/3 - \alpha - \tfrac{1}{3q} \big)} \geq H \geq  2 (\log 2)^{1/3} \big( \tfrac{1}{12} + \tfrac{1}{21\eta} \big)^{1/3} d^{-1/3}$, two properties hold.
\begin{enumerate}
\item {\em Weight mimicry.}  
 Except on an event of probability at most
$$
 2 G \exp \big\{ - 2^{-1} d H^3 \ell \big\} \tau_0^{2^{-1}d (\log 2)^{-1} H^3 \eta} \, + \, 15C \exp \Big\{ - d_0 H  2^{13\ell/14}    \tau_0^{1/24 - 13\eta/14}  \Big\} \, ,  
$$ 

we have the bound
$$
   \Big\vert  \weight^0 \big( \rho_{n,\ell}^{t \to 0} \big) - \weight^t \big( \rho_n^t \big) \Big\vert \leq   H^3 2^{2\ell/3} \tau_0^{1/1002 - 2\eta/3} \Psi   \, ,
$$
where the factor $\Psi$ equals 
$200 (\ell + 1)^{2/3} \big( 1 +  (\ell + 1)^{-1}(\log 2)^{-1} \eta \log \tau_0^{-1} \big)^{2/3}$.
\item {\em Excursion mimicry.} 
For $\xi \in [0,1/2)$, let~$\mc{C}$ denote   the collection of scale~$\ell$ excursions between $\rho_n^0$ and $\rho_n^t$ whose durations are contained in the interval $[\xi,1-\xi]$.
For any given such $\xi$, the construction may be performed in such a way that, for at least one-half of the elements $E$ of $\mc{C}$, both planar endpoints of $E$ lie in $\rho_{n,\ell}^{t \to 0}$.

Moreover, except on an event of probability at most $14  \exp \big\{ - 2^{-4} d H^3 \ell \big\} \tau_0^{(\log 2)^{-1} 2^{-4} d H^3 \eta}$, 
whenever $(x,b),(y,f) \in \R \times n^{-1}\Z$ are the endpoint pair of a normal excursion of scale $\ell$ between $\rho_n^0$ and $\rho_n^t$ as above which are also elements of $\rho_{n,\ell}^{t \to 0}$, 
 they are also the endpoint pair of a weak excursion between $\rho_n^0$ and $\rho_{n,\ell}^{t \to 0}$.
\end{enumerate}
\end{theorem}

If $\mc{C}$ is empty, then Theorem~\ref{t.proxy}(2) is vacuous and the proxy does not perform any geometric mimicry.  If this happens for all $\ell$, there must be a large excursion, and a separate but simpler argument will treat this case.

%Although this case is not hard to handle, it  will require a separate argument.
%Note that the above does not specify what the analogue of (2) is when $\mc{C}$ is empty, and in fact such a case is  not difficult to handle, which we will do  separately.

The proxy  $\rho_{n,\ell}^{t \to 0}$ will mimic the course of $\rho_n^t$ on the scale $2^{-m}$ and thereby substantially succeeds  in replicating the geometric structure of the scale-$\ell$ {\em normal} excursions of $\rho_n^t$. This mimicry is finer for larger values of $\ell$. It comes at the price of a coarser time-zero weight mimicry of~$\rho_n^t$ by the proxy.
Indeed, the weight mimicry upper bound has an $\ell$-dependent factor dominated by the expression $2^{2\ell/3}$, because this factor equals
  $2^{2\ell/3} \Psi$, with the latter term $\Psi$ being in a practical sense insignificant.
%  ---permissibly viewed as of unit order for the purpose of roughly interpreting Theorem~\ref{t.proxy}. 
  The customer who commissions proxy construction may choose the parameter $H$ subject to the two constraints hypothesised in Theorem~\ref{t.proxy}. As $n$ rises, higher choices of $H$ become available in accordance with~(\ref{e.taucondition}); if chosen, they lead to outcomes that are more dependable but whose feats of weight mimicry are less dazzling.
  The customer also chooses the parameter $\xi \in (0,1/2)$. A small choice ensures excursion mimicry along a lengthy bulk $[\xi,1-\xi]$ of the polymer lifetime~$[0,1]$.

%The last phrases seek to communicate the rough import of Theorem~\ref{t.proxy}. 
It is worth reiterating that excursions will be adequately mimicked only if most scale~$\ell$ excursions between $\rho_n^0$ and $\rho_n^t$ are normal. We must show, then, that slender excursions are a rarity. Proposition~\ref{p.noslender} will provide the needed information.

\subsection{Proxy construction}
We turn to the explicit construction of the proxy $\rho_{n,\ell}^{t \to 0}$. Let $J \subset n^{-1} \Z \cap [0,1]$ denote the collection of vertical levels 
$\big\{ n^{-1}\lfloor n 2^{-m}k \rfloor: k \in \llbracket 0, 2^m \rrbracket \big\}$. Note that $J$ numbers $2^m +1$; that $0,1 \in J$; and that the distance between any consecutive pair of elements of $J$
lies in $[2^{-m} - n^{-1},2^{-m} + n^{-1}]$.

A very direct rendering of the plan from Subsection~\ref{s.severalexcursions} would mark points on $\rho_n^t$ at each level in~$J$, and interpolate them via time-zero polymers in order to form $\rho_{n,\ell}^{t \to 0}$. However, this approach may fail to preserve the structure of scale-$\ell$ excursions between $\rho_n^0$ and  $\rho_n^t$ when the latter polymer is replaced by the proxy $\rho_{n,\ell}^{t \to 0}$. For example, and as the upcoming Figure~\ref{f.proxyconstruction}(left) indicates, several consecutive scale-$\ell$ excursions may merge; this effect would diminish, in principle without constraint, the number of scale-$\ell$ excursions between $\rho_n^0$ and the proxy. Since the proxy method is vitally dependent on the lower bound~(\ref{e.proxytimezero}), and this bound arises due to excursion additivity for weight~(\ref{e.excursionadditive}), we cannot tolerate the uncontrolled vanishing of excursions when the proxy enters. We are led to revise our specification of interpolating points with an altered approach that seeks to preserve at least one-half of scale-$\ell$ excursions during the exchange of $\rho_n^t$ for~$\rho_{n,\ell}^{t \to 0}$.

In order that an excursion survive this exchange, it is natural to seek to insist that its starting and ending moments be introduced to the set $J$ of interpolating levels used in proxy construction. For certain excursions, this modification would have the unattractive consequence of introducing into $J$ consecutive elements at vertical separation much less than $2^{-m}$, and we omit them for the purpose of modifying $J$. In precise terms, 
scale-$\ell$ excursions between $\rho_n^0$ and $\rho_n^t$ 
are naturally ordered by increasing vertical coordinate. Examining them in turn according to this order, each is ascribed a status of {\em discarded} or {\em retained}. 
This construction may be carried out in a manner that will serve to ensure the first property of excursion mimicry Theorem~\ref{t.proxy}(2). Namely, taking $\xi \in [0,1/2)$ given, and denoting by~$\mc{C}$ 
 the collection of scale~$\ell$ excursions between $\rho_n^0$ and $\rho_n^t$ whose durations are contained in the interval $[\xi,1-\xi]$, at least one-half
  of the elements of~$\mc{C}$
will be retained; since all retained excursion endpoints will lie in the proxy $\rho^{t \to 0}_{n,\ell}$, the first assertion of Theorem~\ref{t.proxy}(2) will thus be ensured. The rule that permits this outcome is to declare that   
any excursion $E \in \mc{C}$  is discarded if and only if the last examined excursion $E' \in \mc{C}$ was retained and the vertical discrepancy between the starting point of $E$ and the ending point of $E'$
is at most $2^{-m}$. 
All excursions not lying in $\mc{C}$ may be discarded. 
Since 
the first element of $\mc{C}$ is retained, and 
a retain decision for a new element of $\mc{C}$ always follows a discard, the number of elements of~$\mc{C}$  that are retained is indeed at least $\vert \mc{C} \vert/2$.

\begin{figure}[t]
\centering{\epsfig{file=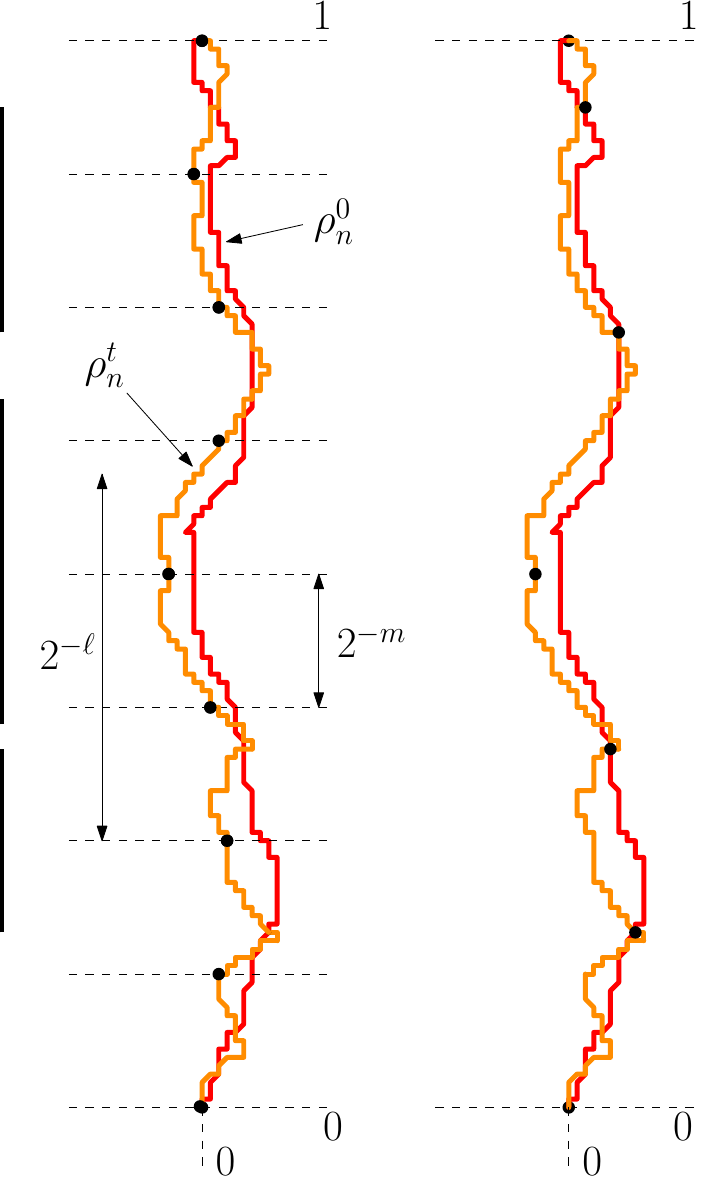, scale=0.6}}
\caption{In the left sketch, a pair $(\rho_n^0,\rho_n^t)$ with three excursions of scale~$\ell$ whose durations are indicated by vertical intervals on the left. Spots on $\rho_n^t$ are marked at vertical coordinates of consecutive separation $2^{-m}$. A naive proposal for the form of the proxy (which is not depicted) would interpolate the consecutive marks by time-zero polymers. But the result may entail as few as one excursion between the time-zero polymer and the proxy. The 
right
sketch indicates for the same example the marked points along $\rho_n^t$ that are interpolated in the actual construction of the proxy. In interpreting the choice of points, note that the middle of the three excursions is discarded, and the other two are retained. In the outcome, three suitably long excursions between the time-zero polymer and the undepicted proxy are secured.
}
\label{f.proxyconstruction}
\end{figure}

We now describe the index set $I \subset n^{-1}\Z \cap [0,1]$ that is the collection of vertical levels of interpolating points for our construction of~$\rho_{n,\ell}^{t \to 0}$. 
Each retained scale-$\ell$ excursion between $\rho_n^0$ and $\rho_n^t$ has a lifetime, $[b,f]$ say. 
Remove from $J$ the elements $\{ j,j'\}$ of any consecutive pair of members of $J$ for which $[j,j']$ contains such a point $b$ or $f$, namely the starting or ending moment of a retained scale-$\ell$ excursion between $\rho_n^0$ and $\rho_n^t$. 
Let $J' \subset n^{-1} \Z \cap [0,1]$ denote the set of elements of $J$ that are not thus removed.
Then form the subset $I$ by adding into the set $J'$
all such starting and ending moments $b$ and $f$. 

By definition this set contains all levels at which a retained scale-$\ell$ excursion between $\rho_n^0$ and $\rho_n^t$ begins and ends.
Consecutive elements of $I$ are at distance at least that between some pair of consecutive elements of $J$ but at most the distance between some pair of elements of $J$ with precisely one element of $J$ between them; thus, any such distance lies in 
$[2^{-m} - n^{-1},2^{1-m} + 2n^{-1}]$.

Let $\big\{ s_i: i \in \llbracket 0, k \rrbracket \big\}$ be an increasing list of $I$'s elements, so that $s_0 =0$ and $s_k = 1$. Note that the last condition holds since $J$ and hence $J'$ contains $0$ and $1$ (the removed elements, namely those in  $J \setminus J'$, are contained in the interval $[\xi,1-\xi]$). We define an associated sequence of locations $\big\{ u_i: i \in \llbracket 0, k \rrbracket \big\}$, with a view to $\big\{ (u_i, s_i): i \in \llbracket 0, k \rrbracket \big\}$ constituting the sequence of interpolating points in proxy construction. Certain pairs $\{ s_i,s_{i+j}\}$ (with $0 \leq i < i+j \leq k$) are the starting and ending heights of retained scale-$\ell$ excursions between  $\rho_n^0$ and $\rho_n^t$; any two such pairs are disjoint. In any such pair, when the excursion has starting point $(x,s_i)$ and ending point $(y,s_{i+j})$, we set $u_i = x$ and $u_{i+j} = y$.  In the remaining cases, namely when $s_i$ is not the ending height of any retained excursion, 
we set~$u_i$ equal to $\rho_n^t(s_i)$, this being the coordinate of departure $\sup \big\{ x \in \R: (x,s_i) \in \rho_n^t \big\}$ of $\rho_n^t$ from level~$s_i$.

We may now specify the proxy $\rho_{n,\ell}^{t \to 0}$, setting
$$
\rho_{n,\ell}^{t \to 0} \, = \, \bigcup_{i=0}^{k-1} \, \rho_n^0 \big[ (u_i,s_i) \to (u_{i+1},s_{i+1}) \big] \, ,
$$
so that indeed the proxy's set of interpolating points is  $\big\{ (u_i,s_i): i \in \llbracket 0, k \rrbracket \big\}$. Weight mimicry Theorem~\ref{t.proxy}(1) will be derived via the basic upper bound on the relevant weight difference:
\begin{equation}\label{e.righthandsum}
  \Big\vert  \weight^0 \big( \rho_{n,\ell}^{t \to 0} \big) - \weight^t \big( \rho_n^t \big) \Big\vert \, \leq \, \sum_{i=0}^{k-1}   \Big\vert  \weight^0 \big[ (u_i,s_i) \to (u_{i+1},s_{i+1}) \big]   - \weight^t \big[ (u_i,s_i) \to (u_{i+1},s_{i+1}) \big]   \Big\vert \, .
\end{equation}
Each increment $s_{i+1} - s_i$ has order $2^{-m}$ (so that $k$ has order $2^m$). Subcritical weight stability Proposition~\ref{p.onepoint}(2) with $\tot = \Theta(1) 2^{-m}$ would appear to bound above the right-hand summand, indicating that it typically has order at most $2^{-m/2} \tau^{1/2}\le 2^{-m/2} \tau_0^{1/2}$ by our assumption~(\ref{e.taucondition}) on $\tau$. The weight difference would thus typically be at most $2^{m/2} \tau_0^{1/2}$, so that a form of Theorem~\ref{t.proxy}(1) would be obtained. 

This reasoning is flawed owing to an issue mentioned briefly in the overview offered in Subsection~\ref{s.oneexcursion}. Given the value $s_i$, the location $u_i$ is not deterministic, but is selected so that $(u_i,s_i)$ lies on $\rho_n^t$. The second problem is that,  in those cases where $(u_i,s_i)$ is chosen to be a starting or ending point of an excursion, even the quantity $s_i$ is not deterministic. The argument may be corrected if we are able to replace the use of the fixed-endpoint-pair Proposition~\ref{p.onepoint}(2) with a more robust tool, in which weight stability is asserted uniformly over a class of endpoint pairs for each of whose members both coordinates are permitted to vary rather freely. Theorem~\ref{t.stable} is exactly such an assertion.

\subsection{Weight stability along much of the proxy}

We continue on the journey to the derivation of Theorem~\ref{t.proxy} by stating and proving Proposition~\ref{p.unstablerarity}. This result asserts that, when the given time polymer $\rho_n^t$ is divided into pieces of lifetime of scale~$2^{-m}$, most of the pieces verify a form of subcritical weight stability. Theorem~\ref{t.stable} will play an important role in the proof of this proposition.
 
 Let $\phi$ be an $n$-zigzag between $(0,0)$ and $(0,1)$. To formalize an already used notion, a {\em set of interpolating points} for $\phi$ is a collection $\big\{ (x_i,s_i): i \in \llbracket 0, k \rrbracket \big\}$ of elements of $\phi$
with $(x_0,s_0) = (0,0)$ and $(x_k,s_k) = (0,1)$ such that $s_{i+1} - s_i$ belongs to $[2^{-m} - n^{-1},2^{1-m} + 2n^{-1}]$
  for each $i \in \llbracket 0, k-1 \rrbracket$. 
Note that the set  $\big\{ (u_i,s_i): i \in \llbracket 0, k \rrbracket \big\}$ used in proxy construction is a set of interpolating points for~$\rho_n^t$ (and also for $\rho_{n,\ell}^{t \to 0}$).
\begin{proposition}\label{p.unstablerarity}
Suppose that $m \geq 4$; that~(\ref{e.proxynlowerbound}) holds; and that  $H \geq 2 (\log 2)^{1/3} \big( \tfrac{1}{12} + \tfrac{1}{21\eta} \big)^{1/3} d^{-1/3}$, where Proposition~\ref{p.maxfluc} furnishes the constant $d > 0$.
It is with probability at most
$$
 \exp \big\{ - d  \log 2  \cdot H^3 m \big\} + 15C \exp \Big\{ - d_0 H   2^{13m/14}    \tau_0^{1/24}  \Big\} 
$$
that there exists a set  $\big\{ (x_i,s_i): i \in \llbracket 0, k \rrbracket \big\}$ of interpolating points for $\rho_n^t \big[ (0,0) \to (0,1) \big]$ such that  
the cardinality of the set of $i \in  \llbracket 0, k-1 \rrbracket$ for which 
$$
 s_{i,i+1}^{-1/3} \Big\vert  \weight_n^t \big[ (x_i,s_i)  \to (x_{i+1},s_{i+1}) \big] -  \weight_n^0 \big[ (x_i,s_i)  \to (x_{i+1},s_{i+1}) \big] \Big\vert \geq  3H \tau_0^{1/501}
$$
is at least $180 H 2^{13m/14} \tau_0^{1/24}$.
\end{proposition}

Here are the basic steps of the proof.
\begin{enumerate}
\item Recall from \eqref{e.righthandsum} that we are seeking to bound above the right-hand side in
$$
  \Big\vert  \weight^0 \big( \rho_{n,\ell}^{t \to 0} \big) - \weight^t \big( \rho_n^t \big) \Big\vert \, \leq \, \sum_{i=0}^{k-1}   \Big\vert  \weight^0 \big[ (u_i,s_i) \to (u_{i+1},s_{i+1}) \big]   - \weight^t \big[ (u_i,s_i) \to (u_{i+1},s_{i+1}) \big]   \Big\vert \, .
$$

\item We tesselate the space into overlapping translates of a box of a certain carefully chosen size.
\item  Lemma \ref{l.notinternallystable} records a consequence of
 Theorem~\ref{t.stable}: most such boxes are stable in the sense that,  uniformly over a class of endpoint pairs that vary rather freely near the base and top of the box, the weight of the interpolating polymer changes little between times zero and~$t$.  The boxes' aspect ratio will be such that, with high probability, polymers starting and ending in the bottom middle and the top middle of a box do not exit the box.  If a box satisfies both of these typical properties, we will call it `good' in this overview; otherwise, we will call it `bad'.
\item Now note that, if the segments of $\rho_n^t$ and $\rho_n^0$ between points $(u_i,s_i)$ and $(u_{i+1},s_{i+1})$ pass through a good box, then the corresponding last-displayed summand is small.  Thus, we bound the number of bad boxes that $\rho_n^t$ passes through. Indeed, finding such an upper bound is a coarse last passage percolation problem with Bernoulli weights---each box becomes a vertex, with weight one if it is bad and weight zero if it is good.  The boxes overlap and these assignations are not independent, but this problem is minor and may be addressed by decomposing the system of boxes into suitable disjoint sets. Further, Bernoulli LPP is not integrable, but its Bernoulli variables may be dominated with a Poisson cloud of points, so that a suitably sharp tail  bound is offered by the integrable  Poissonian LPP model: see Lemma \ref{l.threewaydominate}.
\item We conclude then that no zigzag between $(0,0)$ and $(0,1)$---in particular $\rho_n^{t}$---passes through too many bad boxes. For such a bad box, we bound the corresponding summand by a crude uniform upper bound on the point-to-point weight between two elements in the box. Since the number of such summands is low, the cumulative resulting error is manageable, and we obtain the sought upper bound on $\Big\vert  \weight^0 \big( \rho_{n,\ell}^{t \to 0} \big) - \weight^t \big( \rho_n^t \big) \Big\vert$.
\end{enumerate}

We start implementing our plan with some simple definitions including the dimensions of the boxes.

For  $H > 0$ as in Proposition \ref{p.unstablerarity}, set 
\begin{equation}\label{e.ab}
a = H \cdot 2^{-2m/3}m^{1/3} \, \, \,  \textrm{and} \, \, \, b = n^{-1} \lfloor n 2^{-m} \rfloor \, .
\end{equation}
Note that, when $n \geq 2^{m+1}$, $2^{-m-1} \leq b \leq 2^{-m}$.  The dimensions of a box of width $b^{2/3}$ and height $b$ are governed by the two-thirds spatial  KPZ scaling exponent. The box $[0,a] \times [0,b]$ is a little wider than this, by a factor of $H m^{1/3} = \Theta(1) H \big( \log b^{-1} \big)^{1/3}$. This factor is chosen so that a bound may be provided on the probability that a polymer entering and leaving such a box via the middle-thirds of its horizontal sides escapes the box through its vertical sides. Indeed, the maximum fluctuation Proposition~\ref{p.maxfluc} indicates that this probability is bounded above by $\exp \big\{ - \Theta(1) H^3 \log b^{-1} \big\} = \exp \big\{ - \Theta(1) H^3 m \big\}$. 

 A {\em basic box}
is a translate of $[0,a] \times [0,b]$ by a vector in $(\Z a, \Z b)$. A rectangle of scale $m$ is a translate of $[0,5a ] \times  [ 0, 3b ]$ by a vector of 
the same form. Such a rectangle~$R$ is a union of fifteen basic boxes arranged in a five-by-three pattern; thus, it is naturally divided into a lower, central and an upper third. The {\em lower middle box} in $R$ is the third of its five lower basic boxes. 
The {\em upper middle corridor} of $R$ is the union of its second, third and fourth upper basic boxes. See Figure \ref{f.rhopercolation}.

 Suppose given a rectangle $R$ of scale $m$ 
and two elements of $R$ of the form $(x,s_1),(y,s_2) \in \R \times n^{-1}\Z$ with $s_1 < s_2$. 
The {\em internal weight} between this pair of points is the quantity, to be recorded $\weight_n \big[ (x,s_1) \to (y,s_2) ; R \big]$, 
given by the supremum of the weights of $n$-zigzags that travel between $(x,s_1)$ and $(y,s_2)$ {\em without leaving $R$}.

Recall that $t \geq 0$ denotes dynamic time. The rectangle $R$ 
is called {\em internally stable} if 
$$
 \sup \, \tot^{-1/3}  \Big\vert \weight^t_n \big[ (x,s_1) \to (y,s_2) ; R \big]  - \weight^0_n \big[ (x,s_1) \to (y,s_2) ; R \big] \Big\vert \leq  3H \tau_0^{1/501}  \, ,
$$
where the supremum is taken over $(x,s_1)$ in the lower middle box of $R$ and $(y,s_2)$ in its upper middle corridor.
We will also need to define  the notion of a {\em crossing} polymer at time~$t$ for the rectangle~$R$
which has the form $\rho^t_n \big[ (z_1,h_1) \to (z_2,h_2) \big]$, where $(z_1,h_1)$
lies in the lower middle box of $R$ and
$(z_2,h_2)$ in its upper middle corridor (as Figure~\ref{f.rhopercolation}(left) depicts).

\begin{figure}[t]
\centering{\epsfig{file=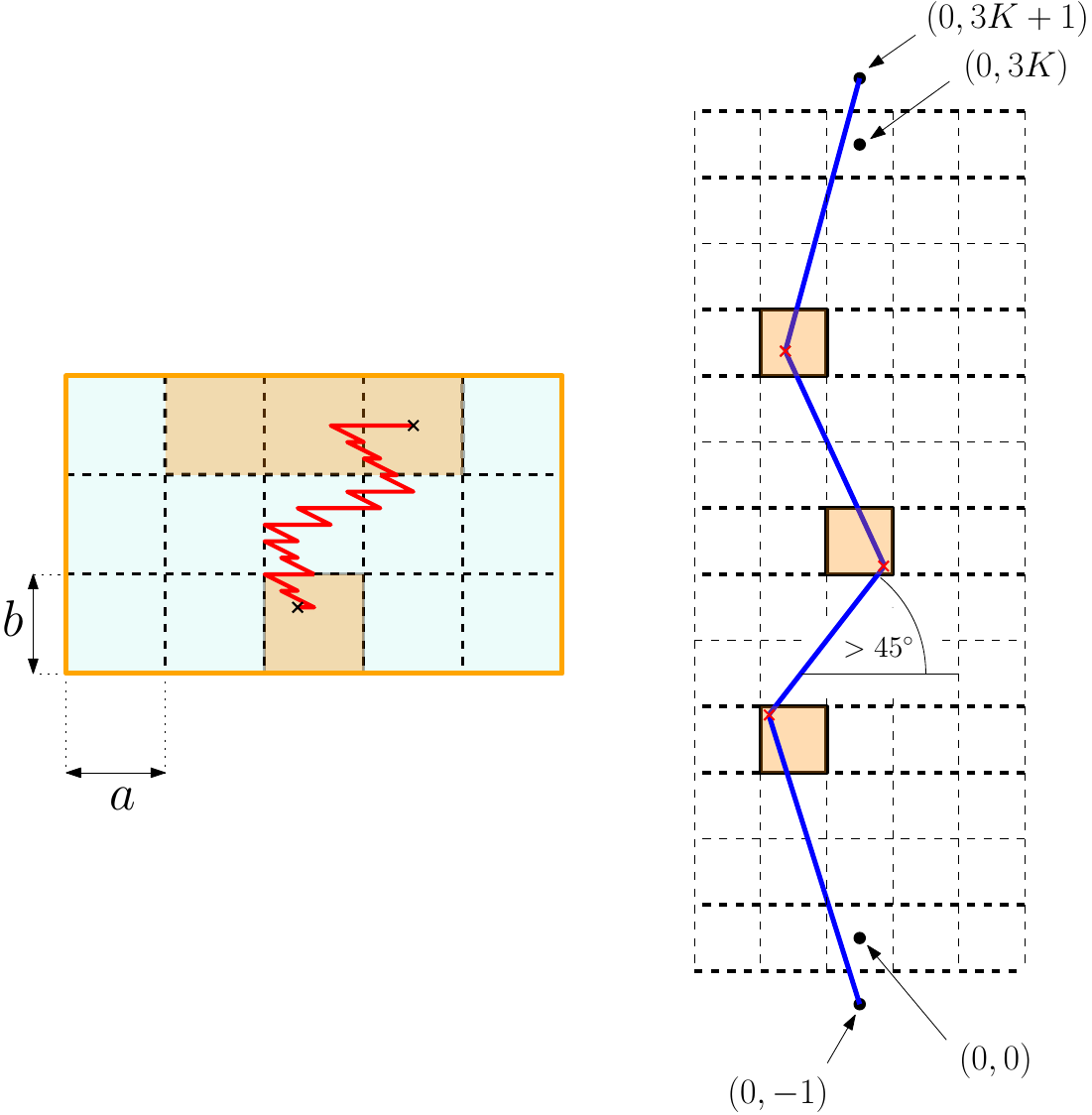, scale=0.6}}
\caption{In the left sketch, a crossing polymer for a rectangle~$R$ of scale $m$ begins in $R$'s lower middle box
and ends in its upper middle corridor.  
The right sketch illustrates the upcoming Lemma~\ref{l.threewaypoisson}. A grid is formed of unit boxes centred at elements of $\Z \times 3\Z$. Three boxes that are open in the Bernoulli-$p$ percolation environment each contain a scattered Poisson point; since the boxes lie on a three-way-up path in the $\Z \times \Z$ index space, the interpolating dotted path makes an angle of at most forty-five degress with the vertical axis.}
\label{f.rhopercolation}
\end{figure} 

\begin{lemma}\label{l.notinternallystable}
Suppose that  $H \geq 2 (\log 2)^{1/3} \big( \tfrac{1}{7} + \tfrac{1}{12\eta} \big)^{1/3} d^{-1/3}$, where $d \in (0,1]$ is specified by Proposition~\ref{p.maxfluc}.
Then the probability that a given rectangle of scale~$m$ is not internally stable is at most~$18 H^2 2^{-m/7}\tau_0^{1/12}$.
\end{lemma}
{\bf Proof.}
We will establish an upper bound on the concerned probability of the form  $9 H^2 2^{-m/7}\tau_0^{1/12} + 4 \exp \big\{ -2^{-3}d H^3 m \big\}$. This will suffice, because the condition $4 \exp \big\{ -2^{-3}d H^3 m \big\} \leq  9 H^2 2^{-m/7}\tau_0^{1/12}$, being implied by
$$
 \exp \Big\{ - \Big( 2^{-3} dH^3 - \big( \tfrac{1}{7} + \tfrac{1}{12\eta} \big) \log 2 \Big) m \Big\} \leq \tfrac{9}{4} H^2 
$$
in view of the consequence $2^m \geq \tau_0^{-\eta}$ of~(\ref{e.mlrelation}),
is also implied by the hypothesis that $H$ be at least $2 (\log 2)^{1/3} \big( \tfrac{1}{7} + \tfrac{1}{12\eta} \big)^{1/3} d^{-1/3}$ alongside the weaker bound $H^2 \geq 4/9$.

To begin demonstrating the just stated bound,
suppose harmlessly that the given rectangle has the label and form $R = [0,5a] \times [0,3b]$. {This rectangle will be  called {\em standard} }at time $t_0 \geq 0$ if  
\begin{equation}\label{e.doublerho}
\rho_n^{t_0} \big[ (a/2,0) \to (a/2,3b) \big] \subset [0,a] \times [0,3b] 
\, \, \textrm{ and } \, \, \rho_n^{t_0} \big[ (5a - a/2,0) \to (5a - a/2,3b) \big] \subset [4a,5a] \times [0,3b] \, .
\end{equation} 

  By polymer ordering Lemma~\ref{l.sandwich}, $R$ being standard at time~$t_0$ entails that
  \begin{equation}\label{e.crossing}
  \textrm{every crossing polymer for $R$ at time~$t_0$ is contained in $R$} \, .
  \end{equation}   
  
In a shorthand convention, the rectangle $R$ will be called {\em standard} if it is standard
at times zero and $t$. 
Note that, when $R$ is standard, $\weight^{t_0}_n \big[ (x,s_1) \to (y,s_2) ; R \big]$ equals $\weight^{t_0}_n \big[ (x,s_1) \to (y,s_2) \big]$ for $t_0 \in \{ 0,t \}$ and for all those pairs
$(x,s_1)$ in the lower middle box of $R$ and $(y,s_2)$ in its upper middle corridor.
Note then that
\begin{eqnarray}
 & & \PP \Big( \textrm{$R$ is not internally stable} \Big) \, \leq \, \PP \Big( \textrm{$R$ is not standard} \Big) \nonumber  \\
 &  & + \, \,  \PP \Big(  \sup \, \tot^{-1/3} \Big\vert \weight^t_n \big[ (x,s_1) \to (y,s_2)  \big]  - \weight^0_n \big[ (x,s_1) \to (y,s_2)  \big]  \Big\vert  \geq  3H \tau_0^{1/501} \Big) \, , \label{e.notinternallystable}
\end{eqnarray}
where the supremum in the latter event is taken over $(x,s_1)$ in the lower middle box of $R$ and $(y,s_2)$ over its upper middle corridor. If $R$ is not standard, then one or other of the polymers indicated in~(\ref{e.doublerho}) has maximum fluctuation greater than $a/2$ either at time $t_0 = 0$ or at time $t_0 = t$. Indeed, four applications of Proposition~\ref{p.maxfluc} with ${\bf n} = n$,
$\mathbf{s_1} = 0$, $\mathbf{s_2} =3b$, 
 and ${\bf R} = 2^{-1}H m^{1/3}$ yield  
\begin{equation}\label{e.risnotstandard}
 \PP \Big( \textrm{$R$ is not standard} \Big) \leq 4  \exp \big\{ - 2^{-3}d  H^3 m \big\}
\end{equation}
in view of $a = H \cdot 2^{-2m/3}m^{1/3}$ and $b \leq 2^{-m}$. Note that~(\ref{e.risnotstandard}) is a rigorous rendering of the discussion after~(\ref{e.ab}). 
The hypothesis of Proposition~\ref{p.maxfluc}
that ${\bf R} \leq ({{\bf n} \cdot \mathbf{s_{1,2}}})^{1/10}$ is implied by $2^{-1} H m^{1/3} \leq (3bn)^{1/10}$;
and, since $2^m \leq n$, the latter condition is implied by $2^{-m}n \geq \big( \log n \big)^4$ provided that $n$ is high enough. Note however this bound is implied by 
(\ref{e.proxynlowerbound}). 

It is uniform stability Theorem~\ref{t.stable} with parameter settings ${\bm \ell} = m-2$, ${\bf \tot} = 3b$ and ${\bf D} = 3H$ and ${\bm \taumac}=\tau_0,$ that permits us to bound above the probability in line~(\ref{e.notinternallystable}). This is because,  if we choose $q$ to be a large enough constant in \eqref{e.taucondition}, the term $(\ell+1)\tau_0^{{2}/{1001}}$ in Theorem~\ref{t.stable} is at most $\tau_0^{{1}/{501}},$ since $\ell$ is at most $\log n.$  Making this choice, 
and also in view of the just stated information about $a$ and $b$, we find that the concerned probability is at most $9 H^2 2^{-m/7}\tau_0^{1/12}$.
We must verify that Theorem~\ref{t.stable}'s hypotheses are verified in this application. The hypothesis that $3b \in [2^{1-m},2^{2-m}]$ is met because $b \geq 2/3 \, \cdot \, 2^{-m}$ is ensured by $n \geq 3 \cdot 2^m$ in view of~(\ref{e.ab}).  In order to verify the hypothesis  $n \geq   H D^{18}   (m - 1)^{18} 2^{m-2} \taumac^{-1/5}$, recall~(\ref{e.mlrelation}), whence 
$m = \ell + (\log 2)^{-1} \eta \log \tau_0^{-1}$. 
From~(\ref{e.taucondition}),
we see that 
$\log \tau_0^{-1} = q \log \log n$. Thus,  trivially, $(m-1)^{18}$ is at most
\begin{eqnarray*}
(m + 1)^{18} & = & (\ell + 1)^{18} \big( 1 + (\ell +1)^{-1} (\log 2)^{-1} \eta q \log \log n \big)^{18} \\
  & \leq  & (\ell + 1)^{18} 2^{18} \Big( 1 +  (\ell +1)^{-18} (\log 2)^{-18} \eta^{18} q^{18} (\log \log n)^{18}  \Big) \\
   & = & (\ell + 1)^{18} 2^{18} +   (\log 2)^{-18} \eta^{18} q^{18} (\log \log n)^{18}  \, .
\end{eqnarray*}
In light of this, the needed lower bound on $n$ is seen to follow from the hypothesis~(\ref{e.proxynlowerbound}) because the parameter $\taumac$ is taken to be $\tau_0$. This choice is permitted because,  if we insist that $n$ be at least a large constant, then $\tau_0$, as specified by~(\ref{e.taucondition}), will drop below the value $a$  given by Theorem~\ref{t.stable}.
Theorem~\ref{t.stable}'s hypothesis $D \geq 1$ amounts to $H \geq 1/3$. The lower bound on $H$ in Theorem~\ref{t.proxy} implies this, since $d \leq 1$.

We thus learn that  
$$
\PP \Big( \textrm{$R$ is not internally stable} \Big) \, \leq \, 9 H^2 2^{-m/7}\tau_0^{1/12} + 4 \exp \big\{ -2^{-3}d H^3 m \big\} \, .
$$
This is the bound to which we reduced the proof of Lemma~\ref{l.notinternallystable} in the proof's first paragraph. Thus is this proof complete.  \qed

Let $\phi$ denote an $n$-zigzag between $(0,0)$ and $(0,1)$; and let $\big\{ (x_i,s_i): i \in \llbracket 0, k \rrbracket \big\}$ be a set of interpolating points for $\phi$. A pair of consecutive indices $(i,i+1)$, with $i \in  \llbracket 0, k-1 \rrbracket$, is called {\em internally unsteady} if there exists a rectangle $R$ of scale~$m$ whose lower middle box and upper middle corridor respectively contain $(x_i,s_i)$ and $(x_{i+1},s_{i+1})$ and for which
$$
 \tot^{-1/3} \Big\vert \weight^t_n \big[ (x_i,s_i) \to (x_{i+1},s_{i+1}) ; R \big]  - \weight^0_n \big[ (x_i,s_i) \to (x_{i+1},s_{i+1}) ; R \big] \Big\vert \geq  3H \tau_0^{1/501} \, .
$$
The {\em unsteadiness energy}  $\mathrm{UE}(\phi)$ of $\phi$ is defined to be the maximum cardinality of the set of internally unsteady index pairs $(i,i+1)$ as the set  $\big\{ (x_i,s_i): i \in \llbracket 0, k \rrbracket \big\}$ ranges over sets of interpolating points for $\phi$. The term `energy' is reminiscent of last passage percolation; and indeed, we will bound above the unsteadiness energy by a coupling with a suitable LPP model.

Recall from Definition~\ref{d.staticpolymer} the notion that 
an $n$-zigzag between $(0,0)$ and $(0,1)$ is $(\kappa,R)$-regular for given positive parameters $\kappa$ and $R$. Here we adopt the shorthand of referring to such a zigzag as {\em regular}
when it is $(b,H)$-regular, where $b$ appears in \eqref{e.ab} 

The maximum regular unsteadiness energy $\mathrm{MaxRegUE} \big[ (0,0) \to (0,1) \big] \in \N$ is the maximum of $\mathrm{UE}(\phi)$ as $\phi$ ranges over regular zigzags between $(0,0)$ and $(0,1)$. An argument that couples to LPP will lead to the next result, which is integral to obtaining Proposition~\ref{p.unstablerarity}.

\begin{proposition}\label{p.maxunsteady}
Suppose that $m \geq 4$ and that~(\ref{e.proxynlowerbound}) holds.
There exists a positive constant  $d_0$ such that
{$$
 \PP \, \Big( \, \mathrm{MaxRegUE} \big[ (0,0) \to (0,1) \big] \geq  180  H 2^{13m/14}\tau_0^{1/24}       \, \Big)  \, \leq \, 15  \exp \big\{ -   d_0 H 2^{13m/14} \tau_0^{1/24}  \big\} \, .
$$}
\end{proposition}

Rectangles of scale $m$ may be classified into fifteen types, indexed by  $I_{5,3} : = \llbracket 0, 4 \rrbracket \times \llbracket 0, 2\rrbracket$. Indeed, any such rectangle $R$ may be translated by a vector in the lattice $(5a\Z , 3b\Z)$ so that its lower-left corner lies in   $[0,5a) \times  [ 0, 3b )$. The position of this translated corner will have the form $(k_1 a,k_2 b)$, where $(k_1,k_2) \in I_{5,3}$. The rectangle's type is $(k_1,k_2)$.

Let $(j,k)$ belong to the index set $I_{5,3}$, and  let $\phi$ be a regular zigzag from $(0,0)$ to $(0,1)$. The $(j,k)$-{\em unsteadiness energy}  $\mathrm{UE}_{j,k}(\phi)$ of $\phi$ 
is defined to be the maximum cardinality of the set of internally unsteady index pairs $(i,i+1)$ for which the associated rectangle $R_i$ has type $(j,k)$, where the maximum is taken by permitting the set  $\big\{ (x_i,s_i): i \in \llbracket 0, k \rrbracket \big\}$ to range over sets of interpolating points for $\phi$. 

Further define the $(j,k)$-maximum regular unsteadiness energy $\mathrm{MaxRegUE}_{j,k} \big[ (0,0) \to (0,1) \big] \in \N$ to be the maximum of $\mathrm{UE}_{j,k}(\phi)$ as $\phi$ ranges over the set of regular zigzags between $(0,0)$ and $(0,1)$.

Note that, for any given regular zigzag $\phi$ from $(0,0)$ to $(0,1)$,
$$
   \mathrm{UE}(\phi)  \, \leq  \,  \sum_{(j,k) \in I_{5,3}} \mathrm{UE}_{j,k}(\phi) \, .
$$
By taking suprema over such $\phi$, we find that
\begin{equation}\label{e.fifteenmax}
\mathrm{MaxRegUE} \big[ (0,0) \to (0,1) \big]
\leq 
\sum_{(j,k) \in I_{5,3}}
\mathrm{MaxRegUE}_{j,k} \big[ (0,0) \to (0,1) \big] \, .
\end{equation}

Proposition~\ref{p.maxunsteady} will be proved by harnessing~(\ref{e.fifteenmax}). Our task is to bound the upper tail of the random variable
$\mathrm{MaxRegUE} \big[ (0,0) \to (0,1) \big]$; and this task reduces to understanding this tail for 
$\mathrm{MaxRegUE}_{j,k} \big[ (0,0) \to (0,1) \big]$, where the index $(j,k) \in I_{5,3}$ is given. The latter task we now attempt.

For given $(j,k) \in I_{5,3}$,
the quantity  $\mathrm{MaxRegUE}_{j,k} \big[ (0,0) \to (0,1) \big]$ will be stochastically dominated by the point-to-point geodesic energy in a suitably specified LPP. In order to make such a comparsion, we consider a simple discrete LPP model.
\begin{definition}\label{d.bernoullithreelpp}
Attach to the lattice $\Z^2$ a directed graph structure---the three-way-up structure---under which three outgoing edges emanate from any given vertex $v$. These three edges point to $v + (-1,1)$, $v + (0,1)$ and $v + (1,1)$.  A directed path is a lattice path each of whose edges is such a directed edge. Bernoulli LPP on this lattice has a noise environment specified by a parameter $p \in (0,1)$. Each vertex independently receives a value of one or zero, with respective probabilities $p$ and $1-p$. For $K \in \N$, let $M_{K,p}^{\rm three}$ denote the maximum energy attached to any directed path between $(0,0)$ and $(0,K)$.
\end{definition}
Let $(j,k) \in I_{5,3}$ be given.  The collection of rectangles of scale $m$ and of type~$(j,k)$ is naturally indexed by $\Z^2$. For definiteness, the label $(0,0) \in \Z^2$ 
may be attached to the rectangle whose lower-left corner lies in $[0,5a) \times [0,3b)$. 
The next lemma shows that the set of rectangles a regular-$n$ zigzag intersects is reasonably smooth. 
\begin{lemma}\label{l.threewayup}
Suppose that $m \geq 1$ and $n \geq 2^{m+1}$. Let $\phi$ be a regular $n$-zigzag from $(0,0)$ to $(0,1)$. Suppose that $\phi$ intersects two rectangles $R_1$ and $R_2$ of scale~$m$ and of given type $(j,k) \in \llbracket 0,4 \rrbracket \times \llbracket 0, 2 \rrbracket$. Let $(u_1,v_1)$ and $(u_2,v_2)$ denote the values in $\Z^2$ that index $R_1$ and $R_2$. 
\begin{enumerate}
\item Suppose that $v_2 > v_1$. Then $\vert u_1 - u_2 \vert \leq \vert v_1 - v_2 \vert$.
\item Suppose that $v_1 = v_2$. Suppose further that $\phi$ intersects the lower middle box of $R_1$. Then $u_1 = u_2$; which is to say, $R_1 = R_2$.
\end{enumerate} 
\end{lemma}
{\bf Proof: (1).} If the assertion fails, then we may find indices $(u_1,v_1)$ and $(u_2,v_2)$ for rectangles visited by $\phi$ that satisfy $v_2 = v_1 + 1$ and $\vert u_1 - u_2 \vert \geq 2$. We may relabel $R_1$ and $R_2$ to be the rectangles so indexed. Let $(x,s_1) \in \phi \cap R_1$ and $(y,s_2) \in \phi \cap R_2$. Note that $\tot \in [0,6b]$, and that $|y-x| \geq 5a$. Here, however, a contradiction arises: indeed, 
since $\phi$ is $(b,H)$-regular (for the definition, see \eqref{e.regulardef}) and, as noted after~(\ref{e.ab}), $2^{-m-1} \leq b \leq 2^{-m}$, we find that 
$$
\vert y - x \vert \leq H \cdot b^{2/3} \big( \log b^{-1} \big)^{1/3} \leq H \cdot 2^{-2m/3} \big((m+1) \log 2\big)^{1/3} <  2H 2^{-2m/3}m^{1/3} = 2a \, ,
$$
where the latter inequality holds since $m \geq 1$, because this condition ensures that the quantity $(\log 2)^{1/3} \big( 1 + m^{-1} \big)^{1/3}$ is less than two.

{\bf (2).} Suppose that $v_1= v_2$. Let $(x,s_1) \in \phi$ lie in the lower middle box of $R_1$, and let $(y,s_2)$ lie in $\phi \cap R_2$. Then $\vert y - x \vert \geq 2a$, while $\vert \tot \vert \leq 3b$. However,
the preceding display shows that $\vert y - x \vert < 2a$. This contradiction establishes the second claim. \qed

We specify a Bernoulli noise environment in the $\Z^2$-index space by declaring a vertex to be open when the associated rectangle of scale $m$ and type $(j,k)$ is not internally stable. Note that the $p$-value of this noise environment is at most $18 H^2 2^{-m/7}\tau_0^{1/12}$
by Lemma~\ref{l.notinternallystable}. Note also that the noise environment is indeed Bernoulli, because the assignations made to differing indices in $\Z^2$ are independent.

\begin{lemma}\label{l.threewaydominate}
Suppose that $m \geq 4$ and $n \geq 6 \cdot 2^m$ and recall $M^{\rm three}_{K,p}$ from Definition \ref{d.bernoullithreelpp}. Let  $(j,k) \in I_{5,3}$ be given.  The random variable  $\mathrm{MaxRegUE}_{j,k} \big[ (0,0) \to (0,1) \big]$ is stochastically dominated by $M^{\rm three}_{K,p}$, where the two parameters satisfy $K \in \big[ 3^{-1}2^m , 2^{m-1} \big]$ and $p = 18 H^2 2^{-m/7}\tau_0^{1/12}$.
\end{lemma}
{\em Remark.} The hypothesis $n \geq 6 \cdot 2^m$ is implied by~(\ref{e.proxynlowerbound}) in view of the paragraph after~(\ref{e.risnotstandard}).

{\bf Proof of Lemma~\ref{l.threewaydominate}.} 
Let $\phi$ be a regular zigzag from $(0,0)$ to $(0,1)$, and let   $\big\{ (x_i,s_i): i \in \llbracket 0, k \rrbracket \big\}$ be a set of interpolating points for $\phi$. Denote by $\mc{R}$
the set of rectangles of type $(j,k)$ each of whose members contains, for some $i \in \llbracket 0, k-1 \rrbracket$, $(x_i,s_i)$ in its lower middle box and $(x_{i+1},s_{i+1})$ in its upper middle corridor. A choice of $\phi$ and its set of interpolating points may be made so that the cardinality of~$\mc{R}$ is equal to $\mathrm{MaxRegUE}_{j,k} \big[ (0,0) \to (0,1) \big]$. By Lemma~\ref{l.threewayup}, we may denote by $P$ a directed three-way-up path in $\Z^2$ that starts at the rectangle of type $(j,k)$ that contains $(0,0)$;
that ends at the rectangle of this type that contains $(0,1)$; and  that visits every element in $\mc{R}$.
Viewed as a path in $\Z^2$, $P$ travels between $(0,0)$ and $(0,K)$, where $K = \lceil (3b)^{-1} \rceil$.
Since $b = n^{-1} \lfloor 2^{-m} n \rfloor$, we have that $3^{-1} 2^m \leq (3b)^{-1} \leq 3^{-1} (2^{-m} - n^{-1})^{-1}$. Using $2^m \geq 10$ and $2^m n^{-1} \leq 6^{-1}$, we confirm that $3^{-1}2^m \leq \lceil (3b)^{-1} \rceil \leq 2^{m-1}$. 
Since every element in $\mc{R}$ is identified with a vertex in $\Z^2$ that is open in the Bernoulli noise environment, we obtain Lemma~\ref{l.threewaydominate} in light of the upper bound on the environment's $p$-value noted before this proof. \qed

Our next task is to comprehend the upper tail of the three-way-up geodesic energy $M^{\rm three}_{K,p}$ for $K \in \N$ and $p \in (0,1)$.
%; and,  in the case of principal interest, with which Proposition~\ref{p.maxunsteady}(1) is concerned,  $M^{\rm three}_{K,p} \gg 1$ is typical.  
The concerned LPP model is not integrable, and, in this principal case, our approach couples this LPP model to an integrable one, namely Poissonian LPP, 
so that the integrable energy dominates its three-way-up counterpart.

In the mentioned coupling, we will replace vertices in $\Z^2$ by unit boxes and dominate the Bernoulli variable associated to a vertex by a Poisson cloud of points in the corresponding box. In analysing Poissonian LPP, we will consider oriented paths whose angle with the vertical axis at any point is at most $\pi/4$. Since the Poisson points are arbitrarily located inside the boxes, we need to separate the boxes a little to ensure that this constraint is satisfied.

%In the mentioned coupling, we will replace vertices in $\Z^2$ by unit boxes and dominate the Bernoulli variable associated  to a vertex by a Poisson cloud of points in the corresponding box.  However, the way we will set up the Poissonian LPP, we would need to consider oriented paths whose angle with the vertical axis at any point is at most $\pi/4.$ However, since the Poisson points are arbitrarily located inside the boxes, we will, in order to ensure this constraint is satisfied, introduce some separation between the boxes.  

To implement this approach, we begin by introducing empty horizontal slices of width two between every horizontal line in $\Z^2$. Indeed, to any assignation of the Bernoulli-$p$ noise environment to $\Z^2$, we associate a counterpart noise to $\Z \times 3\Z$, where the assignation originally made to $(m_1,m_2) \in \Z^2$ is now made to $(m_1,3m_2) \in \Z \times 3\Z$. Picture elements $(m_1,3m_2) \in \Z \times 3\Z$ as the centres of unit boxes $(m_1,3m_2) + [-1/2,1/2]^2$. Note that, if an arbitrary point is selected in each such unit box, then a three-way-up directed path $P$ in $\Z^2$ will correspond to a directed path in $\R^2$, if we associate to $P$ the planar path given by line segments that interpolate the chosen points in the unit boxes corresponding to the vertices visited by $P$; here, by a directed path in $\R^2$, we mean a piecewise affine curve each of whose  line segments makes an angle with the vertical axis that is at most half a right-angle.

 We now couple the newly specified Bernoulli-$p$ noise environment on $\Z \times 3\Z$ to an environment of independent Poisson random variables $P_{m_1,3m_2}$ of mean $\lambda : = -\log(1-p)$ indexed by $(m_1,3m_2) \in \Z \times 3\Z$. The parameter~$\lambda$ has been selected so that each Poisson random variable has probability~$p$ of being positive.
 As such, we may insist that, under the coupling, each Poisson random variable is at least its Bernoulli counterpart. 
 For each $(m_1,3m_2) \in \Z \times 3\Z$, we independently scatter $P_{m_1,3m_2}$ uniform points in the unit box $(m_1,3m_2) + [-1/2,1/2]^2$. 
 The set of points so scattered is a Poisson process of intensity $\lambda$ in the planar region $R$ of points at distance at most one-half from the set $\R \times 3\Z$.
 We further independently scatter Poisson points at intensity $\lambda$ in $\R^2 \setminus R$. The collection~$\mathscr{P}$ of all scattered points is a planar Poisson process of intensity $\lambda$. 
 The associated geodesic energy between $(0,-1)$ and $(0,3K+1)$---a pair of points at Euclidean distance $3K+2$---is the maximum number of points in $\mathscr{P}$ that may be collected on a directed path in $\R^2$ between this pair of endpoints. This energy we will denote by $M_{3K+2,\lambda}^{\mathscr{P}}$.
 
 \begin{lemma}\label{l.threewaypoisson}
 Under the just specified coupling, 
  $M^{\rm three}_{K,p} \leq M_{3K+2,\lambda}^{\mathscr{P}}$.
 \end{lemma}
 \noindent{\bf Proof.} Let $P$ denote a directed three-way-up path between $(0,0)$ and $(0,K)$ that visits $M^{\rm three}_{K,p}$ vertices in $\Z^2$ that are open in the Bernoulli-$p$ environment. For each thus open vertex $(m_1,m_2)$ on~$P$, let $q$ denote a Poisson point in the box  $(m_1,3m_2) + [-1/2,1/2]^2$. 
 The piecewise affine path~$Q$ that begins at the lowest $q$-point and ends at the highest one is a directed path in $\R^2$, as we noted in the penultimate paragraph that precedes this proof---and as Figure~\ref{f.rhopercolation}(right) illustrated. We append to $Q$ the line segment between $(0,-1)$ and the starting point of $Q$; and the segment between the ending point of $Q$ and the point $(0,3K+1)$. The two added segments are readily seen to make an angle of at most forty-five degrees with the vertical axis: this inference is straightforward in view of $P$ being a directed three-way-up path between $(0,0)$ and $(0,K)$; and of $(0,-1)$ and $(0,3K+1)$ respectively lying at distance at least one-half from the unit boxes centred at $(0,0)$ and $(0,3K)$. The extended path collects 
 $M^{\rm three}_{K,p}$ Poisson points and demonstrates the bound asserted by Lemma~\ref{l.threewaypoisson}. \qed
 
The central elements have been assembled for the next proof.
 
{ \bf Proof of Proposition~\ref{p.maxunsteady}.}
Let $(j,k) \in I_{5,3}$ be given. In light of the preceding two lemmas, the key constituent task of proving an upper tail bound on $\mathrm{MaxRegUE}_{j,k} \big[ (0,0) \to (0,1) \big]$ has been reduced to finding a suitable tail bound on $M_{3K + 2,\lambda}^{\mathscr{P}}$, where recall that $\lambda = - \log(1-p)$ and where the parameters $K$ and $p$ are specified by Lemma~\ref{l.threewaydominate}.
 
The Poisson cloud $\mathscr{P}$ has intensity $\lambda$. We prefer to examine a Poisson cloud of unit intensity, and we obtain one from $\mathscr{P}$ by applying to it the contraction  $\R^2 \to \R^2: (x,y) \to \lambda^{1/2} (x,y)$. We learn by so doing that $M_{3K+2,\lambda}^{\mathscr{P}}$ is equal in law to $M_{(3K+2)\lambda^{1/2},1}^{\mathscr{P}}$. 
Thus, for any $R \geq 0$,
$$
 \PP \Big(  \mathrm{MaxRegUE}_{j,k} \big[ (0,0) \to (0,1) \big] \geq R \Big) \leq  \PP \Big(  M_{(3K+2)\lambda^{1/2},1}^{\mathscr{P}} \geq R \Big) \, . 
$$
Since $K \geq 2$ (which is due to $m \geq 3$), $\lambda \leq 2p$ (due to $p \leq 1/2$) and $K \leq 2^{m-1}$, 
we find that
\begin{equation}\label{e.maxregub}
 \PP \Big(  \mathrm{MaxRegUE}_{j,k} \big[ (0,0) \to (0,1) \big] \geq R \Big) \leq  \PP \Big(  M_{2^{3/2} 2^m p^{1/2},1}^{\mathscr{P}} \geq R \Big) \, . 
\end{equation}
Sepp\"{a}l\"{a}inen~\cite{Seppalainen98} identified the rate function $I$ for the upper tail of $M_{h,1}^{\mathscr{P}}$ in the regime of large deviations. He   proved that, for $x \geq 2$, $\lim_{h \to \infty} h^{-1} \log \PP \big( M_{h,1}^{\mathscr{P}} \geq hx \big) = - I(x)$, where $I(x)$ equals $2 x \cosh^{-1}(x/2) - 2 \big( x^2 - 4 \big)^{1/2},$ which, by standard super-multiplicative properties of such probabilities, implies that, for any finite $h$,
$\log \PP \big( M_{h,1}^{\mathscr{P}} \geq hx \big) \le - I(x)$.

 Recalling that $(j,k) \in I_{5,3}$ is given, we learn from~(\ref{e.maxregub}) and this large deviation bound with $h = 2^{3/2 + m}p^{1/2}$ and $x = 3$  that
$$ \PP \Big(  \mathrm{MaxRegUE}_{j,k} \big[ (0,0) \to (0,1) \big] \geq  3  \cdot 2^{3/2 + m}  p^{1/2}  \Big) \leq  \exp \big\{ -   2^{3/2 + m}p^{1/2} I(3) \big\}.
$$

Armed with this bound, we revisit~(\ref{e.fifteenmax}), noting that the latter implies that
$$
 \PP \Big( 
\mathrm{MaxRegUE} \big[ (0,0) \to (0,1) \big] \geq  R \Big) 
\leq 
 \sum_{(j,k) \in I_{5,3}}
 \PP \Big( \mathrm{MaxRegUE}_{j,k} \big[ (0,0) \to (0,1) \big] \geq R/15 \Big) 
$$
for any $R > 0$. Taking $R = 15  \cdot 2^{3/2 + m}  p^{1/2}$
and using   
$p = 18 H^2 2^{-m/7}\tau_0^{1/12}$, we find that
the conclusion of Proposition~\ref{p.maxunsteady} holds with $d_0$ equal to say $2^{1/2}3 I(3)$.
 \qed

{\bf Proof of Proposition~\ref{p.unstablerarity}.} Here we are supposing that the hypothesis of Proposition~\ref{p.maxunsteady} holds.
A set $\big\{ (x_i,s_i): i \in \llbracket 0, k \rrbracket \big\}$ of interpolating points for the polymer $\rho_n^t \big[ (0,0) \to (0,1) \big]$ may be said to be $M$-unstable if
 the cardinality of the set of $i \in  \llbracket 0, k-1 \rrbracket$ that satisfy 
$$
 s_{i,i+1}^{-1/3} \Big\vert  \weight_n^t \big[ (x_i,s_i)  \to (x_{i+1},s_{i+1}) \big] -  \weight_n^0 \big[ (x_i,s_i)  \to (x_{i+1},s_{i+1}) \big] \Big\vert \geq  3H \tau_0^{1/501}
$$
is at least $M \in \N$. Note that the probability that an $M$-unstable interpolating set exists is bounded above by
\begin{equation}\label{e.twotermub}
  \PP \Big( \, \rho_n \big[ (0,0) \to (0,1) \big] \, \, \textrm{is not regular} \, \Big)  
  \, + \, \PP \Big( \mathrm{MaxRegUE} \big[ (0,0) \to (0,1) \big] \geq M \, \Big) \, .
\end{equation}
Indeed, if $\rho_n^t \big[ (0,0) \to (0,1) \big]$ is regular, then  $\mathrm{MaxRegUE} \big[ (0,0) \to (0,1) \big]$ by definition offers an upper bound on the values of $M \in \N$
for which an $M$-unstable interpolating set exists; since $\rho_n^t$ has the law of the static copy $\rho_n$, the above displayed bound is obtained. 

We now select $M = 180 H 2^{13m/14} \tau_0^{1/24}$ and apply Proposition~\ref{p.maxunsteady} to bound the right-hand term in~(\ref{e.twotermub}).
We bound the left-hand term by invoking Proposition~\ref{p.regular} with ${\bm \kappa} = b = n^{-1} \lfloor n 2^{-m} \rfloor \leq 2^{-m}$ and ${\bf R} = H$. Note that~(\ref{e.mlrelation}) and~(\ref{e.proxynlowerbound}) imply that the hypothesis ${\bm \kappa} \geq K_0 n^{-1}$ is satisfied for high $n$. We thus find that the expression~(\ref{e.twotermub}) is at most
$$
 \exp \big\{ - d  \log 2  \cdot H^3 m \big\} + 15C \exp \Big\{ - d_0 H   2^{13m/14}    \tau_0^{1/24}  \Big\} 
 \, ,
$$
as we need to do to prove Proposition~\ref{p.unstablerarity}. \qed

\subsection{Transversal fluctuations of the proxy}
Here we present and prove a result that will be needed to demonstrate the excursion mimicry aspect of Theorem~\ref{t.proxy}.
\begin{definition}\label{d.maxdist}
Let $\phi$ and $\psi$ be $n$-zigzags between $(0,0)$ and $(0,1)$.
Let $\maxdist \big( \phi , \psi \big)$ denote 
the maximum distance between two points $(a,h), (b,h) \in \R \times n^{-1}\Z$ that share their height $h \in [0,1]$
and for which $(a,h) \in \phi$ and $(b,h) \in \psi$. 
\end{definition}
The next result states that $\rho_{n,\ell}^{t \to 0}$ sticks quite close to $\rho_{n}^{t}$ and thereby is indeed a good proxy for the latter.
\begin{proposition}\label{p.proxy}
Suppose that $H \geq 3  d^{-1/3}$ where $d >0$ denotes the minimum value of this constant in Propositions~\ref{p.maxfluc} and~\ref{p.regular}. Then 
$$
\PP \Big( \maxdist \big( \rho_{n,\ell}^{t \to 0} , \rho^t_n \big) > 10H 2^{-2m/3} m^{1/3} \Big) \leq 14  \exp \big\{ - 2^{-4} d  H^3 m \big\} \, .
$$
\end{proposition}

{\bf Proof.} 
 Let $I = I_{H,m}$ denote the interval $\big[ - H  m^{1/3}, H m^{1/3} \big]$.
The {\em entirely standard} event $\mathsf{EntireStand}_n$
is said to occur when every rectangle of scale~$m$ that intersects $I \times [0,1]$ is standard, in the sense of this term specified in the proof of Lemma~\ref{l.notinternallystable}---which is to say, standard at times zero and $t$ (see the paragraph following \eqref{e.crossing}).

Suppose that the event $\mathsf{G} := \big\{  \rho_n^t \, \, \textrm{is regular} \big\} \cap \mathsf{EntireStand}_n \cap \big\{ \rho_n^t \subset 2^{-1}I \times [0,1] \big\}$  occurs. Here, as in the preceding section, by `regular', we mean $(b,H)$-regular in the sense of \eqref{e.regulardef} where $b$ is defined in \eqref{e.ab}. 

Let  $\big\{ (u_i,s_i): i \in \llbracket 0, k \rrbracket \big\} \subset \rho_n^t$ 
be the set of interpolating points for the proxy  $\rho_{n,\ell}^{t \to 0}$ specified in the paragraphs that follow Theorem~\ref{t.proxy}. Recall that $(u_0,s_0) = (0,0)$; that  $(u_k,s_k) = (0,1)$; and that $s_{i+1} - s_i \in [2^{-m},2^{1-m}]$ for each $i \in \llbracket 0, k-1 \rrbracket$. 

For each $i \in \llbracket 0, k-1 \rrbracket$, let $R_i$ denote a rectangle of scale~$m$ whose lower middle box contains $(u_i,s_i)$.
Directly from the definition of $\rho_n^t$ being regular, it follows that 
 the upper middle corridor of $R_i$ contains $(u_{i+1},s_{i+1})$.  Thus,  by definition, $\rho_n^0 \big[ (u_i,s_i) \to (u_{i+1},s_{i+1}) \big]$ and $\rho_n^t \big[ (u_i,s_i) \to (u_{i+1},s_{i+1}) \big]$  are crossing polymers for $R_i$ (see Figure \ref{f.rhopercolation}(left) and the discussion preceding it) at times $0$ and $t$ respectively. Further, on the event $\mathsf{G}$, and since $\big\{ \rho_n^t \subset 2^{-1}I \times [0,1] \big\}$, the rectangle $R_i$ is standard at times $0$ and $t$.
 
Now, for a compact set $A \subset \R^2$, define ${\rm width} \, A$ to be the difference between the maximum and minimum $x$-coordinates adopted by elements of $A$.   

We claim that
\begin{equation}\label{e.widthclaim}
\max \Big\{ {\rm width} \, \rho_n^t \big[ (u_i,s_i) \to (u_{i+1},s_{i+1}) \big] , 
{\rm width} \, \rho_{n,\ell}^{t \to 0} \big[ (u_i,s_i) \to (u_{i+1},s_{i+1}) \big] \Big\} \leq 5H 2^{-2m/3} m^{1/3} \, . 
\end{equation}
Since  $\rho_n^t \big[ (u_i,s_i) \to (u_{i+1},s_{i+1}) \big]$  is a crossing polymer for $R_i$ at time $t$, and  
$R_i$ is standard at this time, this polymer is contained in $R_i$ by~(\ref{e.crossing}); and hence its width is no more than that of $R_i$, namely $5H 2^{-2m/3} m^{1/3}$. The above argument also yields the same bound for  $\rho_{n,\ell}^{t \to 0} \big[(u_i,s_i) \to (u_{i+1},s_{i+1}) \big]$ when we observe that  the latter equals  $\rho_n^0 \big[ (u_i,s_i) \to (u_{i+1},s_{i+1}) \big].$

\begin{figure}[t]
\centering{\epsfig{file=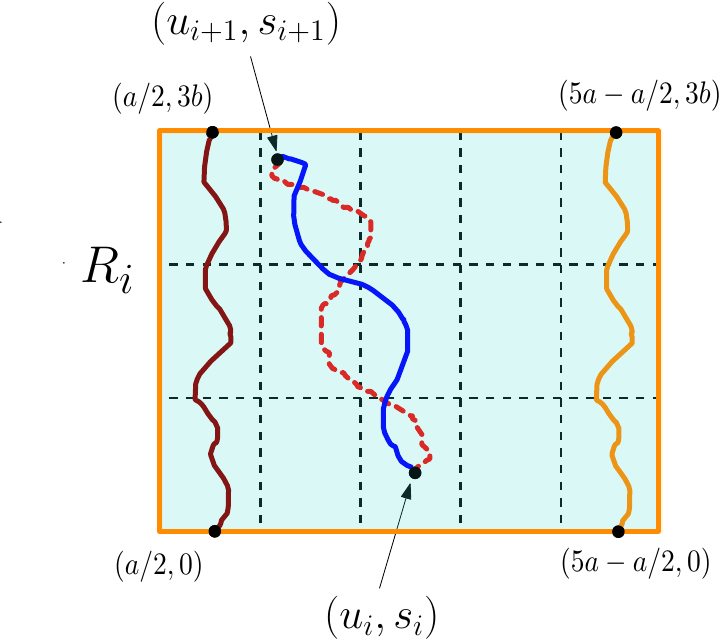, scale=0.8}}
\caption{Illustrating the argument for~(\ref{e.widthclaim}), we depict the
inside of the scale~$m$ rectangle $R_i$. The dotted polymer $\rho_n^t \big[ (u_i,s_i) \to (u_{i+1},s_{i+1}) \big]$ shares its endpoints with the solid polymer  $\rho_n^0 \big[ (u_i,s_i) \to (u_{i+1},s_{i+1}) \big]$.  Since $R_i$ is standard, the two paths are contained in this rectangle.  Indeed, the polymers with routes $(a/2,0) \to (a/2,3b)$ and $(5a-a/2,0) \to (5a-a/2,3b)$ are, at time zero---depicted---and at time $t$, contained in the respective strips $[0,a]\times[0,3b]$ and $[4a,5a]\times[0,3b]$, so that the dotted and solid paths are sandwiched into $R_i$.}
\label{f.leftswing}
\end{figure}

From~(\ref{e.widthclaim}), and the common memberships by $(u_i,s_i)$ and  $(u_{i+1},s_{i+1})$ of  $\rho_n^t \big[ (u_i,s_i) \to (u_{i+1},s_{i+1}) \big]$
and $\rho_{n,\ell}^{t \to 0} \big[ (u_i,s_i) \to (u_{i+1},s_{i+1}) \big]$, we see that
$$
{\rm width}  \Big(  \rho_n^t \big[ (u_i,s_i) \to (u_{i+1},s_{i+1}) \big] \cup \rho_{n,\ell}^{t \to 0} \big[ (u_i,s_i) \to (u_{i+1},s_{i+1}) \big] \Big) \leq 10H 2^{-2m/3} m^{1/3} \, .
$$
Thus, for any given $i \in \llbracket 0, k-1 \rrbracket$, the width of the cross-section of this union set at any given height in $n^{-1}\Z \cap [s_i,s_{i+1}]$ satisfies the same bound.
Since any height in $n^{-1}\Z \cap [0,1]$ lies in  $[s_i,s_{i+1}]$ for some such $i$, we learn that the event $\mathsf{G}$ forces the geometric assertion contained in Proposition~\ref{p.proxy}. To complete the proof of the proposition, it thus suffices to argue that $\PP(\mathsf{G}^c) \leq 14  \exp \big\{ - 2^{-4} d  H^3 m \big\}$.

To this end, note that $\PP \big(   \rho_n^t \, \, \textrm{is not regular}  \big) \leq \exp \big\{- d \log 2 \cdot m  H^3 \big\}$ by Proposition~\ref{p.regular} and $K_0 n^{-1} \leq b \leq 2^{-m}$. We have that
$$
 \PP  \big(  \neg \, \mathsf{EntireStand}_n \big) \leq \big( 2m^{1/3} 2^{2m/3} H^{-1} m^{-1/3} + 1 \big) 2^m \cdot 4 \exp \big\{ - 2^{-3} d m H^3 \big\} \leq 12 \exp \big\{ - 2^{-4} d m H^3 \big\}
$$
where the first factor in the middle expression is a bound on the number of rectangles of scale~$m$ that intersect $I \times [0,1]$, and where the second factor is the upper bound on the probability that a given such rectangle is not standard that is provided by~(\ref{e.risnotstandard}). The latter bound holds because,  in view of $H \geq 3  d^{-1/3}$, we have that
$2^{5m/3} \leq \exp \big\{ 2^{-4}H^3 dm \big\}$.

Note that $\PP \big( \rho_n^t \not\subset 2^{-1}I \times [0,1] \big) \leq \exp \big\{ -d \, 2^{-3} H^3 m \big\}$ by a similar application of Proposition~\ref{p.maxfluc} as in the derivation of~(\ref{e.risnotstandard}); see the text following~(\ref{e.risnotstandard}) for an explanation of why this result's hypotheses are satisfied. 

Thus, we find that
$$
 \PP \big( \mathsf{G}^c \big) \leq  
\exp \big\{- d \log 2 \cdot m  H^3 \big\} 
  + 
  12 \exp \big\{ - 2^{-4} d m H^3 \big\}
+  \exp \big\{ -d \, 2^{-3} H^3 m \big\} \leq 14  \exp \big\{ - 2^{-4} d  H^3 m \big\} \, .
$$
This completes the proof of Proposition~\ref{p.proxy}. \qed

\subsection{Proof of Theorem~\ref{t.proxy}}
\subsubsection{Weight mimicry}

Let $\mathsf{MostlyStable}$
denote the event that a set  $\big\{ (x_i,s_i): i \in \llbracket 0, k \rrbracket \big\}$ of interpolating points for the polymer $\rho_n^t \big[ (0,0) \to (0,1) \big]$ exists such that  
the cardinality of the set of $i \in  \llbracket 0, k-1 \rrbracket$ for which 
$$
 s_{i,i+1}^{-1/3} \Big\vert  \weight_n^t \big[ (x_i,s_i)  \to (x_{i+1},s_{i+1}) \big] -  \weight_n^0 \big[ (x_i,s_i)  \to (x_{i+1},s_{i+1}) \big] \Big\vert \geq  3H \tau_0^{1/501}
$$
is {\em at most} $180 H  2^{13m/14}      \tau_0^{1/24}$. Thus, Proposition~\ref{p.unstablerarity} asserts that
$$
 \PP \Big( \neg \, \mathsf{MostlyStable} \Big) \leq 
 \exp \big\{ - d  \log 2  \cdot H^3 m \big\} + 15C \exp \Big\{ - d_0 H   2^{13m/14}    \tau_0^{1/24}  \Big\}   \, .
$$

In deriving weight mimicry Theorem~\ref{t.proxy}(1), we seek an upper bound on the right-hand sum in~(\ref{e.righthandsum}).
The occurrence of $\mathsf{MostlyStable}$, as well as the bound $s_{i+1} - s_i \leq 2^{1-m}$, furnishes an upper bound on the summand
$$  \Big\vert \, \weight_n^t \big[  (x_i,s_i) \to  (x_{i+1},s_{i+1})  \big]  -   \weight_n^0 \big[  (x_i,s_i) \to  (x_{i+1},s_{i+1})  \big]  \, \Big\vert 
$$
for most choices of the index $i \in \llbracket 0, k-1 \rrbracket$.

We appeal to 
Proposition~\ref{p.weight} with ${\bf k} \in \{ m-1,m\}$ and ${\bf r}$ a constant multiple of $R^{1/2}$ in order to provide an accompanying bound for the remaining indices  $i \in \llbracket 0, k-1 \rrbracket$. Indeed, since $s_{i,i+1} \in [2^{-1-m},2^{1-m}]$, this result implies that, for $\Theta(1)\big(n2^{-m}\big)^{1/32} \geq R \geq R_0 > 0$, and except on an event of probability at most $G \exp \big\{ -dR^{3/2}m \big\}$, it is for all such $i$ that the above $i$-indexed summand is at most $R m^{2/3} 2^{(1-m)/3}$. 

We find then that, except on an event of probability at most   
\begin{eqnarray*}
& & 
 \exp \big\{ - d  \log 2  \cdot H^3 m \big\} + 15C \exp \Big\{ - d_0 H   2^{13m/14}    \tau_0^{1/24}  \Big\} 
 + \exp \big\{ -2^{-1}dR^{3/2}m \big\} \\
 & \leq & \exp \big\{ -  d \log 2 \cdot H^3 \ell \big\} \tau_0^{d H^3 \eta} + 15C \exp \Big\{ - d_0 H  2^{13\ell/14}    \tau_0^{1/24 - 13\eta/14}  \Big\} \\
  & & \qquad \qquad + \, G \exp \big\{ - d R^{3/2} \ell \big\} \tau_0^{d (\log 2)^{-1} R^{3/2} \eta} \, ,
\end{eqnarray*}
the right-hand sum in~(\ref{e.righthandsum}) is at most 
\begin{equation}\label{e.upperbound}
 R m^{2/3} 2^{(1-m)/3} \cdot 180 H  2^{13m/14}      \tau_0^{1/1002} \, + \,  \big( 2^m + 1 \big) \cdot  2^{(1-m)/3} 3H \tau_0^{1/501} \leq  200 RH m^{2/3} 2^{2m/3}  \tau_0^{1/1002}  \, .
\end{equation}
Recalling that  $2^m = 2^\ell \tau_0^{-\eta}$ and $G \geq 1$, and setting $R = H^2$, we obtain the conclusion of Theorem~\ref{t.proxy}(1) subject to verifying that 
$$
 200 H^3 m^{2/3} 2^{2m/3}  \tau_0^{1/1002} =  H^3 2^{2\ell/3} \tau_0^{1/1002 - 2\eta/3} 
200 (\ell + 1)^{2/3} \big( 1 +  (\ell + 1)^{-1}(\log 2)^{-1} \eta \log \tau_0^{-1} \big)^{2/3} \, .
$$  
This equality is a consequence of~(\ref{e.mlrelation}). Note finally that the earlier noted constraint that $R \leq \Theta(1)\big(n2^{-m}\big)^{1/32}$
takes the form $H \leq \Theta(1)\big(n2^{-m}\big)^{1/64}$; it is verified 
due to the form of $\tau_0$ in~(\ref{e.taucondition}) and the bound $n 2^{-m} \geq (\log n)^{64 \big( q(2\eta/3 - \alpha) - 1/3 \big)}$ which is due to~(\ref{e.proxynlowerbound}).

\subsubsection{Excursion mimicry}
Two assertions are made in Theorem~\ref{t.proxy}(2). In regard to the first, 
recall that proxy construction has been performed by consecutively examining 
the elements of the collection~$\mc{C}$ of  
scale-$\ell$ excursions between $\rho_n^0$ and $\rho_n^t$ whose durations are contained in $[\xi,1-\xi]$.  
Both endpoints of any retained excursion lie in $\rho_{n,\ell}^{t \to 0}$. 
 As was noted in specifying the proxy, the proportion of elements of~$\mc{C}$ that are retained is at least one-half.
 Thus is the first assertion demonstrated.

We turn to the second.
By Definition~\ref{d.excursion}, a normal excursion of scale~$\ell$ between $\rho_n^0$ and $\rho_n^t$ that begins at $(x,b)$ and ends at $(y,f)$ satisfies $f - b \in [2^{-\ell-1},2^{-\ell}]$. A proportion of least $1-\chi$  of heights $h \in [b,f] \cap n^{-1}\Z$ are such that the width  of $\rho_n^0 \cup \rho_n^t$ at height~$h$ is at least $(b-f)^{2/3} \tau_0^\alpha \geq 2^{-2/3} 2^{-2\ell/3} \tau_0^\alpha$.
By Proposition~\ref{p.proxy}, it is except on an event of probability at most 
\begin{equation}\label{e.doublefourteen}
14  \exp \big\{ - 2^{-4} d H^3 m \big\} = 
14  \exp \big\{ - 2^{-4} d H^3 \ell \big\} \tau_0^{(\log 2)^{-1} 2^{-4} d H^3 \eta} 
\end{equation}
 that the width of $\rho_{n,\ell}^{t \to 0} \cup \rho_n^t$ at every such height~$h$ is at least 
$$
2^{-2/3} 2^{-2\ell/3} \tau_0^\alpha - 10H 2^{-2m/3} m^{1/3} = 2^{-2/3} 2^{-2\ell/3} \tau_0^\alpha - 10H 2^{-2\ell/3} \tau_0^{2\eta/3} m^{1/3}    \geq 2^{-5/3 -2\ell/3} \tau_0^\alpha \, ,
$$ 
where the equalities in the two preceding displays depend on~(\ref{e.mlrelation}) and where the latter bound is contingent on $10 H \tau_0^{2\eta/3} m^{1/3} \leq 2^{-5/3} \tau_0^\alpha$.
Recalling that $\tau_0 = \big( \log n \big)^{-q}$ and $2^m \leq n$, the needed bound is seen to be implied by 
$H \tau_0^{2\eta/3 - \alpha - (3q)^{-1}} \leq 10^{-1} (\log 2)^{1/3} 2^{-5/3}$. Choosing the positive parameter $h_0$ in Theorem~\ref{t.proxy} to be at most $10^{-1} (\log 2)^{1/3} 2^{-5/3}$, the latter condition is verified.
 This completes the proof of the excursion mimicry aspect of Theorem~\ref{t.proxy}. \qed

\section{High subcritical overlap: proving Theorem~\ref{t.main.scaled}(1)}\label{s.highsubcriticaloverlap}

Section~\ref{s.scaledversion} offered a rough guide to the proof of Theorem~\ref{t.main.scaled} and ended by indicating four elements that would be needed to implement the proof plan. Having assembled these elements in the four preceding sections, we are ready to begin bringing them together to prove  Theorem~\ref{t.main.scaled}(1).
Indeed, we may define the event of {\em low overlap} between these two polymers: for $g > 0$,
$$
 \lowoverlap_n^{0,t}(g) = \Big\{ \, O_n \big( \rho_n^0 , \rho_n^t \big) \leq g \, \Big\} \, .
$$
Theorem~\ref{t.main.scaled}(1) asserts that, for some $g > 0$, $\lowoverlap_n^{0,t}(g)$ is unlikely in the phase $t \ll n^{-1/3}$. 

%Our proof plan was summarised in Section~\ref{s.summaryoverview}. 
%Vital was 
%the concept of an {\em excursion} between $\rho_n^0$ and $\rho_n^t$, which we specified in Definition~\ref{d.excursion}.
%A definition equivalent to the earlier one states that an excursion between two $n$-zigzags $\phi$ and $\psi$  from $(0,0)$ to $(0,1)$  is a connected component of the closure of the set $\phi \cup \psi \, \setminus \, \phi \cap \psi$.  

The argument sketched in Section~\ref{s.scaledversion} and summarised in  Section~\ref{s.summaryoverview} offers a plan of attack: show that $\lowoverlap_n^{0,t}(g)$
entails that the sum of the durations of excursions between $\rho_n^0$ and $\rho_n^t$ is at least a given positive constant. As the second comment in 
Section~\ref{s.summaryoverview} indicates, it may be that {\em short} excursions dominate. We introduce a parameter $\beta \in (0,1)$
to specify a division between {\em long} and {\em short} excursions. An excursion of duration at least $n^{\beta - 1}$ is long, and otherwise short. If long excursions dominate, employ the proxy argument outlined in Subsection~\ref{s.severalexcursions}, recalling from the first comment of Section~\ref{s.summaryoverview} the need to demonstrate that these long excursions are not {\em slender}. If short excursions dominate, a separate analysis is needed.

%In this section, we present the principal stepping stones in our rendering of this plan, and use them to close out the proof of Theorem~\ref{t.main.scaled}(1).
%This proof will harness tools treating the cases where long, or short, excursions dominate. The remaining two sections of the article then respectively furnish the needed results for long, and short, excursions.

We now begin to implement the plan. The plan first claims, ``low overlap entails a high cumulative duration for excursions''. However, though intuitive, this is not quite correct deterministically. In fact, a pair $\phi$ and $\psi$  of $n$-zigzags from $(0,0)$ and $(0,1)$ exists with zero overlap, and with cumulative excursion duration equal to the tiny quantity~$n^{-1}$. To see this consider two staircases between $(0,0)$ and $(n,n)$, the first of which visits $(n,0)$, and the second of which visits $(0,1)$ and $(n,1)$ and let $\phi$ and $\psi$ be the corresponding $n$-zigzags obtained by applying $R_n$ from Subsection~\ref{s.scalingmap}.  

However, the above scenario is rare and we will establish that ``low overlap {\em typically} entails a high cumulative duration for excursions''. In the counterexample, $\phi$ moves right as far as possible to begin, and may be viewed as one huge cliff.  We will introduce a definition of a zigzag advancing with steadiness, meaning that it has few cliffs, or, more precisely, that a positive fraction of its $n$ horizontal intervals have length at least of order $n^{-2/3}$. In Corollary~\ref{c.notallcliffs}, we will assert that the static polymer $\rho_n$ is highly likely to advance steadily. We will see in the overlap-excursion dichotomy Lemma~\ref{l.dichotomy} that in the typical case that $\rho_n^0$ advances steadily: if this zigzag has low overlap with $\rho_n^t$, then the cumulative excursion duration between these polymers is macroscopic; moreover, the two polymers enjoy {\em consistent separation}, departing from a positive fraction of heights in $[0,1] \cap n^{-1}\Z$ with a separation of order $n^{-2/3}$. 

This allows us now to pursue the analysis of the cases where long or short excursions dominate. Theorem~\ref{t.lowell} handles the long excursions' case,
relying on Proposition~\ref{p.noslender}, which shows that slender excursions are typically absent. 
Proposition~\ref{p.shortnonthinexc}
handles the case of short excursions. In light of the conclusion of the preceding paragraph, we see that, in this case, we may assume not merely that the total duration of short excursions is at least a given positive constant, but also that these excursions are not `thin', enjoying separation between their legs of order $n^{-2/3}$ at a positive fraction of heights in $[0,1] \cap n^{-1}\Z$. This stronger consistent separation assumption makes the analysis of the short excursions' case tractable, because it forms a counterpart to the assertion that long excursions are typically not slender. Proposition~\ref{p.shortnonthinexc} and the few cliffs Corollary~\ref{c.notallcliffs}
are proved in Section~\ref{s.shortexc}.%  because, as we have seen, this inference is needed to resolve the case of short excursions.

This section continues with a presentation of the overlap-excursion dichotomy including Lemma~\ref{l.dichotomy}. It finishes by giving the proof of Theorem~\ref{t.main.scaled}(1), invoking the results on the long and short excursions' cases that we have indicated.

\subsection{Overlap-excursion dichotomy}\label{s.overlapanalysis}
We start by recalling some notation: $\rho_n$ denotes the static polymer $\rho_n \big[ (0,0) \to (0,1) \big]$. 
For $i \in \llbracket 0, n \rrbracket$, $\rho_n(i/n)$ equals the supremum of the set $\big\{ x \in \R : (x,i/n) \in \rho_n \big\}$. The sequence $\big\{ \rho_n(i/n): i \in \llbracket 0, n \rrbracket \big\}$
records the horizontal coordinates of departures of the polymer $\rho_n$ from the consecutive horizontal intervals that it traverses. Indeed, the projections to~$\R$ of the horizontal intervals of $\rho_n$ take the form $[0,\rho_n(0)]$ and 
$\big[ \rho_n\big((i-1)/n\big) - 2^{-1}n^{-2/3} , \rho_n(i/n) \big]$ for $i \in \intint{n}$. Thus, by writing $\omega_0 = \rho_n(0)$ and $\omega_i = \rho_n(i/n) - \rho_n \big((i-1)/n\big) + 2^{-1}n^{-2/3}$ for $i \in \intint{n}$, the lengths of the consecutive horizontal intervals of $\rho_n$ are recorded in the sequence $\big\{ \omega_i: i \in \llbracket 0 , n \rrbracket \big\}$. 
The unscaled preimage $R_n^{-1}(\rho_n)$ of $\rho_n$ has endpoints with horizontal coordinates zero and $n$, so 
the form~(\ref{e.scalingmap}) of the scaling map $R_n:\R^2 \to \R^2$  implies that $\sum_{i=0}^n \omega_i = 2^{-1}n^{1/3}$.

Let $\beta_1 > 0$ and $\beta_2 \in (0,1)$. The polymer $\rho_n$ is said to advance horizontally with $(\beta_1,\beta_2)$-steadiness if the cardinality of the set of $i \in \llbracket 0,n \rrbracket$
for which $\omega_i \geq \beta_1 n^{-2/3}$ is at least $\beta_2 n$. That this circumstance is highly typical has been proved in~\cite{NearGroundStates}.
\begin{corollary}\cite[Proposition~$5.6$]{NearGroundStates}\label{c.notallcliffs}
There exist  $\beta_1 > 0$, $\beta_2 \in (0,1)$, $h > 0$ and $n_0 \in \N$ such that, for $n \geq n_0$, the probability that $\rho_n$ fails to  advance horizontally with $(\beta_1,\beta_2)$-steadiness is at most~$e^{-hn}$.
\end{corollary}
\begin{lemma}\label{l.dichotomy}
Let  $\beta_1 > 0$ and $\beta_2 \in (0,1)$ be furnished by Corollary~\ref{c.notallcliffs}, and suppose that $\rho_n$  advances horizontally with $(\beta_1,\beta_2)$-steadiness. 
Let $\phi$ denote an arbitrary $n$-zigzag with starting and ending points $(0,0)$ and $(0,1)$. Set $\rho = \rho_n$.
Then either
\begin{enumerate}
\item the Lebesgue measure of the overlap $\mc{O}(\rho,\phi)$ between $\rho$ and $\phi$ is at least $\frac{\beta_1 \beta_2}{4}  n^{1/3}$; or
\item  the cardinality of the set of $i \in \llbracket 0, n \rrbracket$ such that $\big\vert \rho(i/n) - \phi(i/n) \big\vert \geq \frac{\beta_1}{4} n^{-2/3}$ is at least~$\frac{\beta_2}{4}n$. 
\end{enumerate} 
\end{lemma}
{\bf Proof.} The set of  $i \in \llbracket 0,n \rrbracket$
for which $\omega_i \geq \beta_1 n^{-2/3}$ has cardinality at least $\beta_2 n$.
For such $i$, we claim that at least one of three conditions must be met. The conditions are 
\begin{enumerate}
%[label=\Alph*]
\item $\big\vert \rho\big((i-1)/n\big) - \phi\big((i-1)/n\big) \big\vert \geq \frac{\beta_1}{4} n^{-2/3}$;
\item $\big\vert \rho(i/n) - \phi(i/n) \big\vert \geq \frac{\beta_1}{4} n^{-2/3}$; and
\item the overlap of $\rho$ and $\phi$ at level $i/n$, namely the one-dimensional Lebesgue measure of the set $\rho \cap \phi \cap \big( \R \times \{ i/n\} \big)$, is at least $\frac{\beta_1}{2}n^{-2/3}$.
\end{enumerate}
Indeed, the latter overlap is at least 
\begin{eqnarray*}
 & & \min \big\{ \rho(i/n) , \phi(i/n) \big\} \, - \,  \max \big\{ \rho\big((i-1)/n\big) , \phi\big( (i-1)/n \big) \big\} \, + \, 2^{-1}n^{-2/3} \\
 & \geq & \big[\rho(i/n) -  \rho\big( (i-1)/n \big) + 2^{-1}n^{-2/3}\big] - \big\vert  \rho\big( (i-1)/n \big)  -  \phi\big( (i-1)/n \big)  \big\vert - \big\vert \rho(i/n) - \phi(i/n) \big\vert \, . 
\end{eqnarray*}
The first right-hand term equals $\omega_i$ and thus is at least $\beta_1 n^{-2/3}$. We see then that the third listed condition must occur in the case that the first two do not.
  
Thus, for a set of $i \in \llbracket 0,n \rrbracket$ of cardinality at least $\beta_2 n$, one of these three conditions obtains. If the third condition is satisfied for at least $\tfrac{\beta_2}{2} n$ such indices, then the former alternative presented in Lemma~\ref{l.dichotomy} is forced. If the first or second condition is satisfied for such a set of indices,  then the second condition is satisfied for at least $\tfrac{\beta_2}{4} n$ indices, and it is the latter alternative that occurs. Thus is the proof of Lemma~\ref{l.dichotomy} complete. \qed

\subsection{High subcritical overlap via results on short and long excursions}
The {\em consistent separation} event $\consistsep$ occurs when  the cardinality of the set of $i \in \llbracket 0, n \rrbracket$ such that the bound $\big\vert \rho_n^t(i/n)- \rho_n^0(i/n) \big\vert \geq 4^{-1}\beta_1 n^{-2/3}$ holds is at least $4^{-1}\beta_2 n$.
% This event appears in the second alternative of the dichotomy offered by Lemma~\ref{l.dichotomy} when $\rho_n$ is identified with the time-zero polymer~$\rho_n^0$ and when the choice $\phi = \rho_n^t$ is made. 
Note that, by Corollary~\ref{c.notallcliffs} and Lemma~\ref{l.dichotomy},
\begin{equation}\label{e.steadyadv}
 \PP \Big( \lowoverlap(\beta_1 \beta_2/4) \cap \neg \, \consistsep \Big) \leq e^{-hn} \, .
\end{equation}
Henceforth, it is understood that the term `excursion' when used without elaboration refers to an excursion between $\rho_n^0$ and $\rho_n^t$.
An excursion $E$ of lifetime $[b,f]$, $b,f \in n^{-1}\Z \cap [0,1]$, is called {\em thin} if $\big\vert \rho_n^t(i/n) - \rho_n^0(i/n) \big\vert < 4^{-1}\beta_1 n^{-2/3}$
for every $i/n \in n^{-1}\Z \cap [b,f)$. 
Also recall the notions of long and short excursions defined in the beginning of this section depending on a parameter $\beta.$

For $d \in (0,1)$, let $\shortnonthinexc(d)$ denote the event that the sum of the durations of short excursions that are not thin is at least $d$. Let $\longexc(d)$
denote the event that the sum of the durations of long excursions is at least $d$.
\begin{lemma}\label{l.steadyadv}
When $\consistsep$ occurs, so does $\longexc(\beta_2/8) \cup \shortnonthinexc(\beta_2/8)$.
\end{lemma}
{\bf Proof.} When $\consistsep$ occurs, at least $4^{-1}\beta_2 n$ levels $i/n \in  n^{-1}\Z \cap [0,1]$ satisfy the bound  $\big\vert \rho_n^t(i/n)- \rho_n^0(i/n) \big\vert \geq 4^{-1}\beta_1 n^{-2/3}$. The sum of the durations of excursions that are not thin is thus at least $\beta_2/4$. If the sum of the durations of long excursions is less than $\beta_2/8$,
then the sum of the durations of short excursions that are not thin is at least $\beta_2/8$. \qed

Our principal inference for the short excursions' case asserts that it is rare for short excursions that are not thin to occupy a positive fraction of heights. 
\begin{proposition}\label{p.shortnonthinexc} For $t\le n^{-1/3},$
$$
\PP \Big( \shortnonthinexc\big(\beta_2/8\big) \Big) \leq  \exp \big\{ - \dmac (\log n)^{1/68} \big\}  \, .
$$
\end{proposition}
The proposition will be proved in Section \ref{s.shortexc}.
The next result, treating the long excursions' case,
operates in the regime where $t \leq \tau_0 n^{-1/3}$ is subcritical, so that $\tau_0$ is less than one. Indeed, we have imposed in~(\ref{e.taucondition}) that $\tau_0 = (\log n)^{-q}$, with a notation that suggests that $q$ be treated as a given positive constant. This usage of $q$ has been adequate for our purpose until now, but we must now impose on $q$ an assumption that it grows gradually with $n$. Indeed, we specify 
   \begin{equation}\label{e.qndependence}
   q = q_n = h \big( \log \log n \big)^{67} \, , \, \, \,  \textrm{where $h > 0$ is a constant} \, ,
   \end{equation}
 so that $\tau_0 = \exp \big\{ - h (\log \log n)^{68} \big\}$. The value of $h$ in~(\ref{e.qndependence}) is dictated by the next result, with $g = \beta_2/4$.
\begin{theorem}\label{t.lowell}
For any $g > 0$, there exists $h > 0$ such that
$$
\PP \big( \longexc(g) \big) \leq \exp \big\{ - h \big( \log \tau_0^{-1} \big)^{1/68} \big\} \, ,
$$
where we suppose that $t = \tau n^{-1/3}$ with $\tau \in [0,\tau_0]$ and $\tau_0 = (\log n)^{-q}$ for $q = q_n(h)$ given by~(\ref{e.qndependence}). 
\end{theorem}

Before proving the above two results we show how to quickly deduce Theorem~\ref{t.main.scaled}(1) from them.

{\bf Proof of Theorem~\ref{t.main.scaled}(1).}
There exists $h > 0$ such that
\begin{eqnarray*}
\P  \Big( \lowoverlap_n^{0,t} \big( \beta_1 \beta_2/4 \big) \Big)
& \leq &  \PP \Big( \longexc(\beta_2/8)  \Big) + \PP \Big( \shortnonthinexc(\beta_2/8) \Big)  + e^{-hn} \\
& \leq &  \exp \big\{ - h ( \log \tau_0^{-1})^{1/68}\big\} +  \exp \big\{ - \dmac (\log n)^{1/68} \big\}   + e^{-hn} \\
 & \leq &  \exp \big\{ - 2^{-1}h ( \log \tau_0^{-1})^{1/68}\big\} =  \exp \big\{ - 2^{-1}h \cdot  h^{1/68} \log \log n\big\} \, .
\end{eqnarray*} 
The first bound is due to~(\ref{e.steadyadv}) and Lemma~\ref{l.steadyadv}.
The second  is due to Theorem~\ref{t.lowell} with $g = \beta_2/8$ and to Proposition~\ref{p.shortnonthinexc}.
 The final inequality is valid for high enough $n$, and is due to our setting~$\tau_0$ equal to  $\exp \big\{ - h (\log \log n)^{68} \big\}$. The final right-hand term takes the form $(\log n)^{-\chi}$ where $\chi = 2^{-1} h^{69/68}$. We obtain  Theorem~\ref{t.main.scaled}(1) with $d = \beta_1 \beta_2/4$. \qed

\section{The case of long excursions: deriving Theorem~\ref{t.lowell}}\label{s.longexc}
%Recall that an excursion between $\rho_n^0$ and $\rho_n^t$ is long when its duration is at least $n^{\beta -1}$.
%This section is devoted to the proof of long excursions Theorem~\ref{t.lowell}. 
The backbone of the proof is offered by the heuristic argument in Subsection~\ref{s.severalexcursions}. 
An important component was the excursion weight additivity formula~(\ref{e.excursionadditive}). We begin by stating and proving a rigorous rendering of this, Lemma~\ref{l.additive}. In a second subsection, we give the proof of Theorem~\ref{t.lowell}, stating along the way the crucial assertion Proposition~\ref{p.noslender} about rarity of long slender excursions which is finally proved in a third subsection.
\subsection{Weight excursion additivity}\label{s.additive124}
Recall that $\rho_n$ denotes the $n$-polymer between $(0,0)$ and $(0,1)$ and that $x \to Z_n(\cdot,a)$ as specified in Definition~\ref{d.routedweightprofile} denotes the routed weight profile parameterized by $a \in (0,1) \cap n^{-1}\Z$. 
\begin{lemma}\label{l.additive}
Let $\phi$ denote an $n$-zigzag between $(0,0)$ and $(0,1)$. Let $\big\{ E_i: i \in \intint{K} \big\}$ denote the excursions between $\rho_n$ and $\phi$, recorded in increasing order of height. For $i \in \intint{K}$, let $b_i,f_i \in n^{-1}\Z \cap [0,1]$ be the starting and ending moments of $E_i$.

Let $y_i \in n^{-1}\Z \cap [b_i,f_i)$ for $i \in \intint{K}$. Then
\begin{equation}\label{e.keydiff}
 \weight (\rho_n) - \weight(\phi) \, \geq \, \sum_{i=1}^K \Big( Z_n \big( \rho_n(y_i) , y_i \big) -   Z_n \big( \phi(y_i) , y_i \big) \Big) \,. 
\end{equation}
\end{lemma}
{\bf Proof.} Each excursion $E_i$ is comprised of two legs, one of which is a sub-zigzag of $\rho_n$ and the other of which is a sub-zigzag of $\phi$. Denote these two legs by $\rho_n^i$ and $\phi^i$. Since $\rho_n$ and $\phi$ follow a common course interspersed by $K$ instances where they diverge and then rejoin, and, on the $i\textsuperscript{th}$ of these diversions, $\rho_n$ follows the route of $\rho_n^i$, while $\phi$ follows that of $\phi^i$, we have the formula 
\begin{equation}\label{e.weightformula}
 \weight (\rho_n) - \weight(\phi)  =  \sum_{i=1}^K \Big( \weight\big(\rho_n^i\big)  -  \weight\big( \phi^i \big) \Big) \, . 
\end{equation}
We may write $\rho_n$ in the form $\rho_n^{i,-} \circ \rho_n^i \circ \rho_n^{i,+}$, in the language of Section~\ref{s.split}. Since $y_i \in n^{-1}\Z \cap [b_i,f_i)$,
\begin{equation}\label{e.doublecirc}
 \textrm{$\rho_n^{i,-} \circ \phi^i \circ \rho_n^{i,+}$ is a zigzag from $(0,0)$ to $(0,1)$ that departs from $\R \times \{ y_i\}$ at $\big(\phi(y_i),y_i\big)$} \, . 
 \end{equation}
 See Figure \ref{f.weightadditivity}.
 \begin{figure}[t]
\centering{\epsfig{file=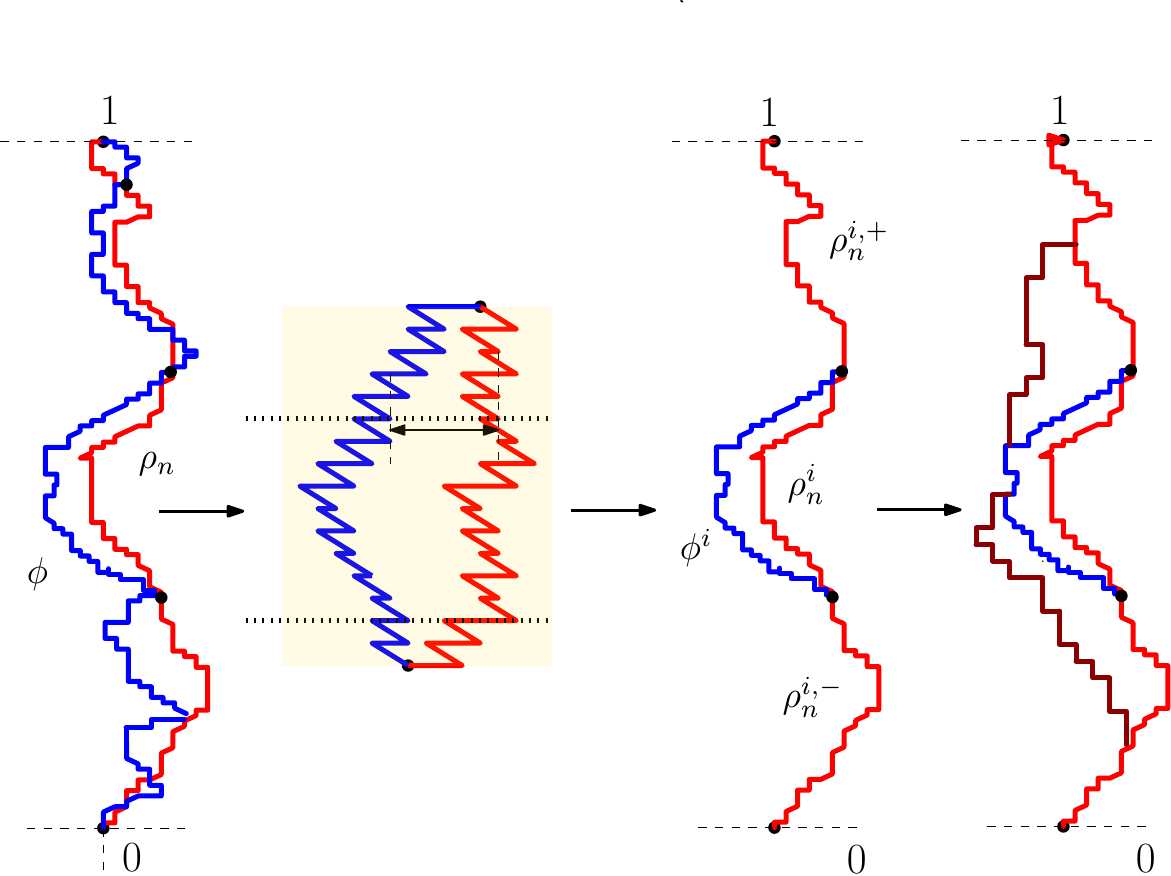, scale=0.6}}
\caption{Illustrating the proof of Lemma \ref{l.additive}. In the left sketch, the polymer $\rho_n$ and a zigzag are shown. We next zoom in on a particular excursion with two legs being $\phi^i$ and $\rho_n^i$. The third sketch shows the paths $\rho_n^{i,-} \circ \rho_n^i \circ \rho_n^{i,+}$ and $\rho_n$ which have exactly one excursion between them, passing through the points $\big( \phi(y_i) , y_i \big)$ and $\big(\rho_n(y_i) , y_i \big)$ respectively. The right sketch shows the geodesic passing through $\big(\phi(y_i) , y_i \big)$ whose weight difference with $\rho_n$ attains the value $$\Big( Z_n \big( \rho_n(y_i) , y_i \big) -   Z_n \big( \phi(y_i) , y_i \big) \Big),$$ which is at most the weight difference between $\rho_n^i$ and  $\phi^i$.
}
\label{f.weightadditivity}
\end{figure} 
 
Note that 
$ \weight \big( \rho_n^{i,-}  \big) + \weight \big( \rho_n^i  \big)  + \weight \big( \rho_n^{i,+}  \big) = \weight\big(\rho_n\big)  = Z_n \big( \rho_n(y_i) , y_i \big)$ and that  $Z_n \big( \phi(y_i) , y_i \big)  \geq    \weight\big( \rho_n^{i,-} \circ \phi^i   \circ \rho_n^{i,+} \big)  = \weight \big( \rho_n^{i,-}  \big) + \weight \big( \phi^i  \big)  + \weight \big( \rho_n^{i,+}  \big)$; the inequality is due to Definition~\ref{d.routedweightprofile} and~(\ref{e.doublecirc}).
 Thus, we find that
\begin{equation}\label{e.crucialargument}
 \weight\big(\rho_n^i\big)  -  \weight\big( \phi^i \big)  =  \weight\big(\rho_n\big)  -  \weight\big( \rho_n^{i,-} \circ \phi^i   \circ \rho_n^{i,+} \big) \geq  Z_n \big( \rho_n(y_i) , y_i \big) -   Z_n \big( \phi(y_i) , y_i \big) \, . 
\end{equation}
Applying this bound to~(\ref{e.weightformula}) proves Lemma~\ref{l.additive}. \qed  

\subsection{Proof of Theorem~\ref{t.lowell}.}
We start with a roadmap.
\begin{enumerate}
\item We divide the analysis of the case of long excursions into two sub-cases, bulk and edge.  The edge case  $\edgelongexc(\cdot)$ occurs when a dominant contribution comes from a single excursion with one endpoint close to $(0,0)$ or $(0,1)$; otherwise, it is the bulk case  $\bulklongexc(\cdot)$ that occurs. How long excursions force one or other case  will be indicated precisely
 in \eqref{e.bulkedgeinclusion}. 
\item The analysis of the edge case is short and follows from the bulk case by changing certain parameters, so we focus on the latter. For this, we identify a scale of excursions with significant cumulative duration (at least of order $(\log n)^{-1},$ since there are at most $\log n$ many scales). To bound below the right-hand side in Lemma \ref{l.additive}, we will consider merely excursions of this scale.
\item Using an input from \cite{NearGroundStates}, we will prove in Proposition \ref{p.noslender} that these excursions are typically not slender  in the sense of  Definition~\ref{d.excursion}.
\item In Lemma \ref{l.evente}, we will further prove that they are not too wide.
\item Excursions thus typically enjoy favourable geometric features.  The key Lemma \ref{l.lowtotdef} allied with this typicality proves a lower bound on the summands in \eqref{e.keydiff}. The lemma is a consequence of the twin peaks' rarity result Theorem~\ref{t.nearmax}.
\end{enumerate}

To begin the thus indicated proof, recalling Definition~\ref{d.excursion},
for $\ell \in \N$, consider the excursions $E$ between $\rho_n^0$ and $\rho_n^t$ of scale~$\ell$ for which $E$ is contained in the strip $\R \times [g/4,1-g/4]$; such $E$ neither begin before moment $g/4$ nor end after moment $1 - g/4$. Let $\scriptdur(\ell) = \sum_E \duration(E)$, where the sum is taken over such excursions $E$.

In a locally employed notational device, we will write $\sum^*_\ell \scriptdur(\ell)$, where $\sum^*_\ell$ indicates that the sum is taken over scales $\ell \in \llbracket 0, \lceil \log_2 n \rceil \rrbracket$ that satisfy $2^\ell \leq n^{1- \beta}$. 
For $s \in n^{-1}\Z \cap [0,1]$, let ${\rm dura}(s)$ denote the duration of the excursion between  $\rho_n^0$ and $\rho_n^t$ whose lifetime contains $s$; if no such excursion exists, set ${\rm dura}(s) = 0$. 
The expression ${\rm dura}(g/4) + {\rm dura}(1 - g/4) + \sum^*_\ell \scriptdur(\ell)$ is at least the cumulative duration of long excursions that intersect the strip  $\R \times [g/4,1-g/4]$. As such, the expression is at least $g/2$ when $\longexc(g)$ occurs. 
%This inequality permits an analysis of  the event $\longexc(g)$ in two cases that we label {\em bulk} and {\em edge}. 

For $\kappa > 0$, let $\bulklongexc(\kappa)$
denote the event that
$\sum^*_\ell \scriptdur(\ell) \geq \kappa$. Let $\edgelongexc(\kappa)$ denote the event that ${\rm dura}(g/4) + {\rm dura}(1 - g/4) \geq \kappa$. We see then that 
\begin{equation}\label{e.bulkedgeinclusion}
 \longexc(g) \subseteq \bulklongexc(g/4) \cup \edgelongexc(g/4) \, .
\end{equation}
Thus in order to prove Theorem~\ref{t.lowell}, it suffices to derive its conclusion when the event $\longexc(g)$ is replaced by  $\bulklongexc(g/4)$ or $\edgelongexc(g/4)$.  We refer to these two cases by the names {\em bulk} and {\em edge}. We first prove Theorem~\ref{t.lowell} in the bulk case. Its derivation in the edge case is merely a slight perturbation of the bulk case's, on which we comment after the latter case is treated.

So for the moment we assume that the $\bulklongexc(g/4)$ occurs.
There may be several indices $\ell$ contributing to the sum $\sum^*_\ell$ for which $\scriptdur(\ell)$ attains its maximum value; let $L \in \llbracket 0, \lceil \log_2n \rceil \rrbracket$ denote the smallest of these values~$\ell$.  If also $n \geq 4$, then
 \begin{equation}\label{e.longexcursionpigeonhole}
\scriptdur(L) \, = \, \sum_{i=1}^N \duration(E_i) \mathbf{1}_{ 2^{-1-L} < \duration(E_i) \leq 2^{-L} , E_i \subset \R \times  [g/4, 1-g/4]}  \, \geq \,   8^{-1}\log 2 \, \frac{g}{\log n}    \, ,
 \end{equation}
 where $\big\{ E_i: i \in \intint{N} \big\}$ is the list of excursions.
 Indeed, any excursion $E$ satisfies $\duration(E) \in [n^{-1},1]$, so there are $1 + \lceil \log_2n \rceil \leq 2 \big( \log 2 \big)^{-1} \log n$ admissible scales for excursions. (The displayed bound is valid due to $n \geq 4$.)

The value $L$, which we call the {\em dominant scale}, satisfies $2^L \leq n^{1 - \beta}$. 
Definition~\ref{d.excursion} specified the notion of an  $(\alpha,1-\chi)$-slender excursion. Recall that we have indicated that ruling out slender excursions in an important part of our argument. The following result, whose proof we provide in the next subsection via an input from~\cite{NearGroundStates}, 
provides the needed input in this regard.  It sets the value of the exponent $\beta>0$ that specifies the division between long and short excursions. 

\begin{proposition}\label{p.noslender}
Let $\mathsf{Slender}(\alpha,\beta,1-\chi)$ denote the event that there exists an $(\alpha,1-\chi)$-slender excursion between $\rho_n^0$ and $\rho_n^t$
whose scale $\ell$ satisfies $2^\ell \leq n^{1-\beta}$. Recall from \eqref{e.taucondition} that $\tau_0=(\log n)^{-q}.$
There exist positive $\alpha_0$, $\chi_0$ and~$h$ such that, for $\alpha \geq \alpha_0$, and any $0< \beta<1$, $\chi \in (0,\chi_0)$ and $n \geq n_0$, $$
  \PP \Big( \mathsf{Slender}(\alpha,\beta,1-\chi) \Big) \leq   \exp \big\{ - h \big( \log n \big)^{q\alpha/2 \, - \, 1} \big\} \, ,
$$  
where $n_0$ is determined by $\beta$.
 \end{proposition}
 
Set  $\mathsf{NoSlender}(\alpha,\beta,1-\chi)  = \neg \, \mathsf{Slender}(\alpha,\beta,1-\chi)$.

Let $\mathsf{ProxySuccess}_n(\ell)$ denote the event that the scale~$\ell$ proxy $\rho_{n,\ell}^{t \to 0}$ validates the two mimicry conditions in proxy construction Theorem~\ref{t.proxy}, where, working as we do in the bulk case, we set the theorem's parameter $\xi$ equal to $g/4$. (The value of its parameter $H$ will be set shortly.)
Set $\mathsf{ProxySuccess}_n$ equal to the intersection of these events over scales~$\ell \in \N$ that satisfy $2^\ell \in [  0 ,  n^{1-\beta} ]$. Note Theorem~\ref{t.proxy} implies that $\PP \big( \neg \, \mathsf{ProxySuccess}_n \big)$ is at most
$$
 \big( (1-\beta)(\log 2)^{-1} \log n + 1 \big) \bigg( 2 G  \tau_0^{2^{-1}d (\log 2)^{-1} H^3 \eta} \, + \, 15C \exp \Big\{ - d_0 H    \tau_0^{1/24 - 13\eta/14}  \Big\} + 14 \tau_0^{(\log 2)^{-1} 2^{-4} d H^3 \eta} \bigg) \, ,  
$$ 
Alongside $\xi$, the parameter $H > 0$ is part of the apparatus of Theorem~\ref{t.proxy}. 
By imposing that 
\begin{equation}\label{e.etalb}
1/24 - 13\eta/14 \leq 0 \, , \, \,  \textrm{i.e.,} \,  
\eta \geq 7/156 \, ,
\end{equation}
and recalling that $\tau_ 0 = (\log n)^{-q}$, we see that, with a suitable decrease in the value of $d$,
there exists $H_0 >0$ such that, when $H$ is set at a value that exceeds $H_0$ (and from now on it is), and when $n$ is at least a level $n_0 \in \N$ determined by $H$,
\begin{equation}\label{e.proxynonsuccess}
\PP \big( \neg \, \mathsf{ProxySuccess}_n \big) \leq \tau_0^{d H^3 \eta} \, .
\end{equation}

Let $\mc{C}$ denote the collection of scale~$L$ excursions between $\rho_n^0$ and $\rho_n^t$
that are contained in the strip $\R \times [g/4 , 1 - g/4]$; this definition indeed corresponds to the object in Theorem~\ref{t.proxy} with $\xi = g/4$ (and with scale $\ell$ equal to $L$). 
Suppose that  $\mathsf{ProxySuccess}_n$ occurs. By the proxy's feat of excursion mimicry, the subset $\mc{C'}$ of $\mc{C}$ consisting of those $E \in \mc{C}$ for which both endpoints of $E$ lie in~$\rho_{n,L}^{t \to 0}$ satisfies $\vert \mc{C}' \vert \geq \vert \mc{C} \vert/2$. 
When $\mathsf{NoSlender}(\alpha,\beta,1-\chi)$ also occurs, every element of~$\mc{C}$ is a normal excursion between $\rho_n^0$ and $\rho_n^t$, so that, for every $E \in \mc{C}'$,
the endpoints of~$E$
form the endpoint pair of a weak excursion between $\rho_n^0$  and $\rho_{n,L}^{t \to 0}$.
 Let $\big\{ (z_i,h_i): i \in \llbracket 0,R \rrbracket \big\}$ record in increasing order of height the union of $(z_0,h_0) = (0,0)$, $(z_R,h_R) =(0,1)$
 and the endpoints of elements of~$\mc{C}'$. Let $J$ denote the set of {\em starting indices for weak excursions}; namely, $J$ contains indices~$i$ for which  $(z_i,h_i)$ and $ (z_{i+1},h_{i+1})$ form the endpoint pair of a weak excursion between $\rho_n^0$  and $\rho_{n,L}^{t \to 0}$. Then the elements of $J$ are a finite set, either of consecutive even numbers starting at zero, or of consecutive odd numbers starting at one.  
 
 Note that 
 $$
  \weight^0_n \big( \rho_n^0 \big) = \sum_{i \in J} \weight_n^0 \big[ (z_i,h_i) \to (z_{i+1},h_{i+1}) \big] +  \sum_{i \in \llbracket 0,R-1 \rrbracket \setminus J} \weight_n^0 \big[ (z_i,h_i) \to (z_{i+1},h_{i+1}) \big]  
 $$
 and that, for each $i \in \llbracket 0, R-1 \rrbracket$, $\weight_n^0 \big[ (z_i,h_i) \to (z_{i+1},h_{i+1}) \big]  \geq \weight_n^0 \, \rho_{n,L}^{t \to 0} \big[ (z_i,h_i) \to (z_{i+1},h_{i+1}) \big]$.

Along with the excursion mimicry property of our proxy, we will need the following proximity result, which rules out wide excursions.  
In a variation of the notation of Definition~\ref{d.maxdist}, let $i \in J$ and denote by  
$$
\maxdist \big( \rho_{n,L}^{t \to 0} \big[(z_i,h_i) \to (z_{i+1},h_{i+1}) \big] , \rho^0_n  \big[(z_i,h_i) \to (z_{i+1},h_{i+1}) \big]  \big)
$$ 
the maximum distance between two points $(a,\heightmac), (b,\heightmac) \in \R \times n^{-1}\Z$ that share their height $\heightmac \in [h_i,h_{i+1}]$
and for which $(a,\heightmac) \in \rho_{n,L}^{t \to 0}  \big[(z_i,h_i) \to (z_{i+1},h_{i+1}) \big]$ and $(b,\heightmac) \in \rho^0_n  \big[(z_i,h_i) \to (z_{i+1},h_{i+1}) \big]$.  

In contrast to Proposition \ref{p.proxy}, the next lemma offers an upper bound on the distance between $\rho_{n,L}^{t \to 0}$  and $\rho_{n}^{0}$.

{For $H > 0$, let $\proxymimicry$ denote the event that}
$$
 \max_{i \in J} \,
\maxdist  \, \Big( \, \rho_{n,L}^{t \to 0} \big[(z_i,h_i) \to (z_{i+1},h_{i+1}) \big] \, , \, \rho^0_n  \big[(z_i,h_i) \to (z_{i+1},h_{i+1}) \big]  \,  \Big) 
$$
is at most $H \eta^{1/3} \big( \log \tau_0^{-1} \big)^{1/3} 2^{- 2L/3}(L+2)^{1/3}$.

\begin{lemma}\label{l.evente}
For any $\eta_0 > 0$, there exist positive $H_0 = H_0(\eta_0)$ and $d$ such that, when $\eta \geq \eta_0$ and $H \geq H_0$, 
$$
\PP \Big(  \neg \, \proxymimicry
\Big)
  \leq \tau_0^{ d \eta H^3 } \, .
$$
\end{lemma}

{\bf Proof.}
Again recalling  Definition~\ref{d.maxdist}, and the notion of regularity from \eqref{e.regulardef}, consider the event 
$$
 \mathsf{E} =  \Big\{ \textrm{$\rho_n^0$ 
and $\rho_n^t$ are
 $\big( 2^{-L-2} , H \eta^{1/3} (\log \tau_0^{-1})^{1/3} \big)$-regular} 
  \Big\} \cap
   \Big\{ \maxdist \big( \rho_{n,L}^{t \to 0} , \rho^t_n \big) \leq  6H 2^{-2m/3} m^{1/3} \Big\} \, ,
$$
where $m \in \N$ is specified by $2^m = 2^L \tau_0^{-\eta}$, in accordance with~(\ref{e.mlrelation}) when the scale $\ell = L$ is chosen.
The lemma will follow from two claims: that $\mathsf{E} \subseteq \proxymimicry$;  and that $\PP \big( \mathsf{E}^c \big) \leq \tau_0^{ d \eta H^3 }$ under the lemma's hypotheses. 

We begin by deriving the probability bound. Proposition~\ref{p.proxy} asserts that, when $H$ is large enough, $\PP \big( \maxdist \big( \rho_{n,\ell}^{t \to 0} , \rho^t_n \big) >  6H 2^{-2m/3} m^{1/3}  \big)$ is at most
the quantity in~(\ref{e.doublefourteen}). By summing this bound over $\ell \in \N$ such that $2^\ell \leq n^{1 - \beta}$ and decreasing the value of $d >0$, we find that, for $n$ large enough,  
$$
\PP \Big( \maxdist \big( \rho_{n,L}^{t \to 0} , \rho^t_n \big) >  6H 2^{-2m/3} m^{1/3} \Big)  \leq \tau_0^{d H^3 \eta}   \, .
$$
Since $2^{-L} \geq n^{\beta - 1}$, we find that, by summing the bound in Proposition~\ref{p.regular} over those $\ell \in \N$, $\ell \geq 2$, for which $2^\ell \leq n^{1 - \beta}$, 
after a suitable decrease in the value of $d > 0$, and for $H$ high enough, 
$$
 \PP  \Big( \textrm{$\rho_n^{t'}$ 
is not
 $\big( 2^{-L-2} , H \eta^{1/3} (\log \tau_0^{-1})^{1/3} \big)$-regular} \Big)  \leq \tau_0^{d \eta H^3} \, \, \textrm{for $t' \in \{ 0,t \}$}  \, .
 $$
 The desired bound on $\PP \big( \mathsf{E}^c \big)$ follows from the two preceding displays.
  
 Suppose that the event $\mathsf{E}$ occurs, and let $(a,\heightmac), (b,\heightmac) \in \R \times n^{-1}\Z$
 satisfy 
 $$
 (a,\heightmac) \in  \rho_{n,L}^{t \to 0}  \big[(z_i,h_i) \to (z_{i+1},h_{i+1}) \big] \, \, \, \textrm{and} \, \, \, (b,\heightmac) \in \rho^0_n  \big[(z_i,h_i) \to (z_{i+1},h_{i+1}) \big]
 $$  
 for some $i \in J$. Recall that $(z_i,h_i) \in \rho_n^0 \cap \rho_n^t$ and that $\heightmac - h_i \in [0,2^{-\ell}]$. Note that
 $$
 \big\vert a - b \big\vert \leq  \big\vert a - \rho_n^t(\heightmac) \big\vert
 +  \big\vert \rho_n^t(\heightmac) - z_i \big\vert +  \big\vert z_i - b \big\vert \, .
 $$
 Note that $\vert a - \rho_n^t(\heightmac) \vert \leq \maxdist \big( \rho_{n,L}^{t \to 0} , \rho^t_n \big)$; that
 $$
 \vert \rho_n^t(\heightmac) - z_i \vert \leq (\log 2)^{1/3} 2^{-4/3} H \eta^{1/3} (\log \tau_0^{-1})^{1/3} 2^{- 2L/3}(L+2)^{1/3}
 $$ 
 when $\rho_n^t$ is $\big(2^{-L-2}, H \eta^{1/3} (\log \tau_0^{-1})^{1/3} \big)$-regular; and that $\vert z_i - b \vert$ satisfies the same bound when $\rho_n^0$ is $\big(2^{-L-2}, H \eta^{1/3} (\log \tau_0^{-1})^{1/3} \big)$-regular. Thus, the occurrence of $\mathsf{E}$ entails that 
 $$
 \big\vert a - b \big\vert \leq   6H 2^{-2m/3} m^{1/3} + 2 (\log 2)^{1/3} 2^{-4/3} H \eta^{1/3} (\log \tau_0^{-1})^{1/3} 2^{- 2L/3}(L+2)^{1/3} \, .   
 $$
 Since~$2^m = 2^L \tau_0^{-\eta}$ implies that $m \geq L +2$, we see that the occurrence of $\mathsf{E}$ implies that 
 $\vert a - b \vert \leq H  \eta^{1/3} \big(\log \tau_0^{-1} \big)^{1/3}  2^{- 2L/3}(L+2)^{1/3}$ after a suitable increase in $H > 0$ that depends on the lower bound $\eta_0 > 0$ on $\eta$. The event $\proxymimicry$ seen to occur, so that the proof of Lemma~\ref{l.evente} is complete. \qed

 Let $i$ belong to the set $J$ of starting indices of weak excursions between $\rho_n^0$ and $\rho_{n,L}^{t \to 0}$ whose durations are contained in the interval $[g/4,1-g/4]$.  
 Let $W_i = [h_i , h_{i+1}) \cap n^{-1}\Z$ denote the set of vertical levels of the excursion of this type associated to the index~$i$. 
 For $\heightmac \in n^{-1} \llbracket 0 , n \rrbracket$, define the {\em deficit} $D(\heightmac)$ to equal 
 \begin{equation}\label{e.deficit}
 Z^0_n(M,\heightmac) \, - \, \sup \Big\{ Z^0_n(x,\heightmac) : x \in \R \, , \, 2^{-2L/3} \cdot 2^{-2/3} \tau_0^\alpha  \leq \vert x - M \vert \leq  2^{- 2L/3} \cdot H \eta^{1/3} \big( \log \tau_0^{-1} \big)^{1/3} (L+2)^{1/3} \Big\}
 \end{equation}
 in the case that $\heightmac \in W_i$ for some $i \in J$ and  $\big\vert \rho_n^0(\heightmac) - \rho_{n,L}^{t \to 0}(\heightmac) \big\vert \geq 2^{-2/3} 2^{-2L/3} \tau_0^\alpha$; for the remaining values of $\heightmac$, set $D(\heightmac) = 0$. Here, $Z_n^0(x,\heightmac)$ refers to the value~$Z_n(x,\heightmac)$ of the routed weight profile  specified by the randomness of the time-zero copy of the noise environment; $M$ denotes the maximizer of $Z_n^0(\cdot,s)$. Choosing $y_i$ in
Lemma~\ref{l.additive} to be a maximizer of $D(\heightmac)$ among $\heightmac \in W_i$ for each $i \in J$, this lemma implies that the occurrence of  $\proxymimicry$ entails that
\begin{equation}\label{e.weightsum}
   \weight^0_n \big( \rho_n^0 \big) -  \weight^0_n \big( \rho_{n,L}^{t \to 0} \big) \geq \sum_{i \in J} \max \big\{ D(\heightmac): \heightmac \in W_i \big\} \, .
\end{equation}

 We will next show that $D(s)$ is reasonably large for most values of $s.$

Towards this end, for a parameter $\zeta$ which will be taken to be a small constant, specify the {\em low total deficit} $\mathsf{LowTotDef}$ event by the condition that $0 < D(\heightmac) < 2^{-1/3} 2^{-L/3} \tau_0^{\alpha/2+\zeta}$ holds for at least  
 $(1-\chi)/2 \cdot (n +1) 
 g \cdot 2^{-5} \log 2 \, \big( \log n \big)^{-1}$
  values of $\heightmac \in n^{-1} \Z \cap [g/4,1-g/4]$. 
 Note that we expect the above to be a rare event since Brownian fluctuations indicate that $D(s)$, when positive, should be at least  comparable to $2^{-L/3} \tau_0^{\alpha/2}.$  This is the claim made by the next result,  whose proof is deferred until after the derivation of Theorem~\ref{t.lowell} is concluded.

\begin{lemma}\label{l.lowtotdef}
Given constants $\zeta, \eta$ as above, for some
$\hmac = \hmac(\eta,\zeta)$, 
 the assumption that  
 $$
 \exp \{- n^{1/10} \} \leq \tau_0   \leq \exp \big\{ - \hmac \big( \log \log n \big)^{68} \big\}
 $$ 
 implies that
 $$
 \PP \Big( \mathsf{LowTotDef} \Big) \leq \exp \big\{ -  \hmac 
   \big( \log \tau_0^{-1} \big)^{1/68}  \big\} \, .
  $$
\end{lemma}  
   Recall that, in~(\ref{e.taucondition}), we set $\tau_0 = (\log n)^{-q}$. The hypothesis on $\tau_0$ in Lemma~\ref{l.lowtotdef} requires that $q$ be treated as a function of $n$; it is in order to satisfy it that we have stipulated for $q$ the formula~(\ref{e.qndependence}).

  The sum of the durations of elements in the excursion set $\mc{C}'$ specified after~(\ref{e.proxynonsuccess}) is at least the quantity $2^{-5}g \log 2  (\log n)^{-1}$, because $\mc{C}'$ constitutes a proportion of at least one-half of the excursions contributing to the sum in~(\ref{e.longexcursionpigeonhole}) and all contributing excursions have duration between $2^{-1-L}$ and~$2^{-L}$.
   
  Note that the cardinality of the set of $\heightmac \in n^{-1}\Z \cap [g/4,1-g/4]$
 for which $\heightmac \in W_i$ for some $i \in J$ and  $\big\vert \rho_n^0(\heightmac) - \rho_{n,L}^{t \to 0}(\heightmac) \big\vert \geq 2^{-2/3} 2^{-2L/3} \tau_0^\alpha$ is at least $(1 - \chi) \cdot (n +1) 
 g \cdot 2^{-5}  \log 2 \, \big( \log n \big)^{-1}$. (Here, we used the containment of each $W_i$ in $[g/4,1-g/4]$, profiting from the assumption that the bulk case event $\bulklongexc(g/4)$ occurs.) Thus, when the event  $\mathsf{LowTotDef}$ fails to occur, there are at least 
 $(1-\chi)/2 \cdot (n +1) 
 g \cdot 2^{-5} \log 2 \, \big( \log n \big)^{-1}$
values of $\heightmac \in \bigcup_{i \in J} W_i$ such that $D(\heightmac) \geq 2^{-1/3} 2^{-L/3} \tau_0^{\alpha/2 + \zeta}$. The number of indices $i \in J$ for which 
  $\max \big\{ D(\heightmac): \heightmac \in W_i \big\} \geq 2^{-1/3} 2^{-L/3} \tau_0^{\alpha/2 +\zeta}$ is then seen to be at least 
$$
 \big( n 2^{-L} + 1 \big)^{-1} \cdot 
 (1-\chi) \cdot (n +1) 
 g \cdot 2^{-6} \log 2 \, \big( \log n \big)^{-1} \, , 
$$
since $\vert W_i \vert \leq n 2^{-L} + 1$ for any such~$i$. Returning to~(\ref{e.weightsum}) having learnt this, we see that, when 
$$
\bulklongexc(g/4)   \cap \mathsf{NoSlender}(\alpha,\beta,1-\chi) \cap \mathsf{ProxySuccess}_n  \cap \proxymimicry \cap \neg \, \mathsf{LowTotDef}
$$ 
occurs,
$$
   \weight^0_n \big( \rho_n^0 \big) -  \weight^0_n \big( \rho_{n,L}^{t \to 0} \big) \geq 
   2^{-1/3} 2^{-L/3} \tau_0^{\alpha/2 +\zeta} \cdot \big( n 2^{-L} + 1 \big)^{-1} \cdot 
 (1-\chi) \cdot (n +1) 
 g \cdot 2^{-6} \log 2 \, \big( \log n \big)^{-1} \, ; 
$$
or, 
more simply,
\begin{equation}\label{e.moresimply}
   \weight^0_n \big( \rho_n^0 \big) -  \weight^0_n \big( \rho_{n,L}^{t \to 0} \big) \geq 
\hmac (1-\chi) \cdot 2^{2L/3} \tau_0^{\alpha/2 + \zeta} \big( \log n \big)^{-1} \, .
\end{equation}
for a suitably small constant $h > 0$.

On the other hand, by the weight mimicry aspect of Theorem~\ref{t.proxy}, the bound
\begin{equation}\label{e.otherhand}
   \Big\vert  \weight^0 \big( \rho_{n,L}^{t \to 0} \big) - \weight^t \big( \rho_n^t \big) \Big\vert \leq   H^3 2^{2L/3} \tau_0^{1/1002 - 2\eta/3} \Psi \, ,
\end{equation}
holds
with  $\Psi = 200 L^{2/3} \big( 1 +  L^{-1}(\log 2)^{-1} \eta \log \tau_0^{-1} \big)^{2/3}$,
since $\mathsf{ProxySuccess}_n$ occurs.
However, as we will argue shortly,
\begin{equation}\label{e.shortly} 
 \hmac(1-\chi)/2 \cdot 2^{2L/3} \tau_0^{\alpha/2 + \zeta} \big( \log n \big)^{-1}
  \geq    H^3 2^{2L/3} \tau_0^{1/1002 - 2\eta/3} \Psi  \, , 
\end{equation}
so that the bounds~(\ref{e.moresimply}) and~(\ref{e.otherhand}) imply that 
$$
 \weight^0_n \big( \rho_n^0 \big) -  \weight^t_n \big( \rho_{n}^{t} \big) \geq 
\hmac (1-\chi)/2 \cdot 2^{2L/3} \tau_0^{\alpha/2 + \zeta} \big( \log n \big)^{-1} \, .
$$

 By Proposition \ref{p.onepoint}(2) with $(x,s_1) = (0,0)$ and $(y,s_2) = (0,1)$,  and by Markov's inequality, it is with probability at least $1- \tau_0^{1/4}$ that
$$
\Big\vert\weight^0_n \big( \rho_n^0 \big) -  \weight^t_n \big( \rho_{n}^{t} \big)\Big\vert <  2^{1/2} \tau_0^{3/8} \, .
$$

Thus if 
\begin{equation}\label{e.constraint}
\hmac (1-\chi)/2 \cdot 2^{2L/3} \tau_0^{\alpha/2 + \zeta} \big( \log n \big)^{-1} \geq 2^{1/2}\tau_0^{3/8} \, ,
 \end{equation}
 we find that
  $$
 \PP \Big( \bulklongexc(g/4)  \cap \mathsf{NoSlender}(\alpha,\beta,1-\chi) \cap \mathsf{ProxySuccess}_n \cap \proxymimicry \cap \neg \, \mathsf{LowTotDef} \Big)
 $$
 is at most $\tau_0^{1/4}$.
We learn then that   $\PP \big( \bulklongexc(g/4)    \big)$ is at most the sum of
$$
 \PP \Big(  \mathsf{Slender}(\alpha,\beta,1-\chi) \Big) +  \PP \Big( \neg \, \mathsf{ProxySuccess}_n \Big) +  \PP \Big( \neg \, \proxymimicry \Big)  +   \PP \Big(  \mathsf{LowTotDef} \Big) 
$$
and  $\tau_0^{1/4}$. 
By Proposition~\ref{p.noslender}, (\ref{e.proxynonsuccess}), Lemma~\ref{l.evente} and Lemma~\ref{l.lowtotdef},
\begin{equation}\label{e.fiveterms}
\PP \Big( \bulklongexc(g/4)   \Big) \leq A_1 + A_2+A_3+A_4+A_5  \, ,
\end{equation}
where $A_1 =  \exp \big\{ - \big( \log n \big)^{q\alpha/2 -1} \big\}$; $A_2 = \tau_0^{d H^3 \eta}$; 
$A_3 = \tau_0^{ d \eta H^3 }$; 
$$
A_4 = \exp \big\{ - \hmac \big( \log \tau_0^{-1} \big)^{1/68}  \big\}
     \, ;
 $$
  and $A_5 = \tau_0^{1/4}$. Imposing that $q\alpha \geq 4$ to render $A_1$ of rapid decay, Theorem~\ref{t.lowell} in the bulk case is obtained after a decrease in the value of $h > 0$.

To complete the proof of this case of the theorem, it remains to verify~(\ref{e.shortly}) and \eqref{e.constraint}. Recalling that $2^{-m} = 2^{-L} \tau_0^\eta$, $2^L \leq n$ and $\tau_0 = (\log n)^{-q}$, it is enough for the former to hold that 
$$\hmac(1-\chi)/2 \cdot  \tau_0^{\alpha/2 + \zeta} \big( \log n \big)^{-1}
  \geq  H^3 \tau_0^{1/1002 - 2\eta/3} 200 (\log 2)^{-2/3}  \big( \log n +   \eta 
  q \log \log  n \big)^{2/3} \, .
  $$
If we impose on $q$ and $\eta$ the condition that   $\eta q \leq \tfrac{\log n}{\log \log n}$, we see that the displayed bound is implied by 
$$
(\log n)^{-5/3 - \varepsilon} \tau_0^{\alpha/2 + \zeta} 
  \geq \tau_0^{1/1002 -2\eta/3}  \, ,
$$
where $\varepsilon > 0$ is an arbitrarily small positive quantity and it is supposed that $n$ exceeds a value determined by $\varepsilon$ and $H$.

Since $L \geq 0$,  \eqref{e.constraint} is implied by 
$$
\tau_0^{\alpha/2 + \zeta} \big( \log n \big)^{-1 - \varepsilon} \geq  \tau_0^{3/8} \,
$$
provided that $n$ is at least a value determined by an arbitrary choice of $\varepsilon > 0$.

Again using  $\tau_0 = (\log n)^{-q}$, we see that~(\ref{e.shortly}) and~(\ref{e.constraint}) are implied by the conditions
$$
 q \big( 1/1002 - 2\eta/3 -  \alpha/2 - \zeta \big) \geq 5/3 + \varepsilon
$$
and
$$
 q \big( 3/8 - \alpha/2 - \zeta \big) \geq 1  + \varepsilon \, ,
$$
at least when $n$ is high enough.
The former condition is the stronger. Recall that the parameter $q$ was chosen to be $h \big( \log \log n \big)^{67}$ in (\ref{e.qndependence})---and recall also that the reason for this choice will become clear in the deferred proof of Lemma \ref{l.lowtotdef}---so that it grows to infinity as $n$ rises. The two last displayed conditions are thus satisfied when $n$ is high enough provided that $2\eta/3 + \alpha/2 + \zeta < 1/1002$. Taking $\eta = 10^{-3}$, we have $\eta < 7/156$ in satisfaction of~(\ref{e.etalb}). Noting that  $\frac{1}{1002} - \frac{2}{3000} = 3.31 \cdots \times 10^{-4}$,
we may set $\alpha = \zeta = \tfrac{1}{4528}$ to achieve the desired bound.
As we have noted, the further condition $q\zeta \geq 4$ is imposed, to assure the exploited control on the right-hand term $A_1$ in~(\ref{e.fiveterms}).
 Since $q = q_n \to \infty$ as $n \to \infty$, this condition is satisfied for $n$ high enough.
 This completes the proof of Theorem~\ref{t.lowell} in the bulk case.

 We now turn attention to the edge case, in which we instead seek a bound on the probability of the event $\edgelongexc(g/4)$ that appears on the right-hand side of the inclusion~(\ref{e.bulkedgeinclusion}). When this event occurs, either the excursion whose lifetime contains $g/4$ has duration at least $g/8$; or the excursion whose lifetime contains $1 - g/4$ does. We redefine the dominant scale $L$ to be the smaller of the scales of these two excursions. Note that $2^{-L} \geq g/8$, so that $L$ is a bounded quantity. The edge-bulk division parameter $\xi$ in Theorem~\ref{t.proxy} was set to equal $g/4$ in the bulk case; {now, we simply take it to equal zero.} At least one excursion will be retained in proxy construction; and, because $L$ is bounded, it suffices to analyse a single such. Thus, we take $J$ to be the singleton set containing the starting index of an arbitrary retained excursion~$E$ of scale~$L$; we set  $W_1 = n^{-1}\Z \cap {\rm dur}(E)$. Shortly after Lemma~\ref{l.lowtotdef}, we noted that, in the bulk case, the cardinality of 
  the set of $\heightmac \in n^{-1}\Z \cap [g/4,1-g/4]$
 for which $\heightmac \in W_i$ for some $i \in J$ and  $\big\vert \rho_n^0(\heightmac) - \rho_{n,L}^{t \to 0}(\heightmac) \big\vert \geq 2^{-2/3} 2^{-2L/3} \tau_0^\alpha$ is at least $(1 - \chi) \cdot (n +1) 
 g \cdot 2^{-5}  \log 2 \, \big( \log n \big)^{-1}$. Since, in the present case, $W_1$ has cardinality that exceeds $gn/8$, and the proportion of elements~$s$ of $W_i$ for which $D(s) > 0$ is at least~$1-\chi$ (this because our highlighted excursion is successfully mimicked in Theorem~\ref{t.proxy}), we see that this cardinality lower bound is obtained in the edge case also, provided that $n$ is chosen high enough, and that the interval $[g/4,1-g/4]$ is replaced by $[g/20,1-g/20]$ (our choice of twenty is arbitrary to the degree that any number that exceeds eight would suffice). With this replacement made, the edge case has been incorporated into the bulk, with the right-hand side of~(\ref{e.weightsum}) taking the form of a single term that is least a constant multiple of $\tau_0^{\alpha/2 +\zeta}$;  making the interval replacement throughout, we obtain the edge case proof of Theorem~\ref{t.lowell} from its bulk case counterpart. \qed

We next provide the proof of Lemma~\ref{l.lowtotdef} which will rely on our choice of $q$.

{\bf Proof of Lemma~\ref{l.lowtotdef}.} For $\ell \in \N$, write $D(\heightmac,\ell)$ for the expression formed from~(\ref{e.deficit}) by the replacement of $L$ by $\ell$.  For any $\ell \in \N,$ we will apply Theorem~\ref{t.nearmax} with ${\bm \e} =  2^{-2\ell/3} \cdot 2^{-2/3} \tau_0^\alpha$, ${\bm \ell'} = 3 \cdot 2^{- 2\ell/3} \cdot H \eta^{1/3} \big( \log \tau_0^{-1} \big)^{1/228} (\ell+2)^{1/3}$, ${\bm \ell} = 30 H \eta^{1/3} \big( \log \tau_0^{-1} \big)^{1/228}$ and ${\bm \hata}= \tau_0^\zeta$. 
According to the convention stated in Section \ref{s.boldface}, we use ${\bm \ell}, {\bm \ell'}$ to denote the parameters $\ell, \ell'$ in Theorem~\ref{t.nearmax}, while $\ell$ is used to denote the  possible values that $L$ may take.

By this application of Theorem~\ref{t.nearmax}, we find that, for any $R \in \R$ with $\vert R \vert \leq h n^{1/9}$, and for $n$ high enough,
\begin{eqnarray*}
& & \PP \Big( D(\heightmac,\ell) \leq 2^{-\ell/3} \cdot 2^{-1/3} \tau_0^{\alpha/2} \cdot \tau_0^\zeta  \, , \, \vert M - R \vert \leq 10  H \eta^{1/3} \big( \log \tau_0^{-1} \big)^{1/228} \Big)
 \\
& \leq &
\log \Big( 3 \cdot 2^{2/3} \tau_0^{-\alpha} \cdot H \eta^{1/3} \big( \log \tau_0^{-1} \big)^{1/228} (\ell+2)^{1/3} \Big) \\
& & \qquad \quad \max \Big\{  (\ell + 1)^{-2} \tau_0^{\zeta/2} \cdot 
 \exp \big\{  - hR^2 + H \ell^{\macroseventeen} \big( 1 + R^2 + \log \tau_0^{-\zeta}  \big)^{5/6} \big\} ,   \exp \big\{ - h n^{1/12}   \big\} \Big\} \, ,
\end{eqnarray*}
where we used $\tau_0^\zeta \leq (\ell + 1)^{-2} \tau_0^{\zeta/2}$, which is due to $2^\ell \leq n$, $\tau_0 = (\log n)^{-q}$, $\zeta q>4,$ 
and $n$ being high enough.
(The upper bound ${\bm \ell} \leq h n^{1/\macrobig}$ needed to apply Theorem~\ref{t.nearmax} is satisfied for high $n$ because we suppose that $\tau_0 \geq \exp \{ - n^{1/10} \}$. The assumed upper bound on $\tau_0$ similarly ensures that $3\e \leq \ell'$.)

Setting $R =0$, we find that, for $2^\ell \leq n$, with suitable $(\eta,\zeta)$-determined adjustments to the values of $H$ and $h$, and when $n$ is high enough,
\begin{eqnarray*}
& &  \PP \Big( D(\heightmac,\ell) \leq 2^{-\ell/3} \cdot 2^{-1/3} \tau_0^{\alpha/2} \cdot \tau_0^\zeta  \, , \, \vert M  \vert \leq 10  H \eta^{1/3} \big( \log \tau_0^{-1} \big)^{1/228} \Big) \\
& \leq &  \max \Big\{  (\ell + 1)^{-2} \tau_0^{\zeta/2}   \exp \big\{  H
\big(  \log \tau_0^{-1}  \big)^{11/12} \big\} ,   \exp \big\{ - h n^{1/12}   \big\} \Big\} \, .
\end{eqnarray*}
Summing over $\ell \in \N$, we find that
\begin{eqnarray*}
& &  \PP \Big( D(\heightmac,L) \leq 2^{-L/3} \cdot 2^{-1/3} \tau_0^{\alpha/2} \cdot \tau_0^\zeta  \, , \, \vert M  \vert \leq 10  H \eta^{1/3} \big( \log \tau_0^{-1} \big)^{1/228} \Big) \\
& \leq &  \max \Big\{  \tau_0^{\zeta/2}   \exp \big\{  H
\big(  \log \tau_0^{-1}  \big)^{11/12} \big\} ,   \exp \big\{ - h n^{1/12}   \big\} \Big\} \, .
\end{eqnarray*}

By Proposition~\ref{p.maxfluc} with ${\bf R} = 10  H \eta^{1/3} \big( \log \tau_0^{-1} \big)^{1/228}$, ${\bf n} = n$ and $\mathbf{s_{1,2}} = 1$, 
$$
 \PP \Big( \vert M  \vert > 10  H \eta^{1/3} \big( \log \tau_0^{-1} \big)^{1/228} \Big) \leq \exp \big\{ - d 10^3  H^3 \eta \big( \log \tau_0^{-1} \big)^{1/68}  \big\} \, .
$$
Note that this proposition's hypothesis that ${\bf R} \leq n^{1/10}$
is satisfied because we suppose that
$\tau_0$  is at least $\exp \{- n^{1/10} \}$ and $n$ is high enough.

 Thus, for suitably small $h = h(\eta,\zeta) > 0$, 
 $$
  \PP \Big( D(\heightmac,L) \leq 2^{-(L +1)/3} \tau_0^{\alpha/2 + \zeta}  \Big) \leq  \exp \big\{ -  H^3 h 
   \big( \log \tau_0^{-1} \big)^{1/68}  \big\} \, .
  $$

Recall that $\mathsf{LowTotDef}$ occurs when $0 < D(\heightmac,L) < 2^{-1/3} 2^{-L/3} \tau_0^\zeta$ holds for a set of 
  values of $\heightmac$ in $n^{-1}\Z \cap [g/4,1-g/4]$ whose cardinality is at least  
 $(1-\chi)/2 \cdot (n +1) 
 g \cdot 2^{-5} \log 2 \, \big( \log n \big)^{-1}$. The proof will now follow from a simple first moment bound.
 Let $A \in [g/4,1-g/4] \cap n^{-1} \llbracket 0 , n \rrbracket$ be picked uniformly at random, independently of other randomness. 
  Then
\begin{eqnarray*}
 & &  \PP \Big(  0 < D(A,L) < 2^{-1/3} 2^{-L/3} \tau_0^\zeta \,  \Big\vert \, \mathsf{LowTotDef} \Big)  \\
 & \geq & (1-\chi)/2 \cdot (n +1) 
 g \cdot 2^{-5} \log 2 \, \big( \log n \big)^{-1}   \cdot \frac{1}{(n+1) (1 -g/2)} \, .
\end{eqnarray*}
Thus, 
 $\PP \big( \mathsf{LowTotDef} \big)$ is at most 
 $$
  \exp \big\{ -  H^3 h 
   \big( \log \tau_0^{-1} \big)^{1/68}  \big\} 
 \cdot 2^6 g^{-1}  (1 -g/2) (1 -\chi)^{-1} (\log 2)^{-1} \log n  \, .
 $$ 
Our hypothesis on $\tau_0$ in \eqref{e.qndependence} permits us to suppose that $\tau_0   \leq \exp \big\{ - \big( 2h^{-1} H^{-3}\log \log n \big)^{68} \big\}$. We then obtain that $\PP \big( \mathsf{LowTotDef} \big)$ is at most a constant multiple of
 $\exp \big\{ -  2^{-1} H^3 h 
    \big( \log \tau_0^{-1} \big)^{1/68}  \big\}$. 
Relabelling $h >0$ completes the proof of Lemma~\ref{l.lowtotdef}. \qed

The next proof is the remaining missing  piece in the derivation of Theorem~\ref{t.lowell}.
\subsection{Slim pickings for slender excursions: deriving Proposition~\ref{p.noslender}}\label{s.slenderresults}

The proof of Proposition~\ref{p.noslender} depends principally on an input~\cite[Theorem~$1.9$]{NearGroundStates} that concerns static Brownian LPP.  We will present notation for the dynamic model and
indicate shortly the relation to this input.
 Recall that~$\rho^0_n$
denotes the time-zero polymer from $(0,0)$ to $(0,1)$. Let $(x,s_1),(y,s_2) \in n^{-1}\Z \cap [0,1]$, $s_1 < s_2$, be two points, neither of which necessarily lies in $\rho^0_n$.
A zigzag 
$\psi$ from $(x,s_1)$ to $(y,s_2)$ that is disjoint from~$\rho^0_n$ will be called a {\em meander}, 
(a word intended to evoke `excursion' which is reserved in our usage for the case where $\psi$'s endpoints lie in $\rho^0_n$).
For $\tza > 0$, a meander 
$\psi$ from $(x,s_1)$ to $(y,s_2)$ is called $(\rho^0_n,\tza,1 - \chi)$-close if 
the set of $s \in [s_1,s_2] \cap n^{-1}\Z$ for which
$\vert \psi(s) - \rho^0_n(s) \vert \leq \tot^{1/3} \tza$
has cardinality at least $(1-\chi) \big\vert [s_1,s_2] \cap n^{-1}\Z \big\vert$
and contains the values $s_1$ and $s_2$.
For $t \geq 0$, the supremum of the time-$t$ weights of $(\rho^0_n,\tza,1-\chi)$-close meanders will be denoted by 
$$
\weight^t_n \big[  (x,s_1) \to (y,s_2)  \, ; \, (\rho^0_n, \tza,1 - \chi)\textrm{-close}  \big] \, .
$$ 

 For $\ell \in \N$ and $d_0 >0$, let $\mathsf{LowSlenderWeight^t}\big( \ell, \tza,1-\chi;  \rho^0_n \big)$
denote the event that   
$$ 
\sup \, \tot^{-1/3} \weight^t_n \big[  (x,s_1) \to (y,s_2) \, ; \, (\rho^0_n,\tza,1-\chi)\textrm{-close}  \big] \leq - 
  \redconst {\tza}^{-1}  \, ,$$
 where the supremum is taken by varying the points $(x,s_1) , (y,s_2) \in \R \, \times \, [0,1] \cap n^{-1}\Z$ over choices such that $2^{-1 - \ell} \leq \tot \leq 2^{-\ell}$.
 The next result sets the value of $d_0$.
 
\begin{proposition}\label{p.slenderpolymer} There exist constants $d_0, C>0$, $\chi_0 \in (0,1)$, $d_2 > 0,$ and $n_0 \in \N$  such that, when   $\chi \in (0,\chi_0)$  $n \geq n_0$, $ \tza^{-1/4} > C\log n$ and  
 $\ell \in \N$ satisfies $2^{\ell}\leq n \tza^{40}$, then  
$$ 
\PP \Big( \neg \,  \mathsf{LowSlenderWeight}^t(\ell,\tza , 1 - \chi;  \rho^0_n) \Big) \leq  \exp \Big\{ - d_2 \tza^{-1/2}\Big\}.
$$ 
\end{proposition}
When $t = 0$, this result is~\cite[Theorem~$1.9$]{NearGroundStates}.
A stochastic comparison involving the dynamical model will permit us to derive the general $t$ version. In fact, even in \cite{NearGroundStates}, Theorem~$1.9$ was derived from \cite[Theorem~$1.10$]{NearGroundStates} using a similar stochastic comparison result which we will say more about after setting up the latter.

Recall that the unscaled dynamical Brownian noise environment is a system $B: \R \times \Z \times \R \to \R$ whose third argument is dynamical time~$t$. This system 
is Ornstein-Uhlenbeck dynamics in~$t$ whose invariant measure is a product of independent two-sided standard Brownian motions. It is convenient to introduce counterpart notation that uses scaled coordinates.
As such,
 $\big\{ \Omega^t: \R \times n^{-1}\Z \to \R \, , \, t \in  [0,\infty) \big\}$ will denote the scaled dynamical Brownian environment that is the image of the time~$t$ marginal of the $B$-system under the scaling map~$R_n$ from~(\ref{e.scalingmap}). For each $t \geq 0$, $\big\{ \Omega^t(\cdot,s): s \in n^{-1} \Z \big\}$
is an independent collection of Brownian motions of rate $2 n^{2/3}$.
The evolution in $t$
of each of these $s$-indexed Brownian motions $\Omega^t(\cdot,s): \R \to \R$
is  an independent Ornstein-Uhlenbeck dynamics (whose invariant measure is Brownian of rate  $2 n^{2/3}$). 

In keeping with the notation, introduced in Subsection~\ref{s.dynamicalnotation}, whereby the absence of a time superscript indicates a static object,
we write $\Omega:  \R \times n^{-1}\Z \to \R$ for the scaled static Brownian environment, so that $\Omega^t$ has the law of $\Omega$ for any given $t \geq 0$. 

For given $t \geq 0$, the fields $\Omega$ and $\Omega^t$ are random functions sending $\R \times n^{-1}\Z$  to~$\R$.
Let $X$ and~$Y$ denote any two such random functions. For any subset $A$ of $\R \times n^{-1}\Z$, $Y$ stochastically dominates $X$ on $A$ if  there exists a coupling of $X$ and $Y$ such that, 
whenever $(j,u,v) \in \Z \times \R^2$ satisfies $u < v$ and $\{ j/n \} \times [u,v] \subset A$,
the bound
$Y(v,j/n)- Y(u,j/n) \geq X(v,j/n)- X(u,j/n)$ holds. An event $E$ is called {\em negative} on $A$ if $\PP (Y \in A) \leq \PP (X \in A)$ whenever $Y$  stochastically dominates $X$ on $A$.

For a given $n$-zigzag $\phi$, let the {\em exterior} ${\rm Ext}(\phi)$ of $\phi$ denote  $\big( \R \times n^{-1} \Z \big) \setminus  \phi$.
%Proposition~\ref{p.slenderpolymer} concerns the copy of $\Omega$ given by $\Omega^t$, for given $t \geq 0$, and the polymer $\rho_n^{0}$ at time zero.
The next result provides information about the field $\Omega^t$ on the {\em exterior} of ${\rm Ext}(\rho_n^0)$.
%Since the latter object is determined by the field~$
%\Omega^0$, on which the field $\Omega^t$ depends, the field $\Omega^t$ reacts to the data $\rho_n^0$; we need to understand how.
%The next result is delivers the pertinent information.

\begin{lemma}\label{l.FKGdynamics}
In this result, we take a copy of the static Brownian environment $\Omega$ independent of the dynamical collection $\big\{ \Omega^t: t \geq 0 \big\}$ by declaring $\Omega$ to be independent of the latter system. For $t \geq 0$, the law of the restriction of $\Omega^t$ to ${\rm Ext}(\rho_n^0)$ is stochastically dominated by the law of $\Omega$'s restriction to ${\rm Ext}(\rho_n^0)$.
\end{lemma}
{\bf Proof.} 
Suppose first that $t = 0$. Consider the noise environment that is given by $\Omega^0$ on $\rho_n^0$ and by~$\Omega$ on ${\rm Ext}(\rho_n^0)$. When this environment is conditioned on the event that there exists no $n$-zigzag from $(0,0)$ to $(0,1)$ whose weight determined by this environment exceeds that of $\rho_n^0$, the result is a distributional copy of $\Omega^0$. The event in the conditioning is negative for $\Omega$ on ${\rm Ext}(\rho_n^0)$. The system~$\Omega$  on ${\rm Ext}(\rho_n^0)$ is a countable collection of Brownian motions whose domains are either copies of the real line or semi-infinite real intervals; indeed, to each height in $y \in n^{-1}\Z$ are associated one or two intervals, formed by the sometimes vacuous removal from $\R \times \{y \}$ of this set's intersection with~$\rho_n^0$. The FKG inequality for products of independent Brownian motions is implied by \cite[Theorems~$3$ and~$4$]{Barbato}. Applying it, we obtain Lemma~\ref{l.FKGdynamics} with $t =0$.

Before continuing to treat the general case,
it is convenient to identify ${\rm Ext}(\rho_n^0)$  with $\R \times n^{-1}\Z$. Recall that  ${\rm Ext}(\rho_n^0)$ is comprised of either one or two infinite intervals at every height in $n^{-1}\Z$; for heights where there are two intervals, we identify the two finite endpoints of these intervals, contracting to a point the interval of passage of $\rho_n^0$ to this vertical coordinate, to obtain the desired identification.

Now suppose that $t > 0$.  
By the case where $t = 0$, and with the above identification in operation,  $\Omega^0$ on $\R \times n^{-1}\Z$
is stochastically dominated by $\Omega$ on this set. 
Consider the Ornstein-Uhlenbeck dynamics mapping $\R \times n^{-1}\Z \times \R$ to $\R$
begun from $\Omega^0$ at dynamic time, or third coordinate,~$t$. This process is $\Omega^t$. But it also satisfies~(\ref{e.correlation}) with $\mc{I} = \R \times n^{-1}\Z$; with ${\bf X}(\cdot,0)$ equal in law to $\Omega^0$; and with~${\bf X}'$ equal in law to $\Omega$. Consider instead the counterpart dynamics begun from $\Omega$. By stationarity, the time-$t$ slice now has the law of $\Omega$. Moreover, this slice is again given by~(\ref{e.correlation}), with the variation that ${\bf X}(\cdot,0)$ is instead equal in law to $\Omega$.
Thus we see that~(\ref{e.correlation}) implies the stochastic domination of $\Omega^t$ on $\R \times n^{-1}\Z$ by~$\Omega$ on this set. This is the assertion of Lemma~\ref{l.FKGdynamics} for $t > 0$ when we recognize the presence of  a notational abuse arising from the identification of ${\rm Ext}(\rho_n^0)$  with~$\R \times n^{-1}\Z$. \qed

{\bf Proof of Proposition~\ref{p.slenderpolymer}.}
The final, three-line, paragraph of~\cite[Section~$4$]{NearGroundStates} ends the proof of~\cite[Theorem~$1.9$]{NearGroundStates}. 
The proof of \cite[Theorem~$1.9$]{NearGroundStates} uses \cite[Theorem~$1.10$]{NearGroundStates} along with 
\cite[Theorem~$4.17$]{NearGroundStates} which is simply Lemma~\ref{l.FKGdynamics} for the $t=0$ case. Replacing the former by the latter in the proof of~\cite[Theorem~$1.9$]{NearGroundStates}
yields Proposition~\ref{p.slenderpolymer}.
\qed

We have just shown that slender meanders typically have low weights. A further ingredient is needed before we prove Proposition~\ref{p.noslender}: no subpath of the polymer $\rho_n^t$ (for fixed $t$) accrues an extreme weight.

%To show that slender excursions between $\rho_n^0$ and $\rho_n^t$ are unlikely, we need one more input, namely that segments of the $\rho_n^t$ do not typically have extremely low weight, thereby ruling out the possibility that they can form slender excursions.

 Let $\nolow(D)$ denote the event that 
$$
 \tot^{-1/3} \weight_n \big[  (x,s_1) \to  (y, s_2) \big]  
$$ 
is at least $- \big( \log n \big)^D$ for all pairs $(x,s_1), (y,s_2) \in \R  \times  \big( n^{-1}\Z \cap [0,1] \big)$ that belong to $\rho_n$ and satisfy $s_1 \leq s_2$, and $\tot \geq n^{\beta/2 - 1}$. Let $\nohigh(D)$ denote the event that the displayed quantity is at most $\big( \log n \big)^D$  for the same set of pairs $(x,s_1), (y,s_2)$. 

The next lemma shows that the events $\nolow(D)$ and $\nohigh(D)$ occur with high probability. 
(In fact, for our application, we will only rely on the rarity of $\neg\nolow(D)$.) 

\begin{lemma}\label{l.onepoint}
There exists $D_0 > 0$ such that, when $D \geq D_0$, 
$$
\PP \Big( \nolow(D) \cup \nohigh(D) \Big) \, \geq \, 1 - 
H \exp \big\{ - h \big( \log n \big)^{3D/2 -1} \big\}
 \, .
$$ 
\end{lemma} 

{\bf Proof.}
The proof follows from Proposition~\ref{p.weight} with  ${\bf r}^2 (\log n)^{2/3} = O( \big( \log n \big)^D)$, so that ${\bf r}$ is a  constant multiple of $\big( \log n \big)^{D/2 -1/3}$, alongside a simple union bound. 
\qed

{\bf Proof of Proposition \ref{p.noslender}.}
With the standard usage of superscript $t$ to indicate this event expressed via the time-$t$ weight $\weight_n^t$, and taking $\tza = \tau_0^\alpha$ (with $\tau_0$ specified in~(\ref{e.taucondition})), Proposition~\ref{p.slenderpolymer}  implies that
\begin{equation}\label{e.lsw}
\PP \Big( \neg \,  \mathsf{LowSlenderWeight}^t \big( \ell,  \tau_0^\alpha ,1-\chi;  \rho_n^0 \big) \Big) \leq \exp \big\{ - d_2 \big( \log n \big)^{q\alpha/2} \big\} 
\end{equation}
for any scale $\ell$ satisfying $2^\ell \leq n^{1-\beta}$.
The lower bound on $\tza^{-1/4}$ hypothesised in Proposition~\ref{p.slenderpolymer} is valid for high $n$ since $ \alpha q > 4$ for such $n$. Note that the bound $n^\beta \geq \tza^{-40}$, valid for $n$ at least a $\beta$-determined level, serves to ensure that $2^\ell \leq n \tza^{40}$.

The occurrence of $\mathsf{LowSlenderWeight}^t \big( \ell, \tau_0^\alpha,1-\chi;  \rho_n^0 \big)$ entails that any $(\alpha,1-\chi)$-slender excursion of scale~$\ell$ between $\rho_n^0$
and $\rho_n^t$ that begins at $(x,s_1)$ and ends at $(y,s_2)$ satisfies  
\begin{equation}\label{e.tweightbound}
  \tot^{-1/3} \weight^t_n \big[  (x,s_1) \to (y,s_2)  \big] \leq - \redconst \big( \log n \big)^{q\alpha} \, .
 \end{equation}
  Since the points $(x,s_1)$ and $(y,s_2)$ lie on the polymer $\rho_n^t$, we see that an application of Lemma~\ref{l.onepoint}  with $D = q\alpha$  to  the time-$t$ noise field  yields that the probability of the occurrence of the bound~(\ref{e.tweightbound}) is at most $H \exp \big\{ - h \big( \log n \big)^{({3q\alpha/2})-1} \big\}$. Thus, the probability of the existence of an  $(\alpha,1-\chi)$-slender excursion of scale~$\ell$ between $\rho_n^0$
and $\rho_n^t$ is at most 
$$
 H \exp \big\{ - h  \big( \log n \big)^{(3q\alpha/2)-1} \big\} \, + \,  \exp \big\{ - d_2 \big( \log n \big)^{q\alpha/2} \big\} \, ,
$$
 where the two terms provide an upper bound on the probability of the said existence 
occurring in tandem with, or alongside the complement of,
 the event  $\mathsf{LowSlenderWeight}^t \big( \ell, \alpha,1-\chi;  \rho_n^0 \big)$. We sum the obtained bound over scales $\ell \in \N$ that satisfy $2^\ell \leq n^{1 - \beta}$ to learn that, in the notation used in Proposition~\ref{p.noslender},
 $$
  \PP \Big( \mathsf{Slender} \big(\alpha,\beta, 1 - \chi \big) \Big) \leq (1 - \beta) \big(  \log_2 n + 2 \big) 
 \cdot 2 \exp \big\{ - d_2 \big( \log n \big)^{q\alpha/2} \big\} \, .
 $$ 
 Noting that $2\big(  \log_2 n + 2 \big) 
 \leq n$ for $n \geq 12$ completes the proof of this proposition.  \qed 
 
 {\bf Remark.} The hypothesis that 
 the exponent $\beta$ be positive (which stipulates a lower bound on the demarcation between short and long excursions), is invoked in the just given argument to secure the bound $n^\beta \geq \tza^{-40}$
 that permits us to derive~(\ref{e.lsw}). Since $\tza = (\log n)^{-q \alpha}$, a sub-power-law growth rate in $n$ might replace $n^\beta$. That is, the division between long and short excursion 
 has been set at $n^{\beta-1}$ for $\beta > 0$ merely for notational convenience; if we wished, rather shorter excursions might be treated as long.

\section{The case of short excursions: deriving Proposition~\ref{p.shortnonthinexc}}\label{s.shortexc}

%
%
%The principal result in the analysis of the case where short excursions dominate is Proposition~\ref{p.shortnonthinexc}, and here we prove this result.  

The value $t\in [0, n^{-1/3}]$ will be fixed throughout this section.
Recall that an excursion is short when its duration is at most $n^{\beta - 1}$.
We will work with a further positive parameter $\lambda$ and will impose  conditions on the parameter pair as they are needed; the bounds $\beta + \lambda < 1/6$ and $\lambda > \beta$ will however  imply these conditions. Recall also that Proposition~\ref{p.shortnonthinexc} makes an assertion in terms of two further parameters $\beta_1 > 0$ and $\beta_2 \in (0,1)$, whose values are determined by Corollary~\ref{c.notallcliffs}. 
Revising the definition made in~(\ref{e.deficit}), we specify the {\em deficit} $D(h)$ at level $h \in n^{-1}\llbracket 0,n \rrbracket$ so that 
$$
 D(h) = \inf \Big\{ Z_n\big( \rho_n^0(h) , h \big) - Z_n(x,h): x \in \R \, , \, \big\vert x - \rho_n^0(h) \big\vert \in \big[ 8^{-1}\beta_1 n^{-2/3} , n^{-2/3 +2\beta/3} (\log n)^{1/3}   \big] \Big\} \, . 
$$
\begin{lemma}\label{l.deficit}
Let $g \in (0,1/2)$ and $\lambda \in (0,1)$. There exist $n_0 \in \N$ and positive $H$ and $h$ such that, for $n \geq n_0$, it is with probability at most  $\exp \big\{ - \dmac (\log n)^{1/68} \big\}$ that
$$
  \Big\vert  \Big\{ i \in \llbracket 0, n \rrbracket:  g \leq i/n \leq 1 -g \, , \, D(i/n) < n^{-1/3 - \lambda} \Big\} \Big\vert \geq 
  n^{1-\lambda}  \exp \big\{ H  (\log n)^{11/12} \big\} 
   \, .
$$
\end{lemma}
{\bf Proof.} 
By Proposition~\ref{p.maxfluc} with ${\bf R} = (\log n)^{1/228}$, ${\bf n} = n$ and $\mathbf{s_{1,2}} = 1$,
$$
 \PP \Big( \sup \big\{ \vert \rho_n(a) \vert : a \in n^{-1}\Z \cap (0,1) \big\} > ( \log n )^{1/228}  \Big) \leq \exp \big\{ - \dmacd (\log n)^{1/68} \big\} \, .
 $$

Let $a \in n^{-1}\Z \cap (g,1-g)$.
By Theorem~\ref{t.nearmax} with ${\bf R} = 0$, ${\bm \e} = 8^{-1}\beta_1 n^{-2/3}$, ${\bm \hata} = 2^{1/2} \beta_1^{-1/2} n^{-\lambda}$, ${\bm \ell'} = 3 n^{-2/3 +2\beta/3} (\log n)^{1/3}$ and  ${\bm \ell} = 3 ( \log n )^{1/228}$, we find that, for $n$ high enough,
\begin{eqnarray*}
& & \PP \Big( D(a) \leq n^{-1/3 - \lambda} \, , \, \vert \rho_n(a) \vert \leq ( \log n )^{1/228}  \Big) \\
&\leq & \PP \Big( \vert \rho_n(a) \vert \leq ( \log n )^{1/228}    \, , \, \sup_{x \in \R: x - \rho_n(a)  \in [\e,\ell'/3]} \big( Z_n(x,a) + \hata  (x - M)^{1/2} \big) \geq Z_n(M,a)  \Big) \\
 & \leq & \log \big( {\bm \ell'} \e^{-1} \big) \max \Big\{ \hata \cdot 
 \exp \big\{   H {\bm \ell}^{\macroseventeen} \big( 1  + \log \hata^{-1}  \big)^{5/6} \big\} ,   \exp \big\{ - h n^{1/12}   \big\} \Big\} \\
 & \leq & \log n \cdot n^{-\lambda} \exp \big\{ H^2 \lambda^{5/6}   (\log n)^{11/12} \big\} \, ,
\end{eqnarray*} 
where the final inequality is due to a suitable increase if need be in the value of $H > 0$.

On the event $\sup \big\{ \vert \rho_n(a) \vert : a \in n^{-1}\Z \cap (0,1) \big\} \leq 3 ( \log n )^{1/228}$, the conditional mean    number of indices $i \in \llbracket 0,n \rrbracket$, $g \leq i/n \leq 1-g$,  such that $D(i/n) \leq n^{-1/3 - \lambda}$ is thus at most 
$$
n  \log n \cdot n^{-\lambda} \exp \big\{ H^2 \lambda^{5/6}   (\log n)^{11/12} \big\} 
 \, . 
$$
Markov's inequality and a relabelling of $H > 0$ now yield Lemma~\ref{l.deficit}. \qed

To apply this lemma,  we need to show the rarity of wide excursions, as we also needed to do in the long excursions' case. The width of an excursion $E$ between $\rho_n^0$ and $\rho_n^t$ equals the infimum of the widths $x_2 - x_1$ of vertical strips $[x_1,x_2] \times \R$ that contain $E$.
An excursion is called {\em wide} if its width exceeds  $H n^{-2/3 + 2\beta/3} \big( \log n \big)^{1/3}$. 
Let $\nowideexc$ denote the event that there exists no wide excursion between $\rho_n^0$ and~$\rho_n^t$ whose duration is at most $n^{\beta - 1}$. 
\begin{lemma}\label{l.nofatexcursion}
Suppose that $\beta \in (0,1/2]$ and  $H > 0$.
There exists $d > 0$ such that $\PP \big( \nowideexc \big)$ is at least $1 - n^{-d H^3}$.
\end{lemma}
{\bf Proof.} The width of any excursion of duration at most $n^{\beta - 1}$ is at most $\alpha(0) + \alpha(t)$, where $\alpha(t')$ for $t' \in \{ 0,t \}$ denotes the supremum of $\vert y - x \vert$
over $(x,s_1),(y,s_2) \in \rho_n^{t'} \cap \big( \R \times (n^{-1} \Z \cap [0,1]) \big)$
such that $\tot \in [0,n^{\beta - 1}]$. When $\rho_n^0$ and $\rho_n^{t'}$ are $\big( n^{\beta - 1},R \big)$-regular in the sense of Definition~\ref{d.staticpolymer}, the width of such an excusion is thus at most $2R  (1-\beta)^{1/3} n^{-2/3 + 2\beta/3} \big( \log n \big)^{1/3}$. Proposition~\ref{p.regular} with ${\bm \kappa} = n^{\beta - 1}$ and ${\bf R} = H/2$ then implies the lemma. \qed

The value $y \in n^{-1}\Z \cap [2^{-4}\beta_2 , 1-2^{-4}\beta_2]$ is said to be of {\em low deficit}
if  $D(y) < n^{-1/3 - \lambda}$.  An excursion is called {\em unlucky} if it intersects a horizontal line whose coordinate $y$ is of low deficit and satisfies $\big\vert \rho_n^t(y) - \rho_n^0(y) \big\vert \in \big[ 8^{-1}\beta_1 n^{-2/3} , n^{-2/3 +2\beta/3} (\log n)^{1/3}   \big]$. 
Set
$$ 
\manyunlucky = \Big\{ \textrm{the number of unlucky excursions exceeds $n^{1 - \lambda} \exp \big\{ H (\log n)^{11/12} \big\}$}\Big\}  \, .
$$

\begin{lemma}\label{l.manyunlucky} 
For $n$ high enough, $\PP \big( \manyunlucky \big) \leq \exp \big\{ -h \big(\log n\big)^{1/68} \big\}$.
\end{lemma}
{\bf Proof.}
The number of unlucky excursions is bounded above by the number of indices of low deficit; thus, the result follows from Lemma~\ref{l.deficit}. \qed

An excursion $E$ of lifetime $[b,f]$ is called {\em normal} if  there exists a value $y \in n^{-1}\Z \cap [b,f-n^{-1}]$ that is not of low deficit and that satisfies $y \in n^{-1}\Z \cap (2^{-4}\beta_2 , 1- 2^{-4}\beta_2]$ and $\big\vert \rho_n^t(y) - \rho_n^0(y) \big\vert \in \big[ 8^{-1}\beta_1 n^{-2/3} , n^{-2/3 +2\beta/3} (\log n)^{1/3}   \big]$.

\begin{lemma}\label{l.shortnonthinimplies}
Suppose that $\lambda > \beta$.
When the event $\shortnonthinexc(\beta_2/8) \cap \nowideexc \cap \neg \, \manyunlucky$ occurs, and $n$ is high enough, the number of normal short excursions is 
at least $2^{-5}\beta_2 n^{1- \beta}$.
\end{lemma}
{\bf Proof.} Note first that  an excursion
whose lifetime is contained in  $(2^{-4}\beta_2,1-2^{-4}\beta_2]$ that is neither thin, nor wide, nor unlucky, is normal.

 Let $\mc{C}$ denote the set of short excursions that are neither thin nor wide  and whose lifetimes are contained in $(2^{-5}\beta_2,1-2^{-5}\beta_2]$. We {\em claim} that, when  $\shortnonthinexc(\beta_2/8) \cap \nowideexc$ occurs, and $n$ is high enough, the bound $\vert \mc{C} \vert \geq 2^{-4} \beta_2 n^{1-\beta}$ holds. 

To verify this claim,  first let $\mc{C}_0$ denote the set counterpart to~$\mc{C}$ for which 
`are contained in' is replaced by `intersect'. 
Note that, on the event with which the claim is concerned,
the sum of the durations of short excursions that are neither thin nor wide is at least $\beta_2/16$.  
The sum of the durations of elements of $\mc{C}_0$ is thus at least $2^{-4}\beta_2$.
Since a short excursion has duration at most~$n^{\beta - 1}$, we see that $\vert \mc{C}_0 \vert \geq 2^{-4}\beta_2 n^{1-\beta}$. Since $\vert \mc{C} \vert \geq \vert \mc{C}_0 \vert - 2$, the claim follows for $n$ high enough.

Note further that, since $\lambda > \beta$, the number of unlucky excursions in the event $\neg \, \manyunlucky$ is at most $2^{-6}\beta_2 n^{1-\beta}$, provided that $n$ is high enough. From this, the lemma follows. \qed

\begin{lemma}\label{l.normalshortexc}
Suppose that $\beta + \lambda < 1/6$. There exists $D > 0$ such that, for $n$ high enough,
  $$
 \PP \Big( \textrm{there are at least  $2^{-5}\beta_2 n^{1-\beta}$ normal short excursions} \Big) \leq D \exp \big\{ - c n^{1 - 3(\beta + \lambda)/2} \big\}  \, .
  $$
\end{lemma}
{\bf Proof.}
Suppose that there are at least  $2^{-5}\beta_2 n^{1-\beta}$ normal short excursions, and denote them by~$E_i$, with the index rising from $i=1$. 
If $E_i$ has lifetime $[b_i,f_i]$, denote by $y_i$ the value in $n^{-1}\Z \cap [b_i,f_i)$ that ensures that $E_i$ is a normal excursion.
Applying Lemma~\ref{l.additive} with this choice of values $y_i$,   we find that
$$
 \weight^0(\rho_n^0) - \weight^0(\rho_n^t) \, \geq \, 2^{-5}\beta_2 n^{1-\beta} \cdot n^{-1/3 - \lambda} \, = \, 2^{-5}\beta_2 n^{2/3 - \beta - \lambda} \, ,
$$
since none of the values $y_i$ is of low deficit.
Let $t \in \big[0,n^{-1/3}\big]$ be given. By the crude form of dynamical stability for weight offered by  Lemma~\ref{l.crudebound}, we know that it is with probability at least $1 - e^{-hn}$ that 
$$
  \sup  \Big\vert \weight^0 (\rho_n^t) - \weight^t(\rho_n^t) \Big\vert \leq 4 n^{1/2} \, .
$$
Since $\beta + \lambda < 1/6$, we find then that, when $n$ is high enough,
 \begin{eqnarray*}
& &
 \PP \Big( \textrm{there are at least  $2^{-5}\beta_2 n^{1-\beta}$ normal short excursions} \Big) \\
 & \leq & \PP \Big( 
 \weight^0(\rho_n^0) - \weight^t(\rho_n^t) \geq  2^{-6}\beta_2 n^{2/3 - \beta - \lambda} \Big) \, + \, e^{-hn} \, .
 \end{eqnarray*}
 The random variables $\weight^0(\rho_n^0)$ and $\weight^t(\rho_n^t)$ share the distribution of $\weight_n \big[ (0,0) \to (0,n) \big]$.
 Writing
 \begin{eqnarray*}
& &  \PP \Big( 
 \weight^0(\rho_n^0) - \weight^t(\rho_n^t) \geq 2^{-6}\beta_2 n^{2/3 - \beta - \lambda} \Big) \\
 & \leq & \PP \Big( 
 \weight^0(\rho_n^0)  \geq 2^{-7}\beta_2  n^{2/3 - \beta - \lambda} \Big) + \PP \Big( 
 \weight^t(\rho_n^t) \leq - 2^{-7}\beta_2 n^{2/3 - \beta - \lambda} \Big) \, , 
 \end{eqnarray*}
 we may apply Lemma~\ref{l.onepointbounds} to learn that the displayed left-hand side is at most $2C \exp \big\{ -c' n^{1 - 3(\beta + \lambda)/2} \big\}$ with $c' = c \, 2^{-21/2} \beta_2^{3/2}$. Lemma~\ref{l.normalshortexc} follows by a suitable choice of $D > 0$. \qed

By Lemma~\ref{l.shortnonthinimplies},
\begin{eqnarray*}
 \PP \Big( \shortnonthinexc\big(\beta_2/8\big) \Big) & \leq &   
 \PP \big( \neg \,  \nowideexc \big) + \PP \big( \manyunlucky \big) \\
  & & \, \,   + \, \, \PP \Big( \textrm{the number of normal short excursions is 
at least $2^{-5}\beta_2 n^{1- \beta}$} \Big) \, . 
\end{eqnarray*}
By Lemmas~\ref{l.nofatexcursion},~\ref{l.manyunlucky} and~\ref{l.normalshortexc}, 
  $$
 \PP \Big( \shortnonthinexc\big(\beta_2/8\big) \Big) \leq 
 n^{-d H^3} +  \exp \big\{ -h \big(\log n\big)^{1/68} \big\} + 
 D \exp \big\{ - c n^{1 - 3(\beta + \lambda)/2} \big\} \, .
  $$
Since $\beta + \lambda < 2/3$,   
  this bound implies Proposition~\ref{p.shortnonthinexc} 
  with a suitable decrease to the value of $\dmac > 0$. \qed

\appendix

\section{Hypothesis verification for Proposition~\ref{p.stablehorizontal}}\label{s.calcder}

Here we demonstrate that the 
 hypotheses of Proposition~\ref{p.stablehorizontal} are adequate to permit the applications of results during its proof.
There were two such applications: the use of Proposition~\ref{p.onepoint}(2) to obtain (\ref{e.claimone}); and the use of Proposition~\ref{p.fluc} to obtain~(\ref{e.claimtwo}). 

{\em Derivation of~(\ref{e.claimone}).} Proposition~\ref{p.onepoint}(2) is applied with ${\bf x} = x'$ and ${\bf y} = y'$. The proposition's hypothesis that $y - x = y' - x'$ is in absolute value at most $2^{-1}(n\tot)^{1/3}$
holds because $x'$ and $y'$ belong to the respective meshes $I_\eta$ and $J_\eta$, and thus to the intervals $I$ and $J$, each of which has length $D 2^{-2\ell/3} (\ell + 1)^{1/3}$, and whose left endpoints are displaced by at most~$c  (n \tot)^{1/9}$; and the thus needed bound, 
$$
2^{-3}3^{-1} c  (n \tot)^{1/18} + D 2^{-2\ell/3} (\ell + 1)^{1/3} \leq 2^{-1} (n \tot)^{1/3} \, .
$$
 This bound follows from $\tot \geq 2^{-1-\ell}$, $2^{\ell + 1} \leq n$, $c \leq 1$ and $n \geq 10^{29} D^{18}$. It is hypothesised in Proposition~\ref{p.stablehorizontal} that $n \geq 10^{29} D^{18} c^{-9} 2^\ell (\ell +1)^{18} \big(\log \taumac^{-1} \big)^9$; since $D \geq 1$, $c \leq 1$ and $\taumac \leq e^{-1}$, we indeed have $n \geq 2^{\ell +1}$.

{\em Derivation of~(\ref{e.claimtwo}).} Recall that the application of  Proposition~\ref{p.fluc} to obtain~(\ref{e.claimtwo}) started with the following  hypotheses on the parameters $n \in \N$, $\eta > 0$, $s_1,s_2 \in n^{-1}\Z$,  $x,y \in \R$  and $K > 0$: that $\eta \tot^{-2/3} \in (0,2^{-4}]$; that $n \tot \geq 10^{32} c^{-18}$; that $D 2^{-2\ell/3} (\ell + 1)^{1/3} \tot^{-2/3} \leq 2^{-2} 3^{-1} \rsc  (n \tot)^{1/18}$; and that
 $K \in \big[10^4 \, , \,   10^3 (n \tot)^{1/18} \big]$.

Thus the application requires the hypotheses
$$
[1] \, \, \eta \tot^{-2/3} \in (0,2^{-4}]
$$ 
$$
[2] \, \, n \tot \geq 10^{32} c^{-18}
$$
$$
[3] \, \,  2^{-3}3^{-1} c  (n \tot)^{1/18} + D 2^{-2\ell/3} (\ell + 1)^{1/3} \tot^{-2/3} \leq 2^{-2} 3^{-1} \rsc  (n \tot)^{1/18}
$$
$$
[4] \, \,   (\ell + 1) c^{-1/2} 2^{12} (127/250)^{1/2} \big( \log \taumac^{-1}\big)^{1/2}   \geq 10^4 
$$
$$
[5] \, \,   (\ell + 1) c^{-1/2} 2^{12} (127/250)^{1/2} \big( \log \taumac^{-1}\big)^{1/2}  \leq 10^3 (n \tot)^{1/18} \, ,
$$
where to obtain the form of [4] and [5], we used  
$K =   (\ell + 1) c^{-1/2} 2^{12} (127/250)^{1/2} \big( \log \taumac^{-1}\big)^{1/2}$.

We also use $\eta \leq v-u$, i.e.,
$$
[6] \, \, \eta \leq  D 2^{-2\ell/3} (\ell + 1)^{1/3} \, .
$$
 That [1] holds is explained in the proof of Proposition~\ref{p.stablehorizontal}; recall that we set  $\eta = 2^{-5} 2^{-2\ell/3}\phi$ with $\phi = \taumac^{1/10}$, so that [1] follows from $\taumac \leq 1$.

Since $\tot \geq 2^{-\ell -1}$, [2] is implied by  
$$
[7] \, \, n  \geq 10^{32} c^{-18} 2^{\ell + 1}
$$

Since $\tot \geq 2^{-\ell -1}$, [3] is implied by 
$$
 \, \, D 2^{-2\ell/3} (\ell + 1)^{1/3} ( 2^{-\ell-1})^{-2/3} \leq 2^{-3} 3^{-1} \rsc  (n 2^{-\ell-1})^{1/18}
$$
or
$$
 \, \, D  (\ell + 1)^{1/3} 2^{\ell/18} 2^{3 + 2/3 + 1/18} 3 \rsc^{-1} \leq n^{1/18}
$$
or
$$
 \, \, n \geq   2^{67} 3^{18} D^{18}  (\ell + 1)^6 2^\ell \rsc^{-18} \, .
$$
which is implied by 
$$
[8] \, \, n \geq   2^{67} 3^{18} D^{18}  (\ell +1)^6 2^\ell \rsc^{-18} \, .
$$
Introducing 
$$
 [9] \, \, D \geq 1  \, ,
$$
and
$$
[10] \, \, n \geq   10^{29} D^{18}  (\ell +1)^6 2^\ell \rsc^{-18} \, , 
$$
note that $[9,10] \to [7,8]$. (This notation means `[9] and [10] imply [7] and [8]'.) Note also that $[7,8] \to [2,3]$; and that $[1,9] \to [6]$, since $\tot \leq 2^{-\ell}$.

Note that [4] is implied by
$
 c^{-1/2} 2^{12} (127/250)^{1/2} \big( \log \taumac^{-1}\big)^{1/2}   \geq 10^4 
 $
 or
$
  \log \taumac^{-1}   \geq  c^{1/2} 10^8  2^{-24} (9/20)^{-1} \, .
 $ 
Since $10^8  2^{-24} (127/250)^{-1} =  11.733 \cdots$ and $c \leq 1$, the last is implied by 
$$
 [11] \, \, \taumac   <  \exp \big\{ -12 \big\} \, .
 $$ 
Thus $[11] \to [4]$.
Since $\tot \geq 2^{-\ell-1}$, [5] is implied by
$$
 \, \,   (\ell + 1) 2^{\ell/18} c^{-1/2} 2^{12 + 1/18} (127/250)^{1/2} \big( \log \taumac^{-1}\big)^{1/2}  \leq 10^3 n^{1/18}
$$
or
$$
 \, \,  n \geq  (\ell + 1)^{18} 2^\ell c^{-9} 10^{-48} 2^{217} (127/250)^9 \big( \log \taumac^{-1}\big)^9  \, .
$$ 
Since  $10^{-48} 2^{217} (127/250)^9 = 4.745 \cdots \times 10^{14}$, the last is implied by
$$
 [12] \, \,  n \geq   10^{15}  c^{-9} (\ell + 1)^{18} 2^\ell \big( \log \taumac^{-1}\big)^9  \, .
$$ 
Let
$$
 [13] \, \,  n \geq   10^{29} D^{18}  c^{-9} (\ell + 1)^{18} 2^\ell \big( \log \taumac^{-1}\big)^9 \, .  
$$
Note that $[9,11,13] \to [10,12]$ since $[11]$ implies that $\taumac \leq e^{-1}$.  Thus, $[1,9,11,13] \to [2,3,4,5,6]$. Recall that $[1]$ has been confirmed, and that $[1,2,3,4,5,6]$ is the set of conditions that permit the application of  Proposition~\ref{p.fluc} in the proof of Proposition~\ref{p.stablehorizontal}. The other application in this proof is that of  Proposition~\ref{p.onepoint}(2), for which
 we have seen that the condition $n \geq 2^4 D^3$ is sufficient. But this is implied by $[11,13]$, since $c \leq 1$.
 
{\em Conclusion.} In summary, we have shown that the hypotheses $[9,11,13]$ are adequate to derive Proposition~\ref{p.stablehorizontal}. Since these conditions are the hypotheses of Proposition~\ref{p.stablehorizontal} if we take $a = e^{-12}$, we have completed hypothesis verification for Proposition~\ref{p.stablehorizontal}.

\bibliographystyle{plain}

\bibliography{airy}

\end{document}